\documentclass[a4paper,10pt]{amsart} 
 
\usepackage{fancybox} 
\usepackage{pifont} 

\usepackage[symbol]{footmisc}

\usepackage{amsmath} 
\usepackage[applemac]{inputenc}
\usepackage{amsfonts}
\usepackage{amssymb}
\usepackage{amsthm}
\usepackage{stmaryrd}
\usepackage{color}
\usepackage{mathrsfs} 

\newcommand{\mysection}{\setcounter{equation}{0} \section}

\renewcommand{\d}{\mathbf{d}}
\newcommand{\F}{\mathcal{F}}  
\renewcommand{\P}{\mathbb{P}}

\newcommand{\E}{\mathbb{E}}

\newcommand{\N}{\mathbb{N}}  
  
\newcommand{\T}{\mathbb{T}}

\newcommand{\R}{\mathbb{R}}  
\newcommand{\I}{\mathbb{I}}

\newtheorem{THM}{Theorem}   
\newtheorem{REM}{Remark}

\newtheorem{PROP}[THM]{Proposition}
\newtheorem{lem}[THM]{Lemma}

\def\1{\mbox{1\hspace{-0.25em}l}}

\def\gH{{{\mathbf H}}}

\newcommand \A[1]{{\bf (#1)}}

\def\leftB{[\![}
\def\rightB{]\!]}
\def\btheta{{\boldsymbol{\theta}}}

\def\bxi{{\boldsymbol{\xi}}}
\def\bzeta{{\boldsymbol{\zeta}}}

\def\det{{{\rm det}}}
\def\X{{\mathbf{X}}}

\def\x{{\mathbf{x}}}

\def\y{{\mathbf{y}}}
\def\z{{\mathbf{z}}}

\def\m{{\mathbf{m}}}

\def\0{{\mathbf{0}}}

\def\gF{{\mathbf{F}}}

\def\K{{\mathbf{K}}}
\def\gR{{\mathbf{R}}}

\title[Schauder Estimates for Degenerate Kolmogorov Equations]{\textbf{Sharp Schauder Estimates for some Degenerate Kolmogorov Equations}}
\author{Paul-\'Eric Chaudru de Raynal,  Igor Honor\'e  and St\'ephane Menozzi}

\begin{document}
\maketitle

\begin{abstract}
We provide here some sharp Schauder estimates for degenerate PDEs of Kolmogorov type when the coefficients lie in some suitable anisotropic H\"older spaces and the first order term is  non-linear and  unbounded. We proceed through a perturbative approach based on forward parametrix expansions.

Due to the low regularizing properties of the degenerate variables, for the procedure to work, we heavily exploit duality results between appropriate Besov spaces.

Our method can be seen as \textit{constructive} and provides, even in the non-degenerate case, an alternative approach  to Schauder estimates.

\end{abstract}



\mysection{Introduction and Main Results}
\subsection{\textcolor{black}{Statement of the problem}}
\textcolor{black}{F}or a fixed time horizon $T>0$ and given integers $n,d\in \N$, \textcolor{black}{we aim at proving} Schauder estimates for degenerate scalar valued Kolmogorov PDEs of the form:
\begin{eqnarray}
\label{KOLMO}
\begin{cases}
\partial_t u(t,\x)+ \langle \gF(t,\x), {\mathbf D}u(t,\x)\rangle +\frac 12{\rm Tr}\big( D_{\x_1}^2u(t,\x)a(t,\x)\big)=-f(t,\x),\ (t,\x)\in [0,T)\times \R^{nd},\\
 u(T,\x)=g(\x),\ \x\in \R^{nd},
 \end{cases}
 \end{eqnarray}
\textcolor{black}{where} $\x=(\x_1,\cdots, \x_n)\in \R^{nd} $ \textcolor{black}{and} for each $i\in \leftB 1,n\rightB $, $\x_i\in \R^d $. \textcolor{black}{Above,} the source $f$ and the terminal condition $g$ are \textcolor{black}{bounded and scalar value\textcolor{black}{d} mappings} \textcolor{black}{and t}he \textcolor{black}{bounded} diffusion matrix $a$ is $\R^d\otimes \R^d $-valued.
Also,  $\gF(t,\x):=(\gF_1(t,\x),\cdots ,\gF_n(t,\x))$ is \textcolor{black}{a} vector of $\R^d $-valued \textcolor{black}{unbounded} mappings $\gF_i$ which have, for $i \in \leftB 2,n\rightB $, the following structure:
\begin{equation}
\label{STRUCT_F}
\forall  
(t,\x)\in[0,T]\times \R^{nd},\ \gF_i(t,\x):=\gF_i(t,\x^{i-1:n}),\ \x^{i-1:n}:=(\x_{i-1},\cdots,\x_n).
\end{equation}
The notation ${\mathbf D}=(D_{\x_1}, \cdots, D_{\x_n}) 
$  stands for the full spatial gradient and 
$D_{\x_i} $ denotes the partial gradient w.r.t. to $\x_i $. \textcolor{black}{For notational convenience we will denote the spatial operator in \eqref{KOLMO} by $(L_t)_{t\in [0,T]} $, i.e. for any $\varphi \in C_0^2(\R^{nd},\R) $ (space of twice continuously differentiable functions with compact support):
\begin{equation}
\label{SPATIAL_OP}
L_t \varphi(\x)=\langle \gF(t,\x), {\mathbf D}\varphi(\x)\rangle +\frac 12{\rm Tr}\big( D_{\x_1}^2\varphi(\x)a(t,\x)\big).
\end{equation}
In this work, this operator is supposed to \textcolor{black}{be} of weak H\"ormander type, i.e. we suppose that \textcolor{black}{t}he diffusion matrix $a$ is uniformly elliptic and that the matrices $\big(D_{{\x}_{i-1}}\gF_i(t,\cdot)\big)_{i\in \leftB 2,n\rightB} $ have full rank.
}
\\

\textcolor{black}{\textbf{Acknowledgments:}
For the first Author, this work has been partially supported by the ANR project ANR-15-IDEX-02. For the third author, 
the study has been funded by the Grant RSF 17-11-01098 from the  Russian Science Foundation.}
\\

\subsection{\textcolor{black}{Preliminaries}}
\textcolor{black}{
Under suitable regularity assumptions on $a,\gF$ it can be shown that the martingale problem associated with \eqref{SPATIAL_OP} is well posed, see e.g.  \cite{meno:10}, \cite{meno:17}, \cite{chau:meno:17}. In that case, there exists a unique weak solution to the stochastic differential equation
\begin{equation}
\label{SYST}
\begin{array}{l}
\displaystyle d\X_t^1 = \gF_1(t,\X_t^1,\dots,\X_t^n) dt + \sigma(t,\X_{t}^1,\dots,\X_{t}^n) dW_t,
\\
\displaystyle d\X_t^2 = \gF_2(t,\X_t^1,\dots,\X_t^n) dt,
\\
\displaystyle d\X_t^3 = \gF_3(t,\X_t^2,\dots,\X_t^n) dt,
\\
\vdots
\\
\displaystyle d\X_t^n = \gF_n(t,\X_t^{n-1},\X_t^n) dt,
\end{array}
\quad t \geq 0
,
\end{equation}
where $(W_t)_{t\ge 0} $ is a Brownian motion on some filtered probability space  $(\Omega,(\F_t)_{t\ge 0},\P) $.
The operator $(L_t)_{t\ge 0}$ then corresponds to the generator of the process in \eqref{SYST} \textcolor{black}{ where $\sigma$ is a square root of $a$}. The well posedness of the martingale problem in particular implies that \eqref{KOLMO} admits a  solution in the mild sense on a suitable function space (see e.g. Kolokoltsov \cite{kolo:11}).\\
}

\textcolor{black}{Starting from this framework and assuming that the coefficients $a,\gF $, the source $f$ as well as the terminal  condition $g$ lie in appropriate  H\"older spaces, we here prove that the Cauchy problem \eqref{KOLMO} is well posed in the weak sense and that its unique weak solution satisfies some appropriate Schauder estimates.\\}

\textcolor{black}{
These Schauder estimates allow in particular to quantify precisely the \textit{parabolic bootstrap} associated with $(L_t)_{t\ge 0}$, hence emphasizing its intrinsic regularization properties. In \cite{chau:meno:17}, the \textcolor{black}{a}uthors show how in a H\"older framework for the coefficients $ \gF,\sigma$, some minimal thresholds allow to guarantee the well posedness of the martingale problem associated with $(L_t)_{t\ge 0}$ and how such thresholds depend on the \textcolor{black}{level} $i\in \leftB 1,n\rightB$ of the chain and the variable $j\in \leftB (i-1)\vee 1,n\rightB$. Adding to these thresholds some suitable regularity (depending also on the \textcolor{black}{level} and on the variable) \textcolor{black}{allows to derive Schauder estimates. They  provide accurate quantitative bounds associated with the global smoothing effect of $(L_t)_{t\ge 0}$ and reflect how the additional regularity propagates  through the operator.} From another point of view, which relies on the associated stochastic system \eqref{SYST}, such estimates quantify the \textit{regularizing} properties of the Brownian motion when propagating through the aforementioned system. In any case, they \textcolor{black}{underline} the degenerate structure of the operator $(L_t)_{t\ge 0}$ (or equivalently of the system \eqref{SYST}) which lead\textcolor{black}{s} to more tricky smoothing effects than th\textcolor{black}{ose} exhibited in the non-degenerate setting.
}

It is indeed known from the seminal work of Lunardi \cite{lun:97} on the topic that Schauder estimates for degenerate Kolmogorov equations differ from those in the usual non-degenerate setting (see e.g. \cite{kryl:96} or \cite{frie:64}). They reflect in some sense the multiple scales in the systems \eqref{KOLMO} and \eqref{SYST} (see Section \ref{SEC_MC_INTRO} below) and are stated in terms of anisotropic H\"older spaces. In particular, those spaces emphasize that the higher is the index of the considered variable in $\leftB 1,n\rightB $, the weaker is the associated regularity gain.\\

\textbf{Mathematical background.} Let us shortly describe particular cases of dynamics of type \eqref{KOLMO} for which some Schauder estimates have already been proved. We again mention the \textcolor{black}{article} by Lunardi \cite{lun:97}, who considered the special case of a homogeneous linear drift $\gF(\x)={\mathbf A}\x $ satisfying the \textcolor{black}{structure} condition \eqref{STRUCT_F}.   Precisely, the matrix ${\mathbf A}$ writes in this case:
\begin{equation*}{\mathbf A}=\left (\begin{array}{ccccc}{\mathbf a}_{1,1} & \cdots & \cdots &\cdots  & {\mathbf a}_{1,n}\\
{\mathbf a}_{2,1} & \cdots &\cdots &\cdots &{\mathbf a}_{2,n}\\
\0_{d,d} & {\mathbf a}_{3,2}&\cdots & \cdots &{\mathbf a}_{3,n}\\
\vdots &  \0_{d,d}                     & \ddots & \vdots\\
\0_{d,d} &\cdots &  \0_{d,d}   &  {\mathbf a}_{n,n-1}& {\mathbf a}_{n,n}
\end{array}\right) ,
\end{equation*}
where the entries $( {\mathbf a}_{i,j})_{ij\in \leftB 1,n\rightB^2} $ are in $\R^d\otimes \R^d  $ s.t. $\big({\mathbf a}_{i,i-1}\big)_{i\in \leftB 2,n\rightB} $  are non-degenerate elements of $\R^d\otimes \R^d $ (which expresses the weak H\"ormander condition)\footnote{Actually the non-zero entries $({\mathbf a}_{i,j}) $ of ${\mathbf A}$ can be non-square and simply have full rank in \cite{lun:97}. We restrict here to square matrices for the sake of simplicity.}. Also, the homogeneous diffusion coefficient $a$ belongs to an appropriate anisotropic H\"older space and asymptotically converges when $|\x|\rightarrow \infty $ to a non-degenerate constant matrix of $\R^{d}\otimes \R^d $. 
The assumptions on the asymptotic behavior on the diffusion coefficient have then been relaxed by Lorenzi \cite{lore:05:schau}, in the kinetic framework, i.e. $n=2$ with the notations of \eqref{SYST}, up to additional regularity assumptions on $a$ which could also be unbounded.

Priola established later in \cite{prio:06} Schauder estimates, without dimensional constraints  for time homogeneous drifts of the form 
\begin{equation}\label{PERTURB_DR_1}
\gF(\x)={\mathbf A}\x+\left(\begin{array}{c}\tilde \gF_1(\x)\\ {\mathbf 0}_{(n-1)d,d} \end{array} \right),
\end{equation}
for a  non-linear drift $\tilde \gF_1 $ acting on the non-degenerate variable in the expected anisotropic H\"older space. The underlying technique consisted in establishing bounds on the derivatives of the semi-group of the perturbed degenerate Ornstein-Uhlenbeck process (i.e. with $\tilde \gF_1$)  from the usual unperturbed one (with $({\mathbf A}\x)_1$ only) through the Girsanov theorem assuming first that $\tilde \gF_1$ is smooth. This initial smoothness  of $\tilde \gF_1 $ is required in order to compute the associated tangent flows. Through the continuity approach, the author then managed to obtain the estimates for a bounded variable diffusion coefficient lying in the natural \textit{expected} H\"older space similar to the one of \cite{lun:97}  with the same asymptotic conditions. The smoothness of $\tilde \gF_1 $ is relaxed through an approximation procedure viewing the difference between the H\"older drift and its mollification as a source term and exploiting the estimates established for the \textit{smooth} drift.

\textcolor{black}{Finally,} we mention the work of Di Francesco and Polidoro \cite{difr:poli:06} who derived Schauder estimates for a linear drift of the previous type  using an alternative notion of continuity regarding the diffusion coefficient $a$, which somehow involves the unbounded transport associated with the drift.\\

Hence, in the current framework of degenerate Kolmogorov equations, focusing on the drift,  the Schauder estimates hold, to the best of our knowledge, for either linear drifts or H\"older perturbations on the non-degenerate variable of a linear drift.\\

\textbf{Mathematical outline.} \textcolor{black}{There will be two main difficulties to overcome in order to prove} Schauder estimates in our framework: the degeneracy as well as the non-linearity and unboundedness of the drift. \textcolor{black}{Concerning this second issue let us also mention that, in the non-degenerate setting, Schauder estimates for unbounded non-linear drift coefficients  were obtained under mild smoothness assumptions by Krylov and Priola \cite{kryl:prio:10} \textcolor{black}{who heavily used} the flow associated with the first order vector field in $ L_t$, i.e. $\dot \btheta_{t}(\x)=\gF(t, \btheta_{t}(\x)) $, to precisely get rid of the unbounded terms.}

In this work we will prove, in the framework of H\"older spaces for the source $f$, the terminal condition $g$ and the coefficients $a,\gF $, Schauder estimates similar to those of the previously quoted works (\cite{lun:97}, \cite{lore:05:schau}, \cite{prio:06}). The diffusion coefficient $a$ and the source term $f$ will have, as in the non-degenerate case, the same regularity. We mention that, in contrast with the non-degenerate case, this will not be the case for the drift $\gF$ for which some additional smoothness on the degenerate entries $(\gF_i)_{i\in \leftB 2,n\rightB} $ is needed 
to guarantee the well posedness of \eqref{SYST}. In particular, the H\"older indexes of $\gF$ will be above the minimal thresholds appearing in \cite{chau:meno:17}. The flow associated with the drift term will again play a key role in our setting. \textcolor{black}{Eventually, we do not impose any particular spatial asymptotic condition on the diffusion $a$}.\\

\textcolor{black}{To prove our result,} we will here proceed through a perturbative approach. The idea is to perform a first order \textit{parametrix} expansion (or Duhamel expansion) of a solution of \eqref{KOLMO} with mollified coefficients around a suitable linearized Ornstein-Uhlenbeck type semi-group. \textcolor{black}{The main idea behind consists in \textcolor{black}{exploiting} this easier framework \textcolor{black}{in order to subsequently obtain a tractable} control on the error expansion. When applying such a strategy, the common way to proceed consists in adopting the so-called \textit{backward parametrix approach}, as successfully considered by Il'in \textit{et al.} \cite{ilin:kala:olei:62}, Friedman \cite{frie:64} or McKean and Singer \cite{mcke:sing:67} in the non degenerate setting. This techni\textcolor{black}{que} has been  extended \textcolor{black}{to} the current degenerate setting, which involves unbounded coefficients, and successfully exploited for handling the corresponding martingale problem or density estimates of the fundamental solution of \eqref{KOLMO} in \cite{chau:meno:17} and \cite{dela:meno:10}. }

Unfortunately, this approach \textcolor{black}{does not seem very adapted to} derive Schauder estimates since it does not allow to \textcolor{black}{easily deal}  with gradient and sensitivity estimates associated with the degenerate directions.

\textcolor{black}{We therefore adopt here an alternative  \textit{forward parametrix approach} which is better designed to deal with gradient estimates. It has indeed been successfully applied e.g. in \cite{chau:17}, \cite{chau:16} to derive strong uniqueness for degenerate kinetic SDEs of type \eqref{SYST} (i.e. $n=2$ with the previous notations). Especially, this approach is better \textcolor{black}{tailored} to exploit cancellation techniques which are crucial when derivatives come in, \textcolor{black}{as opposed} to the \textit{backward} one.\\
}

The perturbative approach is not usual to establish Schauder type estimates. The standard way is to proceed through \textit{a priori} estimates to establish for a given solution of the PDE in a suitable function space, the expected bound. Existence and uniqueness issues, in the considered function space, for the solution of the equation are addressed in a second time.
We can refer to \cite{kryl:96} for a clear presentation of this approach and to \cite{kryl:prio:10} for an extension of this method to non-degenerate operators with unbounded drift coefficients. We will here obtain that the solutions of \eqref{KOLMO} with mollified coefficients satisfy, uniformly w.r.t. the mollification parameter, a Schauder type estimate (see Sections \ref{SEC2} to \ref{HOLDER} below). From the well posedness of the martingale problem established in \cite{chau:meno:17} under our current assumptions,  we will then derive that the martingale solution to \eqref{KOLMO} actually itself satisfies the Schauder controls. Since we want to be in the \textit{sharpest} possible H\"older setting for the coefficients, source and terminal functions, we will need to establish some subtle controls (in particular we have no \textit{true} derivatives of the coefficients) which will heavily rely on duality results for Besov spaces (see Section \ref{SEC_BESOV_DUAL_FIRST} below and e.g. Chapter 3 in Lemari\'e-Rieusset \cite{lemar:02}). 

Let us emphasize that the perturbative approach developed here provides, even in the non-degenerate case, a new alternative to establish Schauder 
estimates. It can be seen as a constructive one in the sense that, from a sequence of smooth solutions, that uniformly satisfy the expected control, we will extract through convergence in law arguments a limit solution which also satisfies the bound. Uniqueness of the solution in the considered class then again follows from uniqueness in law of the underlying limit process.\\

The drawback of our approach is that, for the parabolic problem \eqref{KOLMO}, we first have to establish our estimates in small time. This is intuitively clear since the perturbative methods (expansions along an \textit{ad hoc proxy}) are precisely designed for small times. To obtain the result for an arbitrary given time, we then have to iterate the estimate, \textcolor{black}{which is precisely natural since Schauder estimates provide a kind of stability in the considered function space}. We are therefore \textit{far} from the \textit{optimal} constants for the Schauder estimates established in the non-degenerate setting for time dependent coefficients by Krylov and Priola \cite{kryl:prio:17}. However, to the best of our knowledge, only two approaches  allow to derive Schauder estimates for   parabolic degenerate Kolmogorov equations with fully non-linear drift in H\"older space: ours and the one of Hao \textit{al.} \cite{hao:wu:zhan:19} developed in the kinetic non-local case and based on Littlewood-Paley decompositions. Also, we think the strategy developed in the current work should apply for the elliptic degenerate Kolmogorov equation with \textit{good} potential term, i.e.  the \textit{negative sign} of the potential would allow to integrate on an infinite time horizon (getting therefore rid of the small time constraint) or for unbounded sources $f$ that would need in that cases to be somehow controlled by an associated potential as in \cite{kryl:prio:10}. 
\\

\textcolor{black}{To conclude with this outline, we come back to the stochastic counterpart of \eqref{KOLMO}: system \eqref{SYST}. \textcolor{black}{In connection with the \textit{strong} regularizing properties of Brownian motion, the perturbative approach we develop here allows as well to address the problem of strong well posedness for the SDE \eqref{SYST} in a H\"older framework for the coefficients. This was done by Chaudru de Raynal in \cite{chau:16} for $n=2 $ (let us mention as well in this \textit{kinetic case} the works of Fedrizzi \textit{et al.} \cite{fedr:flan:prio:vove:17} and Zhang \cite{zhan:16} who derived strong uniqueness for $L^p$ drifts on the non-degenerate component and a linear degenerate dynamics in \eqref{SYST}). Namely,  in the companion paper \cite{chau:hono:meno:18:STRONG}}, we establish through \textcolor{black}{a} similar approach strong uniqueness for the full chain for some suitable related H\"older thresholds for the drift. Indeed, from a PDE viewpoint, strong uniqueness for the associated SDE is heavily related to pointwise controls of the gradient of the solution of the PDE in all the directions (including the degenerate ones). These controls hence require some additional regularity and are then obtained under some slightly stronger H\"older regularity assumptions on the coefficients (which for strong uniqueness issues then in turn become the source term in the PDE with the Zvonkin approach). Some related issues were also considered under additional smoothness conditions by Lorenzi \cite{lore:05} in the case of a linear drift. 
}\\

Before stating our main results, we recall some properties associated with the system \eqref{KOLMO}. We first describe in Section \ref{SEC_MC_INTRO} how the intrinsic multi-scales of the degenerate Kolmogorov like equations appear. We then introduce the appropriate setting of H\"older spaces to consider in Section \ref{SEC_HOLDER_SPACE}. We eventually conclude the introduction stating in Section \ref{MAIN_RES_STAT} our main results concerning Schauder estimates associated with \eqref{KOLMO}.

\subsection{Intrinsic scales of the system and associated distance}
\label{SEC_MC_INTRO}
Let us now briefly expose how the system \textit{typically} behaves. To do so, consider the following operator:
\begin{equation}
\label{KOLM_TOY}
{\mathscr L}_0:= \partial_t+ \langle {\mathbf A}_0\x, {\mathbf D} \rangle +\frac 12 \Delta_{\x_1},\  {\mathbf A}_0= \left (\begin{array}{ccccc}{\mathbf 0}_{d,d} & \cdots & \cdots &\cdots  & {\mathbf 0}_{d,d}
\\
{\mathbf I}_{d,d} & \ddots &\cdots &\cdots &
\vdots
\\
{\mathbf 0}_{d,d} & {\mathbf I}_{d,d}&\ddots & \cdots & \vdots
\\
\vdots &   \ddots 
& \ddots & \ddots & \vdots\\
{\mathbf 0}_{d,d} &\cdots &     {\mathbf 0}_{d,d}      &  {\mathbf I}_{d,d}& {\mathbf 0}_{d,d}
\end{array}\right),
 \end{equation}
 which can be viewed as a typical model for the operator in  \eqref{KOLMO}. Introducing now, for $\lambda>0 $, the dilation operator $\delta_\lambda:(t,\x)\in \R^+\times \R^{nd}\mapsto  \delta_\lambda(t,\x)= \big( \lambda^2 t, \lambda \x_1, \lambda^{3}\x_2,\cdots, \lambda^{2n-1}\x_n\big) \in \R^+\times \R^{nd} $, i.e. with a slight abuse of notation, $ \big(\delta_\lambda(t,\x) \big)_0:=\lambda^2 t$ and for each $i\in \leftB 1, n\rightB, \big(\delta_\lambda(t,\x)\big)_i :=\lambda^{2i-1}\x_i  $, we have that 
\begin{equation}\label{delta_lambda}
{\mathscr L}_0 v=0 \implies {\mathscr L}_0( v \circ \delta_\lambda)=0.
\end{equation}
This hence lead\textcolor{black}{s} us to introduce the homogeneous \textcolor{black}{quasi-distance}\footnote{\textcolor{black}{the triangle inequality holds up to some multiplicative constant.}} corresponding to the dilation operator $\delta_\lambda $. Precisely, setting for all $0\le t\le s<+\infty$, $ (\x,\y)\in (\R^{nd})^2$:
 \begin{equation}
\label{DIST_P}
\d_P\big( (t,\x),(s,\y)\big)=(s-t)^{\frac 12}+\sum_{i=1}^d |\y_i-\x_i|^{\frac{1}{2i-1}},
 \end{equation}
 we indeed have $\d_P\big(\delta_\lambda \big((t,\x)\big),\delta_\lambda\big((s,\y)\big)\big)=\lambda \d_P\big((t,\x),(s,\y) \big) $.  In our current setting we will mainly use the spatial part deriving from the parabolic homogeneous quasi-distance $\d_P$ in \eqref{DIST_P}. We set accordingly, 
\begin{equation}
\label{DIST}
\d\big( \x,\y\big)=\sum_{i=1}^d |\y_i-\x_i|^{\frac{1}{2i-1}}.
 \end{equation}
From a technical point of view, these \textcolor{black}{(quasi-)}distances express the spatial homogeneity associated with the intrinsic time scales of the variances for the Gaussian  process with generator ${\mathscr L}_0 $. From a more probabilistic viewpoint the exponents in \eqref{DIST} can be related to the characteristic time-scales of the iterated integrals of the Brownian motion (see e.g. \cite{dela:meno:10}, \cite{chau:meno:17}). 

The non-degeneracy and boundedness assumption on $a$ as well as the H\"ormander condition on the $(\gF_i)_{i\in \leftB 2,n \rightB} $ that we assume for \eqref{KOLMO}-\eqref{SYST} (see assumptions \A{UE} and \A{H} below) will allow us to consider for our analysis the previous quasi-distances associated with the simplest yet typical equation in the class described by \eqref{KOLMO}-\eqref{SYST}.


\subsection{Associated 
H\"older spaces}
\label{SEC_HOLDER_SPACE}
We first recall some useful notations and spaces. We denote for  $k \in \N, \beta\in (0,1) $ by $\|\cdot\|_{C^{k+\beta}(\R^m,\R^\ell)},\ m\in \{1,d,nd\} $, $\ell \in \{1,d,d^2,nd\} $ the usual homogeneous H\"older \textcolor{black}{semi-}norm, see e.g. Krylov \cite{kryl:96}. Precisely, for $\psi \in C^{k+\beta}(\R^m,\R^\ell) $, denoting by  $\vartheta=(\vartheta_1,\cdots,\vartheta_m)\in \N^m $ a generic multi-index and $|\vartheta|=\sum_{i=1}^{m} \vartheta_i $, we define the semi-norm:
\begin{eqnarray}
\|\psi\|_{C^{k+\beta}(\R^m,\R^\ell)}&:=& \sum_{i=1}^k \sup_{|\vartheta|=i} \|D^{\vartheta} \psi\|_{L^\infty(\R^m,\R^\ell)}+ \sup_{|\vartheta|=k} [D^\vartheta \psi]_\beta,\notag \\
\phantom{BOUHHH} [D^\vartheta \psi  ]_\beta &:=& \sup_{(x,y)\in (\R^m)^2,x\neq y} \frac{|D^\vartheta \psi(x)-D^\vartheta \psi(y)|}{|x-y|^\beta},
 \label{USUAL_HOLDER_SPACE}
\end{eqnarray} 
where $|\cdot| $ denotes the Euclidean norm on the considered space. 
We will also need to consider the associated subspace with bounded elements. Namely, we set:
$$C_b^{k+\beta}(\R^m,\R^\ell):=\{\psi \in C^{k+\beta}(\R^m,\R^\ell) : \|\psi\|_{L^{\infty}(\R^m,\R^{\ell})}<+\infty\}.$$  
We define correspondingly the H\"older norm:
\begin{equation}
\label{BD_HOLDER}
\|\psi\|_{C_b^{k+\beta}(\R^m,\R^\ell)}:=\|\psi\|_{C^{k+\beta}(\R^m,\R^\ell)}+\|\psi\|_{L^\infty(\R^m,\R^\ell)}.
\end{equation}

We are now in position to define our \textcolor{black}{anisotropic} H\"older spaces with multi-index of regularity. Let  $\psi: \R^{nd} \rightarrow \R^\ell$
 be a smooth function. We first introduce, for $i\in \leftB 1,n\rightB$, $x\in \R^d$ the perturbation operator that writes:
\begin{equation}
\label{PERT_OP}
\forall \z\in  \R^{nd},   \Pi_i^x (\psi)(\z):= \psi(\z_1,\cdots, \z_i+x,\cdots, \z_n).
\end{equation}
We then define for each $ i\in \leftB 1,n\rightB$, 
 the mapping  
\begin{equation}\label{DEF_PI_I}
(\z,x)\in \R^{nd}\times \R^d \longmapsto \psi_i(\z,x):=\Pi_i^x (\psi) (\z). 
\end{equation}
Let us introduce the following \textcolor{black}{anisotropic} H\"older space in $\d$-\textcolor{black}{quasi-}metric: given a parameter $\gamma\in (0,1) $, and $k\in \N$,
we say that $\psi$ is in $C_\d^{k+\gamma}(\R^{nd},\R^\ell)$,  if 
\begin{equation}\label{INHOM_NORM}
\|\psi\|_{C_\d^{k+\gamma}(\R^{nd},\R^\ell)}:=\sum_{i=1}^n \textcolor{black}{\sup_{\z\in \R^{nd}}}\|\psi_i(\z,\cdot)\|_{C^{\frac{k+\gamma}{2i-1}}(\R^d,\R^\ell)}<+\infty.
\end{equation}
\textcolor{black}{For the sake of simplicity, we will write:
\begin{equation*}
\|\psi\|_{L^\infty}:=\|\psi\|_{L^\infty(\R^{nd},\R^\ell)}, \text{ and }
\|\psi\|_{C_\d^{k+\gamma}}:=\|\psi\|_{C_\d^{k+\gamma}(\R^{nd},\R^\ell)}.
\end{equation*}
}
The subscript $\d$ stands here to indicate the dependence of the H\"older exponents appearing in the r.h.s. on the underlying quasi-distance $\d$ reflecting the scale invariance of the system (see equation \eqref{DIST} and the comments above for details). Note in particular that, for $k=0$, there exists $C:=C(n,d)\ge 1$ s.t.:
\begin{eqnarray}
\label{EQUIV_NORME_H_HD}
C^{-1} [\psi]_{\gamma,\d}\le \|\psi\|_{C_\d^{k+\gamma}(\R^{nd},\R^\ell)} \le C [\psi]_{\gamma,\d},\notag\\
\phantom{BOUH} [\psi]_{\gamma,\d}:=\sup_{\x \neq \x',(\x,\x')\in (\R^{nd})^2}\frac{|\psi(\x)-\psi(\x')|}{\d^\gamma(\x,\x')},
\end{eqnarray}
see also e.g. Lunardi \cite{lun:97}.

From \eqref{BD_HOLDER} and \eqref{INHOM_NORM}, we write that $\psi\in C_{b,\d}^{k+\gamma}(\R^{nd},\R^\ell) $ if
$$\|\psi\|_{C_{b,\d}^{k+\gamma}(\R^{nd},\R^\ell)}:=\sum_{i=1}^n \textcolor{black}{\sup_{\z\in \R^{nd}}}\|\psi_i(\z,\cdot)\|_{C_b^{\frac{k+\gamma}{2i-1}}(\R^d,\R^\ell)}<+\infty.
  $$ 
  \textcolor{black}{From now on, we will denote:
\begin{equation*}
\|\psi\|_{C_{b,\d}^{\gamma}}:= \|\psi\|_{C_{b,\d}^{\gamma}(\R^{nd},\R^\ell)}.
\end{equation*}
Finally, through the article, we use the following notation for all $\varphi_1 \in L^{\infty}\big ([0,T ],C_{b,\d}^{k+\gamma}(\R^{m},\R^\ell) \big) $ and $ \varphi_2 \in   L^{\infty}\big ([0,T ],C_{\d}^{k+\gamma}(\R^{m},\R^\ell) \big)$:
\begin{equation*}
\|\varphi_1\|_{L^\infty(C_{b,\d}^{k+\gamma})}:= \sup_{t \in [0,T]}\|\varphi_1(t, \cdot)\|_{C_{b,\d}^{k+\gamma}(\R^{m},\R^\ell)},
\text{ and }
\|\varphi_2\|_{L^\infty(C_{\d}^{k+\gamma})}:= \sup_{t \in [0,T]} \|\varphi_2(t, \cdot)\|_{C_{\d}^{k+\gamma}(\R^{m},\R^\ell)}
.
\end{equation*}
}

\subsection{Assumptions and main result}\label{MAIN_RES_STAT}
With these notations 	at hand we can now state our assumptions and main results. In the following, we will assume:
\begin{itemize}
  \item[ \A{UE} ] \textbf{Uniform Ellipticity of the diffusion Coefficient.}
  There exists  $\kappa\ge 1$ s.t. for all $(t,\x)\in \R_+\times \R^{nd}$, $z\in \R^d$,
$$ \kappa^{-1}|z|^2\le \langle a(t,\x) z,z\rangle \le \kappa |z|^2,$$ 
where  $|\cdot|$ again denotes the Euclidean norm and $\langle \cdot, \cdot \rangle $ 
is the inner product.
\item[\A{H}] \textbf{Weak H\"ormander like condition.} For each $i\in \leftB 2,n\rightB $, there exists a closed convex subset ${\mathcal E}_{i-1} \subset GL_{d}(\R)$
(set of invertible $d \times d$  matrices)
 \textcolor{black}{such that}, for all
$t \geq 0$ and $(\x_{i-1},\dots,\x_n) \in \R^{(n-i+2)d}$, $D_{\x_{i-1}} F_i(t,\x_{i-1},\dots,\x_n)
\in {\mathcal E}_{i-1}$.
For example, ${\mathcal E}_{i-1}$ 
may be a closed ball
included in $GL_{ d}(\R)$, which is an open set.
\item[\A{S}] \textbf{Smoothness of the Coefficients.} Fix $\gamma\in(0,1) $. We suppose the following conditions hold.
\begin{itemize}
\item[(i)] \textbf{Smoothness of the diffusion coefficient.}  We assume \textcolor{black}{that $a$ is measurable in time and that} $a\in L^\infty\big([0,T],C_{b,\d}^{ \gamma}( \R^{nd},\R^d\otimes \R^d)\big)$. 

\item[(ii)] \textbf{Smoothness of the drift in time.} We only assume here that the drift is measurable in time and bounded \textcolor{black}{at the origin}, i.e. the measurable mapping $t\mapsto  \gF(t,{\mathbf 0})$ is bounded.

\item[(iii)] \textbf{Smoothness of the drift in space.}
We now state, for each level $i\in \leftB 1,n\rightB$ the smoothness assumptions on the drift component $\gF_i$ (see the remark below for more explanations): 
\begin{eqnarray}
\label{DEF_NORME_I_GAMMA_SEUIL}
\gF_{i} \in L^\infty \textcolor{black}{\Big (}[0,T], C_{\d}^{(2i-3)\textcolor{black}{\vee 0}+\gamma}(\R^{((n-i+2) \wedge n)d},\R^d\textcolor{black}{\Big )}.
\end{eqnarray}

\end{itemize}
\end{itemize}
For a fixed parameter  $ \gamma \in (0,1)$, we will say that \A{A} is in force as soon as \A{UE}, \A{H}, \A{S} hold.\\

\begin{REM} Let us come back to assumption \A{S}-(iii), which may seem difficult to understand at first sight. Namely, we here explain a little bit how the particular thresholds appearing in this assumption come from as well as the precise regularity \textcolor{black}{imposed} on each component of the drift $\gF$ w.r.t. any space variable. 

$\bullet$ Note first that for $i=1$ assumption \A{S}-(iii) readily says, with the previous notations for H\"older spaces, that $\gF_1\in L^\infty([0,T],C_{\d}^\gamma(\R^{nd},\R^d)) $.

$\bullet$ For each level $i\in \leftB 2,n\rightB$ of the chain in \eqref{STRUCT_F}, we shall consider different types of assumptions on $\gF_i$ depending on the variables $\x_{i-1} $ and $\x^{i:n} =(\x_i,\cdots,\x_n)$  respectively. Let us now fix  $i\in \leftB 2,n\rightB$.

The component $\x_{i-1} $ is hence the one which transmits the noise. Coherently with the \textit{usual} H\"ormander setting, we need some differentiability  of $\gF_i$ w.r.t. $\x_{i-1} $. In order to have a global coherence, in terms of time-space homogeneity, for all the considered variables, the specific smoothness to be considered for that variable is that $\gF_i(t,\cdot,\x^{i:n})$ is in $ C^{1+\frac{\gamma}{2(i-1)-1}}(\R^d,\R^d )$. Recalling now the previous \textcolor{black}{definition} of $\d$ and \textcolor{black}{of the} associated H\"older spaces, we have $C^{1+\frac{\gamma}{2(i-1)-1}}(\R^d,\R^d )= C^{\frac{2i-3+\gamma}{2(i-1)-1}}(\R^d,\R^d )$\textcolor{black}{.}

Now, at level $i$, the components $\x^{i:n} $ are \textit{above} the current characteristic time-scale, i.e. the vector of their associated time rescaling, which writes according to the homogenous quasi-metric $\d_P$ in \eqref{DIST_P} as $(t^{i-\frac 12}, t^{(i+1)-\frac 12},$\\$ \cdots ,t^{n-\frac 12}) $, has  in small time entries that are actually smaller or equal than the time rescaling of the current variable $\x_i$ in $t^{i- \frac 12} $.  We recall as well that, in order to have the well posedness of the martingale problem associated with the operator $(L_t)_{t\ge 0}$ in \eqref{SPATIAL_OP}, some natural minimal thresholds of H\"older continuity appear for these variables. Precisely, at level $i$, $\gF_i$ must be H\"older continuous in $\x_j,\ j\in \leftB i,n\rightB $, with index strictly greater than $\frac{2i-3}{2j-1} $  (see \cite{chau:meno:17} for details). Here, still to have a global coherence, in terms of time-space homogeneity, for all the considered variables, we assume that $\gF_i $ is $\frac{2i-3}{2j-1}+\frac{\gamma}{2j-1} $ H\"older continuous in its $j^{{\rm th}} $ variable. This precisely corresponds to the minimal threshold required to which we add the intrinsic $\gamma $-H\"older regularity w.r.t to the associated scale
  appearing in $\d$ for the considered entry. Thus, with a slight abuse of notations, $z \mapsto \gF_i(t,\x_{i-1},\x_{i}, \cdots, \x_{j-1},z,\x_{j+1},\cdots,\x_n)$ is supposed to be in $C^{\frac{2i-3}{2j-1}+\frac{\gamma}{2j-1}}(\R^d,\R^d )$\textcolor{black}{.}
  
$\bullet$ \textcolor{black}{``Gathering'' the regularity conditions assumed on each variable for each component of $\gF$} hence gives assumption \A{S}-(iii).\\
\end{REM}

\begin{REM}
Concerning the time regularity in the previous assumptions we can refer to the earlier work of Kruzhkov \textit{et al.} \cite{kruz:cast:lope:75} who first consider this type of regularity to establish Schauder estimates in the classical non degenerate framework with bounded coefficients. We can also mention Lorenzi \cite{lore:11} for extensions to unbounded coefficients.
\end{REM}

We are now in position to state our main result.

\begin{THM}[Schauder Estimates for degenerate Kolmogorov Equations with general drifts.] \label{THEO_SCHAU}\quad
Let $\gamma\in (0,1)$ be given. Suppose that \A{A} is in force and that the terminal condition $g$ and source term $f$ of the Cauchy problem \eqref{KOLMO} satisfy: $g\in C_{b,\d}^{2+\gamma }(\R^{nd},\R)$ and $f\in L^\infty\big([0,T],C_{b,\d}^{\gamma}( \R^{nd},\R)\big)$. 

Then, there exists a unique mild and weak solution $u$ in $C^{2+\gamma}_{b,\d}(\R^{nd},\R)$ to \eqref{KOLMO}. Furthermore, there exists a constant $C_{\ref{THEO_SCHAU}}:=C_{\ref{THEO_SCHAU}}\big(\A{A},T\big)$ s.t. 
\begin{equation}\label{eq:GeneralSchauderEstimate}
\|u\|_{L^\infty(C_{b,\d}^{2+\gamma})}\le C_{\ref{THEO_SCHAU}} \big( \|g\|_{C_{b,\d}^{2+\gamma}}+\|f\|_{L^\infty(C_{b,\d}^{\gamma})}\big).
\end{equation}
\end{THM}

\textcolor{black}{Note that some Schauder estimates results were recently established for a non-local operator associated with a stable process.  Let us mention \cite{chau:meno:prio:20} for the non-degenerate case  ($n=1$) with super-critical drift (stability index strictly less than 1), \cite{hao:wu:zhan:19} for a kinetic case ($n=2$) and \cite{mari:19} for the complete degenerate chain. If the second order term in \eqref{KOLMO} is replaced by an $\alpha$-stable operator, $\alpha \in (0,2)$, then the parabolic regularity gain associated with the parabolic distance $\d$, established therein, is $\alpha$ instead of $2$ in Theorem \ref{THEO_SCHAU}.}
\\

Section \ref{GUIDE_TO_PROOF} below is dedicated to the presentation and description of the various steps that we perform to obtain Theorem \ref{THEO_SCHAU}. From now on we will denote by $C$ a generic constant that may change from line to line but only depends on known parameters in \A{A} and the considered fixed final time $T$, i.e. $C:=C(\A{A},T) $. We reserve the notation $c$ for generic constants that  may also change from line to line, depend on \A{A} but are also independent of $T$, i.e. $c:=c(\A{A}) $.


\mysection{Detailed Guide to the proof}
\label{GUIDE_TO_PROOF}
\textcolor{black}{The various steps of our procedure could be roughly \textcolor{black}{summed up} as follows: we first mollify the \textcolor{black}{coefficients in equation} \eqref{KOLMO} in order to work with well defined objects. 
\textcolor{black}{In the following, we will call, with a slight terminology abuse, by \textit{regularized} or \textit{mollified} solution of \eqref{KOLMO} the solution of \eqref{KOLMO} associated with the \textit{regularized} or \textit{mollified} coefficients.}}

\textcolor{black}{We then  derive the estimate of Theorem \ref{THEO_SCHAU} in this framework but uniformly in the \textcolor{black}{mollification parameter}. To do so, we will expand the regularized solution of  \eqref{KOLMO} around a well chosen proxy. 
This expansion will allow us to obtain an explicit representation of the mollified solution of  \eqref{KOLMO} for which we will derive the desired estimates in small time. A key point is that such a representation is of implicit form, so that, when applying our strategy,  the \textcolor{black}{upper-bound} of the H\"older estimate will involve the H\"older norm of the smoothed solution itself. To overcome this problem, the main idea consists in using a circular argument, \textcolor{black}{bringing} together the H\"older norms of the solutions on the same side of the inequality. This strategy then requires to obtain constants in front of the bounds depending of the solution as small as needed. This property will be fulfilled when working with an appropriately rescaled version of the smoothed solution. We can then transfer estimates on the (regularized) rescaled version of the solution to the original (regularized) one and then extend it \textcolor{black}{to} arbitrary time length interval\textcolor{black}{s} by using a chaining argument. 
We then conclude the proof of the estimates in Theorem \ref{THEO_SCHAU} through \textcolor{black}{a} compactness argument, allowing us to get rid of the regularization parameters, and \textcolor{black}{eventually} show that the mild solution of  \eqref{KOLMO} is a weak solution thanks to suitable controls deriving from our analysis.}

\textcolor{black}{The main objective of this section is to \textcolor{black}{introduce} the approach shortly described above. Especially the derivation of the estimates in Theorem \ref{THEO_SCHAU} for the regularized solution in small time, involving norms of the solution itself. This part \textcolor{black}{is actually} the core of \textcolor{black}{the} paper (we refer to Section \ref{scaling} for the last steps of our procedure: \textcolor{black}{scaling arguments, finite time issues and weak solution property}). The main  point is \textcolor{black}{thus} to emphasize the various difficulties arising when applying our \textcolor{black}{strategy} and to introduce the adapted tools that can be used to circumvent them.\\  
}

\subsection{The mollifying procedure.} The first step of our strategy is to mollify equation \eqref{KOLMO} in order to get a well-posed Cauchy problem in the classical sense.
Precisely, 
for $\varphi \in C_0^2(\R^{nd},\R) $, $m\in \N $  and $t\in [0,T] $ we define the operator:
\begin{equation}
\label{MOLL_OPERATOR}
L_t^m \varphi(x):=\langle \gF_m(t,\x), {\mathbf D}\varphi(\x)\rangle +\frac 12{\rm Tr}\big( D_{\x_1}^2\varphi(\x)a_m(t,\x)\big),
\end{equation}
where $\gF_m ,a_m $ are mollified versions in space of the initial coefficients $\gF, a $ in \eqref{SPATIAL_OP}, i.e. $\gF_m(t,\x)=\gF(t,\cdot)\star \phi_m(\x), a_m(t,\x)=a(t,\cdot)\star \phi_m(\x)  $, where for any $\z\in \R^{nd},\ \phi_m(\z):=m^{nd}\phi(\z m) $ for a smooth, i.e. $C^\infty $, non-negative function $\phi:\R^{nd}\rightarrow \R^+ $ s.t. $\int_{\R^{nd}} \phi(\z) d\z=1$ and the previous convolutions are to be understood componentwise. We introduce correspondingly the stochastic differential equation with generator $(L_t^m)_{t\ge 0}$. Namely, for fixed $ (t,\x)\in [0,T]\times \R^{nd}$ and $s\ge t $,
\begin{equation}
\label{MOLL_SDE}
\X_s^{m,t,\x}=\x+\int_t^s \gF_m(u,\X_u^{m,t,\x})du+\int_t^s B \sigma_m(u,\X_u^{m,t,\x})dW_u,
\end{equation}
where $\sigma_m $ is a square root of $a_m $ \textcolor{black}{and $B=({\rm \mathbf{I}}_{d,d},\mathbf{0}_{d,d},\ldots,\mathbf{0}_{d,d})^{\textcolor{black}{*}}$ is the embedding matrix from $\R^{nd}$ to $\R^d$}. The dynamics in \eqref{MOLL_SDE} is similar to the one in \eqref{SYST} up to the mollification of the coefficients. It can be deduced from the well-posedness of the martingale problem, which holds under our current assumptions from \cite{chau:meno:17}, that $(\X_t^m)_{t\in [0,T]} \Rightarrow_m  (\X_t)_{t\in [0,T]}  $ (\textcolor{black}{convergence in law on the path space}) where $(\X_t)_{t\in [0,T]}$ is the unique weak solution of \eqref{SYST} (see also \cite{stro:vara:79}).

Consider now mollified versions $f_m,g_m $ of the source $f$ and the final condition $g$ in \eqref{KOLMO}. It is then rather direct to derive through stochastic flows techniques, see e.g. Kunita \cite{kuni:97}, that 
\begin{equation}
\label{REP_FK}
u_m(t,\x):=\E[g_m(\X_T^{m,t,\x})]+\int_t^T \E[f_m(s,\X_s^{m,t,\x})] ds,
\end{equation}
belongs for any given $m$ to $C_b^\infty(\R^{nd},\R) $ (space of infinitely differentiable functions with bounded derivatives) and precisely solves:
\begin{equation}
\label{KOLMO_m}
\begin{cases}
\partial_t u_m(t,\x)+ \langle \gF_m(t,\x), {\mathbf D}u_m(t,\x)\rangle +\frac 12{\rm Tr}\big( D_{\x_1}^2u_m(t,\x)a_m(t,\x)\big)=-f_m(t,\x),\ (t,\x)\in [0,T)\times \R^{nd},\\
 u_m(T,\x)=g_m(\x),\ \x\in \R^{nd}.
 \end{cases}
 \end{equation}

\subsection{Proxy and explicit representation of $u_m$.} The idea is now to obtain controls of the norms $\|u_m\|_{L^\infty( C_{b,\d}^{2+\gamma})} $ which are uniform w.r.t. the mollifying parameter $m$. To this end, we will use a perturbative method by expanding $u_m$ around a suitable Ornstein-Uhlenbeck like Gaussian proxy corresponding to an appropriate linearization of the dynamics in \eqref{MOLL_SDE}. Consider first the deterministic dynamics deriving from \eqref{MOLL_SDE} obtained setting $\sigma_m $ to \textcolor{black}{$\mathbf{0}_{d,d}$}, i.e.
\begin{equation}
\label{DYN_DET_SMOOTH}
\dot \btheta_{v,\tau}^m(\bxi)=\gF_m(v,\btheta_{v,\tau}^m(\bxi)),\ v\in [0,T],\  \btheta_{\tau,\tau}^m(\bxi)=\bxi,
\end{equation}
where $(\tau,\bxi) \in [0,T]\times \R^{nd}$ are \textit{freezing parameters}, respectively in time and space to be specified.

Fix $0\le t <s\le T $ and $\x\in \R^{nd} $. The typical linearization of \eqref{MOLL_SDE}  on the time interval $[t,s] $ around $(\btheta_{v,\tau}^m(\bxi))_{v\in [t,s]}$ writes:
\begin{eqnarray}\label{FROZ} 
&&\tilde \X_v^{m,(\tau,\bxi)}=\x + \int_t^v [\gF_m(r,\btheta_{r,\tau}^m(\bxi))+ D\gF_m(r,\btheta_{r,\tau}(\bxi))(\tilde \X_r^{m,(\tau,\bxi)}-\btheta_{r,\tau}^m(\bxi))]dr +\int_t^vB\sigma_m(r,\btheta_{r,\tau}^m(\bxi)) dW_r,\nonumber\\
\label{FROZ_MOL_FOR}
 \end{eqnarray}  
where for any $\z\in \R^{nd}$,\\ 
\begin{equation}\label{DEF_PARTIAL_GRADIENTS}
D\gF_m(v,\z):=\left (\begin{array}{ccccc}\0_{d,d} & \cdots & \cdots &\cdots  & \0_{d,d}\\
D_{\z_1}\gF_{m,2}(v,\z) & \0_{d,d} &\cdots &\cdots &\0_{d,d}\\
\0_{d,d} & D_{\z_2} \gF_{m,3}(v,\z_{2:n})& \0_{d,d}& \0_{d,d} &\vdots\\
\vdots &  \0_{d,d}                     & \ddots & \vdots & \textcolor{black}{\vdots}\\
\0_{d,d} &\cdots &     \0_{d,d}      & D_{\z_{n-1}}\gF_{m,n}(v,\z_{n-1},\z_n) & \0_{d,d}
\end{array}\right) 
\end{equation}
denotes the subdiagonal of the Jacobian matrix ${D_{\mathbf \z} \gF_m(v,\cdot)} $ at point $\z$. From our previous assumptions (non-degeneracy of $ \sigma$ and H\"ormander like condition), the Gaussian process with dynamics \eqref{FROZ_MOL_FOR} admits a well controlled multi-scale density $\tilde p^{m,(\tau,\bxi)}(t,s,\x,\cdot) $ (see e.g. Section \ref{SEC_GAUSS_DENS} below and for instance \cite{dela:meno:10}, \cite{chau:meno:17}). 
Namely, there exists $C:=C(\A{A},T)\ge 1$ s.t. for \textcolor{black}{$j\in \{0,1,2\} $, \textcolor{black}{for} all $k\in \leftB 1,n\rightB^2$, $\ell \in \{0,1\} $}, \textcolor{black}{and}  for all $0\le t<s\le T$, $(\x,\y)\in (\R^{nd})^2 $:
\begin{eqnarray}\label{FIRST_CTR_DENS}
|D_{\x_k}^\ell D_{\x_1}^j \tilde p^{m,(\tau,\bxi)}(t,s,\x,\y)| &\le& \frac{C}{(s-t)^{\ell(k-\frac{1}{2})+\frac{j}{2}+\frac{n^2 d}{2}} } \exp\left (-C^{-1}  (s-t)|\T_{s-t}^{-1}(\m^{m,(\tau,\bxi)}_{s,t}(\x)-\y)|^2\right)\notag\\
&=:& \textcolor{black}{\frac{C}{(s-t)^{\ell(k-\frac{1}{2})+\frac{j}{2}}}\bar p_{C^{-1}}^{(\tau,\bxi)}(t,s,\x,\y)},
\end{eqnarray}
where $\m^{m,(\tau,\bxi)}_{s,t}(\x) $ stands for the mean of $\tilde \X_s^{m,(\tau,\bxi)} $  and  for any $u>0 $, $\T_{u} $ is the \emph{intrinsic} scale matrix:
\begin{equation}
\label{DEF_T_ALPHA}
\T_u=\left( \begin{array}{cccc}
u\mathbf I_{d, d}& \0_{d, d}& \cdots& \0_{d, d}\\
\0_{d, d}   &u^2 \mathbf I_{d, d}&\0_{d,d}& \vdots\\
\vdots & \ddots&\ddots & \vdots\\
\0_{d, d}& \cdots & \0_{d,d}& u^{n}\mathbf I_{d , d}
\end{array}\right),
\end{equation}
that is, the $i^{\rm th} $ diagonal entry of $u^{-\frac 12}\T_u $ reflects the time order of the variances of the $(i-1)^{{\rm th}} $ iterated integral of the standard Brownian motion at time $u$. Observe as well that the time singularities in \eqref{FIRST_CTR_DENS} precisely reflect the typical scale of the associated variable, i.e. differentiating in $\x_k$ yields an additional time singularity in $(s-t)^{-k+\frac 12} $ where $(s-t)^{k-\frac 12} $ is exactly the order of  the standard deviation of the $(k-1)^{\rm th} $ iterated integral of the Brownian motion.

Denoting by 
$(\tilde L_v^{m,(\tau,\bxi)})_{v\in [t,T]} $ the generator of \eqref{FROZ_MOL_FOR}, it also holds that:
\begin{eqnarray*}
&&\big(\partial_s-(\tilde L_s^{m,(\tau,\bxi)})^* \big) \tilde p^{m,(\tau,\bxi)}(t,s,\x,\y)=0, \ \tilde p^{m,(\tau,\bxi)}(t,s,\textcolor{black}{\x,\cdot}) \rightarrow_{s\downarrow t}  \delta_{\textcolor{black}{\x}}(\cdot),\\
&&\big(\partial_t+\tilde L_t^{m,(\tau,\bxi)} \big) \tilde p^{m,(\tau,\bxi)}(t,s,\x,\y)=0, \ \tilde p^{m,(\tau,\bxi)}(t,s,\textcolor{black}{\cdot,\y}) \rightarrow_{t\uparrow s}  \delta_{\textcolor{black}{\y}}(\cdot).
\end{eqnarray*}
The above equations are respectively the \textit{forward} and \textit{backward} Kolmogorov equations. In the first one, the operator $(\tilde L_s^{m,(\tau,\bxi)})^* $ acts on the \textit{forward} variable $\y$ whereas in the  second one, $\tilde L_t^{m,(\tau,\bxi)}$ acts in the \textit{backward} variable $\x$. 
We will use the notation $\tilde P_{T,t}^{m,(\tau,\bxi)} $ for the corresponding semi-group, i.e. 
$$\tilde P_{T,t}^{m,\textcolor{black}{(\tau,\bxi)}} g_m(\x):=\int_{\R^{nd}}\tilde p^{m,\textcolor{black}{(\tau,\bxi)}}(t,T,\x,\y) g_m(\y) d\y,$$
as well as 
\begin{eqnarray}\label{FROZE_SG_GK_FORWARD}
\tilde G^{m,\textcolor{black}{(\tau,\bxi)}} f_m(t,\x)&:=&\int_t^T ds \int_{\R^{nd}}  \tilde p^{m,\textcolor{black}{(\tau,\bxi)}}(t,s,\x,\y)f_m(s,\y)d\y,
\end{eqnarray}
for the associated Green kernel (with fixed final time $T>0$).


 For fixed $(t,\x) \in [0,T]\times \R^{nd}$ and the above Gaussian proxy, for which $(\tau,\bxi) $ still remain to be specified, we recall that  Duhamel's formula (first order parametrix expansion) 
 yields that:
\begin{equation}
\label{DUHAMEL_PERTURB}
u_m(t,\x)=\tilde P_{T,t}^{m,(\tau,\bxi)} g_m(\x)+ \tilde G^{m,(\tau,\bxi)} f_m(t,\x)+\int_t^T ds \int_{\R^{nd}} \tilde p^{m,(\tau,\bxi)}(t,s,\x,\y)(L_s^m-\tilde L_s^{m,(\tau,\bxi)})u_m(s,\y) d\y.
\end{equation}
\textcolor{black}{Note for instance that the superscript $(\tau,\bxi)$, which stands for the freezing parameters, does not appear in the regularized solution $u_m$. This is because the smoothed solution does not depend on the freezing parameters. \textcolor{black}{Hence}, the above representation is valid for \textbf{any} choice of $(\tau,\bxi)$.\\}

\subsection{Estimates of the supremum norm of the second order derivative w.r.t. the non degenerate variables: introduction of the Besov duality argument.} Recall from the statement of our main Theorem \ref{THEO_SCHAU} that we have to give bounds on $\|u_m\|_{L^\infty(C_{b,\d}^{2+\gamma})}$. For this introduction to the proof, we will focus on the contribution $D_{\x_1}^2 u_m $, that already exhibits \textcolor{black}{almost} all the difficulties and for which we want to establish a control in time-space supremum norm and for the $\gamma$-H\"older modulus associated with the distance $\d$. 

Differentiating in $D_{\x_1}^2$ equation \eqref{DUHAMEL_PERTURB} gives:
\begin{eqnarray}
\label{DUHAMEL_PERTURB_DER_NON_DEG}
D_{\x_1}^2 u_m(t,\x)&=&D_{\x_1}^2\tilde P_{T,t}^{m,(\tau,\bxi)} g_m(\x)+ D_{\x_1}^2 \tilde G^{m,(\tau,\bxi)} f_m(t,\x)\notag\\
&&+\int_t^T ds \int_{\R^{nd}} D_{\x_1}^2 \tilde p^{m,(\tau,\bxi)}(t,s,\x,\y)(L_s^m-\tilde L_s^{m,(\tau,\bxi)})u_m(s,\y) d\y.
\end{eqnarray}
Concentrating on the last term, which turns out to be the most delicate, we see that the choice of $(\tau,\bxi) $ must be made in order to \textcolor{black}{balance} the time singularities coming from $D_{\x_1}^2 \tilde p^{m,(\tau,\bxi)}(t,s,\x,\y)  $. Let us first consider the non-degenerate part coming from the difference $(L_s^m-\tilde L_s^{m,(\tau,\bxi)})u_m(s,\y) $ which explicitly writes from \eqref{MOLL_OPERATOR} and \eqref{FROZ_MOL_FOR}:
\begin{eqnarray}\label{def_Delta_a_F1}
&&\langle \gF_{m,1}(s,\y)-\gF_{m,1}(s,\btheta_{s,\tau}^m(\bxi)) , D_{\y_1}u_m(s,\y)\rangle+\frac 12{\rm Tr} \Big( (a_m(s,\y)-a_m(s,\btheta_{s,\tau}^m(\bxi))   )D_{\y_1}^2 u_m(s,\y)\Big)\notag\\
&=:& \Delta_{1,\gF_m,\sigma_m}(\tau,s,\y,\btheta_{s,\tau}^m(\bxi),u_m),\label{DEF_DIFF_NON_DEG}
\end{eqnarray}
and can be upper-bounded 
from the H\"older continuity assumption (w.r.t. the underlying homogeneous metric $\d $) on $\gF_1 $ and $a$  as:
\begin{eqnarray}\label{ineq_Holder_a_F1}
&&|\Delta_{1,\gF_m,\sigma_m}(\tau,s,\y,\btheta_{s,\tau}^m(\bxi),u_m)|\nonumber\\
&\le& \Big( [\gF_1(s,\cdot)]_{\d,\gamma}\|D_{\y_1}u_m(s,\cdot)\|_{L^\infty}+\frac 12 [a(s,\cdot)]_{\d,\gamma}\|D_{\y_1}^2u_m(s,\cdot)\|_{L^\infty}\Big) \d^\gamma(\y,\btheta_{s,\tau}^m(\bxi)).
\end{eqnarray}
The contribution $\d^\gamma(\y,\btheta_{s,\tau}^m(\bxi)) $ in the above r.h.s. must then equilibrate the time singularity in $(s-t)^{-1} $ coming from $D_{\x_1}^2 \tilde p^{m,(\tau,\bxi)}(t,s,\x,\y) $ (see \eqref{FIRST_CTR_DENS} and Proposition  \ref{THE_PROP} below). This is possible if $\d^\gamma(\y,\btheta_{s,\tau}^m(\bxi))$ is \textit{compatible} with the off-diagonal bound $(s-t)|\T_{s-t}^{-1}(\m^{m,(\tau,\bxi)}_{s,t}(\x)-\y)|^2 $ in \eqref{FIRST_CTR_DENS}. This is precisely the case considering $(\tau,\bxi)=(t,\x) $ which gives 
\begin{equation}\label{eq_m_theta_x}
\m^{m,(\tau,\bxi)}_{s,t}(\x)|_{(\tau,\bxi)=(t,\x)}=\btheta_{s,t}^m(\x),
\end{equation}
as it can readily be checked from \eqref{DYN_DET_SMOOTH}, \eqref{FROZ_MOL_FOR} (taking the expectation) and the Gr\"onwall's lemma. Therefore,  observing precisely from the metric homogeneity (see equations \eqref{DIST_P} and \eqref{DIST}) that:
\begin{eqnarray*}
\d^\gamma(\y,\btheta_{s,t}^m(\x))&=& (s-t)^{\frac \gamma 2} \d^\gamma((s-t)^{\frac 12} \T_{s-t}^{-1}\y,(s-t)^{\frac 12} \T_{s-t}^{-1} \btheta_{s,t}^m(\x)) \\
&\le& C(s-t)^{\frac \gamma 2}\Big(\sum_{i=1}^n |(s-t)^{\textcolor{black}{-}\frac{2i-1}2} (\y-\btheta_{s,t}^m(\x))_i|^{\frac{\gamma}{2i-1}}\Big),
\end{eqnarray*}
we get that the terms of last contribution in the above r.h.s. can precisely be absorbed by the exponential off-diagonal bound in \eqref{FIRST_CTR_DENS}.

We therefore eventually derive for the non-degenerate contribution \textcolor{black}{with the notation of \eqref{FIRST_CTR_DENS}}:
\begin{eqnarray}\label{ineq_D2x1_Delta1}
&&\Big|\int_t^Tds \int_{\R^d} D_{\x_1}^2 \tilde p^{m,(\tau,\bxi)}(t,s,\x,\y) \Delta_{1,\gF_m,\sigma_m}(\tau,s,\btheta_{s,t}^m(\bxi) ,\y,u_m)d\y
\Big| \Bigg|_{(\tau,\bxi)=(t,\x)}
\notag\\
&\le& \int_{t}^{T}\frac{ds}{(s-t)^{1-\frac \gamma2}}\int_{\R^{nd}}  \textcolor{black}{C\big(\|a\|_{L^\infty(C_\d^\gamma)}+\|\gF_1\|_{L^\infty(C_\d^\gamma)} \big)}\textcolor{black}{ \bar p_{C^{-1}}^{(\tau,\bxi)}(t,s,\x,\y)}\notag\\
&&\times 
\Big(\|D_{\y_1}u_m(s,\cdot)\|_{L^\infty}+ \|D_{\y_1}^2u_m(s,\cdot)\|_{L^\infty}\Big)d\y \notag\\
&\le& \frac {2\textcolor{black}{\Lambda}}\gamma (T-t)^{\frac \gamma 2} \big(\|D_{\y_1}u_m\|_{L^{\infty}}+ \|D_{\y_1}^2 u_m\|_{L^{\infty}}\big)
\nonumber \\
&\le& \frac {2\textcolor{black}{\Lambda}}\gamma (T-t)^{\frac \gamma 2} \|u_m\|_{L^{\infty}(C_{b,\d}^{2+\gamma})} \label{FIRST_SMOOTHING_NON_DEG}.
\end{eqnarray}
\begin{REM}[Constants depending on the H\"older moduli of the coefficients]\label{REM_LAMBDA}
\textcolor{black}{In equation \eqref{FIRST_SMOOTHING_NON_DEG},  we denoted $\Lambda :=C\big(\|a\|_{L^\infty(C_\d^\gamma)}+\|\gF_1\|_{L^\infty(C_\d^\gamma)} \big)$, i.e.  $\Lambda $ explicitly depends on the H\"older moduli of the coefficients $a$ and $\gF_1$ on the time interval $[0,T] $. Importantly, in the following, we will keep the generic notation $\Lambda $ for any constant depending on the H\"older moduli of  $a,\gF $, but \textbf{not} on the supremum norms of $a$ and $(D_{}\gF_i)_{i\in \leftB 2,n\rightB} $ and such that $\Lambda \rightarrow 0 $ when the H\"older moduli of $a$, $\gF_1 $ and for any  $i\in \leftB 2,n\rightB $, $\gF_i $ w.r.t. to the variables $i$ to $n$,  themselves tend to 0. In other words, $\Lambda $ is meant to tend to 0  when the coefficients do not vary much. In the computations below $\Lambda $ may change from line to line but will always enjoy the previous property.}
\end{REM}

Equation \eqref{FIRST_SMOOTHING_NON_DEG} thus precisely yields a time smoothing effect corresponding exactly to the H\"older continuity exponent $\gamma $ of the coefficients.
The previous choice of $(\tau,\bxi) $ is known as the \textit{forward} parametrix and seems adapted as soon as one is led to estimate derivatives of the solution. 

Let us mention that, as far as one is concerned with density estimates, which formally amounts to replace $u_m(s,\x),\ u_m(s,\y) $ in \eqref{DUHAMEL_PERTURB} with $p^m(t,T,\x,\z),\ p^m(s,T,\y,\z) $ (density at some fixed point  $\z\in \R^{nd} $ of $\X_T^m $ starting from $\x$ at time $t$), or with the well-posedness of the martingale problem, another choice, consisting in freezing in $(\tau,\bxi)=(s,\y) $ in the above equation, could also be considered. Note that the freezing parameters would here depend on the time and spatial integration variables. This \textit{backward} approach was first introduced by Il'in \textit{et al} \cite{ilin:kala:olei:62} (see also Friedman \cite{frie:64}  or McKean and Singer \cite{mcke:sing:67}) and  led successfully to density estimates and well-posedness of the martingale problem for the current model \eqref{SYST} in the respective works \cite{dela:meno:10}, \cite{meno:17}, \cite{chau:meno:17}. 

However, when dealing with derivatives, the forward perturbative approach appears more flexible since it allows to exploit cancellation techniques whereas this is much trickier in the backward case for which $\tilde p^{m,(s,\y)}(t,s,\x,\y) $ is not a density w.r.t. $\y$. Some associated errors associated with this approach are thoroughly discussed in \cite{chau:meno:17}. 

\textcolor{black}{Let us emphasize that, in view of the above calculations, from now on, the choice of the freezing parameters $(\tau,\bxi)$ will \textcolor{black}{\textit{by default}} be $(\tau,\bxi)=(t,\x)$. We will  hence sometimes forget the superscript on quantities that depend on these parameters for \textcolor{black}{the} sake of clarity and assume implicitly this choice. It may \textcolor{black}{happen} in the following that another choice for the freezing space point $\bxi$ will be done. If so, we will \textcolor{black}{specify} it. In any case, when the regularized solution $u_m$ will be evaluated at point $t$ in $[0,T]$, we will choose $\tau=t$ so that this choice shall be assumed in the following.}\\

Let us now turn to the contributions associated with the degenerate variables in the difference $(L_s^m-\tilde L_s^{m,(\tau,\bxi)})u_m(s,\y)$. 
They precisely write:
\begin{eqnarray}
&&\sum_{i=2}^n\Big\langle \Big(\gF_{m,i}(s,\y)-\gF_{m,i}(s,\btheta_{s,\tau}^m(\bxi))-D_{\x_{i-1}} \gF_{m,i}(s,\btheta_{s,\tau}^m(\bxi)) (\y-\btheta_{s,\tau}^m(\bxi))_{i-1} \Big), D_{\y_i} u_m(s,\y)\Big \rangle\notag\\
&=:&\sum_{i=2}^n \big \langle  \Delta_{i,\gF_m}(\tau,s,\btheta_{s,t}^m(\bxi),\y) , D_{\y_i} u_m(s,\y) \big \rangle. \label{DEF_TERMES_DEG}
\end{eqnarray}
Under the current assumptions, we do not expect to have uniform controls w.r.t. to the smoothing parameter $m$ for the derivatives $(D_{\y_i} u_m)_{i\in \leftB 2,n\rightB} $ in the degenerate directions. Our strategy will first consist for those terms in performing an integration by parts leading to:
\begin{eqnarray}
&&\Big|\sum_{i=2 }^n\int_t^T ds \int_{\R^{nd}} d\y D_{\x_1}^2 \tilde p^{m,(\tau,\bxi)} (t,s,\x,\y) \big \langle \Delta_{i,\gF_m}(\tau,s,\btheta_{s,t}^m(\bxi),\y),D_{\y_i} u_m(s,\y) \big \rangle \Big|\Bigg|_{(\tau,\bxi)=(t,\x)}\notag\\
&\le& \sum_{i=2 }^n\int_t^T ds \Big|\int_{\R^{nd}} d\y  D_{\y_i}\cdot  \Big( \big (D_{\x_1}^2 \tilde p^{m,(\tau,\bxi)}(t,s,\x,\y)  \otimes \Delta_{i,\gF_m}(\tau,s,\btheta_{s,t}^m(\bxi),\y) \big )\Big) u_m(s,\y) \Big|\Bigg|_{(\tau,\bxi)=(t,\x)},\notag\\
\label{DUALITE_PREAL_BESOV}
\end{eqnarray}
\textcolor{black}{where the notation ``$\otimes$" stands for the usual tensor product. In particular, the term $\big (D_{\x_1}^2 \tilde p^{m,(\tau,\bxi)}(t,s,\x,\y)  \otimes \Delta_{i,\gF_m}(\tau,s,\btheta_{s,t}^m(\bxi),\y) \big )$ is a tensor lying in $(\R^d)^{\otimes 3}$. Furthermore, $D_{\y_i}\cdot$ refers to an extended form of the divergence over the $i^{\rm {th}}$ variable ($\y_i \in \R^d$). Precisely, from \eqref{DUALITE_PREAL_BESOV}, we rewrite for all $i \in \leftB 2,n \rightB$, $(s,\y) \in [t,T] \times \R^{nd}$: 
\begin{eqnarray*}
&&D_{\y_i}\cdot  \Big( \big (D_{\x_1}^2 \tilde p^{m,(\tau,\bxi)}(t,s,\x,\y)  \otimes \Delta_{i,\gF_m}(\tau,s,\btheta_{s,t}^m(\bxi),\y) \big )\Big)\\
&=& \sum_{j=1}^d  \partial_{\y_i^j} \Big (D_{\x_1}^2 \tilde p^{m,(\tau,\bxi)}(t,s,\x,\y)   \big ( \Delta_{i,\gF_m}(\tau,s,\btheta_{s,t}^m(\bxi),\y) \big )_j \Big ),
\end{eqnarray*}
with $\y_i= (\y_i^1, \cdots, \y_i^n)$.
In other words, this ``enhanced"  divergence form decreases by one the order of the input tensor.
As a particular case, if $d=1$, $ \Delta_{i,\gF_m}(\tau,s,\btheta_{s,t}^m(\bxi),\y) $ is a scalar and the divergence form corresponds to the standard differentiation, i.e. $D_{\y_i} \cdot=\partial_{\y_i} $.}

Introduce now for notational convenience, for a multi-index $\vartheta \in \N^{\textcolor{black}{ n}}$ the quantity:
$$ \Theta_{i,(t,\x)}^{m,\vartheta}(s,\y):=D_{\x}^\vartheta \tilde p^{m,(\tau,\bxi)}(t,s,\x,\y)  \otimes \Delta_{i,\gF_m}(\tau,s,\btheta_{s,t}^m(\bxi),\y).$$
With this notation at hand equation \eqref{DUALITE_PREAL_BESOV} rewrites:
\begin{eqnarray}
&&\Big|\sum_{i=2 }^n\int_t^T ds \int_{\R^{nd}} d\y D_{\x_1}^2 \tilde p^{m,(\tau,\bxi)} (t,s,\x,\y) \big \langle \Delta_{i,\gF_m}(\tau,s,\btheta_{s,t}^m(\bxi),\y),D_{\y_i} u_m(s,\y) \big \rangle \Big|\Bigg|_{(\tau,\bxi)=(t,\x)}\notag\\
&=&\sum_{i=2 }^n\int_t^T ds \Big|\int_{\R^{nd}} d\y D_{\y_i} \cdot \big (\Theta_{i,(t,\x)}^{m,\vartheta}(s,\y) \big )  u_m(s,\y)  \Big| \Bigg|_{(\tau,\bxi)=(t,\x)},\label{DUALITE_PREAL_BESOV_WITH_THETA}
\end{eqnarray}
with $\vartheta=(2,0,\cdots,0) $, i.e. the multi-index here involves the second order derivatives of the frozen heat-kernel
w.r.t. to its non-degenerate components.
\\

In view of our main estimates in Theorem \ref{THEO_SCHAU}, we will use the duality between suitable Besov spaces to derive  bounds for the spatial integrals in \eqref{DUALITE_PREAL_BESOV_WITH_THETA}. 
Introduce, for each fixed $i\in \leftB 2, n\rightB  $ and any spatial point $(\y_1,\cdots, \y_{i-1},\y_{i+1}, \cdots,\y_n)=:(\y_{1:i-1},\y_{i+1:n})\in \R^{(n-1)d} $ the mappings 
\begin{eqnarray}
u_m^{i, (s,\y_{1:i-1},\y_{i+1:n})}:\y_i &\mapsto& u_m(s,\y_{1:i-1},\y_i,\y_{i+1:n}),\notag\\
\Psi_{i,(t,\x),(s,\y_{1:i-1},\y_{i+1:n})}^{m,\vartheta}: \y_i&\mapsto & D_{\y_i} \cdot \big (\Theta_{i,(t,\x)}^{m,\vartheta}(s,\y) \big ). 
\label{DEF_PARTIAL_FUNC_TO_LEMMA_BESOV}
\end{eqnarray}
The underlying idea is that we actually want, for the $i^{\rm th}  $ variable, to control uniformly in $m$ the H\"older modulus $[u_m^{i, (s,\y_{1:i-1},\y_{i+1:n})}]_{\frac{2+\gamma}{2i-1}} $ uniformly in $(s,\y_{1:i-1},\y_{i+1:n})\in [t,T]\times \R^{(n-1)d} $.
To prove this property we recall that, setting $\tilde \alpha_i:=\frac{2+\gamma}{2i-1} $, $C_{b}^{\tilde \alpha_i}(\R^d,\R)=B_{\infty,\infty}^{\tilde \alpha_i}(\R^d,\R) $ with the usual notations for Besov spaces (see e.g. Triebel \cite{trie:83}). 

Let us now recall some definitions/characterizations from Section 2.6.4 of Triebel \cite{trie:83}.
For $\tilde \alpha \in \R, q\in (0,+\infty] ,p \in (0,\infty] $, $B_{p,q}^{\tilde \alpha}(\R^d):=\{f\in {\mathcal S}'(\R^d): \|f\|_{{\mathcal H}_{p,q}^{\tilde \alpha}}<+\infty \} $ where ${\mathcal S}(\R^d) $ stands for the Schwarz class and 
\begin{equation}
\label{THERMIC_CAR_DEF}
\|f\|_{{\mathcal H}_{p,q}^{\tilde \alpha}}:=\|\varphi(D) f\|_{L^p(\R^d)}+ \Big(\int_0^1 \frac {dv}{v} v^{(m-\frac{\tilde \alpha} 2)q}    \|\partial_v^m h_v\star f\|_{L^p(\R^d)}^q \Big)^{\frac 1q},
\end{equation}
with $\varphi \in C_0^\infty(\R^d)$ (smooth function with compact support) is s.t. $\varphi(0)\neq 0 $, $\varphi(D)f := (\varphi \hat f)^{\vee} $ where $\hat f$ and $(\varphi \hat f)^\vee $ respectively denote the Fourier transform of $f$ and the inverse  Fourier transform of $\varphi \hat f $. \textcolor{black}{When $\tilde \alpha > d  (1/p-1) \vee 0 $ then  $\|\varphi(D) f\|_{L^p(\R^d)}$ in \eqref{THERMIC_CAR_DEF} can be replaced by $\|f\|_{L^p(\R^d)}$}. 
The parameter $m$ is an integer s.t. $m>\frac{\tilde \alpha} 2 $ and for $v>0$, $z\in \R^d $, $h_v(z):=\frac{1}{(2\pi v)^{\frac d2}}\exp \big(-\frac{|z|^2}{2v} \big)$ is the usual heat kernel of $\R^d$. We point out that the quantities in \eqref{THERMIC_CAR_DEF} are well defined for $q<\infty $. The modifications for $q=+\infty $ are obvious and can be written passing to the limit.

Observe that the quantity $\|f\|_{{\mathcal H}_{p,q}^{\tilde \alpha}} $, where the subscript ${\mathcal H} $ stands to indicate the dependence on the heat-kernel, depends on the considered function $\varphi $ and the chosen $m\in \N$. It also defines a quasi-norm on $B_{p,q}^s(\R^d) $. The previous definition of $B_{p,q}^{\tilde \alpha}(\R^d) $ is known as the thermic characterization of Besov spaces and is particularly well adapted to our current framework.
By abuse of notation we will write as soon as this quantity is finite $\|f\|_{{\mathcal H}_{p,q}^{\tilde \alpha}}=:\|f\|_{B_{p,q}^{\tilde \alpha}} $.


As indicated above, it is easily seen from \eqref{THERMIC_CAR_DEF} that $C_{b}^{\tilde \alpha_i}(\R^d,\R)=B_{\infty,\infty}^{\tilde \alpha_i}(\R^d,\R) $.
It is also well known that $B_{\infty,\infty}^{\tilde \alpha_i} (\R^d,\R)$ and $B_{1,1}^{-\tilde \alpha_i} (\R^d,\R)$ are in duality (see e.g. Proposition 3.6 in \cite{lemar:02}). \textcolor{black}{Namely $B_{\infty,\infty}^{\tilde \alpha_i}$ is the dual of the closure of the Schwartz class $\mathcal S$ in $B_{1,1}^{-\tilde \alpha_i}$. But $\mathcal S$ is dense in  $B_{1,1}^{-\tilde \alpha_i}$ (see for instance Theorem 4.\textcolor{black}{1}.3 in \cite{adam:hedb:96})}.
We will therefore write from \eqref{DUALITE_PREAL_BESOV} and with the notations of \eqref{DEF_PARTIAL_FUNC_TO_LEMMA_BESOV} 
\begin{eqnarray}
&&\sum_{i=2 }^n\int_t^T ds \Big|\int_{\R^{nd}} d\y D_{\y_i} \cdot \big (\Theta_{i,(t,\x)}^{m,\vartheta}(s,\y) \big )  u_m(s,\y)  \Big| \Bigg|_{(\tau,\bxi)=(t,\x)}\notag\\
&\le& \sum_{i=2 }^n\int_t^T ds \int_{\R^{(n-1)d}} d(\y_{1:i-1},\y_{i+1:n}) \|\Psi_{i,(t,\x),(s,\y_{1:i-1},\y_{i+1:n})}^{m,\vartheta}\|_{B_{1,1}^{-\tilde \alpha_i}} \|u_m^{i, (s,\y_{1:i-1},\y_{i+1:n})}\|_{B_{\infty,\infty}^{\tilde \alpha_i}}\notag\\
&\le & \textcolor{black}{C}\|u_m\|_{L^\infty(C_{b,\d}^{2+\gamma})} \sum_{i=2 }^n\int_t^T ds \int_{\R^{(n-1)d}} d(\y_{1:i-1},\y_{i+1:n}) \|\Psi_{i,(t,\x),(s,\y_{1:i-1},\y_{i+1:n})}^{m,\vartheta}\|_{B_{1,1}^{-\tilde \alpha_i}}.\label{PREAL_CTR_BESOV_DEG_UG}\end{eqnarray}
Exploiting the thermic characterization of Besov spaces (see again  \eqref{THERMIC_CAR_DEF} and Section \ref{CTR_DER_SUP}
),  it will be shown in Lemma \ref{LEMME_BESOV_DEG} below that there exists a constant $\Lambda:=\Lambda(\A{A},T)$ as in Remark \ref{REM_LAMBDA} s.t. for all $i\in \leftB 2,n\rightB $ and  $m\in \N $:
\begin{equation}
\label{CTR_BESOV_TO_INTEGRATE_IN_TIME}
\int_{\R^{(n-1)d}} d(\y_{1:i-1},\y_{i+1:n}) \|\Psi_{i,(t,\x),(s,\y_{1:i-1},\y_{i+1:n})}^{m,\vartheta}\|_{B_{1,1}^{-\tilde \alpha_i}}\le \frac{\Lambda}{(s-t)^{1-\frac \gamma 2}}.
\end{equation}
Therefore:
\begin{equation}
\sum_{i=2 }^n\int_t^T ds \int_{\R^{(n-1)d}} d(\y_{1:i-1},\y_{i+1:n}) \|\Psi_{i,(t,\x),(s,\y_{1:i-1},\y_{i+1:n})}^{m,\vartheta}\|_{B_{1,1}^{-\tilde \alpha_i}}\le \Lambda(T-t)^{\frac \gamma2},
\end{equation}
which plugged into \eqref{PREAL_CTR_BESOV_DEG_UG} eventually gives the following  global smoothing effect for the degenerate contributions. Namely,
\begin{eqnarray}
&&\Big|\sum_{i=2 }^n\int_t^T ds \int_{\R^{nd}} d\y D_{\x_1}^2 \tilde p^{m,(\tau,\bxi)}  (t,s,\x,\y)\big \langle \Delta_{i,\gF_m}(\tau,s,\btheta_{s,t}^m(\bxi),\y)D_{\y_i} u_m(s,\y) \big \rangle  \Big|\Bigg|_{(\tau,\bxi)=(t,\x)}\notag\\
&\le& \Lambda(T-t)^{\frac{\gamma}2} \textcolor{black}{\|u_m\|_{L^\infty(C_{b,\d}^{2+\gamma})}}, \label{THE_DEG_POINT_DX1_2}
\end{eqnarray}
which is precisely homogeneous to the bound obtained for the non-degenerate variables in \eqref{FIRST_SMOOTHING_NON_DEG}. In both cases, the contribution $(T-t)^{\frac{\gamma}2} $ derives from the assumed smoothness of the coefficients $a,\gF $ w.r.t. $\d $ which exactly leads to the same global control for the \textit{a priori} most singular part of expansion \eqref{DUHAMEL_PERTURB_DER_NON_DEG}.

From the previous bounds  and  \eqref{DUHAMEL_PERTURB_DER_NON_DEG} we thus obtain:
\begin{equation}
|D_{\x_1}^2 u_m(t,\x)|\le \Big(|D_{\x_1}^2\tilde P_{T,t}^{m,(\tau,\bxi)}g_m (\x)|+|D_{\x_1}^2 \tilde G^{m,(\tau,\bxi)}f_m(t,\x)|\Big)\Big|_{(\tau,\bxi)=(t,\x)}+\Lambda(T-t)^{\frac{\gamma}2}\|u_m\|_{L^\infty(C_{b,\d}^{2+\gamma})}. 
 \label{CTR_BORNE_SUP_VAR_DEG}
\end{equation}
Since $\tilde P_{T,t}^{m,(\tau,\xi)}$ is a \textit{true} semi-group, and $\tilde G^{m,(\tau,\bxi)} $ the associated Green kernel (precisely because we used a \textit{forward} perturbative expansion), it will be derived in Lemma \ref{CTR_SEMI_GROUP_GREEN_FROZEN} (thanks to cancellation techniques) that there exists $C:=C(\A{A})$ s.t. for any $(t,\x)\in [0,T] \times \R^{nd} $:
\begin{eqnarray}
\label{CTR_TRUE_D2X1_SG_AND_G}
\Big(|D_{\x_1}^2\tilde P_{T,t}^{m,(\tau,\bxi)}g_m (\x)|+|D_{\x_1}^2 \tilde G^{m,(\tau,\bxi)}f_m(t,\x)|\Big)\Big|_{(\tau,\bxi)=(t,\x)}&\le& C(\| g_m\|_{C_{b,\d}^{2+\gamma}}+(T-t)^{\frac \gamma 2} \|f_m\|_{L^\infty(C_{b,\d}^\gamma)})\notag\\
&\le & C (\| g\|_{C_{b,\d}^{2+\gamma}}+ (T-t)^{\frac \gamma 2}\|f\|_{L^\infty(C_{b,\d}^\gamma)}).\notag\\
\end{eqnarray}
Equation \eqref{CTR_TRUE_D2X1_SG_AND_G}
eventually leads to the following estimate on $|D_{\x_1}^2 u_m(t,\x)| $:
\begin{equation}
\label{CTR_D2_PROV_UM}
|D_{\x_1}^2 u_m(t,\x)|\le C\Big(\textcolor{black}{\| g\|_{C_{b,\d}^{2+\gamma}}}+ (T-t)^{\frac \gamma 2}( \|f\|_{L^\infty(C_{b,\d}^\gamma)}\Big)+\Lambda (T-t)^{\frac \gamma2}\|u_m\|_{L^\infty(C_{b,\d}^{2+\gamma})}.
\end{equation}
For $T$ small enough, this equation would be compatible with the estimates of Theorem \ref{THEO_SCHAU}. Equation \eqref{CTR_D2_PROV_UM} might even seem \textit{too strong} since it also exhibits, additionally to the control of the term associated with the perturbation, a small contribution (in $(T-t)^{\frac \gamma 2} $ for a small enough $T$) w.r.t. to the source $ f_m$. 
This is precisely because $\|D_{\x_1}^2u_m(t,\cdot)\|_{L^\infty} $ is not one of the \textit{critical} terms in the H\"older norm $\|u_m(t,\cdot)\|_{C_{b,\d}^{2+\gamma}} $, i.e. the regularity of the coefficients still gives that it can be viewed as a remainder at first sight.\\

\subsection{Estimates on the H\"older modulus of the second order derivative w.r.t. the non degenerate variables: introduction of the various change of regime of the system.}\label{subsec_detailedguide_regime}
Now, a typical critical term  of the H\"older norm, for which we precisely exploit totally the spatial regularity of the coefficients, is $[D_{\x_1}^2 u_m(t,\cdot)]_{\d}^\gamma $ \textcolor{black}{(see assumption \A{S} and \eqref{EQUIV_NORME_H_HD})}. Let us now detail how we can handle it and in which sense it can be viewed as critical.

\textcolor{black}{Of course, if $g=0$, for $t\in [0,T]  $ and  given spatial points $(\x,\x')\in (\R^{nd})^2 $ we can assume w.l.o.g. that, \textcolor{black}{for some constant $c_0$ to be specified later on and meant to be \textit{small}},  $c_0^{\frac 12}\d(\x,\x')\le (T-t)^{\frac 12}  $, i.e. the spatial points are \textit{close} w.r.t. the characteristic time scale $(T-t)^{\frac 12} $ for the homogeneous metric $\d $. Indeed, if $c_0^{\frac 12 }\d(\x,\x')>(T-t)^{\frac 12}   $, equation \eqref{CTR_D2_PROV_UM} readily gives:
\begin{eqnarray}\label{ineq_D2u_T_sup_dxx}
|D_{\x_1}^2 u_m(t,\x)-D_{\x_1}^2 u_m(t,\x')|&\le& |D_{\x_1}^2 u_m(t,\x)|+|D_{\x_1}^2 u_m(t,\x')|\nonumber \\
&\le& 2 (T-t)^{\frac \gamma 2}( C\|f\|_{L^\infty(C_{b,\d}^\gamma)}+\textcolor{black}{\Lambda} \|u_m\|_{L^\infty(C_{b,\d}^{2+\gamma})}) \nonumber \\
&\le& 2 \textcolor{black}{c_0^{\frac \gamma 2}}\d^\gamma(\x,\x')( C\|f\|_{L^\infty(C_{b,\d}^\gamma)}+\textcolor{black}{\Lambda} \|u_m\|_{L^\infty(C_{b,\d}^{2+\gamma})}).
\end{eqnarray}}
\textcolor{black}{Note that the above bound is now critical in the sense discussed above.} Let us now focus, as before, on the H\"older control associated with the perturbative contribution in \eqref{DUHAMEL_PERTURB_DER_NON_DEG} when $c_0^{\frac 12}\d(\x,\x')\le (T-t)^{\frac 12} $. Namely,
\begin{eqnarray}
\label{DECOU_PREAL_HOLDER_PERTURB}
D_{\x_1}^2  \Delta_m^{\tau,\bxi,\bxi'}(t,T,\x,\x') &:=&
\int_t^T ds \int_{\R^{nd}} D_{\x_1}^2 \tilde p^{m,(\tau,\bxi)}(t,s,\x,\y)(L_s^m-\tilde L_s^{m,(\tau,\bxi)})u_m(s,\y) d\y \notag \\
&&-\int_t^T ds \int_{\R^{nd}} D_{\x_1}^2 \tilde p^{m,(\tau,\bxi')}(t,s,\x',\y)(L_s^m-\tilde L_s^{m,(\tau,\bxi')})u_m(s,\y) d\y ,
\end{eqnarray}
where we recall that \textit{a priori} the spatial freezing points $(\bxi,\bxi') $ in \eqref{DECOU_PREAL_HOLDER_PERTURB} (see also \eqref{DUHAMEL_PERTURB_DER_NON_DEG}) should be different for $ \x$ and $\x'$ and 
 depend on the position of $\d(\x,\x') $ w.r.t. the current characteristic time scale in the time integral. Following the terminology of heat kernels, we will say that at time $s\in [t,T] $ the points $\x,\x' $ are in the \textit{diagonal regime} if $c_0\d^2(\x,\x')\le s-t $, i.e. their homogeneous distance is small w.r.t. the characteristic time for a parameter $c_0 $ to be specified later on. 

We insist again that we have the usual equivalence between time and space, i.e. time has to be compared with the square of the spatial metric \textcolor{black}{$\d$}. Similarly, we will say that  the \textit{off-diagonal} regime holds when $c_0\d^2(\x,\x')>(s-t) $. Observing that in the diagonal case $s\ge t+c_0\d^2(\x,\x') $ (and in the off-diagonal one $ s< t+c_0\d^2(\x,\x')$) we split the time integral in \eqref{DECOU_PREAL_HOLDER_PERTURB} as:
\begin{equation*}
\Delta_m^{\tau,\bxi,\bxi'}(t,T,\x,\x'):=\Delta_{m,{\rm \textbf{diag}}}^{\tau,\bxi,\bxi'}(t,T,\x,\x')+\Delta_{m,{\rm \textbf{off-diag}}}^{\tau,\bxi,\bxi'}(t,\x,\x'),
\end{equation*}
with
\begin{eqnarray}
\Delta_{m,{\rm \textbf{off-diag}}}^{\tau,\bxi,\bxi'}(t,\x,\x')&:=&\int_t^{t+c_0\d^2(\x,\x')}\!\!\!\!\!\!\!\!\!\! ds \int_{\R^{nd}}  \tilde p^{m,(\tau,\bxi)}(t,s,\x,\y)(L_s^m-\tilde L_s^{m,(\tau,\bxi)})u_m(s,\y) d\y\notag\\
&&-\int_t^{t+c_0\d^2(\x,\x')} \!\!\!\!\!\!\!\!\!\! ds \int_{\R^{nd}} \tilde p^{m,(\tau,\bxi')}(t,s,\x',\y)(L_s^m-\tilde L_s^{m,(\tau,\bxi')})u_m(s,\y) d\y,\notag\\
\Delta_{m,{\rm \textbf{diag}}}^{\tau,\bxi,\bxi'}(t,T,\x,\x')&:=&\int_{t+c_0\d^2(\x,\x')}^T\!\!\!\!\!\!\!\! ds \int_{\R^{nd}}  \tilde p^{m,(\tau,\bxi)}(t,s,\x,\y)(L_s^m-\tilde L_s^{m,(\tau,\bxi)})u_m(s,\y) d\y\notag\\
&&-\int_{t+c_0\d^2(\x,\x')}^T \!\!\!\!\!\!\!\!ds \int_{\R^{nd}} \tilde p^{m,(\tau,\bxi')}(t,s,\x',\y)(L_s^m-\tilde L_s^{m,(\tau,\bxi')})u_m(s,\y) d\y.\notag\\
\label{DECOUP_MOD_HOLDER}
\end{eqnarray}
Intuitively, for the term $ D_{\x_1}^2 \Delta_{m,{\rm \textbf{off-diag}}}^{\tau,\bxi,\bxi'}(t,\x,\x')$, since $\x,\x' $ are \textit{far} at the characteristic 
time scale $(s-t)^{\frac12}$, there is no expectable gain in expanding $D_{\x_1}^2 \tilde p^{m,(\tau,\bxi')}(t,s,\x',\y) -D_{\x_1}^2 \tilde p^{m,(\tau,\bxi)}(t,s,\x,\y)$. One therefore writes:
\begin{eqnarray}
&&| D_{\x_1}^2 \Delta_{m,{\rm \textbf{off-diag}}}^{\tau,\bxi,\bxi'}(t,\x,\x')|\notag \\
&\le& |\int_t^{t+c_0\d^2(\x,\x')} ds \int_{\R^{nd}}D_{\x_1}^2  \tilde p^{m,(\tau,\bxi)}(t,s,\x,\y)(L_s^m-\tilde L_s^{m,(\tau,\bxi)})u_m(s,\y) d\y|\notag \\
&&+|\int_t^{t+c_0\d^2(\x,\x')} ds \int_{\R^{nd}}D_{\x_1}^2   \tilde p^{m,(\tau,\bxi')}(t,s,\x',\y)(L_s^m-\tilde L_s^{m,(\tau,\bxi')})u_m(s,\y) d\y|.\notag
\end{eqnarray}
Now, provided $\bxi=\x$, $\bxi'=\x' $ one derives from the previous equation, similarly to \eqref{PREAL_CTR_BESOV_DEG_UG}, \eqref{CTR_BESOV_TO_INTEGRATE_IN_TIME}, that 
\begin{equation}
| D_{\x_1}^2 \Delta_{m,{\rm \textbf{off-diag}}}^{\tau,\bxi,\bxi'}(t,\x,\x')|\le \Lambda \|u_m\|_{L^\infty(C_{b,\d}^{2+\gamma})}  \int_{t}^{t+c_0 \d^2(\x,\x')}\frac{ds}{(s-t)^{1-\frac \gamma 2}}\le \Lambda \|u_m\|_{L^\infty(C_{b,\d}^{2+\gamma})} c_0^\frac\gamma 2{}\d^\gamma(\x,\x').
\label{PREAL_CTR_OFF_DIAG}
\end{equation}
For $c_0$ small enough, we obtain again an estimate that would be compatible with the global bound on $\|u_m\|_{L^\infty(C_{b,\d}^{2+\gamma})} $ stated in Theorem \ref{THEO_SCHAU}.

Turning now to $ D_{\x_1}^2 \Delta_{m,{\rm \textbf{diag}}}^{\tau,\bxi,\bxi'}(t,T,\x,\x')$ one would therefore be tempted to carry on the analysis with the previous freezing points $\bxi=\x$, $\bxi'=\x' $. Intuitively, in the diagonal regime this should not have too much impact. This is only partly true, since if we proceed so we will be led to investigate the difference of operators at different freezing spatial points and this leads to compare quantities like $\btheta_{s,t}^m(\x)-\btheta_{s,t}^m(\x') $ for which we want a uniform  control w.r.t. $m$. Since the initial (unmollified) coefficients $a,\gF$ are only H\"older continuous in space,  this quantity is typically controlled (see \textcolor{black}{Lemma \ref{lem_theta_theta}}) as:
\begin{equation}
\label{CTR_DIFF_SYST_MOL}
\d\big(\btheta_{s,t}^m(\x),\btheta_{s,t}^m(\x')\big)\le C\big( \d(\x,\x')+(s-t)^{\frac 12}\big),
\end{equation}
\textcolor{black}{where the time contribution is precisely due to the quasi-distance $\d$ (see the proof of Lemma \ref{lem_theta_theta} in Appendix \ref{sec:ProofLem_d_theta}).}

Unfortunately, this approach would lead to a final control of order $ \big( \d(\x,\x')+(s-t)^{\frac 12}\big)^\gamma\le\textcolor{black}{C}\big ( \d^\gamma(\x,\x')+(s-t)^{\frac \gamma 2} \big )$  which is not enough on the considered integration set. Recall indeed that, in the diagonal regime $c_0\d^2(\x,\x')\le (s-t) $ and the term $(s-t)^{\frac \gamma 2}$ in the previous r.h.s. is \textit{too big}. This means that for $\Delta_{m,{\rm \textbf{diag}}}^{\tau,\bxi,\bxi'}(t,T,\x,\x') $, it would be more appropriate to consider the same spatial freezing point. In that case, taking $\bxi=\bxi'=\x $ and expanding the difference of the derivatives of the frozen Gaussian densities yields:
\begin{eqnarray}
&&\Delta_{m,{\rm \textbf{diag}}}^{\tau,\bxi,\bxi'}(t,T,\x,\x')\notag\\
&=&-\sum_{j=1}^n \int_{t+c_0 \d^2(\x,\x')}^{T} ds \int_{\R^{nd}} d\y \int_0^1 d\mu  D_{\x_j} D_{\x_1}^2 \tilde p^{m,(\tau,\bxi)}(t,s,\x+\mu (\x'-\x),\y)\cdot  (\x'-\x)_j \notag\\
&& (L_s^m-\tilde L_s^{m,(\tau,\bxi)} )u_m(s,\y) \notag\\
&=&-\sum_{j=1}^n \int_{t+c_0 \d^2(\x,\x')}^{T} ds \int_{\R^{nd}} d\y \int_0^1 d\mu  D_{\x_j} D_{\x_1}^2 \tilde p^{m,(\tau,\bxi)}(t,s,\x+\mu (\x'-\x),\y)\cdot  (\x'-\x)_j \notag\\
&& \Bigg(\Delta_{1,\gF_m,\sigma_m}(t,s,\btheta_{s,t}^m(\x),\y,u_m) 
+\sum_{i=2}^n  \big \langle \Delta_{i,\gF_m,\sigma_m}(t,s,\btheta_{s,t}^m(\x),\y) D_{\y_i}u_m(s,\y) \big \rangle \Bigg), \label{ANA_PREAL_HD_UG}
\nonumber  \\
\end{eqnarray}
using the notations introduced in \eqref{DEF_DIFF_NON_DEG} and \eqref{DEF_TERMES_DEG} for the last equality. 
\textcolor{black}{In the previous identities \eqref{ANA_PREAL_HD_UG} and from now on, the symbol ``$\cdot$" between two tensors means the usual tensor contraction. In particular $D_{\x_j} D_{\x_1}^2 \tilde p^{m,(\tau,\bxi)}(t,s,\x+\mu (\x'-\x),\y)\cdot  (\x'-\x)_j$ is a $d \times d$ matrix.}

In the current diagonal regime, 
it can be shown from \eqref{FIRST_CTR_DENS} and the homogeneity of the distance $\d $ that \textcolor{black}{there is $C>1$ such that} for \textcolor{black}{$(\tau,\bxi)=(t,\x)$}:
\begin{eqnarray}
&&|D_{\x_j}D_{\x_1}^2 \tilde p^{m,(\tau,\bxi)}(t,s,\x+\mu (\x'-\x),\y)|\Big|_{\textcolor{black}{(\tau,\bxi)=(t,\x)}}\notag\\
&\le& \frac{C}{(s-t)^{j-\frac 12+ 1+\frac{n^2d}2}}\exp( -C^{-1}(s-t)|\T_{s-t}^{-1}(\m_{s,t}^{m,(\textcolor{black}{\tau},\x)} (\x+\mu(\x'-\x))-\y|^2)\Big|_{\textcolor{black}{(\tau,\bxi)=(t,\x)}}\notag\\
&\le& \frac{C}{(s-t)^{j-\frac 12+ 1+\frac{n^2d}2}}\exp(C^{-1}(s-t)|\T_{s-t}^{-1} (\m_{s,t}^{m,(\textcolor{black}{\tau},\bxi)}(\x-\x'))|^2 )
\nonumber \\
&&\times \exp( -\frac{C^{-1}}2(s-t)|\T_{s-t}^{-1}(\m_{s,t}^{m,(\textcolor{black}{\tau},\bxi)} (\x))-\y|^2)\Big|_{\textcolor{black}{(\tau,\bxi)=(t,\x)}}\notag\\
&\le &\frac{C}{(s-t)^{j-\frac 12+ 1+\frac{n^2d}2}}\exp( -\frac{C^{-1}}2(s-t)|\T_{s-t}^{-1}(\btheta_{s,t}^m (\x))-\y|^2),
\label{PREAL_BD_OF_DIAG_M}
\end{eqnarray}
using for the last inequality that $\m_{s,t}^{m,(\textcolor{black}{\tau},\bxi)}(\x)|_{\textcolor{black}{(\tau,\bxi)=(t,\x)}}=\btheta_{s,t}^m(\x) $ and the fact that, from the linear structure of ODE satisfied by $\m_{s,t}^{m,(\textcolor{black}{\tau},\bxi)}(\x) $ (which can be read from system \eqref{FROZ_MOL_FOR} taking the expectation), $(s-t)^{\frac 12}|\T_{s-t}^{-1} (\m_{s,t}^{m,(\textcolor{black}{\tau},\bxi)})(\x-\x')|^2 \le C (s-t)^{\frac 12}|\T_{s-t}^{-1}(\x-\x')|$. Since $c_0 \d^2(\x,\x')\le s-t \Leftrightarrow c_0 \d^{2}\big((s-t)^{\frac 12}\T_{s-t}^{-1}\x,(s-t)^{\frac 12}\T_{s-t}^{-1}\x'\big)\le 1 $, we readily derive from the definition of $\d$ in \eqref{DIST} that $(s-t)|\T_{s-t}^{-1} (\m_{s,t}^{m,(\textcolor{black}{\tau},\bxi)})(\x-\x')|^2\le C 
$. These points are thoroughly discussed in Sections \ref{SEC2}.

From \eqref{ANA_PREAL_HD_UG},  \eqref{PREAL_BD_OF_DIAG_M} reproducing the previously described analysis, we finally derive:
\begin{eqnarray}
 &&D_{\x_1}^2 \Delta_{m,{\rm \textbf{diag}}}^{\tau,\bxi,\bxi'}(t,T,\x,\x')
 \nonumber \\
 &\le& \Lambda\|u_m\|_{L^\infty(C_{b,\d}^{2+\gamma})} \sum_{j=1}^n\int_{t+c_0\d^2(\x,\x')}^T \frac{ds}{(s-t)^{1+(j-\frac 12)- \frac \gamma 2}} |(\x-\x')_j|\notag\\
&\le & \Lambda\|u_m\|_{L^\infty(C_{b,\d}^{2+\gamma})} \sum_{j=1}^n \frac{|(\x-\x')_j|}{(c_0\d^2(\x,\x'))^{j-\frac 12-\frac \gamma 2}}\notag\\
&\le&  \frac{\Lambda}{c_0^{n-\frac 12-\frac \gamma 2}}\|u_m\|_{L^\infty(C_{b,\d}^{2+\gamma})} \d^\gamma (\x,\x'), 
\label{CTR_DIAG_PREAL}
\end{eqnarray}
using again the definition of $\d$ in \eqref{DIST} for the last inequality and where $\Lambda $ is an in Remark \ref{REM_LAMBDA}. We have again globally gained, thanks to the smoothness of the coefficients, a power $\frac \gamma 2 $ in the time singularities of equation \eqref{PREAL_BD_OF_DIAG_M}.\\


From the previous discussion we now have to specify how to modify the freezing parameter depending on the position of the current time variable w.r.t. to the homogeneous spatial distance between the considered points. This can actually been done from the Duhamel formulation up to an additional discontinuity term. Restarting from \eqref{DUHAMEL_PERTURB} we can indeed rewrite for given $(t,\x')\in [0,T]\times \R^{nd} $ and all $ r\in (t,T]$, $\bxi'\in \R^{nd}$:
\begin{eqnarray}\label{GREEN_SUR_SEGMENT_EN_TEMPS}
u_m(t,\x')&=&\tilde P_{r,t}^{m,(\tau,\bxi')} u_m(r, \x')+ \tilde G_{r,t}^{m,(\tau,\bxi')} f_m(t,\x')\notag\\
&&+\int_t^{r} ds \int_{\R^{nd}} d\y\tilde p^{m,(\tau,\bxi')}(t,s,\x',\y)(L_s^m-\tilde L_s^{m,(\tau,\bxi')})u_m(s,\y) , \\
\forall 0\le v<r\le T,\ \tilde G_{r,v}^{m,(\tau,\bxi')} f_m(t,\x)&=&\int_v^{r} ds\int_{\R^{nd}} d\y \tilde p^{m,(\tau,\bxi')}(t,s,\x',\y)f_m(s,\y).
\notag %
\end{eqnarray}
Differentiating the above expression in $r\in (t,T]$ yields for any $\bxi'\in \R^{nd} $:
\begin{eqnarray}
\label{DER_DUHAMEL}
0=\partial_r \tilde P_{r,t}^{m,(\tau,\bxi')} u_m(r, \x')+\int_{\R^{nd} }d\y\tilde p^{m,(\tau,\bxi')}(t,r,\x',\y)f_m(r,\y)\notag\\
+\int_{\R^{nd}} d\y \tilde p^{m,(\tau,\bxi')}(t,r,\x',\y)(L_r^m-\tilde L_r^{m,(\tau,\bxi')})u_m(r,\y) d\y.
\end{eqnarray}
Integrating \eqref{DER_DUHAMEL} between $t $ and $t_0\in (t,T]$ for a first given $\bxi' $ and between $t_0 $ and $T$ with a possibly different $\tilde \bxi'  $ yields:
\begin{eqnarray}
0&=&\tilde P_{t_0,t}^{m,(\tau,\bxi')} u_m(t_0, \x')-u_m(t,\x')+\int_{t}^{t_0}ds \int_{\R^{nd}}d\y \tilde p^{m,(\tau,\bxi')}(t,s,\x',\y)f_m(s,\y)\notag\\
&&+\int_t^{t_0} ds \int_{\R^{nd}}d\y \tilde p^{m,(\tau,\bxi')}(t,s,\x',\y)(L_s^m-\tilde L_s^{m,(\tau,\bxi')})u_m(s,\y)\notag\\
&&+\tilde P_{T,t}^{m,(\tau,\tilde \bxi')} u_m(T, \x')-\tilde P_{t_0,t}^{m,(\tau,\tilde \bxi')}u_m(t_0,\x')+\int_{t_0}^Tds \int_{\R^{nd}}d\y \tilde p^{m,(\tau,\tilde \bxi')}(t,s,\x',\y)f_m(s,\y)\notag\\
&&+\int_{t_0}^T ds \int_{\R^{nd}}d\y \tilde p^{m,(\tau,\tilde\bxi')}(t,s,\x',\y)(L_s^m-\tilde L_s^{m,(\tau,\tilde \bxi')})u_m(s,\y)\notag.
\end{eqnarray}
Recalling that $u_m(T,\x')=g_m(\x') $ (terminal condition), and with the notations of \eqref{GREEN_SUR_SEGMENT_EN_TEMPS} the above equation \textcolor{black}{becomes}:
\begin{eqnarray}
u_m(t,\x')&=&\tilde P_{T,t}^{m,(\tau,\tilde \bxi')}g_m(\x')+\tilde G_{t_0,t}^{m,(\tau,\bxi')} f_m(t,\x')+\tilde G_{T,t_0}^{m,(\tau,\tilde \bxi')} f_m(t,\x')\notag\\
&&+\tilde P_{t_0,t}^{m,(\tau,\bxi')} u_m(t_0, \x')-\tilde P_{t_0,t}^{m,(\tau,\tilde \bxi')} u_m(t_0, \x')\notag\\
&&+\int_t^T ds \int_{\R^{nd}}d\y\Big(\I_{s\le t_0}\tilde p^{m,(\tau,\bxi')}(t,s,\x',\y)(L_s^m-\tilde L_s^{m,(\tau, \bxi')})\notag\\
&&\quad +\I_{s>t_0} \tilde p^{m,(\tau,\tilde\bxi')}(t,s,\x',\y)(L_s^m-\tilde L_s^{m,(\tau,\tilde \bxi')})\Big)  u_m(s,\y).
\label{INTEGRATED_DIFF_BXI}
\end{eqnarray}
We see that for $\bxi ' \neq\tilde  \bxi ' $ we have an additional discontinuity term deriving from the change of freezing point along the time variable. Of course expression \eqref{INTEGRATED_DIFF_BXI} can be differentiated in space and taking  \textbf{then}
\begin{equation}\label{def_t0}
t_0=t+c_0 \d^2(\x,\x'),
\end{equation}
i.e. \textcolor{black}{$t_0$ precisely corresponds to the critical time at which a change of regime occurs},
 and $\bxi'=\x'$, $\tilde \bxi'=\x $ precisely allows, when expanding $D_{\x_1}^2u_m(t,\x)-D_{\x_1}^2 u_m(t,\x')$ using \eqref{DUHAMEL_PERTURB_DER_NON_DEG} for the first term and \eqref{INTEGRATED_DIFF_BXI} for the second one, to exploit the previous analysis that led to \eqref{PREAL_CTR_OFF_DIAG} and \eqref{CTR_DIAG_PREAL} and which relied on the suitable choice of freezing point. We again insist on the fact that, in the analysis, $t_0$ is an additional freezing parameter, which is \textit{a posteriori} chosen according to \eqref{def_t0} as a function of $(t,\x,\x') $. In particular the parameter $t_0$ does not intervene in the various possible differentiations of the considered perturbative expansions.
 
 This approach eventually leads to:
\begin{eqnarray}
&&|D_{\x_1}^2u_m(t,\x)-D_{\x_1}^2u_m(t,\x')|\notag\\
&\le&  \d^{\gamma}(\x,\x')\Big[ C\big(\|g\|_{C_{b,\d}^{2+\gamma}}+\|f\|_{L^\infty(C_{b,\d}^\gamma)}\big) +\Lambda \big( c_0^{-n+\frac 12+\frac \gamma 2}+c_0^{\frac\gamma 2}\big) \|u_m\|_{L^\infty(C_{b,\d}^{2+\gamma})}  \Big]\notag\\
&&+ \Big|\big(D_{\x_1}^2\tilde P_{t_0,t}^{m,(\tau,\bxi')} u_m(t_0, \x')-D_{\x_1}^2\tilde P_{t_0,t}^{m,(\tau,\tilde \bxi')} u_m(t_0, \x')\big)\big|_{t_0=t+c_0\d^2(\x,\x')}\Big|\label{PREAL_DISC}.
\end{eqnarray}
The last contribution can be controlled through cancellation techniques and the key estimate \eqref{CTR_DIFF_SYST_MOL} on the difference of the flows. The specific choice of $t_0=t+c_0\d^2(\x,\x') $ then precisely provides the required order leading to:
\begin{eqnarray*}
&&|D_{\x_1}^2u_m(t,\x)-D_{\x_1}^2u_m(t,\x')|\\
&\le&  \d^{\gamma}(\x,\x')\Big[ C\big(\|g\|_{C_{b,\d}^{2+\gamma}}+\|f\|_{L^\infty(C_{b,\d}^\gamma)}\big) +\Big(\Lambda\big( c_0^{-n+\frac 12+\frac \gamma 2}+c_0^{\frac\gamma 2}\big)+ C c_0^{\textcolor{black}{\frac \gamma{2n-1}}}\Big) \|u_m\|_{L^\infty(C_{b,\d}^{2+\gamma})}  \Big].
\end{eqnarray*}
We refer to Lemma \ref{CTR_TERME_DISC} for results associated with the discontinuity term in \eqref{PREAL_DISC}.\\

\subsection{Conclusion and outline of the derivation of estimate in Theorem \ref{THEO_SCHAU}.} We have detailed up to now what happens with the second order derivatives w.r.t. the non-degenerate variables. The previous procedure can be applied as well to control the H\"older moduli w.r.t. the degenerate ones. We therefore end up with the following kind of estimate:
\begin{eqnarray}
\|u_m\|_{L^\infty(C_{b,\d}^{2+\gamma})}&\le&  C\big(\|g\|_{C_{b,\d}^{2+\gamma}}+\|f\|_{L^\infty(C_{b,\d}^\gamma)}\big)+\|u_m\|_{L^\infty(C_{b,\d}^{2+\gamma})}\Big[\Lambda\big( c_0^{-n+\frac 12+\frac \gamma 2}+c_0^{\frac\gamma 2}+T^{\frac \gamma 2}\big) +Cc_0^{\frac \gamma{2n-1}}  \Big].\notag \\
\label{PREAL_FINAL}
\end{eqnarray}
Equation \eqref{PREAL_FINAL} would provide exactly the expected control if $\Lambda$ and $c_0$ are small enough. On the one hand, the final parameter $c_0$ can always be chosen small enough (cutting threshold).
On the other hand, it will appear from the proofs that the constant $\Lambda$ in \eqref{PREAL_FINAL} actually depends on the H\"older norms of the considered coefficients (see again Remark \ref{REM_LAMBDA}). If these quantities are small, i.e. the coefficients do not vary much and the components that transmit the noise are almost linear, then $\Lambda$ will be small. For $\Big[\Lambda\big( c_0^{-n+\frac 12+\frac \gamma 2}+c_0^{\frac\gamma 2} +T^{\frac \gamma 2}\big)+Cc_0^{\frac \gamma{2n-1}}  \Big]\le k_0<1 $, we eventually derive:
\begin{equation}
\label{SCHAU_UG}
\|u_m\|_{L^\infty(C_{b,\d}^{2+\gamma})}\le \frac{C}{1-k_0}\big(\|g\|_{C_{b,\d}^{2+\gamma}}+\|f\|_{L^\infty(C_{b,\d}^\gamma)}\big),
\end{equation}
which is precisely the expected control. The general case, is proved through a scaling argument which also allows to balance the opposite effects of $c_0 $ (meant to be small, in particular $c_0\le 1$) in the above bounds. This last point will be discussed in Section \ref{scaling}. 
\\

\subsection{Organization of this paper.} 
The remaining part of this article is organized as follows. We prove in Section \ref{SEC2} various properties for the density of the linearized Gaussian proxy:  precise pointwise estimates for the density itself and its  derivatives (see equation  \eqref{FIRST_CTR_DENS}) and some useful controls allowing cancellation arguments in our perturbative analysis. 
Section \ref{Sectio_Sup_D2} is then devoted to the control of the supremum norms of the non-degenerate derivatives, corresponding to the previous equation \eqref{CTR_D2_PROV_UM}. Section \ref{HOLDER} addresses the issues of H\"older controls. 
 Section \ref{scaling} is concerned with the above mentioned scaling issues and we also conclude there the final proof of Theorem \ref{THEO_SCHAU}. Eventually, some auxiliary, but crucial, technical results are proved in Appendix \ref{APP} \textcolor{black}{ for the regularity of the flow and the mean, in Appendix \ref{sec_resolv_cov} for the regularity of the resolvent and the covariance,  and in Appendix \ref{SEC_BESOV_DUAL_FIRST_Scalling} for the technical points related to the scaling analysis}. 

\mysection{Gaussian proxy and associated controls}\label{SEC2}

We first aim here at proving the control of equation \eqref{FIRST_CTR_DENS}. We recall that our point is to control the density of $(\tilde \X_{s}^{m,(\tau,\bxi)})_{s\in (t,T]} $ satisfying \eqref{FROZ_MOL_FOR}.\\

\textbf{WARNING:} for notational simplicity, for the rest of the document we drop the sub and superscripts in $m$ associated with the regularizations. We rewrite, with some notational abuse, for fixed $(\tau,\bxi) \in [0,T]\times \R^{nd} $,  the dynamics in \eqref{FROZ_MOL_FOR} as:
\begin{eqnarray} 
&&\hspace*{-.5cm}d\tilde \X_v^{(\tau,\bxi)}=[\gF(v,\btheta_{v,\tau}(\bxi))+ D\gF(v,\btheta_{v,\tau}(\bxi))(\tilde \X_v^{(\tau,\bxi)}-\btheta_{v,\tau}(\bxi))]dv +B\sigma(v,\btheta_{v,\tau}(\bxi)) dW_v,\nonumber\\
&& \hspace*{.75cm}\forall v\in  [t,s],\
 \tilde \X_t^{(\tau,\bxi)}=\x, \label{FROZ_MOL_FOR_NO_SUPER_SCRIPTS}
\end{eqnarray}
keeping in mind that $ \gF,\btheta,\sigma$ in \eqref{FROZ_MOL_FOR_NO_SUPER_SCRIPTS} are \textit{smooth} coefficients.
We will give in the next subsection some key-controls to investigate the terms appearing the perturbative expansions \eqref{DUHAMEL_PERTURB} and \eqref{DUHAMEL_PERTURB_DER_NON_DEG}.


\subsection{Controls for the frozen density} \label{SEC_GAUSS_DENS}
 We explicitly integrate \eqref{FROZ_MOL_FOR_NO_SUPER_SCRIPTS} to obtain for any $v\in [t,s] $:
 \begin{equation}
 \label{INTEGRATED}
\begin{split}
\tilde  \X_v^{(\tau,\bxi)}&=\textcolor{black}{\tilde\gR}^{(\tau,\bxi)}(v,t)\x+\int_t^v  \textcolor{black}{\tilde \gR}^{(\tau,\bxi)}(v,u)\Big( \gF(u,\btheta_{u,\tau}(\bxi))-D\gF(u,\btheta_{u,\tau}(\bxi))\btheta_{u,\tau}(\bxi)\Big )du\\
 & +\int_t^v \textcolor{black}{\tilde \gR}^{(\tau,\bxi)}(v,u)B\sigma(u,\btheta_{u,\tau}(\bxi)) dW_u\notag\\
&=:\m_{v,t}^{(\tau,\bxi)}(\x)+\int_t^v \textcolor{black}{\tilde \gR}^{(\tau,\bxi)}(v,u)B\sigma(u,\btheta_{u,\tau}(\bxi)) dW_u,
 \end{split}
 \end{equation}
where $(\tilde \gR^{(\tau,\bxi)}(v,u))_{t\le u,v \le s} $ stands for the resolvent associated with the collection of partial gradients in $(D {\mathbf F}(v,\btheta_{v,\tau}(\bxi)))_{ v\in [t,s]} $, introduced in \eqref{DEF_PARTIAL_GRADIENTS}, which satisfies for $v\in [t,s]$:
\begin{equation}
\begin{split}
\partial_{v}\tilde \gR^{(\tau,\bxi)}(v,t)&={ D \gF}(v,\btheta_{v,\tau}(\bxi))\tilde \gR^{(\tau,\bxi)}(v,t) , 
\ \tilde \gR^{(\tau,\bxi)}(t,t)=\mathbf I_{nd, nd}.
\end{split}
\label{DYN_RES_FORWARD}
\end{equation}
Note in particular that since the partial gradients are subdiagonal $ {\rm det}(\tilde \gR^{(\tau,\bxi)}(v,t))=1$.\

Also, for $v\in [t,s]$, we recall that  $\m_{v,t}^{(\tau,\bxi)}(\x) $ stands for the mean of $\tilde  \X_v^{(\tau,\bxi)} $ and  corresponds as well  to the solution of \eqref{FROZ_MOL_FOR_NO_SUPER_SCRIPTS} when $\sigma =0 $ and the  starting point is $\x $. We write:
 \begin{equation}
 \label{INTEGRATED_FLOW}
\tilde  \X_v^{(\tau,\bxi)}=\m_{v,t}^{(\tau,\bxi)}(\x) +\int_t^v \textcolor{black}{\tilde \gR}^{(\tau,\bxi)}(v,u)B\sigma(u,\btheta_{u,\tau}(\bxi)) dW_u, \ v\in [t,s].
\end{equation}
Importantly, we point out that $\x \in \R^{nd}\mapsto \m_{v,t}^{(\tau,\bxi)}(\x)$ is \textbf{affine}  w.r.t. the starting point $\x$. Precisely, for $\x,\x'\in \R^{nd} $:
\begin{equation}
\label{AFFINE_FLOW}
\m_{v,t}^{(\tau,\bxi)}(\x+\x')=\tilde \gR^{(\tau,\bxi)}(v,t)\x'+ \m_{v,t}^{(\tau,\bxi)}(\x).
\end{equation}

We first give in the next proposition a key estimate on the covariance matrix associated with \eqref{INTEGRATED_FLOW} and its properties w.r.t. a suitable scaling of the system. 
\begin{PROP}[Good Scaling Properties of the Covariance Matrix]
\label{PROP_SCALE_COV}
The covariance matrix of $\tilde  \X_v^{(\tau,\bxi)}$ in \eqref{INTEGRATED_FLOW} writes:
$$\tilde \K_{v,t}^{(\tau,\bxi)}:=\int_t^v \textcolor{black}{\tilde  \gR}^{(\tau,\bxi)}(v,u) Ba(u,\btheta_{u,\tau}(\bxi))B^*\textcolor{black}{\tilde \gR}^{(\tau,\bxi)}(v,u)^* du. 
$$
Uniformly in $(\tau,\bxi)\in [0,T]\times \R^{nd} $ and $s\in [0,T] $, it satisfies a \textit{good scaling property} in the sense of Definition 3.2 in \cite{dela:meno:10} (see also Proposition 3.4 of that reference). That is, for any fixed $T>0$, there exists $C_{\ref{GSP}}:=C_{\ref{GSP}}(\A{A},T)\ge 1$ s.t. for all $0\le t<v \le s\le T $, $(\tau,\bxi) \in [0,T]\times \R^{nd} $:
\begin{equation}
\label{GSP}
\forall \bzeta \in \R^{nd},\   C_{\ref{GSP}}^{-1} (v-t)^{-1}|\T_{v-t} \bzeta|^2\le \langle  \tilde \K_{v,t}^{(\tau,\bxi)}\bzeta,\bzeta\rangle \le C_{\ref{GSP}} (v-t)^{-1}|\T_{v-t} \bzeta|^2,
\end{equation} 
where for any $u>0$, we denote by $\T_u $ the \emph{intrinsic} scale matrix introduced in \eqref{DEF_T_ALPHA}. Namely:
\begin{equation*}
\T_u=\left( \begin{array}{cccc}
u\mathbf I_{d, d}& \0_{d, d}& \cdots& \0_{d, d}\\
\0_{d, d}   &u^2 \mathbf I_{d, d}&\0_{d,d}& \vdots\\
\vdots & \ddots&\ddots & \vdots\\
\0_{d, d}& \cdots & \0_{d,d}& u^{n}\mathbf I_{d , d}
\end{array}\right).
\end{equation*}
\end{PROP}
The proof of the above proposition readily follows from Proposition 3.3 and Lemma 3.6 in \cite{dela:meno:10}.
We now state some important density bounds for the linearized model. 
\begin{PROP}[Density of the linearized dynamics]\label{THE_PROP}
Under \A{A}, we have that, for any $s\in (t,T]$ the random variable $\tilde  \X_s^{(\tau,\bxi)} $ in \eqref{INTEGRATED_FLOW} admits a Gaussian density $\tilde p^{(\tau,\bxi)}(t,s,\x,\cdot) $ which writes for any $ \y \in \R^{nd}$:
\begin{equation}\label{CORRESP}
 \tilde p^{(\tau,\bxi)}(t,s,\x,\y):=\frac{1}{(2\pi)^{\frac{nd}2}\det(\tilde \K_{s,t}^{(\tau,\bxi)})^{\frac 12}}\exp\left( -\frac 12 \left\langle (\tilde \K_{s,t}^{(\tau,\bxi)})^{-1} (\m_{s,t}^{(\tau,\bxi)}(\x)-\y),\m_{s,t}^{(\tau,\bxi)}(\x)-\y\right\rangle\right),
\end{equation}
with $\tilde \K_{s,t}^{(\tau,\bxi)} $ as in Proposition \ref{PROP_SCALE_COV}.
Also, there exists $C:=C(\A{A},T)>0$ s.t. for each multi-index $\vartheta=(\vartheta_1,\cdots,\vartheta_n)\in \N^n,\ |\vartheta|\le 3$ and denoting by $D_\x^\vartheta:=D_{\x_1}^{\vartheta_1}\cdots D_{\x_n}^{\vartheta_n}  $, we have: 
\begin{eqnarray}
|D_{\x}^\vartheta\tilde p^{(\tau,\bxi)}(t,s,\x,\y)|&\le& \frac{C}{(s-t)^{\sum_{i=1}^n \vartheta_i(i-\frac{1}{2}) +\frac{n^2d}2}}\exp\left(-C^{-1}(s-t) \big|\T_{s-t}^{-1}\big(\m_{s,t}^{(\tau,\bxi)}(\x)-\y\big)\big|^2\right)\notag\\
&=:&\frac{C}{(s-t)^{\sum_{i=1}^n \vartheta_i(i-\frac{1}{2})}}\bar p_{C^{-1}}^{(\tau,\bxi)}(t,s,\x,\y),\label{CTR_GRAD}
\end{eqnarray}
with $\int_{\R^{nd}} d\y \bar p_{C^{-1}}^{(\tau,\bxi)}(t,s,\x,\y)$,  up to a modification of the constants in \eqref{CTR_GRAD}.
\end{PROP}

\begin{REM}[A slight abuse of notation] \label{NOTA_ABUSE} To ease the reading we denote, when there is no possible ambiguity, $\bar p_{C^{-1}}(t,s,\x,\y):=\bar p_{C^{-1}}^{(t,\x)}(t,s,\x,\y)$.
\end{REM}
\begin{REM}[Regularizing effect of the quasi-distance]\label{THE_REM_REG}
From equation \eqref{CTR_GRAD}, we derive from the definition of $\d $ in \eqref{DIST} that for any given $\beta>0 $, there exists $C_\beta$ s.t.
\begin{equation}
\d^{\textcolor{black}{\beta}}\big ((s-t)^{\frac 12}\T_{s-t}^{-1} \m_{s,t}^{(\tau,\bxi)}(\x),(s-t)^{\frac 12}\T_{s-t}^{-1} \y \big ) |D_{\x}^\vartheta\tilde p^{(\tau,\bxi)}(t,s,\x,\y)|
\le \frac{C_\beta}{(s-t)^{\sum_{i=1}^n \vartheta_i(i-\frac{1}{2}) -\frac \beta 2} }
\bar p_{C_\beta^{-1}}^{(\tau,\bxi)}(t,s,\x,\y),\label{CTR_GRAD_ET_DIST}
\end{equation}
i.e. equation \eqref{CTR_GRAD_ET_DIST} quantifies the regularizing effect of the scaled arguments in the quasi-distance.
\end{REM}
\begin{proof} 
Expression \eqref{CORRESP} readily follows from \eqref{INTEGRATED} and  \textcolor{black}{\eqref{INTEGRATED_FLOW}}. 
Differentiating w.r.t. $\x$ recalling from \eqref{AFFINE_FLOW} that $\x\mapsto 
\m_{s,t}^{(\tau,\bxi)}(\x) $ is affine yields:
\begin{eqnarray}
\label{EXPLICIT_DERIVATIVES}
D_{\x_j} \tilde p^{(\tau,\bxi)}(t,s,\x,\y)=- \Big[ \big[{\tilde  \gR}^{(\tau,\bxi)}(s,t)\big]^* (\tilde \K_{s,t}^{(\tau,\bxi)})^{-1} (\m_{s,t}^{(\tau,\bxi)}(\x)-\y) \Big]_j \tilde p^{(\tau,\bxi)}(t,s,\x,\y).
\end{eqnarray}
The point is now to use scaling arguments.
We can first rewrite 
\begin{equation}\label{FIRST_SCALING}
\big[{\tilde  \gR}^{(\tau,\bxi)}(s,t)\big]^* (\tilde \K_{s,t}^{(\tau,\bxi)})^{-1}=(s-t) \big[{\tilde  \gR}^{(\tau,\bxi)}(s,t)\big]^* \T_{s-t}^{-1}(\widehat { \tilde\K_1^{s,t}})^{-1}\T_{s-t}^{-1},
\end{equation}
 where $\widehat {\tilde \K_1^{s,t}} $ is the covariance matrix of the rescaled process $\big((s-t)^{\frac 12}\T_{s-t}^{-1}\tilde \X_{t+v(s-t)}^{t,\x}\big)_{v\in [0,1]} $ at time $1$. From the good-scaling property of Proposition \ref{PROP_SCALE_COV}, it is plain to derive that
$\widehat {\tilde \K_1^{s,t}} $
 a non-degenerate bounded matrix, i.e. there exists $\hat C\ge 1 $ s.t. for any $\bzeta \in \R^{nd} $, $(\hat C)^{-1} |\bzeta|^2 \le \langle \widehat {\tilde \K_1^{s,t}}\bzeta,\bzeta\rangle \le \hat C |\bzeta|^2$. A similar rescaling argument yields on the deterministic system \eqref{DYN_RES_FORWARD} of the resolvent yields that ${\tilde  \gR}^{(\tau,\bxi)}(s,t) $ can also be written as:
 \begin{equation}\label{RESCALED_RES}
 [{\tilde  \gR}^{(\tau,\bxi)}(s,t)]^*=\T_{s-t}^{-1}\Big[ \widehat{\tilde  \gR}^{(\tau,\bxi),s,t}(1,0)\Big]^*\T_{s-t},
 \end{equation}
where again $\widehat{\tilde  \gR}^{(\tau,\bxi),s,t}(1,0)$ is the resolvent at time 1 of  the rescaled system $$ \Big(\T_{s-t}[{\tilde  \gR}^{(\tau,\bxi)}(t+v(s-t),t)]^*\T_{s-t}^{-1}\Big)_{v\in [0,1]}=\Big(\big[ \widehat{\tilde  \gR}^{(\tau,\bxi),s,t}(v,0)\big]^*\Big)_{v\in [0,1]}$$ associated with \eqref{DYN_RES_FORWARD}. 
 From the analysis performed in Lemma 5.1 in \cite{huan:meno:15} (see also the proof of Proposition 3.7 in \cite{dela:meno:10}) one derives that there exists $\hat C_1$ s.t.  for any $\bzeta \in \R^{nd}$,
 \begin{equation}\label{borne_tilde_R}
  |\big[\widehat{\tilde  \gR}^{(\tau,\bxi),s,t}(1,0)\big]^* \bzeta|\le\hat C_1 |\bzeta| .
 \end{equation} Equations \eqref{EXPLICIT_DERIVATIVES}, \eqref{FIRST_SCALING} and \eqref{RESCALED_RES} therefore yield:
 \begin{eqnarray*}
 &&|D_{\x_j} \tilde p^{(\tau,\bxi)}(t,s,\x,\y)|\\
  &\le&  (s-t)^{-j+\frac 12} \Bigg|\bigg( \Big[ \widehat{\tilde  \gR}^{(\tau,\bxi),s,t}(1,0)\Big]^* \big(\widehat { \tilde\K_1^{s,t}}\big)^{-1}\big((s-t)^{\frac 12} \T_{s-t}^{-1}(\m_{s,t}^{(\tau,\bxi)}(\x)-\y)\big) \bigg)_j\Bigg| \tilde p^{(\tau,\bxi)}(t,s,\x,\y)  \\
  &\le& C (s-t)^{-j+\frac 12}   (s-t)^{\frac 12} | \T_{s-t}^{-1}(\m_{s,t}^{(\tau,\bxi)}(\x)-\y)| \tilde p^{(\tau,\bxi)}(t,s,\x,\y).
 \end{eqnarray*}
 From the explicit expression \eqref{CORRESP}, Proposition \ref{PROP_SCALE_COV} and the above equation, we eventually derive:
 \begin{eqnarray*}
&&  |D_{\x_j} \tilde p^{(\tau,\bxi)}(t,s,\x,\y)|
\nonumber \\
&\le& \frac{C}{(s-t)^{j-\frac 12 }}\Big((s-t)^{\frac 12} | \T_{s-t}^{-1}(\m_{s,t}^{(\tau,\bxi)}(\x)-\y)| \Big)
\frac{1}{(s-t)^{\frac{n^2d} 2}}\exp \left(-C^{-1} (s-t) | \T_{s-t}^{-1}(\m_{s,t}^{(\tau,\bxi)}(\x)-\y)|^2\right)\\
  &\le & \frac{C}{(s-t)^{j-\frac 12}} \bar p_{C^{-1}}(t,s,\x,\y),
\end{eqnarray*}
up to a modification of $C$, 
which gives the statement for one partial derivative. The controls on the higher order derivatives are obtained similarly (see e.g. the proof of Lemma 5.5 of \cite{dela:meno:10} for the bounds on $D_{\x_1}^2 \tilde p^{(\tau,\bxi)}(t,s,\x,\y) $).


\end{proof}

As a direct consequence of Proposition \ref{THE_PROP} we have the following result for the semi-group $\tilde P^{(\tau,\bxi)} $ associated with \eqref{INTEGRATED_FLOW}:
\begin{lem}\label{lemme_nabla_Pg} 
For $\gamma\in (0,1) $, \textcolor{black}{under \A{A}, there exists  $C:=C(\A{A},T)$, s.t.  for any function $\psi\in C_{b,\d}^{\gamma}(\R^{nd},\R) $, and any given multi-index $\vartheta,\ |\vartheta|\in \leftB 1,3\rightB $, for all $0\le t<s\le T,\ \x\in \R^{nd} $:}
\begin{equation}
|  D_{\x}^\vartheta \tilde P_{s,t}^{(\tau,\bxi)} \psi(\x)| \big |_{\bxi=\x}
 \leq 
  C\|\psi\|_{C_{\d}^\gamma}(s-t)^{-\sum_{i=1}^n |\vartheta_i|(i-\frac 12)+\frac \gamma 2}.
\end{equation}
\end{lem}
\begin{proof}
It suffices to write:
\begin{eqnarray*}
D_{\x}^\vartheta \tilde P_{s,t}^{(\tau,\bxi)} \psi(\x)&=&\int_{\R^{nd}} d\y D_\x^\vartheta \tilde p^{(\tau,\bxi)} (t,s,\x,\y) \psi(\y)
=\int_{\R^{nd}} d\y D_\x^\vartheta \tilde p^{(\tau,\bxi)} (t,s,\x,\y) [\psi(\y)-\psi(\m_{s,t}^{(\tau,\bxi)}(\x))],
\end{eqnarray*}
so that from Proposition \ref{THE_PROP} and the smoothness of $\psi $:
\begin{eqnarray*}
|D_{\x}^\vartheta \tilde P_{s,t}^{(\tau,\bxi)} \psi(\x)|&\le &\frac{C}{(s-t)^{\sum_{i=1}^n \vartheta_i (i-\frac 12)}} \|\psi\|_{C_\d^\gamma}\int_{\R^{nd}} \bar p_{C^{-1}}^{(\tau,\bxi)}(t,s,\x,\y) \d^\gamma(\y,\m_{s,t}^{(\tau,\bxi)}(\x)),
\end{eqnarray*}
which yields the result thanks to Remark \ref{THE_REM_REG} recalling from the homogeneity of $\d$ that $ \d^\gamma(\y,\m_{s,t}^{(\tau,\bxi)}(\x))=\big[(s-t)^{\frac 12}\d( (s-t)^{\frac 12}\T_{s-t}^{-1} \y,(s-t)^{\frac 12}\T_{s-t}^{-1} \m_{s,t}^{(\tau,\bxi)}(\x))\big]^\gamma$.
\end{proof}

Note carefully that in the above lemma $|\vartheta| \geq 1$. Indeed, if $|\vartheta|=0$ we cannot benefit from any regularizing effects which are precisely due to cancellation techniques.

We now give some useful controls involving the previous Gaussian kernel which will be \textcolor{black}{used} in our perturbative analysis. The main interest of the estimates below is that they precisely allow to exploit cancellation techniques.
\begin{PROP} \label{Prop_moment_D2_tilde_p}
For all $0 \leq t \leq s \leq T$, $(\x,\bxi) \in \R^{nd} \times \R^{nd}$, the following identities hold:
\begin{eqnarray} \label{eq_D2moment2}
\int_{\R^{nd}}  \tilde p^{(\tau,\bxi)}(t,s,\x,\y) (\y-\m^{(\tau,\bxi)}_{s,t}(\x))_1^{\otimes 2} d \y&=& [\tilde \K_{s,t}^{(\tau,\bxi)}]_{1,1},\label{COV_SYNDICALE}\\
\int_{\R^{nd}} D^2_{\x_1} \tilde p^{(\tau,\bxi)}(t,s,\x,\y) (\y-\m^{(\tau,\bxi)}_{s,t}(\x))_1 d \y&=&\0_{d},
\label{center_odd} \\
\int_{\R^{nd}} D_{\x_k}D^2_{\x_1} \tilde p^{(\tau,\bxi)}(t,s,\x,\y) \cdot(\y-\m^{(\tau,\bxi)}_{s,t}(\x))_1 d \y&=&\0_{d,d}, \ k\in \leftB 1,n \rightB
\label{GROS_CENTER},\\
\textcolor {black}{\int_{\R^{nd}} } D^2_{\x_1} \tilde p^{(\tau,\bxi)}(t,s,\x,\y)   
{\rm Tr}\Big (\mathbf{M} (\y-\m^{(\tau,\bxi)}_{s,t}(\x))_1^{\otimes 2} \Big)d \y&=&2 \mathbf M, \ 
\mathbf{M} \in \R^d\otimes \R^d,
\label{GROS_CENTER_TER}
\\
\int_{\R^{nd}} D_{\x_k}D^2_{\x_1} \tilde p^{(\tau,\bxi)}(t,s,\x,\y)  
{\rm Tr}\Big (\mathbf{M} (\y-\m^{(\tau,\bxi)}_{s,t}(\x))_1^{\otimes 2} \Big) d \y&=&\0_{d,d,d}, \ k\in \leftB 1,n \rightB, \mathbf{M} \in \R^d\otimes \R^d.\nonumber\\
\label{GROS_CENTER_BIS} 
\end{eqnarray}
\end{PROP}
\textcolor{black}{Where, in \eqref{COV_SYNDICALE}, we define for all $(i,j) \in \leftB 1,n \rightB^2$ and $M \in (\R^{nd})^{\otimes 2}$ , $[M]_{i,j}$ is the $d \times d$ block matrix  corresponding to the entry of $M$ on the $i^{\rm th}$ line and the $j^{\rm th}$ column.}
\begin{proof}
First of all, remark that equation \eqref{COV_SYNDICALE} simply follows from a direct covariance computation. 

Observe now that from Proposition \ref{THE_PROP}, we have $ \int_{\R^{nd}} \tilde p^{(\tau,\bxi)}(t,s,\x,\y) (\y-\m^{(\tau,\bxi)}_{s,t}(\x))_1 d \y=\0_d$.
Differentiating twice this expression w.r.t. $\x_1$ and using the Leibniz formula (recalling as well the identity  \eqref{AFFINE_FLOW} which yields $D_{\x_1}[\m_{s,t}^{(\tau,\bxi)}(\x)]_1=[\tilde \gR^{(\tau,\bxi)}(s,t)]_{1,1}=\textcolor{black}{{\mathbf I}}_{d,d} $) gives \eqref{center_odd}. Iterating the differentiation w.r.t. $D_{\x_k}$ then yields \eqref{GROS_CENTER} (observing again that $D_{\x_k}[\m_{s,t}^{(\tau,\bxi)}(\x)]_1=[\tilde \gR^{(\tau,\bxi)}(s,t)]_{1,k}$, i.e. $D_{\x_k}[\m_{s,t}^{(\tau,\bxi)}(\x)]_1)=\textcolor{black}{{\mathbf I}}_{d, d}$ if $k=1$ and $\0_{d, d} $ for $k>1 $). 
\textcolor{black}{Observe that $D_{\x_1} \int_{\R^{nd}} {\rm Tr}\Big (\mathbf{M} (\y-\m^{(\tau,\bxi)}_{s,t}(\x))_1^{\otimes 2} \Big) \tilde p^{(\tau,\bxi)}(t,s,\x,\y)  d \y= D_{\x_1}  {\rm Tr}\Big (\mathbf{M}   [\tilde \K_{s,t}^{(\tau,\bxi)}]_{1,1} \Big) =\0_d$. Differentiating again w.r.t. $D_{\x_1}$, the Leibniz formula and identity  $D_{\x_1}[\m_{s,t}^{(\tau,\bxi)}(\x)]_1=[\tilde \gR^{(\tau,\bxi)}(s,t)]_{1,1}=\mathbf I_{d , d}$ yield \eqref{GROS_CENTER_TER}.}
Eventually, \eqref{GROS_CENTER_BIS} can be derived again from derivation or observing that the sum of the length of the multi-derivation index, here 3, and the power integrated, here 2, is an odd number.
\end{proof}

\subsection{Additional sensitivity controls : covariance, (mollified) flow, mean}
We now state three important estimates associated with our proxy. The first one concerns the sensitivity of the covariance w.r.t. the frozen point, the second and third one concern the linearization or sensitivity w.r.t. the initial point for the frozen (mollified) differential system \eqref{DYN_DET_SMOOTH}. For the sake of simplicity, their proof\textcolor{black}{s} are postponed to appendixes \ref{APP} and \ref{sec_resolv_cov}. We have:

\begin{lem}[Sensitivities of the covariance] \label{SENS_COV}
There exists  \textcolor{black}{$\Lambda:=\Lambda(\A{A},T)$} as in Remark \ref{REM_LAMBDA} s.t. for given  $(\bxi,\bxi')\in (\R^{nd})^2 $  and $0\le t<s \le T$, $(\x,\x')\in (\R^{nd})^2 $:
\begin{equation}
\label{CTR_SENSI_COV}
|[\tilde \K_{s,t}^{(\tau,\bxi)}]_{1,1}-[\tilde \K_{s,t}^{(\tau,\bxi')}]_{1,1}|\le \Lambda(s-t)\big( \d^\gamma(\bxi,\bxi')+ (s-t)^{\frac \gamma 2}\big).
\end{equation} 
\end{lem}
The proof is given in Appendix \ref{proof_lemma_SENCOV}. We importantly point out that, in Lemma \ref{SENS_COV}, the constant $\Lambda$ mainly depends on the H\"older norms of the coefficients and is \textit{small} provided the coefficients do not \textit{vary much}. Precisely, it can be shown that $\Lambda$ writes:
\begin{equation}\label{EXPL_DEP_COEFF}
\Lambda:=\tilde C\big (  \|a\|_{L^\infty(C_\d^\gamma)}+\sum_{i=2}^n  \|\gF_i\|_{L^\infty( C_{\d,\gH}^{2i-3+\gamma})}  \big )
\end{equation}
for some \textit{universal} constant $\tilde C$, where, in the above equation, we write  
with the notation of Section \ref{SEC_HOLDER_SPACE}:

\begin{eqnarray}
\label{DEF_NORME_HOLDER_HOMO}
\|\gF_i\|_{L^\infty( C_{\d,\gH}^{2i-3+\gamma})}&:=&\sup_{(t,\z) \in [0,T]\times \R^{(n-i+2)d}}\|(D_{\x_{i-1}}\gF_i)_{i-1}(t, \z,\cdot)\|_{C^{\frac{\gamma}{2(i-1)-1}}(\R^d,\R^d\otimes \R^d)}\notag\\
&&+\sum_{j=i}^n \sup_{(t,\z) \in [0,T]\times \R^{(n-i+2)d}} 
\|(\gF_i)_j(t, \z,\cdot)\|_{C^{\frac{2i-3+\gamma}{2j-1}}(\R^d,\R^d)}.
\end{eqnarray}
Namely, the quantity $\|\gF_i\|_{L^\infty( C_{\d,\gH}^{2i-3+\gamma})}$ gathers the H\"older moduli of $\gF_i $ at the intrinsic associated scales according to the distance $\d$ in the variables $j\in \leftB i, n\rightB$ as well as the H\"older norm of the gradient w.r.t. the component which transmits the noise (but importantly not its supremum norm). Said differently, the $L^\infty( C_{\d,\gH}^{2i-3+\gamma}) $ semi-norm of $\gF_i $ gathers the H\"older norms of the fractional parts of $\big(D_{\x_j}^{\lfloor \frac{2i-3+\gamma}{2j-1} \rfloor}\gF_i\big)_{j\in \leftB i-1,n\rightB} $ in the $j^{th} $ variable with corresponding H\"older index $\frac{2i-3+\gamma}{2j-1}- \lfloor \frac{2i-3+\gamma}{2j-1} \rfloor$, where $\lfloor\cdot \rfloor $ stands for the integer part.
\textcolor{black}{Eventually, observe as well that, with respect to the notation \eqref{INHOM_NORM} of Section \ref{SEC_HOLDER_SPACE},  $\|\gF_i\|_{L^\infty( C_{\d,\gH}^{2i-3+\gamma})}=\|\gF_i\|_{L^\infty( C_{\d}^{2i-3+\gamma})}-\|D_{\x_{i-1}}\gF_i\|_{L^\infty(L^\infty)} $}.

We again refer to Appendix \ref{APP} for a precise  statement  and proof of this  assertion \eqref{EXPL_DEP_COEFF} (See also Lemma \ref{LEM_SENSI_TO_CONCLUDE} below and its proof for similar properties).\\

The second result is the following.
\begin{lem}\label{lem_theta_theta}
\label{lem_d_theta}
There exists $C:=C(\A{A})$ s.t. for all  $0 \leq t \leq s \leq T$, $(\x,\x') \in \R^{nd} \times \R^{nd}$: 
\begin{equation*}
\d(\btheta_{s,t}(\x),\btheta_{s,t}(\x'))
\leq  C\Big( \d(\x,\x') +(s-t)^{\frac 12} \Big).
\end{equation*}
\end{lem}
The proof of Lemma \ref{lem_theta_theta} is postponed to Appendix \ref{sec:ProofLem_d_theta}.\\

Eventually, this last lemma concerns the impact of the freezing point in the linearization procedure. Namely,
\begin{lem}[Sensitivity of the linearized flow w.r.t. the freezing parameter ]
There exists $C:=C(\A{A})$ s.t. for all  $\tau= t $,  $(\x,\x') \in (\R^{nd})^2$ at the change of regime time  $t_0$ defined in \eqref{def_t0} (i.e. $t_0 = (t + c_0\d^2(\x,\x'))\wedge T$):
\label{Lemme_d_theta_theta_x_x}
\begin{equation*}
\d\big ( \m_{t_0,t}^{(\tau,\x)}(\x'),\m_{t_0,t}^{(\tau,\x')}(\x') \big )=\d\big ( \m_{t_0,t}^{(\tau,\x)}(\x'),  \btheta_{t_0,t}(\x') \big ) \leq  Cc_0^{\frac 1{
2n-1
}} \d(\x,\x').
\end{equation*}
\end{lem}
Again, the proof of Lemma \ref{Lemme_d_theta_theta_x_x} is postponed to Appendix \ref{sec:ProofLemma_flow}.


\mysection{Control of the supremum of the derivatives w.r.t. the non-degene\-ra\-te variables}
\label{Sectio_Sup_D2}
\textbf{WARNING:} for notational simplicity,  we drop from now on the sub and superscripts $ \tau$ associated with the linearization since, \textcolor{black}{as soon as the function $u$ in  \eqref{DUHAMEL_PERTURB} is evaluated at time $t$ in $[0,T]$, we choose this parameter to be equal to $t$.}
For example, $\tilde p^{(\tau,\bxi)}, \m_{s,t}^{(\tau,\bxi)}, \btheta_{u,\tau} $ become respectively $\tilde p^{\bxi}, \m_{s,t}^{\bxi}, \btheta_{u,t}$. \textcolor{black}{Recall as well that we decided to \textcolor{black}{omit} the dependence of such a function $u$ in the regularization parameter $m$. This dependence is implicitly assumed and we will derive the desired control uniformly in \textcolor{black}{$m$}.
}\\

The result we aim at proving in this \textcolor{black}{section concerns} the supremum norm of the derivatives w.r.t. the non degenerate variables. Namely, we here prove the following \textcolor{black}{result}.
\begin{PROP}[] \label{PROP_SUP_CTRL}\quad
Let $\gamma\in (0,1)$ be given. Suppose that \A{A} is in force and that the terminal condition $g$ and source term $f$ of the Cauchy problem \eqref{KOLMO} satisfy: $g\in C_{b,\d}^{2+\gamma }(\R^{nd},\R)$ and $f\in L^\infty\big([0,T],C_{b,\d}^{\gamma}( \R^{nd},\R)\big)$. Then, there exist $C:=C(\A{A},T)$ and $\Lambda $ as in Remark \ref{REM_LAMBDA} such that for any $(t,\x) \in [0,T] \times \R^{nd}$,
\begin{equation}\label{ESTI_SUP_NORMS}
\textcolor{black}{|u(t,\x)|}+ |D_{\x_1} u(t,\x) | + |D_{\x_1}^2 u(t,\x) |\leq C \Big\{\| g\|_{C^{2+\gamma}_{b,\d}}  +  (T-t)^{\frac \gamma 2} \|f\|_{L^\infty(C_{b,\d}^{\gamma})} \Big\} + {\Lambda}(T-t)^{\frac \gamma 2}\|u\|_{L^\infty(C^{2+\gamma}_{b,\d}) }.
\end{equation}
\end{PROP}

\textcolor{black}{Note first that the control for the function itself readily follows from \A{A} and the Feynman-Kac representation of the solution of \eqref{KOLMO_m} under \A{A}, i.e. recall from \eqref{REP_FK} that $u(t,\x)=\E[g(\X_T^{t,\x})]+\int_t^T ds \E[f(s,\X_s^{t,\x})] $ and that $T\le 1$.}

For the derivatives, let us now start 
from  \eqref{DUHAMEL_PERTURB} to control pointwise the second order \textcolor{black}{derivative of $u$ in the non-degenerate variable}, i.e. $\|D_{\x_1}^2u\|_{L^\infty} $. The first \textcolor{black}{one} can be controlled similarly and more directly. Write for any  $(t,\x)\in [0,T]\times \R^{nd} $:
\begin{eqnarray}
&&|D_{\x_1}^2u(t,\x)| \le  \|D_{\x_1}^2\tilde P_{T,t}^\bxi g\|_{\textcolor{black}{L^\infty}}+\|D_{\x_1}^2\tilde G^{\bxi}f\|_{\textcolor{black}{L^\infty}} \nonumber\\
&&+  \Big |\int_{t}^T ds \int_{\R^{nd}} D_{\x_1}^2\tilde p^{\bxi}(t,s,\x,\y)\Delta_{1,\gF,\sigma}(t,s,\btheta_{s,t}(\bxi) ,\y,u) d\y\Big | 
 \notag\\
&&+\Big | \sum_{i=2}^n\int_t^T ds \int_{\R^{nd}}
D_{\y_i} \cdot \Big( \big (D_{\x_1}^2\tilde p^{\bxi}(t,s,\x,\y) \otimes \Delta_{i,\gF}(t,s,\btheta_{s,t}(\bxi),\y) \big )\Big )  u(s,\y) d\y \Big | ,\label{def_R1_R2}
\end{eqnarray}
where we recall from \eqref{def_Delta_a_F1} and \eqref{DEF_TERMES_DEG} that:
\begin{eqnarray}
\Delta_{1,\gF,\sigma}(t,s,\btheta_{s,t}(\bxi) ,\y,u)\! \!&\! \!=\! \!&\! \! \langle \big (\gF_{1}(s,\y)-\gF_{1}(s,\btheta_{s,t}(\bxi))\big), D_{\y_1}u(s,\y)\rangle
\notag\\
&&+\frac 12{\rm Tr}  \Big(\big ( a(s,\y)-a(s,\btheta_{s,t}(\bxi))\big) D_{\y_1}^2 u(s,\y)\Big),
\nonumber \\
\Delta_{i,\gF}(t,s,\btheta_{s,t}(\bxi),\y)\! \!&\! \!=\! \!& \! \!\gF_{i}(s,\y)-\gF_{i}(s,\btheta_{s,t}(\bxi))-D_{\x_{i-1}}\gF_{i}(s,\btheta_{s,t}(\bxi))(\y-\btheta_{s,t}(\bxi))_{i-1}.\label{LES_DELTA_REG_LOC}
\end{eqnarray}

This section is then organized as follows: we first estimate the non degenerate part of \eqref{def_R1_R2} (first term in the r.h.s. of the equation) thanks to Lemma \ref{CTR_DER_SUP_non_deg} in Section \ref{SUBSEC_CTRL_DERIV_NONDEG}, we then estimate the degenerate part  of \eqref{def_R1_R2} (second term in the r.h.s. of the equation) thanks to Lemma \ref{LEMME_BESOV_DEG} in Section \ref{SEC_BESOV_DUAL_FIRST} and eventually estimate the remainder of \eqref{def_R1_R2} (third and fourth terms in the r.h.s. of the equation) thanks to Lemma \ref{CTR_SEMI_GROUP_GREEN_FROZEN} in Section \ref{pragestiderivsemietgreen}. Proposition \ref{PROP_HOLDER_CTRL} then follows from the previous Lemmas.\qed

\subsection{Control of the non-degenerate part of the perturbative term }\label{SUBSEC_CTRL_DERIV_NONDEG}
The aim of this section is to prove identity \eqref{ineq_D2x1_Delta1} appearing in the \textcolor{black}{detailed guide to the proof}. To this end, we provide a general differentiation result,  which will be useful as well in Section \ref{HOLDER} to deal with the H\"older norms.
Under the current assumptions on $a,\gF$,  the following lemma holds.
\begin{lem}
\label{CTR_DER_SUP_non_deg}
\textcolor{black}{There exists $\Lambda:=\Lambda(\A{A},T) $ as in Remark \ref{REM_LAMBDA} s.t.}  for each multi-index $\vartheta=(\vartheta_1, \hdots, \vartheta_n) \in \N^{\textcolor{black}{n}}$, $\textcolor{black}{|\vartheta|\le 3}$:
\begin{equation}
\Big| \int_{\R^{nd}} D_\x^\vartheta \tilde p^{\bxi} (t,s,\x,\y)\Delta_{1,\gF,\sigma}(t,s,\btheta_{s,t}(\bxi) ,\y,u)   d\y\Big| \bigg|_{\bxi=\x}
\leq \Lambda\|u\|_{L^\infty(C_{b,\d}^{2+\gamma})}(s-t)^{-\sum_{j=1}^n \vartheta_j (j-\frac 12)+\frac{\gamma}2}.
\end{equation}
\end{lem}
\begin{proof}[Proof of Lemma \ref{CTR_DER_SUP_non_deg}]
We first recall the control \eqref{ineq_Holder_a_F1}
\begin{equation*}
|\Delta_{1,\gF,\sigma}(t,s,\y,\btheta_{s,t}(\bxi),u)|\le \Big( [\gF_1]_{\d,\gamma}\|D_{\x_1}u(s,\cdot)\|_{L^\infty}+\frac 12 [a(s,\cdot)]_{\d,\gamma} \|D_{\x_1}^2u(s,\cdot)\|_{L^\infty}\Big) \d^\gamma(\y,\btheta_{s,t}(\bxi)).
\end{equation*}
From this control and Proposition \ref{THE_PROP}, we  directly obtain:
\begin{eqnarray*}
&&\Big| \int_{\R^{nd}} D_{\x}^\vartheta \tilde p^{\bxi} (t,s,\x,\y)\Delta_{1,\gF,\sigma}(t,s,\btheta_{s,t}(\bxi) ,\y,u)   d\y\Big|\bigg|_{\bxi=\x} \notag\\
&\le& \Big( \int_{\R^{nd}} |D_{\x}^\vartheta \tilde p^{\bxi}(t,s,\x,\y)|\ \d^\gamma\big(\y,\btheta_{s,t}(\bxi)\big)\big(\|\gF_1\|_{L^\infty(C_\d^\gamma
)}\|D_{\x_1}u\|_{L^\infty}
+\|a\|_{L^\infty(C_{\d}^\gamma 
)} \|D_{\x_1}^2 u\|_{L^\infty}\big)d\y\Big) \Big|_{\bxi=\x}\notag\\
&\le& \Lambda\|u\|_{L^\infty(C_{b,\d}^{2+\gamma})} (s-t)^{-\sum_{j=1}^n \vartheta_j (j-\frac 12)} \int_{\R^{nd}} \bar p_{C^{-1}}(s,t,\x,\y) \d^\gamma( \btheta_{s,t}(\x), \y) d\y
\notag\\
&\le& \Lambda\|u\|_{L^\infty(C_{b,\d}^{2+\gamma})}(s-t)^{-\sum_{j=1}^n \vartheta_j (j-\frac 12)+\frac{\gamma}2}\label{CTR_RM1},
\end{eqnarray*} 
with the notations of Remark \ref{NOTA_ABUSE} for the last but one inequality.
\end{proof}
Equation \eqref{ineq_D2x1_Delta1} readily follows from Lemma \ref{CTR_DER_SUP_non_deg} taking $\vartheta=(2,0,\hdots,0)$. Namely:
\begin{eqnarray}\label{ineq_sup_D2_a_F1}
&&\Big|\int_t^T ds \int_{\R^{nd}} D_\x^\vartheta \tilde p^{\bxi} (t,s,\x,\y)\Delta_{1,\gF,\sigma}(t,s,\btheta_{s,t}(\bxi) ,\y,u)   d\y\Big|\ \Bigg|_{\bxi=\x}\notag\\
& \leq &\Lambda\|u\|_{L^\infty(C_{b,\d}^{2+\gamma})}\int_{t}^T \frac{ds}{(s-t)^{1-\frac \gamma 2}}\notag\\
&\leq &\Lambda\|u\|_{L^\infty(C_{b,\d}^{2+\gamma})}(T-t)^{\frac{\gamma}2}.
\end{eqnarray}

\subsection{Control of the degenerate part of the perturbative term}
\label{SEC_BESOV_DUAL_FIRST}
The point is here to control the terms 
$ \sum_{i=2}^n\int_t^T ds \int_{\R^{nd}} 
D_{\x_1}^2\tilde p^{\bxi}(t,s,\x,\y) \big \langle \Delta_{i,\gF}(t,s,\btheta_{s,t}(\bxi),\y),  D_{\y_i} u(s,\y)  \big \rangle d\y $ \- appearing in equation \eqref{DUALITE_PREAL_BESOV} of the \textcolor{black}{detailed guide to the proof}. We precisely want to derive equation \eqref{THE_DEG_POINT_DX1_2}.

The bound will actually follow from the more general following result, which will again be useful for the H\"older norm in Section \ref{HOLDER}.
\label{CTR_DER_SUP}
\begin{lem}[First Besov Control Lemma] \label{LEMME_BESOV_DEG}
\textcolor{black}{There exists $\Lambda:=\Lambda(\A{A},T) $ as in Remark \ref{REM_LAMBDA} s.t.}  for each multi-index $\vartheta=(\vartheta_1, \hdots, \vartheta_n) \in \N^{\textcolor{black}{n}}$, $\textcolor{black}{|\vartheta|\le 3}$:
\begin{equation}\label{ineq_lemme_Besov}
\sum_{i=2}^n\Big|\int_{\R^{nd}} D_\x^\vartheta \tilde p^{\bxi}(t,s,\x,\y) \big \langle  \Delta_{i,\gF}(t,s,\btheta_{s,t}(\bxi),\y),  D_{\y_i}  u(s,\y) \big \rangle 
   d\y\Big| \bigg|_{\bxi=\x} \!\!\! \leq\! \Lambda \|u\|_{L^\infty(C_{b,\d}^{2+\gamma})}(s-t)^{-\sum_{j=1}^n \vartheta_j (j-\frac 12)+\frac{\gamma}2} \!.
\end{equation}
\end{lem}
\begin{proof}[Proof of Lemma \ref{LEMME_BESOV_DEG}] \textcolor{black}{Let us first emphasize that
from the H\"older continuity assumption \A{S}-(iii) (w.r.t. the underlying homogeneous metric $\d $) on $\gF_i $ :
\begin{eqnarray}\label{ineq_Holder_Fi}
&&|\Delta_{i,\gF}(t,s,\btheta_{s,t}(\bxi),\y)| \le [\gF_i(s,\cdot)]_{\d,2i-3+\gamma} \d^{2i-3+\gamma}(\y,\btheta_{s,t}(\bxi)).
\end{eqnarray}
}
Similarly to \textcolor{black}{\eqref{DUALITE_PREAL_BESOV_WITH_THETA}},
 we \textcolor{black}{take} for each $i\in \leftB 2, n\rightB $, 
\begin{equation}
\label{DEF_GI_TO_BELONG_TO_BESOV_SPACE}
D_{\y_i} \cdot  \big (\Theta_{i,(t,\x)}^\vartheta(s,\y) \big )=D_{\y_i}\cdot \big(D_\x^\vartheta\tilde p^{\bxi}(t,s,\x,\y) \otimes  \Delta_{i,\gF}(t,s,\btheta_{s,t}(\bxi),\y)
\big).
\end{equation}
The contribution of the l.h.s. in \eqref{ineq_lemme_Besov} then rewrites:
\begin{equation}
\label{FROM_WHERE_TO_USE_BESOV_DUALITY}
\sum_{i=2}^n\Big|\int_{\R^{nd}} D_\x^\vartheta \tilde p^{(t,\bxi)}(t,s,\x,\y) \big \langle \Delta_{i,\gF}(t,s,\btheta_{s,t}(\bxi),\y),  D_{\y_i}  u(s,\y) \big \rangle 
   d\y\Big| \bigg|_{\bxi=\x} 
=\sum_{i=2}^n\Big| \int_{\R^{nd}} \big[D_{\y_i} \cdot  \big (\Theta_{i,(t,\x)}^\vartheta(s,\y) \big ) \big]u(s,\y) d\y\Big|.
\end{equation}
The point now is to observe that for any fixed $i\in \leftB 2,n\rightB $ and $\z=(\z_1, \cdots, \z_{i-1},\z_{i+1},\cdots, \z_n)\in \R^{(n-1)d} $ the mapping $\y_i \mapsto u(s,\z_{1:i-1},\y_i,\z_{i+1:n}) $ is in $C_{b}^{\frac{2+\gamma}{2i-1}}(\R^{d})=B_{\infty,\infty}^{\frac{2+\gamma}{2i-1}}(\R^{d})$ using the Besov space terminology, see e.g. Triebel \cite{trie:83}, uniformly in $s\in [0,T] $. \textcolor{black}{We can hence} put in duality the mappings $\y_i \mapsto \textcolor{black}{u}(s,\y_{1:i-1},\y_i,\y_{i+1:n}) $ and 
\begin{equation}
\label{DEF_PSI_M}
\Psi_{i,(t,\x),(s,\y_{1:i-1},\y_{i+1:n})}^\vartheta: \y_i\mapsto  D_{\y_i} \cdot  \big (\Theta_{i,(t,\x)}^\vartheta(s,\y) \big ),
\end{equation}
see e.g. Proposition 3.6  in \cite{lemar:02}.
\textcolor{black}{To do so, we thus } have to prove that \textcolor{black}{ $\Psi_{i,(t,\x),(s,\y_{1:i-1},\y_{i+1:n})}^\vartheta$ lies in the suitable Besov space, namely}
 $\Psi_{i,(t,\x),(s,\y_{1:i-1},\y_{i+1:n})}^\vartheta\in B_{1,1}^{-\frac{2+\gamma}{2i-1}}(\R^d) $ and to control the associated norm. 
 We will actually prove that those norms  provide an integrable quantity w.r.t. $\y_{1:i-1}, \y_{i+1:n} $ as well as an integrable time singularity. This will be done through the thermic characterization of Besov spaces, see e.g. Section 2.6.4 in \cite{trie:83} as well as \eqref{THERMIC_CAR_DEF} above. Precisely, we recall that for a function $\psi:\R^d\rightarrow \R$  in $B_{1,1}^{-\tilde \alpha_i}(\R^d)$, $\tilde \alpha_i:=\frac{2+\gamma}{2i-1}$ a quasi-norm is given by:
 \begin{equation}\label{THERMIC_CAR}
 \|\psi\|_{B_{1,1}^{-\tilde \alpha_i}(\R^d)}:=\|\varphi(D) \psi\|_{L^1(\R^{d},\R)}+\int_0^1 \frac{dv}{v}v^{\frac{\tilde \alpha_i}2} \|h_v\star \psi\|_{L^1(\R^d,\R)}
 ,\ \forall z\in \R^d,\ h_v(z):=\frac{1}{(2\pi v)^{\frac d2}}\exp\big(-\frac{|z|^2}{v}\big), 
 \end{equation}
\textcolor{black}{$h$} being the usual \textit{heat kernel} of $\R^d$, ``$\star$'' standing for the usual convolution on $\R^d $ for $\varphi \in C_0^\infty(\R^{d},\R)$ s.t. $ \varphi(0)\neq 0$.

Taking $\psi=\Psi_{i,(t,\x),(s,\y_{1:i-1},\y_{i+1:n})}^\vartheta $ in the above characterization
and from definition \textcolor{black}{\eqref{DEF_PSI_M}}, the main advantage of using \eqref{THERMIC_CAR} consists in rebalancing the derivative appearing in the definition \textcolor{black}{\eqref{DEF_PSI_M}} to the heat kernel or to the smooth compactly supported function $\varphi $. \textcolor{black}{Namely, focusing on the $L^1$ norm of the convolution product} in \eqref{THERMIC_CAR}, \textcolor{black}{ we write}:
\begin{eqnarray}\label{IPP_theta}
\|h_v\star \Psi_{i,(t,\x),(s,\y_{1:i-1},\y_{i+1:n})}^\vartheta\|_{L^1(\R^d,\R)}
&=&\int_{\R^d}\Big |\int_{\R^d} h_v(z-\y_i)D_{\y_i} \cdot \big ( \Theta_{i,(t,\x)}^\vartheta(s,\y) \big )d\y_i \Big| dz\notag \\
&=&\int_{\R^d}\Big |\int_{\R^d} 
 D_\x^\vartheta \tilde p^{\bxi}(t,s,\x,\y)   
 \big \langle \Delta_{i,\gF}(t,s,\btheta_{s,t}(\bxi),\y), D_z h_v(z-\y_i) \big \rangle d\y_i\Big | dz
\label{NORME_L1_CONV_HK}.
\end{eqnarray}
\textcolor{black}{To estimate $\|\Psi_{i,(t,\x),(s,\y_{1:i-1},\y_{i+1:n})}^\vartheta\|_{B_{1,1}^{-\tilde \alpha_i}} $}, we split the time integral in \eqref{THERMIC_CAR} into two parts writing:
\begin{eqnarray}
 &&\int_0^1 \! \frac{dv}{v}v^{\frac{\tilde \alpha_i}2 }\|h_v\star \Psi_{i,(t,\x),(s,\y_{1:i-1},\y_{i+1:n})}^\vartheta\|_{L^1(\R^d,\R)}\notag\\
 &=& 
\int_0^{(s-t)^{\beta_i}}  \frac{dv}{v}v^{\frac{\tilde \alpha_i}2 }\|h_v\star \Psi_{i,(t,\x),(s,\y_{1:i-1},\y_{i+1:n})}^\vartheta\|_{L^1(\R^d,\R)} \! +\! \int_{(s-t)^{\beta_i}}^1 \!\frac{dv}{v}v^{\frac{\tilde \alpha_i}2 }\|h_v\star \Psi_{i,(t,\x),(s,\y_{1:i-1},\y_{i+1:n})}^\vartheta\|_{L^1(\R^d,\R)}\notag\\ \label{DECOUP_HK_INT}
\end{eqnarray}
for a parameter $\beta_i>0 $ to be specified. Precisely, in order to have a similar smoothing effect in time than for the terms appearing in 
\eqref{CTR_RM1}, we now want to calibrate $\beta_i $ to obtain:
\begin{equation}
\label{RES_CALIBRATION_BETA_I}
\int_{(s-t)^{\beta_i}}^1 \frac{dv}{v}v^{\frac{\tilde \alpha_i}2 }\|h_v\star \Psi_{i,(t,\x),(s,\y_{1:i-1},\y_{i+1:n})}^\vartheta\|_{L^1(\R^d,\R)}\le \frac{\Lambda}{(s-t)^{\sum_{j=1} \vartheta_j(j-\frac 12)-\frac \gamma 2}}\hat q_{c\setminus i}(t,s,\x,(\y_{1:i-1},\y_{i+1:n})),
\end{equation}
where introducing: 
\begin{equation*}
\bar p_{c^{-1}}(t,s,\x,\y)=:\prod_{j=1}^n {\mathcal N}_{c(s-t)^{2j-1}}\big( (\btheta_{s,t}(\x)-\y)_j\big),
\end{equation*}
where for $\textcolor{black}{\varsigma}>0,\ z\in \R^d, {\mathcal N}_{\varsigma}(z)=\frac{1}{(2\pi \varsigma)^{\frac d2}}\exp\big(-\frac{|z|^2}{2\varsigma}\big) $ stands for the standard Gaussian density of $\R^d $ with covariance matrix $\varsigma\textcolor{black}{{\mathbf I}_{d,d}}$, we introduce:
\begin{eqnarray}
\label{DEF_HAT_SETMINUS}
\hat q_{c\setminus i}(t,s,\x,(\y_{1:i-1},\y_{i+1:n}))=\prod_{j\in \leftB 1, n\rightB, j\neq i} {\mathcal N}_{c(s-t)^{2j-1}}\big( (\btheta_{s,t}(\x)-\y)_j\big).
\end{eqnarray}
To choose properly the parameter $\beta_i $ leading to \eqref{RES_CALIBRATION_BETA_I}, we now write \textcolor{black}{from \eqref{IPP_theta}}:
\begin{eqnarray*}
&&\int_{(s-t)^{\beta_i}}^1 \frac{dv}{v}v^{\frac{\tilde \alpha_i}2 }\|h_v\star \Psi_{i,(t,\x),(s,\y_{1:i-1},\y_{i+1:n})}^\vartheta\|_{L^1(\R^d,\R)}\\
&\le& \int_{(s-t)^{\beta_i}}^1 \frac{dv}{v}v^{\frac{\tilde \alpha_i}2 } \int_{\R^d} dz 
 \Big|\int_{\R^d}
 D_\x^\vartheta \tilde p^{\bxi}(t,s,\x,\y)   
 \big \langle \Delta_{i,\gF}(t,s,\btheta_{s,t}(\bxi),\y), D_z h_v(z-\y_i) \big \rangle
d\y_i\Big| \bigg |_{\bxi=\x}\\
&\le& \Lambda\int_{(s-t)^{\beta_i}}^1 \frac{dv}{v}v^{\frac{\tilde \alpha_i}2 } \int_{\R^d} dz 
 \int_{\R^d} d\y_i \frac{h_{cv}(z-\y_i)}{v^{\frac 12}}\frac{\bar p_{c^{-1}}(t,s,\x,\y)}{(s-t)^{\sum_{j=1}^n \vartheta_j(j-\frac 12)}} \d^{2i-3+\gamma}(\btheta_{s,t}(\x),\y)\\
&\le& \Lambda\int_{(s-t)^{\beta_i}}^1 \frac{dv}{v}v^{\frac{\tilde \alpha_i}2 } \int_{\R^d} dz 
 \int_{\R^d} d\y_i \frac{h_{cv}(z-\y_i)}{v^{\frac 12}}\frac{\bar p_{c^{-1}}(t,s,\x,\y)}{(s-t)^{\sum_{j=1}^n \vartheta_j(j-\frac 12)}} (s-t)^{\frac{2i-3+\gamma}2}\\
&\le& \Lambda\hat q_{c\setminus i}(t,s,\x,(\y_{1:i-1},\y_{i+1:n}))\int_{(s-t)^{\beta_i}}^1 dvv^{-\frac 32+\frac{\tilde \alpha_i} 2 }(s-t)^{-\sum_{j=1}^n \vartheta_j(j-\frac 12)+\frac{2i-3+\gamma}2}\\
&\le& \Lambda\hat q_{c\setminus i}(t,s,\x,(\y_{1:i-1},\y_{i+1:n})) (s-t)^{[-\frac 12+\frac{\tilde \alpha_i} 2]\beta_i -\sum_{j=1}^n \vartheta_j(j-\frac 12)+\frac{2i-3+\gamma}2}
\end{eqnarray*}
using \eqref{ineq_Holder_Fi}, Proposition \ref{THE_PROP}
for the third inequality (see also Remark \ref{THE_REM_REG}, equation \eqref{CTR_GRAD_ET_DIST})
recalling as well that $0\le t<s\le T $ is small. 

To obtain \eqref{RES_CALIBRATION_BETA_I}, we then take:
\begin{equation}
\label{DEF_BETA_I} [-\frac 12+\frac{\tilde \alpha_i} 2]\beta_i -\sum_{j=1}^n \vartheta_j(j-\frac 12)+\frac{2i-3+\gamma}2=-\sum_{j=1}^n \vartheta_j(j-\frac 12)+\frac \gamma 2 \iff \beta_i=\frac{(2i-3)(2i-1)}{2i-3-\gamma}.
\end{equation}

The key point is now to check that the previous choice of $\beta_i $ also yields a bound similar to \eqref{RES_CALIBRATION_BETA_I} for the contribution in \eqref{DECOUP_HK_INT} associated with $v\in [0,(s-t)^{\beta_i}] $.
To this end, we restart from identity \eqref{NORME_L1_CONV_HK}, which allows to exploit partial cancellations w.r.t. the integration variable $\y_i $. Namely, write:
\begin{eqnarray}
&&\int_{\R^d}  h_v(z-\y_i) D_{\y_i} \cdot \Theta_{i,(t,\x)}^\vartheta(s,\y) d\y_i
\nonumber \\
&=&\int_{\R^d}  h_v(z-\y_i)  D_{\y_i} \cdot \Big(\Theta_{i,(t,\x)}^\vartheta (s,\y)- \Theta_{i,(t,\x)}^\vartheta(s,\y_{1:i-1},z,\y_{i+1:n}) \Big) d\y_i\notag \\
&=& \int_{\R^d}   D_\x^\vartheta \tilde p^{\bxi}(t,s,\x,\y)  \big\langle \gF_{i}(s,\y)-\gF_{i}(s,\y_{1:i-1},z,\y_{i+1:n}),  D_z h_v(z-\y_i)  \big \rangle d\y_i \notag \\
&&+\int_{\R^d}  \big(D_\x^\vartheta \tilde p^{\bxi}(t,s,\x,\y)-D_\x^\vartheta\tilde p^{\bxi}(t,s,\x,\y_{1:i-1},z,\y_{i+1:n})\big)\notag \\
&& \big\langle \gF_{i}(s,\y_{1:i-1},z,\y_{i+1:n})-\gF_{i}(s,\btheta_{s,t}(\bxi)) -D_{\x_{i-1}}\gF_{i}(s,\btheta_{s,t}(\bxi))(\y-\btheta_{s,t}(\bxi))_{i-1},  D_z h_v(z-\y_i) \big\rangle d\y_i\notag\\
&=:& \Big({\mathscr T}_1+{\mathscr T}_2\Big) \big(v,t,s,\x,(\y_{1:i-1},z,\y_{i+1:n})\big), \label{DECOUP_CAR_THERMIC}
\end{eqnarray}
using the definition in \eqref{DEF_GI_TO_BELONG_TO_BESOV_SPACE} and \eqref{IPP_theta} for the last decomposition. Write now from \textcolor{black}{Proposition \ref{THE_PROP} and the H\"older regularity assumed on $\gF_i$ from \A{S}-(iii)}:
\begin{eqnarray}
|{\mathscr T}_1\big(v,t,s,\x,(\y_{1:i-1},z,\y_{i+1:n})\big)|&\le& \Lambda \int_{\R^d} \frac{h_{cv}(z-\y_i)}{v^{\frac12}} \frac{\bar p_{c^{-1}}(t,s,\x,\y)}{(s-t)^{\sum_{j=1}^n \vartheta_j(j-\frac 12)}}|z-\y_i|^{\frac{2i-3+\gamma}{2i-1}} d\y_i\notag \\
&\le & \Lambda\int_{\R^d} \frac{h_{cv}(z-\y_i)}{v^{\frac{2-\gamma}{4i-2}}} \frac{\bar p_{c^{-1}}(t,s,\x,\y)}{(s-t)^{\sum_{j=1}^n \vartheta_j(j-\frac 12)}} d\y_{i}.\label{PREAL_T1}
\end{eqnarray}
We thus derive from \eqref{PREAL_T1}:
\begin{eqnarray}
\label{CTR_T1}
&& |{\mathscr T}_1\big(v,t,s,\x,(\y_{1:i-1},z,\y_{i+1:n})\big)|\notag\\
&\le& \frac{\Lambda}{v^{\frac{2-\gamma}{4i-2}} (s-t)^{\sum_{j=1}^n \vartheta_j(j-\frac 12)}} \hat q_{c\setminus i}(t,s,\x,(\y_{1:i-1},\y_{i+1:n})) {\mathcal N}_{cv+(s-t)^{2i-1}}\big ( z-\btheta_{s,t}(\x)_i\big).
\end{eqnarray}
\textcolor{black}{We now deal with the term  $|{\mathscr T}_2\big(v,t,s,\x,(\y_{1:i-1},z,\y_{i+1:n})\big)|$ in \eqref{DECOUP_CAR_THERMIC}}. From the Taylor formula applied to the $i^{\rm th}$ variable for the difference of the derivatives of the densities \textcolor{black}{we obtain}:
\begin{eqnarray*}
&&|{\mathscr T}_2\big(v,t,s,\x,(\y_{1:i-1},z,\y_{i+1:n})\big)|
\nonumber \\
&\le& C\int_{\R^d} d\y_i\frac{h_{cv}(z-\y_i)}{v^{\frac 12}}\int_0^1 d\mu  \frac{\bar p_{c^{-1}}(t,s,\x,\y_{1:i-1},z+\mu (\y_i-z), \y_{i+1:n})}{(s-t)^{\sum_{j=1}^n \vartheta_j(j-\frac 12)+\frac{2i-1}{2}}}\\
&&\times |\y_i-z|\bigg (\Big|\gF_{i}(s,\y_{1:i-1},z,\y_{i+1:n})-\gF_{i}(s,\y_{1:i-1},\btheta_{s,t}(\x)_{i:n})\Big|\\
&&+\Big|\gF_{i}(s,\y_{1:i-1},\btheta_{s,t}(\x)_{i:n})-\gF_{i}(s,\btheta_{s,t}(\bxi)) -D_{\x_{i-1}}\gF_{i}(s,\btheta_{s,t}(\x))(\y-\btheta_{s,t}(\x))_{i-1} \Big|\bigg)\\
&\le& \Lambda\int_{\R^d} d\y_ih_{cv}(z-\y_i)\int_0^1 d\mu  \frac{\bar p_{c^{-1}}(t,s,\x,\y_{1:i-1},z+\mu (\y_i-z), \y_{i+1:n})}{(s-t)^{\sum_{j=1}^n \vartheta_j(j-\frac 12)+\frac{2i-1}{2}}}\\
&&\times \Bigg(
|z-\btheta_{s,t}(\x)_i|^{\frac{2i-3+\gamma}{2i-1}} +|(\btheta_{s,t}(\x)-\y)_{i-1}|^{1+\frac{\gamma}{2(i-1)-1}}+\sum_{k=i+1}^n |(\btheta_{s,t}(\x)-\y)_k|^{\frac{2i-3+\gamma}{2k-1}}\Bigg).
\end{eqnarray*}
Writing, for any  $\mu \in [0,1]  $,
\begin{equation*}
|z-\btheta_{s,t}(\x)_i|\le \mu  |z-\y_i|+|z +\mu (\y_i-z)-(\btheta_{s,t}(\x))_i|, 
\end{equation*}
we thus derive
\begin{eqnarray}
&& |{\mathscr T}_2\big(v,t,s,\x,(\y_{1:i-1},z,\y_{i+1:n})\big)|
\nonumber \\
&\le& \Lambda\int_{\R^d} d\y_ih_{cv}(z-\y_i)\int_0^1 d\mu  \frac{\bar p_{c^{-1}}(t,s,\x,\y_{1:i-1},z+\mu (\y_i-z), \y_{i+1:n})}{(s-t)^{\sum_{j=1}^n \vartheta_j(j-\frac 12)+\frac{2i-1}{2}}}\notag \\
&&\times \Big(|\y_i-z|^{\frac{2i-3+\gamma}{2i-1}}+\d^{2i-3+\gamma}\Big(\btheta_{s,t}(\x), (\y_{1:i-1}, z +\mu (\y_i-z)
,\y_{i+1:n})\Big) \Big) \notag \\
&\le&  \Lambda\int_{\R^d} d\y_ih_{cv}(z-\y_i)\int_0^1 d\mu  \bar p_{c^{-1}}(t,s,\x,\y_{1:i-1},z+\mu (\y_i-z), \y_{i+1:n})
\nonumber \\
&&\times \Bigg(\frac{v^{\frac{2i-3+\gamma}{2(2i-1)}}}{(s-t)^{\sum_{j=1}^n \vartheta_j(j-\frac 12)+\frac{2i-1}{2}}}    +\frac{1}{(s-t)^{\sum_{j=1}^n \vartheta_j(j-\frac 12)+1-\frac \gamma 2}}\Bigg)\notag\\
&\le&  \Lambda\hat q_{c\setminus i}(t,s,\x,\y_{1:i-1}, \y_{i+1:n})\int_0^1 d\mu \int_{\R^d} h_{cv}(z-\y_i) {\mathcal N}_{c(s-t)^{2i-1}}(z+\mu (\y_i-z)-(\btheta_{s,t}(\x))_i) d\y_i\notag\\
&&\times \Bigg(\frac{v^{\frac{2i-3+\gamma}{2(2i-1)}}}{(s-t)^{\sum_{j=1}^n \vartheta_j(j-\frac 12)+\frac{2i-1}{2}}}    +\frac{1}{(s-t)^{\sum_{j=1}^n \vartheta_j(j-\frac 12)+1-\frac \gamma 2}}\Bigg),
\label{CTR_T2}
\end{eqnarray}
using again \eqref{CTR_GRAD_ET_DIST} for the second inequality.
From \eqref{DECOUP_CAR_THERMIC}, \eqref{CTR_T1} and \eqref{CTR_T2} we derive, with the notation introduced in \eqref{DEF_HAT_SETMINUS}:
\begin{eqnarray*}
&&\|h_v\star \Psi_{i,(t,\x),(s,\y_{1:i-1},\y_{i+1})}^\vartheta\|_{L^1(\R^d,\R)} 
\\
&\le& \Bigg(\frac{1}{v^{\frac{2-\gamma}{4i-2}} (s-t)^{\sum_{j=1}^n \vartheta_j(j-\frac 12)}}
+\frac{v^{\frac{2i-3+\gamma}{2(2i-1)}}}{(s-t)^{\sum_{j=1}^n \vartheta_j(j-\frac 12)+\frac{2i-1}{2}}}    +\frac{1}{(s-t)^{\sum_{j=1}^n \vartheta_j(j-\frac 12)+1-\frac \gamma 2}} \Bigg)
\\
&&\times \Lambda\hat q_{c\setminus i}(t,s,\x,(\y_{1:i-1},\y_{i+1:n})) \int_0^1 d\mu \int_{\R^d} dz \int_{\R^d} d\y_i h_{cv}(z-\y_i) {\mathcal N}_{c(s-t)^{2i-1}}(z+\mu (\y_i-z)-(\btheta_{s,t}(\x))_i)\\
&\le& \Lambda\hat q_{c\setminus i}(t,s,\x,(\y_{1:i-1},\y_{i+1:n}))\\
&&\times \Bigg(\frac{1}{v^{\frac{2-\gamma}{4i-2}} (s-t)^{\sum_{j=1}^n \vartheta_j(j-\frac 12)}}+\frac{v^{\frac{2i-3+\gamma}{2(2i-1)}}}{(s-t)^{\sum_{j=1}^n \vartheta_j(j-\frac 12)+\frac{2i-1}{2}}}    +\frac{1}{(s-t)^{\sum_{j=1}^n \vartheta_j(j-\frac 12)+1-\frac \gamma 2}}\Big),
\end{eqnarray*}
using the change of variable $(w_1, w_2)=(z-\y_i,z+\mu (\y_i-z)-(\btheta_{s,t}(\x))_i) $ for the last inequality.
%
From the above computations and with the notations of \eqref{DECOUP_HK_INT}, we derive:
\begin{eqnarray*}
&&\int_0^{(s-t)^{\beta_i}} dv v^{\frac{\tilde \alpha_i}2-1}\|h_v\star \Psi_{i,(t,\x),(s,\y_{1:i-1},\y_{i+1})}^\vartheta\|_{L^1(\R^d,\R)}\\
&\le& \Lambda\hat q_{c\setminus i}(t,s,\x,(\y_{1:i-1},\y_{i+1:n}))\int_0^{(s-t)^{\beta_i}} \frac{dv}{v}v^{\frac{\tilde \alpha_i}{2}}
\\
&&\times \Bigg(\frac{1}{v^{\frac{2-\gamma}{4i-2}} (s-t)^{\sum_{j=1}^n \vartheta_j(j-\frac 12)}}+\frac{v^{\frac{2i-3+\gamma}{2(2i-1)}}}{(s-t)^{\sum_{j=1}^n \vartheta_j(j-\frac 12)+\frac{2i-1}{2}}}    +\frac{1}{(s-t)^{\sum_{j=1}^n \vartheta_j(j-\frac 12)+1-\frac \gamma 2}}\Bigg)\\
&=:&\Lambda\hat q_{c\setminus i}(t,s,\x,(\y_{1:i-1},\y_{i+1:n}))  {\mathcal B}_{\vartheta,\beta_i}(t,s).
\end{eqnarray*}
Let us now prove that for $\beta_i=\frac{(2i-3)(2i-1)}{2i-3-\gamma}
$ defined in \eqref{DEF_BETA_I}, we have:
\begin{equation}
\label{CTR_EXP_TEMPS_PETIT}
{\mathcal B}_{\vartheta,\beta_i}(t,s) \le \frac{C}{(s-t)^{\sum_{j=1}^n \vartheta_j(j-\frac 12)-\frac \gamma 2}}.
\end{equation}
To prove \eqref{CTR_EXP_TEMPS_PETIT}, we now write: 
\begin{eqnarray*}
&&{\mathcal B}_{\vartheta,\beta_i}(t,s)
\nonumber \\
&\le& C 
\Big[\frac{v^{\frac{\tilde \alpha_i}{2}-\frac{2-\gamma}{4i-2}}}{(s-t)^{\sum_{j=1}^n \vartheta_j(j-\frac 12)}}+\frac{v^{\frac{\tilde \alpha_i}2+\frac{2i-3+\gamma}{2(2i-1)}}}{(s-t)^{\sum_{j=1}^n \vartheta_j(j-\frac 12)+\frac{2i-1}{2}}}    +\frac{v^{\frac{\tilde \alpha_i}2}}{(s-t)^{\sum_{j=1}^n \vartheta_j(j-\frac 12)+1-\frac \gamma 2}}\Big]_{v=0}^{v=(s-t)^{\beta_i}}\\
&\le&  C \Big[(s-t)^{\beta_i(\frac{\tilde \alpha_i}{2}-\frac{2-\gamma}{4i-2})-\sum_{j=1}^n \vartheta_j(j-\frac 12)}+(s-t)^{\beta_i(\frac{\tilde \alpha_i}2+\frac{2i-3+\gamma}{2(2i-1)}) -(\sum_{j=1}^n \vartheta_j(j-\frac 12)+\frac{2i-1}{2})}  \\
   &&
   +(s-t)^{\beta_i \frac{\tilde \alpha_i}2-\sum_{j=1}^n \vartheta_j(j-\frac 12)-1+\frac \gamma 2}\Big].
\end{eqnarray*}
From the above equation, \eqref{CTR_EXP_TEMPS_PETIT} holds as soon as $\beta_i $ can be chosen so that the three following conditions hold:
\begin{eqnarray*}
\beta_i\bigg(\frac{\tilde \alpha_i}{2}-\frac{2-\gamma}{4i-2}\bigg)-\frac{\gamma}2\ge 0,\quad \beta_i\bigg(\frac{\tilde \alpha_i}{2}+\frac{2i-3+\gamma}{2(2i-1)}\bigg) -\frac{2i-1}{2}-\frac \gamma 2\ge 0,\quad \beta_i \frac{\tilde \alpha_i}2-1\ge 0.
\end{eqnarray*}
Recalling that $\frac{\tilde \alpha_i}2=\frac{1+\frac{\gamma}2}{2i-1} $ and for the previous choice of $\beta_i $, the above conditions rewrite:
\begin{eqnarray*}
\bigg(\frac{(2i-3)(2i-1)}{2i-3-\gamma}\bigg)\bigg(\frac{1+\frac{\gamma}2}{2i-1}-\frac{1-\frac \gamma 2}{2i-1}\bigg)-\frac{\gamma}2\ge 0 \iff \frac{(2i-3)}{2i-3-\gamma} \gamma-\frac \gamma 2\ge 0,\\
\bigg(\frac{(2i-3)(2i-1)}{2i-3-\gamma}\bigg)\bigg(\frac{2+\gamma}{2(2i-1)}+\frac{2i-3+\gamma}{2(2i-1)}\bigg) -\frac{2i-1}{2}-\frac \gamma 2\ge 0 \\
\iff \bigg(\frac{2i-3}{2i-3-\gamma}\bigg)(2i-1+2\gamma)-(2i-1+\gamma)\ge 0,
 \\
  \bigg(\frac{(2i-3)(2i-1)}{2i-3-\gamma}\bigg)\frac{1+\frac{\gamma}2}{2i-1} -1\ge 0\iff  \bigg(\frac{2i-3}{2i-3-\gamma}\bigg)(1+\frac \gamma 2)-1 \ge 0.
\end{eqnarray*}
All the above conditions are true for $i\in \leftB 2,n\rightB,\ \gamma\in (0,1] $. Note that the chosen $\beta_i $  seems to be rather sharp in the sense that letting $\gamma $ go to 0 the above constraints become equalities. This proves \eqref{CTR_EXP_TEMPS_PETIT}. We finally get:
\begin{equation}
\label{CTR_IN_TIME_BESOV_NORM}
\int_0^1 \frac{dv}{v}v^{\frac{\tilde \alpha_i}2 }\|h_v\star \Psi_{i,(t,\x),(s,\y_{1:i-1},\y_{i+1:n})}^{\vartheta}\|_{L^1(\R^d,\R)}\le \frac{\Lambda}{(s-t)^{\sum_{j=1}^n \vartheta_j(j-\frac 12)-\frac \gamma 2}}\hat q_{c\setminus i}(t,s,\x,(\y_{1:i-1},\y_{i+1:n})).
\end{equation}
Reproducing the previous computations we also write for a $C^\infty $ compactly supported function $\varphi $:
\begin{eqnarray*}
&&\|\varphi (D) \Psi_{i,(t,\x),(s,\y_{1:i-1},\y_{i+1:n})}^\vartheta\|_{L^1(\R^d,\R)}\\
&\le& \int_{\R^d } \Big|\int_{\R^d} D_{\y_i} \textcolor{black}{\varphi^\vee} (z-\y_i) \cdot \big ( D_\x^\vartheta \tilde p^{\bxi}(t,s,\x,\y)
\otimes \Delta_{i,\gF}(t,s,\btheta_{s,t}(\bxi),\y)\rangle d\y_i\Big|dz  \bigg |_{\bxi=\x} \\
&\le& \frac{\Lambda}{(s-t)^{\sum_{j=1}^n \vartheta_j(j-\frac 12)}}\int_{\R^d} \bar p_{c^{-1}}(t,s,\x,\y) \d^{2i-3+\gamma }(\btheta_{s,t}(\x),\y)d\y_i\\
&\le & \frac{\Lambda}{(s-t)^{\sum_{j=1}^n \vartheta_j(j-\frac 12)-(i-\frac 32+\frac \gamma 2)}} \hat q_{c\setminus i}(t,s,\x,(\y_{1:i-1},\y_{i+1:n})).
\end{eqnarray*}
From  \eqref{THERMIC_CAR} and \eqref{CTR_IN_TIME_BESOV_NORM}, we finally obtain:
\begin{equation}
\|\Psi_{i,(t,\x),(s,\y_{1:i-1},\y_{i+1:n})}^\vartheta\|_{B_{1,1}^{-\tilde \alpha_i}}\le \frac{\Lambda}{(s-t)^{\sum_{j=1}^n \vartheta_j(j-\frac 12)-\frac \gamma 2}}\hat q_{c\setminus i}(t,s,\x,(\y_{1:i-1},\y_{i+1:n})),
\end{equation}
which together with  \eqref{FROM_WHERE_TO_USE_BESOV_DUALITY} and \eqref{DEF_PSI_M}  gives the result.
\end{proof}

Equation \eqref{CTR_BORNE_SUP_VAR_DEG} now follows from Lemma \ref{LEMME_BESOV_DEG} taking $\vartheta=(2,0,\hdots,0)$. Namely, 
\begin{equation}\label{ineq_Rm_2_n}
\Big|\sum_{i=2}^n\int_t^T ds \int_{\R^{nd}}
D_{\x_1}^2\tilde p^{\bxi}(t,s,\x,\y) \big \langle \Delta_{i,\gF}(t,s,\btheta_{s,t}(\bxi),\y),  D_{\y_i} u(s,\y)  \big \rangle \Big )\Big| \bigg |_{\bxi=\x}
\le  \Lambda (T-t)^{\frac \gamma 2} \|u\|_{L^\infty(C_{b,\d}^{2+\gamma})}.
\end{equation}

\subsection{Non-degenerate derivatives  for the frozen semi-group : terminal condition and source
}\label{pragestiderivsemietgreen}
The main result of this section is the following lemma.
\begin{lem}[Derivatives of frozen semi-group and Green kernel] \label{CTR_SEMI_GROUP_GREEN_FROZEN}
There exists a constant $C:=C(\A{A})$ s.t. for any $(t,\x)\in [0,T]\times \R^{nd}$,
\begin{eqnarray*}
|D_{\x_1}^2\tilde P_{T,t}^\bxi g(\x)| \Big|_{\bxi=\x}&\le& C \|D_{\x_1}^2g\|_{L^\infty} \le C \|g\|_{C_{b,\d}^{2+\gamma}},\\
|D_{\x_1}^2\tilde G^{\bxi} f(t,\x)|\Big|_{\bxi=\x}
&\leq& C (T-t)^{\frac \gamma 2} \|f\|_{L^\infty(C_{b,\d}^\gamma)}.
\end{eqnarray*}
\end{lem}
\begin{proof}[Proof of Lemma \ref{CTR_SEMI_GROUP_GREEN_FROZEN}]

\textcolor{black}{Note first that,} 
\begin{eqnarray}\label{D2_tildeP_g}
\big|D_{\x_1}^2\tilde P_{T,t}^\bxi g(\x)\big|\Big|_{\bxi=\x}&=&\Big|\int_{\R^{nd}} D_{\x_1}^2 \tilde p^{\bxi}(t,T,\x,\y) [g(\y)-g(\m_{T,t}^\bxi(\x))] d\y\Big| \bigg|_{\bxi=\x}\notag\\
&\le& \Big|\int_{\R^{nd}} D_{\x_1}^2 \tilde p^{\bxi}(t,T,\x,\y) [g(\y)-g(\y_1,\m_{T,t}^\bxi(\x)_{2:n})] d\y\Big| \bigg|_{\bxi=\x}
\nonumber \\
&&+ \Big|\int_{\R^{nd}} D_{\x_1}^2 \tilde p^{\bxi}(t,T,\x,\y) [g(\y_1,(\m_{T,t}^\bxi(\x))_{2:n})-g(\m_{T,t}^\bxi(\x))] d\y\Big| \bigg|_{\bxi=\x} .
\end{eqnarray}
The first term in the r.h.s. of the previous identity is readily controlled thanks to Proposition \ref{THE_PROP}
\begin{eqnarray}\label{D2_tildeP_g1}
&&\Big|\int_{\R^{nd}} D_{\x_1}^2 \tilde p^{\bxi}(t,T,\x,\y) [g(\y)-g(\y_1,(\m_{T,t}^\bxi(\x))_{2:n})] d\y\Big| \bigg|_{\bxi=\x}\nonumber \\
&\leq& C \|g\|_{C_{b,\d}^{2+\gamma}} \int_{\R^{nd}}(T-t)^{-1} \bar p_{C^{-1}}(t,T,\x,\y) \d^{2+\gamma}(\m_{T,t}^\bxi(\x),\y)  d\y \Big|_{\bxi=\x} 
\nonumber \\
 &\leq&  C (T-t)^{\frac \gamma 2}\|g\|_{C_{b,\d}^{2+\gamma}}  .
\end{eqnarray}
The second term of \eqref{D2_tildeP_g} is more subtle. We need to expand $g(\y_1,(\m_{T,t}^\bxi(\x))_{2:n})$ in its non-degenerate variable to take advantage of the corresponding regularity of $g$.
Namely, recalling from Proposition \ref{Prop_moment_D2_tilde_p} that 
\begin{equation*}
\int_{\R^d} d\y D_{\x_1}^2 \tilde p^{m,\bxi}(t,T,\x,\y) \big \langle  D_{\x_1} g(\m_{T,t}^\bxi(\x)), (\y-\m_{T,t}^\bxi(\x))_1 \big \rangle
 \Big|_{\bxi=\x}=\0_{d,d},
\end{equation*}
we obtain
\begin{eqnarray}
&&\Big|\int_{\R^{nd}} D_{\x_1}^2 \tilde p^{\bxi}(t,T,\x,\y) [g(\y_1,(\m_{T,t}^\bxi(\x))_{2:n})-g(\m_{T,t}^\bxi(\x))] d\y\Big| \bigg|_{\bxi=\x}
\nonumber \\
&=&  \bigg|\int_{\R^{nd}} D_{\x_1}^2 \tilde p^{\bxi}(t,T,\x,\y) \bigg( \big  \langle D_{\x_1} g(\m_{T,t}^\bxi(\x)),(\y-\m_{T,t}^\bxi(\x))_1\big \rangle 
\nonumber \\
&&+ \int_0^1 d\mu  (1-\mu ) \rm{Tr} \Big ( D_{\x_1}^2 g\big(\m_{T,t}^\bxi(\x)_1
+\mu (\y-\m_{T,t}^\bxi(\x))_1,(\m_{T,t}^\bxi(\x))_{2:n} \big) 
\big (\y-\m_{T,t}^\bxi(\x) \big)_1^{\otimes 2} \Big )
 \bigg) d\y\bigg| \Bigg|_{\bxi=\x} \nonumber \\
&\leq& C \|D_{\x_1}^2g\|_{L^\infty} \int_{\R^{nd}}  \frac{\bar p_{C^{-1}}(t,T,\x,\y)}{T-t}|(\y-\m_{T,t}^\bxi(\x))_1|^2 d\y\bigg|_{\bxi=x} 
\nonumber \\
&\le& C\|g\|_{C_{b,\d}^{2+\gamma}}\label{D2_tildeP_g2}.
\end{eqnarray}
Gathering identities \eqref{D2_tildeP_g1}, \eqref{D2_tildeP_g2} 
into \eqref{D2_tildeP_g}, we obtain the stated control for $|D_{\x_1}^2\tilde P_{T,t}^\bxi g(\x)|\Big|_{\bxi=\x}$.

Let us now turn to the Green kernel. We  directly get from Proposition \ref{THE_PROP}:
\begin{eqnarray*}
|D_{\x_1}^2\tilde G^{\bxi} f(t,\x)|\Big|_{\bxi=\x}&\le& \Big|\int_t^T ds \int_{\R^{nd}} D_{\x_1}^2\tilde p^{\bxi}(t,s,\x,\y) [f(s,\y)-f(s,\btheta_{s,t}(\bxi))] d\y\Big| \Big|_{\bxi=\x} \notag \\
&\le& C\|f\|_{L^\infty(C_\d^\gamma)}\int_t^T ds \int_{\R^{nd}} \frac{1}{s-t}\bar p_{C^{-1}}(s,t,\x,\y) \d( \btheta_{s,t}(\x), \y)^\gamma d\y\\
&\le& C\|f\|_{L^\infty(C_\d^\gamma)}(T-t)^{\frac \gamma 2}, 
\end{eqnarray*}
which gives the result.

\end{proof}


\mysection{H\"older controls} \label{HOLDER}

In this section, we aim at giving suitable controls on the H\"older moduli $ [D_{\x_1}^2 u(t,\cdot)]_{\gamma,\d}$ and $\sup_{z \in \R^d} [u(t,z,\cdot)]_{2+\gamma,\d}$ in order to derive our main Schauder estimate of Theorem \ref{THEO_SCHAU}. Namely, we want to establish the following result.

\begin{PROP}\label{PROP_HOLDER_CTRL}\quad
Let $\gamma\in (0,1)$ be given. Suppose that \A{A} is in force and that the terminal condition $g$ and source term $f$ of the Cauchy problem \eqref{KOLMO} satisfy: $g\in C_{b,\d}^{2+\gamma }(\R^{nd},\R)$ and $f\in L^\infty\big([0,T],C_{b,\d}^{\gamma}( \R^{nd},\R)\big)$. Then, there exists $C:=C(\A{A},T)$ and $ \Lambda:=\Lambda (\A{A})$ as in Remark \ref{REM_LAMBDA} such that \textcolor{black}{for any $c_0\in (0,1]$}:
\begin{equation}\label{ESTI_HOLDER_MODULI}
[D_{\x_1}^2 u(t,\cdot)]_{\gamma,\d} + \sup_{z \in \R^d} [u(t,z,\cdot)]_{2+\gamma,\d} \leq C \Big\{\| g\|_{C^{2+\gamma}_{b,\d}}  +   \|f\|_{L^\infty(C_{b,\d}^{\gamma})} \Big\}+\Big(\Lambda( c_0^{-(n-\frac 12)+\frac \gamma 2}+c_0^{\frac \gamma 2})+C c_0^{\frac{\gamma}{2n-1}} \Big)\|u\|_{L^\infty(C^{2+\gamma}_{b,\d}) }.
\end{equation}
\end{PROP}

To prove this result the point is to consider for a fixed time $t\in [0,T] $ and for fixed $\x\in \R^{nd} $ the perturbative expansion \eqref{DUHAMEL_PERTURB} and for another spatial point $\x'\in {\R^{nd}}$, the possibly more refined version provided by equation \eqref{INTEGRATED_DIFF_BXI} which precisely allows to take into account the various regimes depending on $\d(\x,\x') $ and $(s-t)^{1/2}, s\in [t,T] $ detailed in the previous detailed guide to the proof (see Section \ref{GUIDE_TO_PROOF} paragraph \ref{subsec_detailedguide_regime}).

Hence, we will address separately two cases. For a constant $c_0$ to be specified later on (but formally meant to be \textit{small}), we consider:
\begin{trivlist}
\item[$\bullet$] The globally \textit{off-diagonal} regime $T-t< c_0 \d^2(\x,\x')$. In that case, the spatial points $\x,\x'$ are globally \textit{far} for the corresponding homogeneous distance $\d$ over   the time horizon $s\in [t,T]$. Hence, there is no specific need to exploit \eqref{INTEGRATED_DIFF_BXI}. Expanding the quantities  $D_{\x_1}^2 u(t,\x)-D_{\x_1}^2 u(t,\x') $ and $u(t,z,\x_{2:n})-u(t,z,\x_{2:n}') $ with \eqref{DUHAMEL_PERTURB} is enough to get the result.

Indeed, writing from \eqref{DUHAMEL_PERTURB} the for any $(\x,\x') \in (\R^{nd})^2$ s.t. $(T-t)<c_0\d^2(\x,\x') $,
\begin{eqnarray}\label{EXP_HD_GLOBAL_DER_NON_DEG}
&&D_{\x_1}^2 u(t,\x)- D_{\x_1}^2 u(t,\x')\notag\\
&=&\Bigg\{D_{\x_1}^2 \tilde P_{T,t}^{\bxi}g(\x)-D_{\x_1'}^2 \tilde P_{T,t}^{\bxi'}g(\x')+D_{\x_1}^2\tilde G^\bxi f(t,\x)-D_{\x_1}^2\tilde G^{\bxi'} f(t,\x')\nonumber \\
&&+ \int_t^T ds \int_{\R^{nd}} \Big(D_{\x_1}^2 \tilde p^\bxi(t,s,\x,\y) (L_s-\tilde L_s^\bxi)u(s,\y) - D_{\x_1}^2 \tilde p^{\bxi'}(t,s,\x',\y) (L_s-\tilde L_s^{\bxi'})u(s,\y) \Big)d\y\Bigg\}\Bigg|_{(\bxi,\bxi')=(\x,\x')},
\end{eqnarray} 
and similarly, for any $z\in \R^d$,
\begin{eqnarray}\label{EXP_GLOBAL_HOLDER_DEG}
&& u\big(t,(z,\x_{2:n})\big)-  u\big(t,(z,\x'_{2:n})\big)\notag\\
&=&\Bigg\{\tilde P_{T,t}^{\bxi}g(z,\x_{2:n})- \tilde P_{T,t}^{\bxi'}g(z,\x'_{2:n})+\tilde G^\bxi f\big(t,(z,\x_{2:n})\big)-\tilde G^{\bxi'} f\big(t,(z,\x'_{2:n})\big)\notag\\
&&+ \int_t^{T} ds \int_{\R^{nd}} \Big( \tilde p^\bxi(t,s,(z,\x_{2:n}),\y) (L_s-\tilde L_s^\bxi)u(s,\y)\notag\\
&&\hspace*{2cm}-  \tilde p^{\bxi'}(t,s,(z,\x_{2:n}'),\y) (L_s-\tilde L_s^{\bxi'})u(s,\y) \Big)d\y\Bigg\}\Bigg|_{(\bxi,\bxi')=\big((z,\x_{2:n}),(z,\x_{2:n}')\big)}.
\end{eqnarray} 
In this case, we derive from Lemmas \ref{CTR_DER_SUP_non_deg}, \ref{LEMME_BESOV_DEG}, \ref{CTR_SEMI_GROUP_GREEN_FROZEN} and equation \eqref{def_R1_R2} that:
\begin{eqnarray*}
&&|D_{\x_1}^2 u(t,\x)- D_{\x_1}^2 u(t,\x')|\\
&\le& C\Big[ |D_{\x_1}^2 \tilde P_{T,t}^{\bxi}g(\x)-D_{\x_1'}^2 \tilde P_{T,t}^{\bxi'}g(\x')|\big|_{(\bxi,\bxi')=(\x,\x')}+\big(\|f\|_{L^\infty(C_{b,\d}^\gamma)}+ \Lambda \|u\|_{L^\infty(C_{b,\d}^{2+\gamma})}\big)(T-t)^{\frac \gamma 2}\Big]\\
&\le& C\Big[ |D_{\x_1}^2 \tilde P_{T,t}^{\bxi}g(\x)-D_{\x_1'}^2 \tilde P_{T,t}^{\bxi'}g(\x')|\big|_{(\bxi,\bxi')=(\x,\x')}+\big(\|f\|_{L^\infty(C_{b,\d}^\gamma)}+ \Lambda \|u\|_{L^\infty(C_{b,\d}^{2+\gamma})}\big)c_0^{\frac\gamma 2}\d^{\gamma }(\x,\x')\Big],
\end{eqnarray*}
and 
\begin{eqnarray*}
&&|u\big(t,(z,\x_{2:n})\big)-  u\big(t,(z,\x'_{2:n})\big)|\\
&\le& C \bigg(|\tilde P_{T,t}^{\bxi}g(z,\x_{2:n})- \tilde P_{T,t}^{\bxi'}g(z,\x'_{2:n})|\big|_{(\bxi,\bxi')=\big((z,\x_{2:n}),(z,\x_{2:n}')\big)}\\&&+|\tilde G^\bxi f\big(t,(z,\x_{2:n})\big)-\tilde G^{\bxi'} f\big(t,(z,\x'_{2:n})\big)|\big|_{(\bxi,\bxi')=\big((z,\x_{2:n}),(z,\x_{2:n}')\big)} +\|u\|_{L^\infty(C_{b,\d}^{2+\gamma})} c_0^{\frac \gamma 2} \d^\gamma(\x,\x')\bigg).
\end{eqnarray*}
Therefore, these equations give the expected controls up to appropriate estimates for the H\"older moduli of the frozen semigroup and Green kernel (which are obtained below, for the so-called \textit{mixed} regime).\\

\item[$\bullet$] The  \textit{mixed} regime $T-t\ge c_0 \d^2(\x,\x')$. In that case, up to the \textit{transition time} $t_0$ defined in \eqref{def_t0}, $s-t<c_0 \d^2(\x,\x') $, i.e. the \textit{off-diagonal} regime holds for the times $s$ \textit{close} to $t$. Things are hence more involved. In particular, it is then crucial to exploit the more refined perturbative expansion \eqref{INTEGRATED_DIFF_BXI} to derive suitable bounds for $D_{\x_1}^2 u(t,\x)-D_{\x_1}^2 u(t,\x') $ and $u\big(t,(z,\x_{2:n})\big)-u\big(t,(z,\x_{2:n}')\big) $. In particular, this leads to handle carefully the additional terms appearing from the change of freezing parameter. 

To handle our controls in such a case, we will start from identity \eqref{INTEGRATED_DIFF_BXI} for the expansion of $u(t,\x') $, where we have chosen $\bxi'=\x'$ and $\tilde \bxi'=\x$. Namely,
\begin{eqnarray*}
u(t,\x')&=&\tilde P_{T,t}^{\tilde \bxi'}g(\x')+\tilde G_{t_0,t}^{ \bxi'} f(t,\x')+\tilde G_{T,t_0}^{\tilde \bxi'} f(t,\x')
+\tilde P_{t_0,t}^{ \bxi'} u(t_0, \x')-\tilde P_{t_0,t}^{\tilde \bxi'} u(t_0, \x')\notag\\
&&+\int_t^T ds \int_{\R^{nd}}d\y\Big(\I_{s\le t_0} \tilde p^{ \bxi'}(t,s,\x',\y)(L_s-\tilde L_s^{\bxi'})+ \I_{s>t_0} p^{\tilde \bxi'}(t,s,\x',\y)(L_s-\tilde L_s^{\tilde \bxi'})  \Big)u(s,\y).\notag
\end{eqnarray*}
Again, $t_0$ must be here seen as a frozen parameter, which is \textit{a posteriori}, i.e. after possible differentiation, chosen as in \eqref{def_t0}.

According to the notations of the detailed guide to the proof (see eq. \eqref{DECOUP_MOD_HOLDER} and \eqref{DECOU_PREAL_HOLDER_PERTURB} in Section \ref{GUIDE_TO_PROOF}), the terms to control then write for the H\"older norm of the derivatives w.r.t. the non-degenerate variables:
\begin{eqnarray}
&&D_{\x_1}^2 u(t,\x)- D_{\x_1}^2 u(t,\x')\notag\\
&=&\Bigg\{ \Big(D_{\x_1}^2 \tilde P_{T,t}^{\bxi}g(\x)-D_{\x_1'}^2 \tilde P_{T,t}^{\bxi}g(\x')\Big)+ \Big(D_{\x_1}^2\tilde G_{t_0,t}^{ \bxi} f(t,\x)-D_{\x_1}^2\tilde G_{t_0,t}^{ \bxi'} f(t,\x')\Big)\notag\\
&&+\Big(D_{\x_1}^2\tilde G_{T,t_0}^{ \bxi} f(t,\x)-D_{\x_1}^2\tilde G_{t_0,t}^{ \bxi} f(t,\x')\Big)+\Big(D_{\x_1}^2\tilde P_{t_0,t}^{ \bxi'} u(t_0, \x')-D_{\x_1}^2\tilde P_{t_0,t}^{\bxi} u(t_0, \x') \Big)\notag\\
&&+D_{\x_1}^2\Delta_{{\rm \textbf{diag}}}^{\bxi,\tilde \bxi'}(t,T,\x,\x')+D_{\x_1}^2\Delta_{{\rm \textbf{off-diag}}}^{\bxi,\bxi'}(t,\x,\x')\Bigg\}\Bigg|_{(\bxi,\bxi',\tilde \bxi') = (\x,\x',\x)},\label{EXP_MIXED_DER_NON_DEG}
\end{eqnarray}  
with
\begin{eqnarray}
\Delta_{{\rm \textbf{off-diag}}}^{\bxi,\bxi'}(t,\x,\x')&=&
\int_t^T ds \int_{\R^{nd}}d\y \tilde p^{\bxi}(t,s,\x,\y)\I_{s\le t_0} (L_s-\tilde L_s^{\bxi})u(s,\y)
\nonumber \\
&&-\int_t^T ds \int_{\R^{nd}}d\y \tilde p^{\bxi'}(t,s,\x',\y)\I_{s\le t_0} (L_s-\tilde L_s^{\bxi'})u(s,\y),
\notag\\
\Delta_{{\rm \textbf{diag}}}^{\bxi,\tilde \bxi'}(t,T,\x,\x')&=&
\int_t^T ds \int_{\R^{nd}}d\y \tilde p^{ \bxi}(t,s,\x,\y)\I_{s>t_0}  (L_s-\tilde L_s^{\bxi})u(s,\y)
\nonumber \\
&&-\int_t^T ds \int_{\R^{nd}}d\y \tilde p^{\tilde  \bxi'}(t,s,\x',\y)\I_{s>t_0}  (L_s-\tilde L_s^{\tilde \bxi'})u(s,\y),
\label{DECOUP_MOD_HOLDER_SANS_M}
\end{eqnarray}
and 
\begin{eqnarray}
&& u\big(t,(z,\x_{2:n})\big)-  u\big(t,(z,\x_{2:n}')\big)\notag\\
&=&\Bigg\{ \Big(\tilde P_{T,t}^{\bxi}g\big(z,\x_{2:n}\big)- \tilde P_{T,t}^{\bxi}g\big(z,\x_{2:n}'\big)\Big)+ \Big(\tilde G_{t_0,t}^{ \bxi} f\big(t,(z,\x_{2:n})\big)-\tilde G_{t_0,t}^{ \bxi'} f\big(t,(z,\x_{2:n}')\big)\Big)\notag\\
&&+\Big(\tilde G_{T,t_0}^{ \bxi} f\big(t,(z,\x_{2:n})\big)-\tilde G_{t_0,t}^{ \bxi} f\big(t,(z,\x_{2:n})\big)\Big)+\Big(\tilde P_{t_0,t}^{ \bxi'} u\big(t_0,(z,\x_{2:n}')\big)-\tilde P_{t_0,t}^{\bxi} u\big(t_0,(z,\x_{2:n}')\big) \Big)\notag\\
&&+\Delta_{{\rm \textbf{diag}}}^{\bxi,\tilde \bxi'}(t,T,(z,\x_{2:n}),(z,\x_{2:n}'))+\Delta_{{\rm \textbf{off-diag}}}^{\bxi,\bxi'}(t,(z,\x_{2:n}),(z,\x_{2:n}'))\Bigg\}\Bigg|_{(\bxi,\bxi',\tilde \bxi') = ((z,\x_{2:n}),(z,\x_{2:n}'),(z,\x_{2:n}))}\!\!\!.\notag\\
\label{EXP_MIXED_HOLDER_DEG}
\end{eqnarray}
for the H\"older moduli w.r.t. the degenerate variables according to the previously introduced notations. 

To derive the expected bounds we will then devote a subsection to the H\"older controls for the frozen semi-group (see Lemma \ref{lem_Holder_tildePg} in Section \ref{SUB_lem_Holder_tildeP}), for the frozen Green kernel (see Lemma \ref{lemme_Holder_Gg_VAR_DEG_ET_D2_NON_DEG} in Section \ref{SUB_lemme_Holder_Gg_VAR_DEG_ET_D2_NON_DEG}), for the discontinuity term coming from the change of freezing point (see Lemma \ref{CTR_TERME_DISC} in Section \ref{SEC_DISC}) and for the perturbative contribution (see Lemma \ref{SUB_HOLDER_CTR_PERT_PART} in Section \ref{SUB_HOLDER_CTR_PERT_PART}). Aggregating the previously mentioned \textcolor{black}{l}emmas directly yields Proposition \ref{PROP_HOLDER_CTRL}. \qed 
\end{trivlist}

\subsection{H\"older norms for the frozen semi-group 
}\label{SUB_lem_Holder_tildeP}
We precisely want to establish the following result. 
\begin{lem}\label{lem_Holder_tildePg}
There exists $C:=C(\A{A})$ s.t. for any $(t,\x,\x') \in [0,T] \times \R^{nd} \times \R^{nd}$, taking $$(\bxi,\bxi')=\begin{cases} (\x,\x'), \ {\rm if}\ (T-t)^{1/2} <c_0 \d(\x,\x'),\\
(\x,\x), \ {\rm if}\ (T-t)^{1/2} \ge c_0 \d(\x,\x'),
\end{cases}$$
one has:
\begin{eqnarray*}
\big | D_{\x_1}^2 \tilde P_{T,t}^{\bxi}\textcolor{black}{g}(\x)-D_{\x_1}^2 \tilde P_{T,t}^{\bxi'}\textcolor{black}{g}(\x') \big | &\leq& C \| g\|_{C^{2+\gamma}_{b,\d}} \d^{\gamma}(\x,\x'),
\\
\big |  \tilde P_{T,t}^{\bxi}\textcolor{black}{g}(\x)- \tilde P_{T,t}^{\bxi'}\textcolor{black}{g}(\x') \big | &\leq& C \| g\|_{C^{2+\gamma}_{b,\d}}\d^{2+\gamma}(\x,\x') , \text{ for } \x_1=\x_1'.
\end{eqnarray*}
\end{lem}
\textcolor{black}{Before entering into the proof of such a result, let us emphasize that, as suggested by the above off-diagonal and diagonal splitting, the constants $C$ appearing in the r.h.s. of the above equations should depend on $c_0$. This is true, but since these terms are not planned to be passed in the r.h.s. of the final estimate, see \eqref{PREAL_FINAL} and the associated comments, we do not keep track of this dependence.}

\subsubsection{H\"older norms of the derivatives w.r.t. the non-degenerate variables}
Let us deal with the first inequality of Lemma \ref{lem_Holder_tildePg}, i.e. the H\"older norms of the derivatives w.r.t. the non-degenerate variables $\x_1$. For the frozen semi-group, we say that the off-diagonal regime (resp. diagonal regime) holds when $T-t\leq c_0 \d^2(\x,\x')$ (resp. $T-t\ge c_0 \d^2(\x,\x')$).\\

$\bullet$ Off-diagonal regime. If $T-t\leq c_0 \d^2(\x,\x')$,
like in \eqref{D2_tildeP_g}, we write:
\begin{eqnarray}\label{D2_tildeP_g_Holder}
&&D_{\x_1}^2 \tilde P_{T,t}^{\bxi}g(\x)- D_{\x_1}^2 \tilde P_{T,t}^{\bxi'}g(\x')
\nonumber \\
&=&
\bigg [ \int_{\R^{nd}} D_{\x_1}^2 \tilde p^{\bxi}(t,T,\x,\y) [g(\y)-g(\y_1,(\m_{T,t}^\bxi(\x))_{2:n})] d\y
\nonumber \\
&& \quad -\int_{\R^{nd}} D_{\x_1}^2 \tilde p^{\bxi'}(t,T,\x',\y) [g(\y)-g(\y_1(,\m_{T,t}^{\bxi'}(\x'))_{2:n})] d\y \bigg ]
\nonumber \\
&&+\bigg [\int_{\R^{nd}} D_{\x_1}^2 \tilde p^{\bxi}(t,T,\x,\y) [g(\y_1,(\m_{T,t}^\bxi(\x))_{2:n})-g(\m_{T,t}^\bxi(\x))] d\y
\nonumber \\
&&\quad -\int_{\R^{nd}} D_{\x_1}^2 \tilde p^{\bxi'}(t,T,\x',\y) [g(\y_1,(\m_{T,t}^{\bxi'}(\x'))_{2:n})-g(\m_{T,t}^{\bxi'}(\x'))] d\y \bigg ]
%
\nonumber \\
&=:& \Delta_{t,T,\bxi,\bxi'} D_{\x_1}^2\tilde P_1g(\x,\x') +\Delta_{t,T,\bxi,\bxi'} D_{\x_1}^2\tilde P_2g(\x,\x').
\end{eqnarray}
The first term, which is associated with the degenerate variables, is controlled directly thanks to \eqref{D2_tildeP_g1}, which again readily follows from Proposition \ref{THE_PROP} (see as well Remark \ref{CTR_GRAD_ET_DIST}), for $(\bxi,\bxi')=(\x,\x')$. One hence gets:
\begin{equation}\label{D2_tildeP_g1_Holder}
\Big|\Delta_{t,T,\bxi,\bxi'} D_{\x_1}^2\tilde P_1g(\x,\x')\Big| \Big|_{\bxi=\x}
\leq  2C(T-t)^{\frac \gamma 2}\|g\|_{C_{b,\d}^{2+\gamma}}  \leq 2 Cc_0^{\frac \gamma 2}\|g\|_{C_{b,\d}^{2+\gamma}} \d^\gamma(\x,\x').
\end{equation}
The second term is more delicate, we proceed like in \eqref{D2_tildeP_g2}: 
\begin{eqnarray}\label{D2_tildeP_g2_Holder}
&&\Delta_{t,T,\bxi,\bxi'} D_{\x_1}^2\tilde P_2g(\x,\x')
\nonumber \\
&=& \Bigg [ \int_{\R^{nd}} D_{\x_1}^2 \tilde p^{\bxi}(t,T,\x,\y)  \int_0^1 d\mu  (1-\mu )
\nonumber \\
&&{\rm{Tr}} \Big ( \big [D_{\x_1}^2 g \big (\m_{T,t}^\bxi(\x)_1
+\mu (\y-\m_{T,t}^\bxi(\x))_1,(\m_{T,t}^\bxi(\x))_{2:n}\big  )-D_{\x_1}^2 g \big (\m_{T,t}^\bxi(\x)\big )\big ]\big (\y-\m_{T,t}^\bxi(\x) \big)_1^{\otimes 2} \Big )
  d\y \nonumber \\
&-& \int_{\R^{nd}} D_{\x_1}^2 \tilde p^{\bxi'}(t,T,\x',\y)  \int_0^1 d\mu  (1-\mu )
\nonumber \\
&&{\rm{Tr}} \Big ( \big [ D_{\x_1}^2 g \big (\m_{T,t}^{\bxi'}(\x')_1
+\mu (\y-\m_{T,t}^{\bxi'}(\x'))_1,(\m_{T,t}^{\bxi'}(\x'))_{2:n}\big  )-D_{\x_1}^2 g\big (\m_{T,t}^{\bxi'}(\x') \big )\big ] \big (\y-\m_{T,t}^{\bxi'}(\x') \big )_1^{\otimes 2}
\Big) d\y \Bigg ]\nonumber \\
&+& \Bigg [ \frac 12   \int_{\R^{nd}}  D_{\x_1}^2 \tilde p^{\bxi}(t,T,\x,\y) {\rm{Tr}}\Big ( D_{\x_1}^2 g(\m_{T,t}^\bxi(\x)) (\y-\m_{T,t}^\bxi(\x))_1^{\otimes 2} \Big )d\y
\nonumber \\
&-& \frac 12   \int_{\R^{nd}}  D_{\x_1}^2 \tilde p^{\bxi'}(t,T,\x',\y) {\rm{Tr}} \Big ( D_{\x_1}^2 g(\m_{T,t}^{\bxi'}(\x')) (\y-\m_{T,t}^{\bxi'}(\x'))_1^{\otimes 2} d \y \Bigg ]
\nonumber \\
&=:& \Delta_{t,T,\bxi,\bxi'} D_{\x_1}^{2}\tilde P_{21}g(\x,\x') + \Delta_{t,T,\bxi,\bxi'} D_{\x_1}^{2}\tilde P_{22}g(\x,\x').
\end{eqnarray} 
The first contribution of the previous identity is handled exploiting the smoothness of $D_{\x_1}^2 g$ and Proposition \ref{THE_PROP}. Namely,
\begin{eqnarray*}
|\Delta_{t,T,\bxi,\bxi'} D_{\x_1}^{2}\tilde P_{21}g(\x,\x')|
&\le& C[D_{\x_1}^2g]_{\d}^{\gamma} \int_{\R^{nd}}  \frac{d\y}{(T-t)}\Big(\bar p_{C^{-1}}^\bxi(t,T,\x,\y)|(\y-\m_{T,t}^{\bxi}(\x))_1|^{2+\gamma}\\
&&+\bar p_{C^{-1}}^{\bxi'}(t,T,\x',\y)|(\y-\m_{T,t}^{\bxi'}(\x'))_1|^{2+\gamma} \Big)\Big|_{(\bxi,\bxi')=(\x,\x')}.
\end{eqnarray*}
Hence,
\begin{equation}\label{D2_tildeP_g4_Holder}
\Big |\Delta_{t,T,\bxi,\bxi'} D_{\x_1}^{2}\tilde P_{21}g(\x,\x') \Big|_{(\bxi,\bxi')=(\x,\x')}
\leq 2 C(T-t)^{\frac \gamma 2}  \|g\|_{C_{b,\d}^{2+\gamma}}
\leq 2 C c_0^{\frac \gamma 2}  \|g\|_{C_{b,\d}^{2+\gamma}} \d^\gamma(\x,\x').
\end{equation}
Let us now decompose the last contribution of \eqref{D2_tildeP_g2_Holder}:
\begin{eqnarray}
&&\Big |\Delta_{t,T,\bxi,\bxi'} D_{\x_1}^{2}\tilde P_{22}g(\x,\x') \Big|_{(\bxi,\bxi')=(\x,\x')}\nonumber\\
&\le& \Bigg\{\frac 12   \int_{\R^{nd}} \frac{d\y}{(s-t)}  \bar p_{C^{-1}}^{\bxi}(t,T,\x,\y) | D_{\x_1}^2 g(\m_{T,t}^\bxi(\x))-D_{\x_1}^2 g(\m_{T,t}^{\bxi'}(\x')) | |(\y-\m_{T,t}^\bxi(\x))_1|^2
\nonumber \\
&&+ \frac 12   \Big|\int_{\R^{nd}} D_{\x_1}^2 \tilde p^{\bxi}(t,T,\x,\y) {\rm{Tr}} \Big (D_{\x_1}^2 g(\m_{T,t}^{\bxi'}(\x')) (\y-\m_{T,t}^{\bxi}(\x))_1^{\otimes 2} \Big ) 
\nonumber\\
 &&-D_{\x_1}^2 \tilde p^{\bxi'}(t,T,\x',\y) {\rm{Tr}}\Big ( \langle D_{\x_1}^2 g(\m_{T,t}^{\bxi'}(\x')) (\y-\m_{T,t}^{\bxi'}(\x'))_1^{\otimes 2} \Big )d\y \Big | \Bigg\} \Bigg|_{(\bxi,\bxi')=(\x,\x')}
 \nonumber \\
&\le& C | D_{\x_1}^2 g(\m_{T,t}^\bxi(\x))-D_{\x_1}^2 g(\m_{T,t}^{\bxi'}(\x')) | \big|_{(\bxi,\bxi')=(\x,\x')}\notag\\
&=&C \big| \big(D_{\x_1}^2 g\big)(\btheta_{T,t}(\x))-\big(D_{\x_1}^2 g\big)(\btheta_{T,t}(\x')) \big|,\label{THE_DECOUP_QUI_BIEN_MOD_HOLDER_DX12}
\end{eqnarray}
exploiting Proposition \ref{THE_PROP} and equation \eqref{GROS_CENTER_TER} in Proposition \ref{Prop_moment_D2_tilde_p} to observe that the second contribution of the first inequality above vanishes and recalling as well \eqref{eq_m_theta_x} for the last equality to identify the linearized flows, respectively frozen in $\bxi=\x$, $\bxi'=\x' $, with the initial non-linear ones.
\\

From Lemma \ref{lem_theta_theta}, 
we derive that for $T-t \leq c_0 \d^2(\x,\x')$:
\begin{equation}\label{D2_tildeP_g2_Holder_diff_R12}
\Big |\Delta_{t,T,\bxi,\bxi'} D_{\x_1}^{2}\tilde P_{22}g(\x,\x') \Big|_{(\bxi,\bxi')=(\x,\x')}
 \leq  C \|g \|_{C^{2+\gamma}_{b,\d}}
 \big (\d^{\gamma} (\x,\x') +(T-t)^{\frac \gamma 2} \big ) .
\end{equation}
Plugging \eqref{D2_tildeP_g1_Holder}, \eqref{D2_tildeP_g4_Holder}, \eqref{D2_tildeP_g2_Holder_diff_R12} 
into \eqref{D2_tildeP_g_Holder} yields the result.
\\

$\bullet$ \textcolor{black}{Diagonal regime.} If $T-t > c_0 \d^2(\x,\x')$,
we directly write:
\begin{eqnarray}
\label{control_Holder_D2_P_g}
&&|D_{\x_1}^2 \tilde P_{T,t}^{\bxi}g(\x)- D_{\x_1}^2 \tilde P_{T,t}^{\bxi}g(\x')| 
\nonumber \\
 &\leq& 
\big |\int_{\R^{nd}} [D_{\x_1}^2\tilde p^{\bxi}(t,T,\x,\y) -D_{\x_1}^2\tilde p^{ \bxi}(t,T,\x',\y) ] g(\y) d\y  \big |
\nonumber \\
 &\leq& 
\sum_{k=1}^n \big |\int_{\R^{nd}} D_{\x_k} D_{\x_1}^2 \tilde p^{\bxi}(t,T,\x'+\mu(\x-\x'),\y) \cdot (\x-\x')_k g(\y) d\y  \big |
\nonumber 
\end{eqnarray}
This contribution is dealt through the cancellation tools of Proposition \ref{Prop_moment_D2_tilde_p} (see equations \eqref{GROS_CENTER}, \eqref{GROS_CENTER_BIS}). We get from the above estimate that:
\begin{eqnarray}
\label{control_Holder_D2_P_3g}
&&|D_{\x_1}^2 \tilde P_{T,t}^{\bxi}g(\x)- D_{\x_1}^2 \tilde P_{T,t}^{\bxi}g(\x')| 
\nonumber \\
 &\leq& \sum_{k=1}^n
\bigg |\int_{\R^{nd}} D_{\x_k} D_{\x_1}^2 \tilde p^{\bxi}(t,T,\x'+\mu(\x-\x'),\y) \cdot (\x-\x')_k \bigg [g \big (\y \big )-g\big (\m^\bxi_{T,t} \big (\x'+\mu(\x-\x')\big ) \big )
\nonumber \\
&&-\big \langle  D_{\x_1}g\big ( \m^\bxi_{T,t}  (\x'+\mu(\x-\x') )  \big ),  \big (\y-\m^\bxi_{T,t}  (\x'+\mu(\x-\x'))\big )_1\big \rangle
\nonumber \\
&&-\frac 12  {\rm Tr} \Big (   D_{\x_1}^2g\big ( \m^\bxi_{T,t} \big (\x'+\mu(\x-\x')\big )  \big (\y-\m^\bxi_{T,t}  (\x'+\mu(\x-\x') ) \big)_1^{\otimes 2}
 \Big )
  \bigg] d\y \bigg |.
\end{eqnarray}
Because $g \in C^{2+\gamma}_{b,\d}(\R^{nd},\R)$, we readily deduce, reproducing the Taylor expansion on $g$ employed for equations \eqref{D2_tildeP_g_Holder}-\eqref{D2_tildeP_g2_Holder} above,  that:
\begin{eqnarray}\label{Holder_D2_g}
&&\bigg| g \big (\y \big )-g\big (\m^\bxi_{T,t} \big (\x'+\mu(\x-\x')\big ) \big )
-\big \langle  D_{\x_1}g ( \m^\bxi_{T,t}  (\x'+\mu(\x-\x') )   ), (\y-\m^\bxi_{T,t}  (\x'+\mu(\x-\x') ) )_1\big \rangle
\nonumber \\
&&-\frac 12  {\rm Tr} \bigg ( 
D_{\x_1}^2g\big ( \m^\bxi_{T,t} \big (\x'+\mu(\x-\x')\big )  \big ) \Big (\y-\m^\bxi_{T,t} \big (\x'+\mu(\x-\x')\big ) \Big)_1^{\otimes 2} \bigg ) \bigg |\Bigg |_{\bxi=\x}
\nonumber \\
&\leq& \|g\|_{C_{b,\d}^{2+\gamma}} \d^{2+\gamma}\big (\y,\m^\bxi_{T,t} \big (\x'+\mu(\x-\x')\big )\big ) \Big |_{\bxi=\x}.
\end{eqnarray}
Plugging this inequality into \eqref{control_Holder_D2_P_3g} yields:
\begin{eqnarray}
\label{control_Holder_D2_P_g1}
&&|D_{\x_1}^2 \tilde P_{T,t}^{\bxi}g(\x)- D_{\x_1}^2 \tilde P_{T,t}^{\bxi}g(\x')| \big |_{\bxi=\x}
\nonumber \\
 &\leq& C \|g\|_{L^\infty(C_{b,\d}^{2+\gamma})} \sum_{k=1}^n
 \int_{\R^{nd}}(T-t)^{-1-(k-\frac 12)} \Big\{\bar p_{C^{-1}}^\bxi(t,T,\x'+\mu(\x-\x'),\y) |(\x-\x')_k|
\notag\\
 &&\times 
\d^{2+\gamma} \big (\y, \m^\bxi_{T,t} (\x'+\mu(\x-\x')\big ) \Big\}\Big |_{\bxi=\x}
 \nonumber \\
&\le&  C \|g\|_{L^\infty(C_{b,\d}^{2+\gamma})} \sum_{k=1}^n(T-t)^{-(k-\frac 12)+\frac \gamma 2}|(\x-\x')_k|\notag\\
&\le& C\|g\|_{C^{2+\gamma}},
 \d^{\gamma}(\x,\x'),
\end{eqnarray}
using \eqref{CTR_GRAD_ET_DIST} for the second inequality and recalling that, since $c_0\d^2(\x,\x')< (T-t) $, we indeed have $(T-t)^{-(k-\frac 12)+\frac \gamma 2}|(\x-\x')_k|\le (c_0 \d^2(\x,\x'))^{-(k-\frac 12)+\frac \gamma2} \d^{2k-1}(\x,\x')\le C\d^\gamma(\x,\x') $.This concludes the proof of the first inequality of Lemma \ref{lem_Holder_tildePg}.
\\

\subsubsection{H\"older control for the degenerate variables}\label{DEAL_SG_DEG}
We are here interested in proving the second estimate in  Lemma \ref{lem_Holder_tildePg} relying on the H\"older regularity of the frozen semi-group w.r.t. the degenerate variables. This proof is also based on the previous techniques. In particular,
we still take advantage of cancellation tools. For the whole paragraph we consider two arbitrary given spatial points $(\x,\x')\in (\R^{nd})^2$ s.t. $\x_1=\x_1' $, i.e. their first entry, corresponding to the non-degenerate variable, coincide.
\\

$\bullet$ Off-diagonal regime. If $T-t \leq c_0\d^2(\x,\x') $, we proceed to an expansion similar to \eqref{D2_tildeP_g_Holder} for  $D_{\x_1}^2\tilde P_{T,t}^{\x}g(\x)$.
In particular, with the notations introduced in \eqref{D2_tildeP_g_Holder}, we write:
\begin{equation}\label{tildeP_g_Holder}
\tilde P_{T,t}^{\bxi}g(\x)- \tilde P_{T,t}^{\bxi'}g(\x')
=: \Delta_{t,T,\bxi,\bxi'} \tilde P_1g(\x,\x') +\Delta_{t,T,\bxi,\bxi} \tilde P_2g(\x,\x').
\end{equation}
We directly obtain from the Proposition \ref{THE_PROP}, similarly to \eqref{D2_tildeP_g1_Holder}, that:
\begin{equation}\label{tildeP_g1_Holder}
\Big|\Delta_{t,T,\bxi,\bxi'} \tilde P_1g(\x,\x')\Big| \Big|_{\bxi=\x}
\leq  2C(T-t)^{\frac {2+\gamma} 2}\|g\|_{C_{b,\d}^{2+\gamma}}  \leq 2 C\|g\|_{C_{b,\d}^{2+\gamma}} \d^{2+\gamma}(\x,\x').
\end{equation}
We indeed recall that the difference w.r.t. \eqref{D2_tildeP_g1_Holder} is that we do not have anymore 
the time-singularities coming therein from the spatial derivatives.  
\\

 With the notations of \eqref{D2_tildeP_g2_Holder}, the second contribution of \eqref{tildeP_g_Holder} writes: 
\begin{equation*}
\Delta_{t,T,\bxi,\bxi'} \tilde P_2g(\x,\x')
= \Delta_{t,T,\bxi,\bxi'} \tilde P_{21}g(\x,\x') +\Delta_{t,T,\bxi,\bxi'} \tilde P_{22}g(\x,\x').   
\end{equation*}
Proposition \ref{THE_PROP} again yields, similarly to \eqref{D2_tildeP_g4_Holder}, that:
\begin{equation}\label{tildeP_g4_Holder}
\Big |\Delta_{t,T,\bxi,\bxi'} \tilde P_{21}g(\x,\x') \Big|_{(\bxi,\bxi')=(\x,\xi')}
\leq 2 C(T-t)^{\frac {2+\gamma }2}  \|g\|_{C_{b,\d}^{2+\gamma}}
\leq 2 C  \|g\|_{C_{b,\d}^{2+\gamma}} \d^{2+\gamma}(\x,\x').
\end{equation}
On the other hand, we readily get from Proposition \ref{Prop_moment_D2_tilde_p} that:
\begin{equation*}
\Delta_{t,T,\bxi,\bxi'} \tilde P_{22}g(\x,\x')=\frac 12{\rm Tr} \Big( D_{\x_1}^2 g(\m_{T,t}^\bxi(\x)) [\tilde \K^{\bxi}_{T,t}]_{1,1}
-   D_{\x_1}^2g(\m_{T,t}^{\bxi'}(\x'))[\tilde \K^{\bxi'}_{T,t}]_{1,1} \Big ).
\end{equation*}
Write now:
\begin{eqnarray}
\Delta_{t,T,\bxi,\bxi'} \tilde P_{22}g(\x,\x')&=&\frac 12 {\rm  Tr}\Big ([D_{\x_1}^2g(\m_{T,t}^\bxi(\x))- D_{\x_1}^2g(\m_{T,t}^{\bxi'}(\x'))] [\tilde \K^{\bxi}_{T,t}]_{1,1}\Big )
\nonumber \\  
&&+  \frac 12  {\rm Tr}\Big (D_{\x_1}^2g(\m_{T,t}^{\bxi'}(\x'))\big[ [\tilde \K^{\bxi}_{T,t}]_{1,1}-[\tilde \K^{\bxi'}_{T,t}]_{1,1}\big] \Big )
.\label{DECOUP_HOLDER_DEG_HD}
\end{eqnarray}

Since $T-t \leq c_0\d^2(\x,\x') $, recalling as well that for $\bxi=\x$, $\bxi'=\x' $,  $\m_{T,t}^{\bxi}(\x)=\btheta_{T,t}(\x)$, $\m_{T,t}^{\bxi'}(\x')=\btheta_{T,t}(\x') $, we readily  deduce from Proposition \ref{PROP_SCALE_COV} and Lemma \ref{lem_theta_theta} that:
\begin{eqnarray}
  \frac 12 \Big | [D_{\x_1}^2g(\m_{T,t}^\bxi(\x))-D_{\x_1}^2 g(\m_{T,t}^{\bxi'}(\x'))] [\tilde \K^{\bxi}_{T,t}]_{1,1} \Big | \bigg|_{(\bxi,\bxi')=(\x,\x')}
&\leq & C \|g\|_{C_{b,\d}^{2+\gamma}} \d^{\gamma}\big(\btheta_{T,t}(\x),\btheta_{T,t}(\x')\big) (T-t)\notag\\
& \le&  C \|g\|_{C_{b,\d}^{2+\gamma}} \d^{2+\gamma}(\x,\x').\label{THE_CTR_HD_1_DEG}
\end{eqnarray}
For the last contribution, we directly obtain from Lemma \ref{SENS_COV} (equation \eqref{CTR_SENSI_COV} for $j=1$):
\begin{eqnarray}\label{tildeP_g2_Holder}
\big |\frac 12  D_{\x_1}^2g(\m_{T,t}^\bxi(\x'))\big([ \tilde \K^{\bxi}_{T,t}]_{1,1}-[\tilde \K^{\bxi'}_{T,t}]_{1,1}\big) \big| \Big|_{(\bxi,\bxi')=(\x,\x')}
\!\!\!&\!\!\!\leq\!\!\!&\!\!\!  C \|g\|_{C_{b,\d}^{2+\gamma}} \big( (T-t)^{\frac {2+\gamma} 2}+(T-t)\d^{\gamma}(\x,\x') \big )
\nonumber \\
&\leq & C \|g\|_{C_{b,\d}^{2+\gamma}} \d^{2+\gamma}(\x,\x'),
\end{eqnarray}
using again that $T-t \leq c_0\d^2(\x,\x') $ for the last inequality. Plugging \eqref{THE_CTR_HD_1_DEG} and \eqref{tildeP_g2_Holder}  into \eqref{DECOUP_HOLDER_DEG_HD} gives 
\begin{eqnarray}\label{tildeP_ginfty_Holder}
\Big |\Delta_{t,T,\bxi,\bxi'} \tilde P_{22}g(\x,\x') \Big|_{(\bxi,\bxi')=(\x,\x')} \leq  C \|g\|_{C_{b,\d}^{2+\gamma}} \d^{2+\gamma}(\x,\x').
\end{eqnarray}
Bringing together \eqref{tildeP_g4_Holder}, \eqref{tildeP_ginfty_Holder} and  \eqref{tildeP_g1_Holder} in \eqref{tildeP_g_Holder} yields the result.\\

$\bullet$ Diagonal regime. If $T-t\ge c_0\d^2(\x,\x')$, we write:
\begin{eqnarray*}
&&| \tilde P_{T,t}^{\x}g(\x)-  \tilde P_{T,t}^{\x}g(\x')| 
\nonumber \\
&\leq&  
\big|\int_{\R^{nd}}[\tilde p^{\bxi}(t,T,\x,\y) - \tilde p^{ \bxi}(t,T,\x',\y) ] g(\y) d\y  \big |\Big |_{\bxi=\x}
\nonumber \\
&\leq& \bigg |\int_0^1 d\mu \int_{\R^{nd}} \langle {\mathbf D} \tilde p^{\bxi}(t,T,\x'+\mu(\x-\x'),\y) , (\x-\x')\rangle \bigg [g \big (\y \big )-g\big (\m^\bxi_{T,t} \big (\x'+\mu(\x-\x')\big ) \big )
\nonumber \\
&&- \big \langle  D_{\x_1}g  ( \m^\bxi_{T,t}  (\x'+\mu(\x-\x') )   ),  (\y-\m^\bxi_{T,t}  (\x'+\mu(\x-\x') ) )_1 \big \rangle 
\nonumber \\
&&-\frac 12  {\rm Tr} \Big ( 
 D_{\x_1}^2g ( \m^\bxi_{T,t}  (\x'+\mu(\x-\x') )   )  (\y-\m^\bxi_{T,t}  (\x'+\mu(\x-\x') ) )_1^{\otimes 2}\Big ) \bigg] d\y \bigg | \Bigg |_{\bxi=\x}\!\!\!,
\end{eqnarray*}
with the same cancellation argument as in \eqref{control_Holder_D2_P_g}. Observe anyhow that the cancellation involving the gradient in the above equation is possible precisely because $\x_1=\x_1'$ and therefore ${\mathbf D}  \tilde p^{\bxi}(t,T,\x'+\mu(\x-\x'),\y) \cdot (\x-\x')=D_{\x_{2:n}}\tilde p^{\bxi}(t,T,\x'+\mu(\x-\x'),\y) \cdot (\x-\x') _{2:n}$.
We then obtain thanks to the previous identity and \eqref{Holder_D2_g}:
\begin{eqnarray}
\label{control_Holder_D2P_g_final}
&&\!| \tilde P_{T,t}^{\x}g(\x)- \tilde P_{T,t}^{\x}g(\x')|
\nonumber \\
 &\!\leq \!&
 \|g\|_{C_{b,\d}^{2+\gamma}}
 \sum_{k=\textcolor{black}{2}}^n
\int_{\R^{nd}} \!\! \big |D_{\x_k} \tilde p^{\bxi}(t,T,\x'+\mu(\x-\x'),\y) \big | |\x_k-\x'_k|\d^{2+\gamma}\big (\y,\m^\bxi_{T,t} \big (\x'+\mu(\x-\x')\big )\big ) d\y
\nonumber \\
& \!\leq \!&
 \|g\|_{C_{b,\d}^{2+\gamma}}
 \sum_{k=\textcolor{black}{2}}^n
\int_{\R^{nd}}\!\! \frac{C}{(T-t)^{1+(k-\frac 12) - \frac{2+\gamma}{2}}} \bar p_{C^{-1}}^{\x}(t,T,\m^\x_{T,t}  (\x'+\mu(\x-\x') ),\y)  |\x_k-\x'_k|d\y
 \nonumber \\
& \!\leq \! &C
 \|g\|_{C_{b,\d}^{2+\gamma}}
 \d^{2+\gamma}(\x,\x'),
\end{eqnarray}
reproducing the arguments used to establish \eqref{control_Holder_D2_P_g1} for the last inequality. This concludes the proof of the second assertion in Lemma \ref{lem_Holder_tildePg}.

\subsection{H\"older norms associated with the Green kernel}\label{SUB_lemme_Holder_Gg_VAR_DEG_ET_D2_NON_DEG}
 
Let us recall that in \eqref{INTEGRATED_DIFF_BXI}, for a source $f\in L^\infty\big([0,T],C_{b,\d}^\gamma(\R^{nd},\R)\big) $, we have to control the H\"older norms of the Green kernel which we split into two parts 
according to the position of the time integration variable w.r.t. the change of regime time $t_0$ (see \eqref{def_t0})  \textit{a posteriori} chosen to be 
$
t_0:=\big(t+ c_ 0\d^2(\x,\x')\big)\wedge T.
$ 
This is again the splitting according to the off-diagonal and diagonal regime. The point is that for the Green kernel, if $t_0<T $ both regimes appear.

\begin{lem}\label{lemme_Holder_Gg_VAR_DEG_ET_D2_NON_DEG} Under \A{A}, for fixed spatial points $(\x,\x')\in (\R^{nd})^2 $, we have that there exists a constant $C:=C(\A{A},T)$, s.t. for any $f\in L^\infty\big([0,T], C_{b,\d}^\gamma(\R^{nd},\R) \big)$:
\begin{eqnarray*}
&&\sup_{t \in [0,T]}\big ( |D_{\x_1}^2 \tilde G_{t_0,t}^{\bxi} f(t,\x)- D_{\x_1}^2\tilde G_{t_0,t}^{\bxi'} f(t,\x')|
+|D_{\x_1}^2 \tilde G_{T,t_0}^{ \bxi} f(t,\x)- D_{\x_1}^2\tilde G_{T,t_0}^{\tilde \bxi'} f(t,\x') |\big)\notag\\
&\leq &C 
\|f\|_{L^\infty(C_{b,\d}^{\gamma})} \d^{\gamma}(\x,\x'),
\nonumber \\
\text{\textcolor{black}{and}}&&
\\
&&\sup_{t \in [0,T]}\big ( | \tilde G_{t_0,t}^{\bxi} f(t,\x)- \tilde G_{t_0,t}^{\bxi'} f(t,\x')|
+| \tilde G_{T,t_0}^{ \bxi} f(t,\x)- \tilde G_{T,t_0}^{\tilde \bxi'} f(t,\x') |\big)\notag\\
&\leq& C 
\|f\|_{L^\infty(C_{b,\d}^{\gamma})} \d^{2+\gamma}(\x,\x'), \text{ if } \x_1=\x_1',
\end{eqnarray*}
where $\bxi=\x$, $\bxi'=\x'$, $\tilde \bxi'=\x  $.
\end{lem}

\begin{proof}[Proof of Lemma \ref{lemme_Holder_Gg_VAR_DEG_ET_D2_NON_DEG}]
Let us begin with the statement concerning the second order derivatives of the frozen Green kernel w.r.t. the non-degenerate variable $\x_1$.

For the \textit{off-diagonal regime}, involving the term $D_{\x_1}^2 \tilde G_{t_0,t}^{\bxi'} f(t,\x') $, we readily get from Lemma \ref{lemme_nabla_Pg} that
\begin{eqnarray}
&&\big | D_{\x_1}^2 \tilde G_{t_0,t}^{\bxi} f(t,\x) \!-\! D_{\x_1}^2 \tilde G_{t_0,t}^{\bxi'} f(t,\x') \big | \Big|_{(\bxi,\bxi')=(\x,\x')} 
\nonumber \\
& \le&
\Big |\int_t^{t_0} \!\!ds D_{\x_1}^2 \tilde P_{s,t}^{\bxi}f(s,\x)\Big | \bigg|_{\bxi=\x}
\!\!\!+\!\Big |\int_t^{t_0} \!\!ds D_{\x_1}^2 \tilde P_{s,t}^{\bxi'}f(s,\x')\Big | \bigg|_{\bxi'=\x'}
\nonumber \\
&\!\!\!\leq& C\|f\|_{L^\infty(C_{b,\d}^\gamma)}
\int_t^{t_0} ds (s-t)^{-1+\frac \gamma 2}\notag\\
\label{CTR_D1}
&\!\!\!\leq& C\|f\|_{L^\infty(C_{b,\d}^\gamma)}\d^\gamma(\x,\x').
\end{eqnarray}
For the \textit{diagonal regime}, involving the term $D_{\x_1}^2\tilde G_{T,t_0}^{\tilde \bxi'} f(t,\x')$, we have to be more subtle and perform again a Taylor expansion of $D_{\x_1}^2 \tilde P_{s,t}^{\bxi}f(s,\cdot) $. Namely:
\begin{eqnarray*}
&& \big | D_{\x_1}^2 \tilde G_{T,t_0}^{\bxi} f(t,\x) \!-\! D_{\x_1}^2 \tilde G_{T,t_0}^{\tilde\bxi'} f(t,\x') \big | \Big|_{(\bxi,\tilde \bxi')=(\x,\x)} 
\nonumber \\
&\le& \int_{t_0}^{T } ds \Big |\int_0^1 d\mu  {\mathbf D} D_{\x_1}^2 \tilde P_{s,t}^{\bxi}f(s,\x'+\mu(\x-\x')) \cdot (\x-\x') \Big | \bigg|_{\bxi=\x} \\
&\le& \sum_{i=1}^n|(\x-\x')_i|\int_{t_0}^{T } ds \int_0^1 d\mu \Big|D_{\x_i}D_{\x_1}^2 \tilde P_{s,t}^{\bxi} f\big(s,\x'+\lambda (\x-\x')\big)\Big | \bigg|_{\bxi=\x}\\
&\le& C \|f\|_{L^\infty(C_{b,\d}^\gamma)}\sum_{i=1}^n|(\x-\x')_i|\int_{t_0}^{T } ds (s-t)^{-1-(i-\frac 12)+\frac \gamma 2},
\end{eqnarray*}
using again Lemma \ref{lemme_nabla_Pg} and the arguments of \eqref{control_Holder_D2_P_g1} for the last inequality. 
This finally yields, recalling the definition of $\d$ in \eqref{DIST} (especially that $|(\x-\x')_i| \leq \d(\x,\x')^{2i-1}$) and the fact that we chose now $t_0=\big(t+ c_ 0\d^2(\x,\x')\big)\wedge T$:
\begin{eqnarray}
\label{CTR_D2}
\big | D_{\x_1}^2 \tilde G_{T,t_0}^{\bxi} f(t,\x) \!-\! D_{\x_1'}^2 \tilde G_{T,t_0}^{\bxi'} f(t,\x') \big | \Big|_{(\bxi,\bxi')=(\x,\x)} 
&\le& C\|f\|_{L^\infty(C_{b,\d}^\gamma)}\sum_{i=1}^n|(\x-\x')_i| (\d^2(\x,\x'))^{-(i-\frac 12)+\frac \gamma 2}\notag\\
&\le& 
 C\|f\|_{L^\infty(C_{b,\d}^\gamma)}\d^\gamma(\x,\x'). 
\end{eqnarray}
Gathering \eqref{CTR_D1}  and  \eqref{CTR_D2} gives the first estimate of the \textcolor{black}{l}emma.\\

Let us now turn to the H\"older controls on the degenerate variables. 
 The idea is here again to perform a Taylor expansion at order one \textcolor{black}{ for $\x_1=\x_1'$}. Namely, from Lemma \ref{lemme_nabla_Pg}, the diagonal control is direct. We get
\begin{eqnarray}\label{ineq_D4}
&&\big |  \tilde G_{T,t_0}^{\bxi} f(t,\x) \!-\!  \tilde G_{T,t_0}^{\tilde \bxi'} f(t,\x') \big | \Big|_{(\bxi,\tilde \bxi')=(\x,\x)} 
\nonumber \\
&\le& \int_{t_0}^T ds \int_0^1 d\mu \big| \big \langle {\mathbf D} \tilde P_{s,t}^{\bxi}f\big(s,\x'+\mu (\x-\x')\big) 
,(\x-\x') \big\rangle  \big | \notag\\
&\le &  C \|f\|_{L^\infty(C_{b,\d}^\gamma)}
\int_{(t+c_0\d^2(\x,\x'))\wedge T}^T ds \sum_{i=2}^n |(\x-\x')_i| (s-t)^{-(i-\frac12)+\frac \gamma 2}\nonumber \\
&\le &  C \|f\|_{L^\infty(C_{b,\d}^\gamma)} \sum_{i=2}^n |(\x-\x')_i| \d(\x,\x')^{2-(2i-1)+ \gamma }\nonumber \\
&\le &   C \|f\|_{L^\infty(C_{b,\d}^\gamma)} \d^{2+\gamma}(\x,\x').
\end{eqnarray}

Now, for the off-diagonal bound, associated with the term $ \tilde G_{t_0,t}^{\bxi} f(t,\cdot) $, we precisely need to exploit the smoothness of $f$ associated to the fact that the semi-group $\tilde P_{s,t}^{\bxi}$ has a density.
Indeed, we cannot take advantage of the cancellation tools of  Lemma \ref{lemme_nabla_Pg}, but we have for all $\x,\x' \in \R^{nd}$ s.t. $\x_1=\x_1'$:
\begin{eqnarray}\label{ineq_tilde_Gf_off_diag1}
&&\Big |\int_t^{t_0} \!\!ds \Big(\tilde P_{s,t}^{\bxi}f(s,\x)-\tilde P_{s,t}^{\bxi'}f(s,\x')\Big)
\Big | \bigg|_{(\bxi,\bxi')=(\x,\x')} \\
&\leq &
\Big |\int_t^{t_0} \!\!ds \tilde P_{s,t}^{\bxi}f(s,\x)-f(s,\m^{\bxi}_{t_0,t}(\x))\Big | \bigg|_{\bxi=\x} 
+\Big |\int_t^{t_0} ds  \tilde P_{s,t}^{\bxi'}f(s,\x')-f(s,\m^{\bxi'}_{t_0,t}(\x'))\big | \bigg|_{\bxi'=\x'}\nonumber\\
&& + \Big |\int_t^{t_0} ds f(s,\m^{\bxi}_{t_0,t}(\x))-f(s,\m^{\bxi'}_{t_0,t}(\x')) \Big |\bigg|_{(\bxi,\bxi')=(\x,\x')}.\nonumber
\end{eqnarray}
Note that the first two terms in the r.h.s. of inequality \eqref{ineq_tilde_Gf_off_diag1} are handled like in  the previous section. Precisely, we write for the first contribution:
\begin{eqnarray*}
&&\Big |\int_t^{t_0} \!\!ds \tilde P_{s,t}^{\bxi}f(s,\x)-f(s,\m^{\bxi}_{t_0,t}(\x)) d\y\Big | \bigg|_{\bxi=\x} 
\nonumber \\
&=&
\Big |\int_t^{t_0} \!\!ds \int_{\R^{nd}} \tilde p^{\bxi}(t,T,\x,\y) [f(s,\y)-f(s,\btheta_{t_0,t}(\x)]\Big | \bigg|_{\bxi=\x}
\nonumber \\
&\leq &\|f\|_{L^\infty(C_{b,\d}^\gamma)}
\int_t^{t_0} \!\!ds \int_{\R^{nd}} \tilde p^{\bxi}(t,T,\x,\y) \d^\gamma(\y,\btheta_{t_0,t}(\x))d\y\Big|_{\bxi=\x}
\nonumber  \\
&\leq& 
 C  \|f\|_{L^\infty(C_{b,\d}^\gamma)}
\int_t^{t_0} \!\!ds (s-t)^{\frac \gamma2}d\y\\
&\leq&
C\|f\|_{L^\infty(C_{b,\d}^\gamma)}\d^{2+\gamma}(\x,\x'),
\end{eqnarray*}
by definition of $t_0$ in \eqref{def_t0}. The second term of \eqref{ineq_tilde_Gf_off_diag1} is handled similarly. 
We thus obtain:
\begin{eqnarray}\label{ineq_tilde_Gf_off_diag2}
&&\Big |\int_t^{t_0} \!\!ds \tilde P_{s,t}^{\bxi}f(s,\x)-f(s,\m^{\bxi}_{t_0,t}(\x) d\y\Big | \bigg|_{\bxi=\x} 
+\Big |\int_t^{t_0} \!\!ds \tilde P_{s,t}^{\bxi'}f(s,\x')-f(s,\m^{\bxi'}_{t_0,t}(\x') d\y\Big | \bigg|_{\bxi'=\x'} 
\nonumber \\
&\!\!\!\leq& C\|f\|_{L^\infty(C_{b,\d}^\gamma)}\d^{2+\gamma}(\x,\x').
\end{eqnarray}
For the last contribution in \eqref{ineq_tilde_Gf_off_diag1}, we have directly that:
\begin{equation}\label{ineq_tilde_Gf_off_diag3}
\Big |\int_t^{t_0} \!\!ds f(s,\m^{\bxi}_{t_0,t}(\x))-f(s,\m^{\bxi'}_{t_0,t}(\x')) \Big |\bigg|_{(\bxi,\bxi')=(\x,\x')}
\leq C\|f\|_{L^\infty(C_{b,\d}^\gamma)}
\int_t^{t_0} ds \d^{ \gamma } \big ( \btheta_{t_0,t}(\x), \btheta_{t_0,t}(\x') \big ).
\end{equation}
Lemma \ref{lem_d_theta} and \eqref{ineq_tilde_Gf_off_diag3} eventually yield:
\begin{eqnarray*}
\Big |\int_t^{t_0} \!\!ds f(s,\m^{\bxi}_{t_0,t}(\x))-f(s,\m^{\bxi}_{t_0,t}(\x')) \Big |\bigg|_{(\bxi,\bxi')=(\x,\x')}&\leq& C\|f\|_{L^\infty(C_{b,\d}^\gamma)} \big ( (t_0-t)^{1+\frac \gamma2}+ (t_0-t)\d^\gamma(\x,\x') \big )
\nonumber \\
&=&2 C\|f\|_{L^\infty(C_{b,\d}^\gamma)} \d^{2+\gamma}(\x,\x').
\end{eqnarray*}
This, together with \eqref{ineq_D4} gives the second estimate of the \textcolor{black}{l}emma.
\end{proof}

\subsection{H\"older norms of the perturbative contribution}\label{SUB_HOLDER_CTR_PERT_PART}
This section is dedicated to the investigation of the spatial H\"older continuity of the perturbative term in \eqref{EXP_MIXED_DER_NON_DEG} and \eqref{EXP_MIXED_HOLDER_DEG}. Recalling the notations \eqref{DECOUP_MOD_HOLDER_SANS_M} introduced at the beginning of this section, we prove the following \textcolor{black}{l}emma.

\begin{lem}\label{HOLDER_CTR_PERT_PART} Under \A{A}, for fixed spatial points $(\x,\x')\in (\R^{nd})^2 $, we have that there exists a constant $\Lambda :=\Lambda(\A{A},T)$ as in Remark \ref{REM_LAMBDA}, s.t.
\begin{eqnarray}
&&|D^2_{\x_1}\Delta_{{\rm \mathbf{diag}}}^{\bxi,\tilde \bxi'}(t,T,\x,\x')|\big|_{(\bxi,\tilde \bxi')=(\x,\x)} + |D^2_{\x_1} \Delta_{{\rm \mathbf{off-diag}}}^{\bxi,\bxi'}(t,\x,\x')|\big|_{(\bxi, \bxi')=(\x,\x')} \notag\\
& \leq & \Lambda (c_0^{-(n-\frac 12) + \frac \gamma{2}}+ c_0^{\frac \gamma 2})\|u\|_{L^\infty(C^{2+\gamma}_{b,\d}) }  \d^\gamma(\x,\x'),\label{esti_holder_deriv_sec_perturpart}
\end{eqnarray}
and, if we assume in addition that $\x_1=\x_1'$,
\begin{equation}\label{esti_holder_perturpart}
|\Delta_{{\rm \mathbf{diag}}}^{\bxi,\tilde \bxi'}(t,T,\x,\x')|\big|_{(\bxi,\tilde \bxi')=(\x,\x)} + |\Delta_{{\rm \mathbf{off-diag}}}^{\bxi,\bxi'}(t,\x,\x')|\big|_{(\bxi, \bxi')=(\x,\x')}  \leq   \Lambda(c_0^{-(n-\frac 12) + \frac \gamma{2}}+ c_0)\d^{2+\gamma}(\x,\x') \|u\|_{L^\infty(C_{b,\d}^{2+\gamma})}.
\end{equation}
\end{lem}

As already successfully used to establish in the previous sections to derive the regularity of the semi-group and the Green kernel we split the investigations into two parts: the first one is done when the system is in the \emph{off-diagonal} regime (i.e. for time $s \leq t_0$) and the other one when the system is in the  \emph{diagonal} regime (i.e. for time $s > t_0$). We aslo recall that the critical time giving the change of regime is \textcolor{black}{(chosen after potential differentiation)} $t_0=t+c_0\d^2(\x,\x')\wedge T$. We can assume here w.l.o.g. that $ t_0<T$ (otherwise there is a globally off-diagonal regime and the analysis becomes easier).\\

$\bullet$  Control of \eqref{esti_holder_deriv_sec_perturpart}. We decompose from definitions \eqref{DECOUP_MOD_HOLDER} and \eqref{INTEGRATED_DIFF_BXI}:
\begin{eqnarray*}
&&|D^2_{\x_1} \Delta_{{\rm \textbf{diag}}}^{\bxi,\tilde \bxi'}(t,T,\x,\x')|\Big|_{(\bxi,\tilde \bxi')=(\x,\x)
}
\notag \\
&\leq& \Big |\int_t^T ds \int_{\R^{nd}} d\y \Big [ D^2_{\x_1}  \tilde p^{\bxi} (t,s,\x,\y)- D^2_{\x_1}  \tilde p^{\bxi} (t,s,\x',\y) \Big ] \Big [ \langle \gF_{1}(s,\y)-\gF_{1}(s,\btheta_{s,t}(\bxi)), D_{\y_1} \rangle\notag\\
 &&+ \frac{1}2{\rm Tr} \Big( \big(a(s,\y)-a(s,\btheta_{s,t}(\bxi))\big) D_{\y_1}^2\Big) \Big ]u(s,\y)\I_{s>t_0} \Big| \bigg |_{
 \bxi=\x}
\notag\\
&&+ \Big| \sum_{i=2}^n \int_t^T ds\int_{\R^{nd}}d\y \Big[ D^2_{\x_1}  \tilde p^{\bxi} (t,s,\x,\y)- D^2_{\x_1}  \tilde p^{\bxi} (t,s,\x',\y)  \I_{s>t_0}\Big]
\notag \\
&&\times \Big \langle \gF_{i}(s,\y)-[\gF_{i}(s,\btheta_{s,t}(\bxi))+D_{\x_{i-1}} \gF_{i}(s,\btheta_{s,t}(\bxi))(\y-\btheta_{s,t}(\bxi))_{i-1}], D_{\y_i} u(s,\y) \Big \rangle \I_{s>t_0}\Big|\bigg |_{
 \bxi=\x},
  \end{eqnarray*}
which readily yields with the notations of \eqref{LES_DELTA_REG_LOC} that:
\begin{eqnarray}\label{def_Delta_diag1_2n}
&&|D^2_{\x_1} \Delta_{{\rm \textbf{diag}}}^{\bxi,\tilde \bxi'}(t,T,\x,\x')|\Big|_{(\bxi,\tilde \bxi')=(\x,\x)}
\notag \\
&\leq& \Big | \int_t^T \!\!ds \! \int_{\R^{nd}}  \!\! d\y \! \int_0^1  d\mu 
{\mathbf D} D^2_{\x_1}  \tilde p^{\bxi} (t,s,\x+\mu(\x'-\x),\y) \cdot (\x-\x')  
 \Delta_{1,\gF,\sigma}(t,s,\y,\btheta_{s,t}(\bxi),u)  \I_{s>t_0}\Big | \bigg |_{
 \bxi=\x}
  \nonumber \\
 &&+\Big | \int_t^T \!\! ds \! \int_{\R^{nd}} \!\!d\y \! \int_0^1 \!\!\!\! d\mu 
 {\mathbf D} D^2_{\x_1}  \tilde p^{\bxi} (t,s,\x+\mu(\x'-\x),\y) \cdot (\x-\x')
  \big \langle \Delta_{i,\gF}(t,s,\btheta_{s,t}(\bxi),\y),D_{\y_i} u(s,\y) \big \rangle  \I_{s>t_0}\Big | \bigg |_{
 \bxi=\x}
 \nonumber \\
 &=:& \big |D^2_{\x_1} \Delta_{{\rm \textbf{diag}}}^{1}(t,T,\x,\x') \big |+\big | D^2_{\x_1} \Delta_{{\rm \textbf{diag}}}^{2:n}(t,T,\x,\x')\big |.
 \nonumber \\
 \end{eqnarray}
We will now control the first term of the above right hand side. In other words, we specify the control of \eqref{ineq_D2x1_Delta1}. 
We obtain directly thanks to the smoothness assumption \A{S} on the coefficients and Proposition \ref{THE_PROP} (see also equation \eqref{CTR_GRAD_ET_DIST}) that
for each $k \in \leftB1,n \rightB$:
\begin{eqnarray}
&&|\gF_{1}(s,\y)-\gF_{1}(s,\btheta_{s,t}(\bxi))| |D_{\x_k}D_{\x_1}^2 \tilde p^{\bxi}(t,s,\x+\mu(\x'-\x),\y)|
\Big |_{  \bxi=\x}
\notag \\
&\le& C\Big (\|\gF_1(s,\cdot)\|_{C_\d^{\gamma}} \d^\gamma\big(  \y,\btheta_{s,t}(\bxi)\big) \Big)\times
(s-t)^{-1-(k-\frac 12)}\bar p_{{C}^{-1}}(t,s,\x,\y)
\Big |_{  \bxi=\x}
\notag\\
&\le& C 
\|\gF_1\|_{L^\infty(C^\gamma_\d)}
(s-t)^{-1-(k-\frac 12)+\frac \gamma 2}\bar p_{{C}^{-1}}(t,s,\x,\y)\label{CTR_DRIFT_1}.
\end{eqnarray}
Similarly, 
\begin{eqnarray}
 &&|a(s,\y)-a (s,\btheta_{s,t}(\bxi)) ||D_{\x_k}D_{\x_1}^2 \tilde p^{\bxi}(t,s,\x+\mu(\x'-\x),\y)|
\Big |_{  \bxi=\x}
 \notag \\
 &\le& C\Big(\|a(s,\cdot)\|_{C_\d^{\gamma}} \d^\gamma\big( \y,\btheta_{s,t}(\bxi)\big) \Big)\times
(s-t)^{-1-(k-\frac 12)}\bar p_{{C}^{-1}}(t,s,\x,\y) \Big |_{  \bxi=\x}
\notag \\
&\le &C\|a\|_{L^\infty(C^\gamma_\d)} (s-t)^{-1-(k-\frac 12)+\frac \gamma 2}\bar p_{{C}^{-1}}(t,s,\x,\y). \label{CTR_DIFF_1}
\end{eqnarray}

We carefully point-out that the indicated bound only depend on the supremum in time of the H\"older modulus of the coefficients (denoted $ \|\gF_1\|_{L^\infty(C^\gamma_\d)},\|a\|_{L^\infty(C^\gamma_\d)}$ respectively) and not on their supremum norm. In particular, we get from \eqref{CTR_DRIFT_1}, \eqref{CTR_DIFF_1}:
\begin{eqnarray}
\label{CTR_DIAG_NON_DEG}
|D^2_{\x_1} \Delta_{{\rm \textbf{diag}}}^{1}(t,T,\x,\x')|
&\leq &\sum_{k=1}^n \int_{t_0}
^T ds \int_{\R^{nd}} d\y (s-t)^{-1-(k-\frac 12)+\frac \gamma 2}  \nonumber \\
&&\times \Big(\|D_{\x_1}u\|_{L^\infty} \|\gF_1\|_{L^\infty(C^\gamma_\d)}+ \|D_{\x_1}^2 u\|_{L^\infty}\|a\|_{L^\infty(C^\gamma_\d)}\Big) 
\bar p_{{C}^{-1}}(t,s,\x,\y )|\x_k-\x_k'|
\nonumber \\
&
\le& \Lambda(\|D_{\x_1}u\|_{L^\infty}+\|D_{\x_1}^2 u\|_{L^\infty})\d^\gamma(\x,\x'),
\end{eqnarray}
where again the constant $\Lambda$ is \textit{small} provided the coefficients \textit{do not vary much}.
Thanks to Lemma \ref{LEMME_BESOV_DEG} 
 and the previous analysis of Section \ref{SEC_BESOV_DUAL_FIRST}, we  directly deduce:
\begin{eqnarray}
&&|D^2_{\x_1}\Delta_{{\rm \textbf{diag}}}^{2:n}(t,T,\x,\x')|\notag \\
& \leq &
 \Big| \sum_{i=2}^n \int_t^T ds\int_{\R^{nd}}d\y \int_0^1 d\mu   D_{\y_i}\cdot \bigg \{ 
 \Big (
 {\mathbf D} D^2_{\x_1}  \tilde p^{\bxi} (t,s,\x+\mu(\x'-\x),\y)\cdot (\x-\x') \Big ) \notag \\
 && \quad \otimes 
 \Big (\gF_{i}(s,\y)-[\gF_{i}(s,\btheta_{s,t}(\bxi))+D_{\x_{i-1}}\gF_{i}(s,\btheta_{s,t}(\bxi))(\y-\btheta_{s,t}(\bxi))_{i-1}] \Big )
  \bigg \}u(s,\y) \I_{s>t_0} \Big|\bigg |_{\bxi=\x} \notag \\
&\le& \Lambda\|u\|_{L^\infty(C^{2+\gamma}_{b,\d}) } \sum_{k=1}^n \int_{t_0}^T \frac{|\x_k-\x'_k|ds}{(s-t)^{1+(k-\frac 12)-\frac \gamma 2}}\notag\\
&\le& \Lambda c_0^{-(n-\frac 12)+\frac \gamma 2} \|u \|_{L^\infty (C^{2+\gamma}_{b,\d})} \sum_{k=1}^n\d^{\gamma-(2k-1)}(\x,\x') |\x_k-\x'_k|
\nonumber \\
&\le& \Lambda c_0^{-(n-\frac 12)+\frac \gamma 2} \|u\|_{L^\infty(C^{2+\gamma}_{b,\d}) }  \d^\gamma(\x,\x').\label{CTR_DIAG_DEG}
\end{eqnarray}
Plugging \eqref{CTR_DIAG_NON_DEG}  and \eqref{CTR_DIAG_DEG} into \eqref{def_Delta_diag1_2n} yields the stated control for the diagonal contribution.\\

Let us now turn to the control of $|D^2_{\x_1}\Delta_{{\rm \textbf{off-diag}}}^{\bxi,\bxi'}(t,\x,\x')|$ in \eqref{DECOUP_MOD_HOLDER_SANS_M} (or \eqref{DECOUP_MOD_HOLDER} in the detailed guide).
In the off-diagonal case, we choose $\bxi=\x$ and $\bxi'=\x'$ and
 \begin{eqnarray}
 \label{SPLIT_LAMBDA_DIAG}
|D^2_{\x_1}\Delta_{{\rm \textbf{off-diag}}}^{\bxi,\bxi'}(t,\x,\x')|\big|_{(\bxi,\bxi')=(\x,\x')} 
 &\leq & \Big |\int_t^{t_0} \!\! ds \int_{\R^{nd}}\!\!\! d\y
D^2_{\x_1}  \tilde p^{\bxi} (t,s,\x,\y) (L_s-\tilde L_s^{\bxi}) u(s,\y) \Big | \bigg |_{\bxi=\x}
\notag\\
&&+ \Big | \int_t^{t_0} \!\! ds \int_{\R^{nd}}\!\!\! d\y
  D^2_{\x_1} \tilde p^{\bxi '} (t,s,\x',\y) (L_s-\tilde L_s^{\bxi'})   u(s,\y)  \Big | \bigg |_{\bxi'=\x'}.
\end{eqnarray}
We readily get thanks to Lemmas \ref{CTR_DER_SUP_non_deg} and \ref{LEMME_BESOV_DEG}:

\begin{equation}
\label{LE_LAB_MANQUANT}
|D^2_{\x_1} \Delta_{{\rm \textbf{off-diag}}}^{\bxi,\bxi'}(t,\x,\x')|\big|_{(\bxi,\bxi')=(\x,\x')}\leq
\Lambda \sum_{i=2}^n \int_t^{t_0} \frac{ds}{(s-t)^{1-\frac \gamma 2}} \|u\|_{L^\infty(C_{b,\d}^{2+\gamma})}
\leq \Lambda c_0^{\frac \gamma 2}\d^\gamma(\x,\x') \|u\|_{L^\infty(C_{b,\d}^{2+\gamma})}.
 \end{equation}
 We point out from \eqref{CTR_DIAG_DEG} and \eqref{LE_LAB_MANQUANT} that there are  opposite impacts of the threshold $c_0$ on the constants, depending on the diagonal and off-diagonal regimes at hand.
 \\
 We eventually get the estimate \eqref{esti_holder_deriv_sec_perturpart} plugging \eqref{CTR_DIAG_NON_DEG}, \eqref{CTR_DIAG_DEG} into \eqref{def_Delta_diag1_2n} and \eqref{LE_LAB_MANQUANT}  recalling that $c_0\le 1$.\newpage
 
$\bullet$ Control of \eqref{esti_holder_perturpart}. We proceed as above from definitions \eqref{DECOUP_MOD_HOLDER_SANS_M} considering spatial points $(\x,\x') \in (\R^{nd})^2$ s.t. $\x_1=\x_1'$.
In the diagonal case, we also choose $\bxi=\tilde \bxi'=\x$ and we write similarly to \eqref{def_Delta_diag1_2n}:
\begin{eqnarray}
\label{R_DIAG_LAMBDA_2n}
&&|
\Delta_{{\rm \textbf{diag}}}
^{\bxi,\bxi'}(t,T,\x,\x')| \Big| _{(\bxi,\bxi')=(\x,\x)}\notag \\
&\leq& \Big|\int_t^T \!\! ds \int_{\R^{nd}} \!\! d\y \int_0^1 d\mu \big  \langle  D_\x   \tilde p^{\bxi} (t,s,\x+\mu(\x'-\x),\y), \x-\x' \big\rangle
\Delta_{1,\gF,\sigma}(t,s,\y,\btheta_{s,t}(\bxi)) \I_{s>t_0}\Big | \bigg |_{\bxi=\x}\notag\\
&&+ \Big| \sum_{i=2}^n \int_t^T \!\!ds\int_{\R^{nd}} \!\! d\y \int_0^1 d\mu  \big\langle D_\x   \tilde p^{\bxi} (t,s,\x+\mu(\x'-\x),\y),  \x-\x'\big \rangle
\Delta_{i,\gF}(t,s,\btheta_{s,t}(\bxi),\y)D_{\y_i}   u(s,\y)  \I_{s>t_0}\Big|\bigg |_{\bxi=\x} .
\end{eqnarray}
We have an expression rather similar to the one that appeared for the control of $[D_{\x_1}^2 u(t,\cdot)]_\gamma$ but with a weaker time singularity.

In other words, thanks to identities  \eqref{CTR_DRIFT_1}, \eqref{CTR_DIFF_1} and Lemma \ref{LEMME_BESOV_DEG} we obtain:
\begin{eqnarray}
\label{R_DIAG_LAMBDA_2n_BIS}
\Delta_{{\rm \textbf{diag}}}
^{\bxi,\bxi'}(t,T,\x,\x')| 
&\leq& \Lambda \|u\|_{L^\infty(C^{2+\gamma}_{b,\d})} \sum_{k=1}^n \int_{t_0}
^T ds (s-t)^{-(k-\frac 12)+ \frac \gamma 2} |\x_k-\x_k'|
\notag\\
&\leq& \Lambda c_0^{-(n-\frac 12)+\frac \gamma 2 }\|u\|_{L^\infty(C^{2+\gamma}_{b,\d})} \sum_{k=1}^n \d^{2(1-(k-\frac 12)+ \frac \gamma 2)}(\x,\x') |\x_k-\x_k'|
\notag\\
&\leq& \Lambda c_0^{-(n-\frac 12)+\frac \gamma 2 } \|u\|_{L^\infty(C^{2+\gamma}_{b,\d})} \d^{2+\gamma }(\x,\x'),
\end{eqnarray}
where again the constant $C$ is small if the H\"older moduli of the coefficients are small.

For the off-diagonal contribution, we get for $\x_1=\x_1'$, $\bxi=\x$ and $\bxi'=\x'$:
 \begin{eqnarray}
 \label{SPLIT_LAMBDA_DIAG_2n}
&&  |
  \Delta_{{\rm \textbf{off-diag}}}
  ^{\bxi,\bxi'}(t,T,\x,\x')|\big|_{(\bxi,\bxi') = (\x,\x')}
  \nonumber \\
 &\leq & \Big |\int_t^{t_0} \!\! ds \int_{\R^{nd}}\!\!\! d\y
 \tilde p^{\bxi} (t,s,\x,\y)  (L_s-\tilde L_s^{\bxi}) u(s,\y) \Big | \bigg |_{\bxi=\x}
+ \Big | \int_t^{t_0} \!\! ds \int_{\R^{nd}}\!\!\! d\y
 \tilde p^{\bxi '} (t,s,\x',\y)  (L_s-\tilde L_s^{\bxi'}) u(s,\y)  \Big | \bigg |_{\bxi'=\x'}
\nonumber \\
&\leq& \Lambda \sum_{k=2}^n \int_t^{t_0} ds(s-t)^{\frac{\gamma }2} \|u\|_{L^\infty(C_{b,\d}^{2+\gamma})}
\nonumber \\
&=& \Lambda c_0\d^{2+\gamma}(\x,\x') \|u\|_{L^\infty(C_{b,\d}^{2+\gamma})}.
\nonumber \\
 \end{eqnarray}
The last but one inequality is a consequence of Lemmas \ref{CTR_DER_SUP_non_deg} and \ref{LEMME_BESOV_DEG}. Equations \eqref{R_DIAG_LAMBDA_2n_BIS} and \eqref{SPLIT_LAMBDA_DIAG_2n} yield \eqref{esti_holder_perturpart}. This concludes the proof of Lemma \ref{HOLDER_CTR_PERT_PART}.

\subsection{Controls of the discontinuity terms arising from the change of freezing point}
\label{SEC_DISC}
It now remains to control the contribution arising from the change of freezing point in equation \eqref{EXP_MIXED_DER_NON_DEG}.
The main result of this section is the following lemma.
\begin{lem}[Control of the discontinuity terms]
\label{CTR_TERME_DISC} 
There exists $C:=C(\A{A})$ s.t. for any $(t,\x,\x') \in [0,T] \times \R^{nd}\times \R^{nd}$ taking $\bxi'=\x'$, $\tilde \bxi'=\x $,
\begin{eqnarray*}
 \big|D_{\x_1}^2 \tilde P_{t_0,t}^{\bxi'} u(t_0, \x')- D_{\x_1}^2 \tilde P_{t_0,t}^{\tilde \bxi'} u(t_0, \x') \big | &\leq & C c_0^{\frac \gamma{(2n-1)}} 
 \|u\|_{L^\infty(C^{2+\gamma}_{b,\d})}  \d^\gamma(\x,\x'),
 \nonumber \\
   \big |\tilde P_{t_0,t}^{\bxi'} u(t_0, \x')- \tilde P_{t_0,t}^{\tilde \bxi'} u(t_0, \x') \big | &\leq & C c_0  \|u\|_{L^\infty(C^{2+\gamma}_{b,\d})}  \d^{2+\gamma}(\x,\x'), \text{ for } \x_1=\x_1'.
\end{eqnarray*}
\end{lem}

We prove the above statement in the next paragraphs respectively dedicated to the control of the derivatives w.r.t. the non-degenerate variables (first estimate) and the control of the H\"older moduli associated with the degenerate ones (second estimate).

\subsubsection{Control of the derivatives w.r.t. the non-degenerate variables}
As done in \eqref{D2_tildeP_g_Holder} and  \eqref{D2_tildeP_g2_Holder}, we can write:
\begin{eqnarray}
\label{control_hoped_mathbf_R_ineq_trig}
&&D_{\x_1}^2 \tilde P_{t_0,t}^{\bxi'} u(t_0, \x')- D_{\x_1}^2 \tilde P_{t_0,t}^{\tilde \bxi'} u(t_0, \x')  \nonumber \\
&=&
\Bigg [ \int_{\R^{nd}} D_{\x_1}^2 \tilde p^{\bxi'}(t,t_0,\x',\y) [u(t_0,\y)-u(t_0,\y_1,\m_{t_0,t}^{\bxi'}(\x')_{2:n})] d\y
\nonumber \\
&&\quad -\int_{\R^{nd}} D_{\x_1}^2 \tilde p^{\tilde \bxi'}(t,t_0,\x',\y) [u(t_0,\y)-u(t_0,\y_1,\m_{t_0,t}^{\tilde \bxi'}(\x')_{2:n})] d\y \Bigg ]
\nonumber \\
&&+
 \Bigg [ \int_{\R^{nd}} D_{\x_1}^2 \tilde p^{\bxi'}(t,t_0,\x',\y)  \int_0^1 d\mu  (1-\mu )
\nonumber \\
&&\quad{\rm{Tr}} \Big (\big [D_{\x_1}^2 u \big (t_0,\m_{t_0,t}^{\bxi'}(\x')_1
+\mu (\y-\m_{t_0,t}^{\bxi'}(\x'))_1,\m_{t_0,t}^{\bxi'}(\x')_{2:n} \big )-D_{\x_1}^2 u \big (t_0,\m_{t_0,t}^{\bxi'}(\x')\big) \big] \big (\y-\m_{t_0,t}^{\bxi'}(\x') \big)_1^{\otimes 2}\Big) 
d\y \nonumber \\
&&\quad - \int_{\R^{nd}} D_{\x_1}^2 \tilde p^{\tilde \bxi'}(t,t_0,\x',\y)  \int_0^1 d\mu  (1-\mu )
\nonumber \\
&&\quad {\rm{Tr}} \Big ( \big[D_{\x_1}^2 u \big (t_0,\m_{t_0,t}^{\tilde \bxi'}(\x')_1
+\mu (\y-\m_{t_0,t}^{\tilde \bxi'}(\x'))_1,\m_{t_0,t}^{\tilde \bxi'}(\x')_{2:n}\big )-D_{\x_1'}^2 u \big(t_0,\m_{t_0,t}^{\bxi'}(\x')\big)\big]\big(\y-\m_{t_0,t}^{\tilde \bxi'}(\x')\big)_1^{\otimes 2}  \Big) d\y \Bigg ]\nonumber \\
\nonumber \\
&&+ \frac 12 \Bigg[  \int_{\R^{nd}}  D_{\x_1}^2 \tilde p^{\bxi'}(t,t_0,\x',\y) {\rm{Tr}} \Big ( D_{\x_1}^2 u(t_0,\m_{t_0,t}^{\bxi'}(\x')) (\y-\m_{t_0,t}^{\bxi'}(\x'))_1^{\otimes 2} \Big ) d \y
\nonumber \\
&&\quad -    \int_{\R^{nd}}  D_{\x_1}^2 \tilde p^{\tilde \bxi'}(t,t_0,\x',\y) {\rm{Tr}} \Big ( D_{\x_1}^2 u(t_0,\m_{t_0,t}^{\tilde \bxi'}(\x')) (\y-\m_{t_0,t}^{\tilde \bxi'}(\x'))_1^{\otimes 2} \Big )
d\y\Bigg]
\nonumber \\
&=:& \Delta_{t,t_0,\bxi',\tilde \bxi'} D_{\x_1}^2\tilde P_1u(t_0,\x,\x') +\Delta_{t,t_0,\bxi',\tilde \bxi'} D_{\x_1}^2\tilde P_{21}u(t_0,\x,\x')+\Delta_{t,t_0,\bxi',\tilde \bxi'} D_{\x_1}^2\tilde P_{22}u(t_0,\x,\x').
\nonumber \\
\end{eqnarray}
We first directly write from \eqref{D2_tildeP_g1_Holder} and \eqref{D2_tildeP_g4_Holder}:
\begin{eqnarray}\label{D2_tildeP_u1_Holder}
\Big|\Delta_{t,t_0,\bxi',\tilde \bxi'} D_{\x_1}^2\tilde P_1u(t_0,\x,\x') +\Delta_{t,t_0,\bxi',\tilde \bxi'} D_{\x_1}^2\tilde P_{21}u(t_0,\x,\x')\Big| \bigg|_{(\bxi',\tilde \bxi')=(\x',\x)}
&\leq&  2C(t_0-t)^{\frac \gamma 2}\|u\|_{C_{b,\d}^{2+\gamma}}  
\nonumber \\
&\leq& Cc_0^{\frac \gamma 2}\|u\|_{C_{b,\d}^{2+\gamma}} \d^\gamma(\x,\x').
\end{eqnarray}
Let us now deal with the last term in \eqref{control_hoped_mathbf_R_ineq_trig}. We proceed similarly to equation \eqref{THE_DECOUP_QUI_BIEN_MOD_HOLDER_DX12} (control of the frozen semigroup). Write:
\begin{eqnarray}
&&\Big |\Delta_{t,t_0,\bxi',\tilde \bxi'} D_{\x_1}^{2}\tilde P_{22}u(t_0,\x,\x') \Big|_{(\bxi',\tilde \bxi')=(\x',\x)}\nonumber\\
&\le& \bigg\{\frac 12   \int_{\R^{nd}} \frac{d\y}{(s-t)}  \bar p_{C^{-1}}^{\bxi'}(t,t_0,\x',\y) | D_{\x_1}^2 u(t_0,\m_{t_0,t}^{\bxi'}(\x'))-D_{\x_1}^2 u(t_0,\m_{t_0,t}^{\tilde \bxi'}(\x')) | |(\y-\m_{t_0,t}^{\bxi'}(\x'))_1|^2
\nonumber \\
&&+ \frac 12   \Big|\int_{\R^{nd}} D_{\x_1}^2 \tilde p^{\bxi'}(t,t_0,\x',\y) {\rm{Tr}} \Big ( D_{\x_1}^2 u(t_0,\m_{t_0,t}^{\tilde \bxi'}(\x')) (\y-\m_{t_0,t}^{\bxi'}(\x'))_1^{\otimes 2} \Big )
\nonumber\\
 &&-D_{\x_1}^2 \tilde p^{\tilde \bxi'}(t,t_0,\x',\y) {\rm{Tr}} \Big ( D_{\x_1}^2 u(t_0,\m_{t_0,t}^{\tilde \bxi'}(\x')) (\y-\m_{t_0,t}^{\tilde \bxi'}(\x'))_1^{\otimes 2} 
 \Big )d\y \Big| \bigg\}\bigg|_{(\bxi',\tilde \bxi')=(\x',\x)}
\nonumber \\
&\le& C | D_{\x_1}^2 u(t_0,\m_{t_0,t}^{\bxi'}(\x'))-D_{\x_1}^2 u(t_0,\m_{t_0,t}^{\tilde \bxi'}(\x')) |\Big|_{(\bxi',\tilde \bxi')=(\x',\x)},
\notag
\end{eqnarray}
\textcolor{black}{
exploiting Proposition \ref{THE_PROP} for the first contribution and  identity \eqref{GROS_CENTER_TER} in Proposition \ref{Prop_moment_D2_tilde_p} for the second contribution in the last inequality.}

From Lemma \ref{Lemme_d_theta_theta_x_x}, we derive that for $t_0-t = c_0 \d^2(\x,\x')$:
\begin{equation}\label{D2_tildeP_g2_Holder_diff_R12_DISC}
\Big |\Delta_{t,t_0,\bxi',\tilde \bxi'} D_{\x_1}^{2}\tilde P_{22}u(t_0,\x,\x') \Big|_{(\bxi',\tilde \bxi')=(\x',\x)} 
 \leq  C  \|u(t_0,\cdot)\|_{C^{2+\gamma}_{b,\d}}c_0^{\frac \gamma{2n-1}} \d^\gamma(\x,\x').
\end{equation}

Eventually, from \eqref{D2_tildeP_u1_Holder} and \eqref{D2_tildeP_g2_Holder_diff_R12_DISC}, we get the following control:
\textcolor{black}{
\begin{eqnarray*}
\label{control_hoped_mathbf_R_ineq_trig1}
\big |D_{\x_1}^2 \tilde P_{t_0,t}^{\bxi'} u(t_0, \x')- D_{\x_1}^2 \tilde P_{t_0,t}^{\tilde \bxi'} u(t_0, \x')\big |  \Big |_{(\bxi',\tilde \bxi ')=(\x,\x')} \leq C c_0^{\frac \gamma{2n-1}} \|u\|_{L^\infty(C_{b,\d}^{2+\gamma})}  \d^\gamma(\x,\x'), 
\end{eqnarray*}
which gives the first statement of Lemma \ref{CTR_TERME_DISC} for the second order derivatives w.r.t. the non-degenerate variables. 
}
\subsubsection{H\"older controls for the degenerate variables}
Again, for the H\"older norm w.r.t. the degenerate variables, the difficulty is that we cannot take any advantage of cancellation tools. We adapt here the arguments employed in Section \ref{DEAL_SG_DEG} for the frozen semigroup.
Precisely, for any $(\x,\x') \in (\R^{nd})^2$, $\x_1=\x_1'$, $\bxi'=\x',\ \tilde \bxi'=\x $, we have similarly to \eqref{control_hoped_mathbf_R_ineq_trig} (but without the spatial derivatives $D_{\x_1}^2$):
\begin{eqnarray}
\label{control_hoped_mathbf_R_ineq_trigu1}
 \tilde P_{t_0,t}^{\bxi'} u(t_0, \x')-  \tilde P_{t_0,t}^{\tilde \bxi'} u(t_0, \x') 
&=& \Delta_{t,t_0,\bxi', \tilde \bxi'} \tilde P_1u(t_0,\x',\x') +\Delta_{t,t_0,\bxi',\tilde\bxi'} \tilde P_{21}u(t_0,\x',\x')
\nonumber \\
&&+   \frac 12  {\rm Tr}\Big ([D_{\x_1}^2u(t_0,\m_{t_0,t}^{\bxi'}(\x'))- D_{\x_1}^2u(t_0,\m_{t_0,t}^{\tilde \bxi'}(\x'))] [\tilde \K^{\bxi'}_{t_0,t}]_{1,1}\Big )
\nonumber \\
&&+  \frac 12  {\rm Tr}\Big (D_{\x_1}^2u(t_0,\m_{t_0,t}^{\tilde \bxi'}(\x'))\big([ \tilde \K^{\bxi'}_{t_0,t}]_{1,1}-[\tilde \K^{\tilde \bxi'}_{t_0,t}]_{1,1}\big) \Big ),
\end{eqnarray}
where accordingly with \eqref{control_hoped_mathbf_R_ineq_trig}:
\begin{eqnarray*}
\Delta_{t,t_0,\bxi',\tilde \bxi'} \tilde P_1u(t_0,\x',\x')&:=&
 \int_{\R^{nd}}  \tilde p^{\bxi'}(t,t_0,\x',\y) [u(t_0,\y)-u(t_0,\y_1,\m_{t_0,t}^{\bxi'}(\x')_{2:n})] d\y
\nonumber \\
&&-\int_{\R^{nd}}  \tilde p^{\tilde \bxi'}(t,t_0,\x',\y) [u(t_0,\y)-u(t_0,\y_1,\m_{t_0,t}^{\tilde \bxi'}(\x')_{2:n})] d\y ,
\nonumber \\
\end{eqnarray*}
and
\begin{eqnarray*}
&&\Delta_{t,t_0,\bxi',\tilde \bxi'} \tilde P_{21}u(t_0,\x',\x') 
\nonumber \\
&:=& \int_{\R^{nd}}  \tilde p^{\bxi '}(t,t_0,\x',\y)  \int_0^1 d\mu  (1-\mu )
\nonumber \\
&&{\rm{Tr}}\Big ( \big [D_{\x_1}^2 u \big (t_0,\m_{t_0,t}^{\bxi'}(\x')_1
+\mu (\y-\m_{t_0,t}^{\bxi'}(\x'))_1,\m_{t_0,t}^{\bxi'}(\x')_{2:n} \big)-D_{\x_1}^2 u\big (t_0,\m_{t_0,t}^{\bxi'}(\x')\big)\big]\big(\y-\m_{t_0,t}^{\bxi'}(\x')\big)_1^{\otimes 2} \Big )
 d\y \nonumber \\
&&- \int_{\R^{nd}}  \tilde p^{\tilde \bxi'}(t,t_0,\x',\y)  \int_0^1 d\mu  (1-\mu )
\nonumber \\
&& {\rm{Tr}} \Big ( \big [D_{\x_1}^2 u\big(t_0,\m_{t_0,t}^{\tilde \bxi'}(\x')_1
+\mu (\y-\m_{t_0,t}^{\tilde \bxi'}(\x'))_1,\m_{t_0,t}^{\tilde \bxi'}(\x')_{2:n} \big)-D_{\x_1}^2 u\big(t_0,\m_{t_0,t}^{\tilde \bxi'}(\x')\big)\big] \big(\y-\m_{t_0,t}^{\tilde \bxi'}(\x')\big)_1^{\otimes2 } \Big) d\y .
\end{eqnarray*}
Reproducing the arguments that led to equations \eqref{tildeP_g1_Holder} and \eqref{tildeP_g4_Holder}, we derive:
\begin{eqnarray}\label{tildeP_u1_Holder}
&& \big | \Delta_{t,t_0,\bxi',\tilde \bxi'} \tilde P_1u(t_0,\x',\x') +\Delta_{t,t_0,\bxi',\tilde \bxi'} \tilde P_{21}u(t_0,\x',\x') \big | \Big |_{(\bxi',\tilde \bxi ')=(\x',\x)} 
\notag\\
&\leq& C(t_0-t)^{\frac{2+\gamma}2} \|u\|_{C_{b,\d}^{2+\gamma}}
\nonumber \\
&\leq& Cc_0^{\frac{2+\gamma}2} \|u\|_{L^\infty(C_{b,\d}^{2+\gamma})}  \d^{2+\gamma}(\x,\x').
\end{eqnarray}
Let us now turn to the last two contributions in \eqref{control_hoped_mathbf_R_ineq_trigu1}.
As done in \eqref{THE_CTR_HD_1_DEG}, from Proposition \ref{PROP_SCALE_COV} and Lemma \ref{Lemme_d_theta_theta_x_x} we obtain:
\begin{eqnarray}\label{tildeP_u2_Holder2}
 && \frac 12 \big | [D_{\x_1}^2u(t_0,\m_{t_0,t}^{\bxi'}(\x'))-D_{\x_1}^2u(t_0,\m_{t_0,t}^{\tilde \bxi'}(\x'))] [\tilde \K^{\bxi'}_{t_0,t}]_{1,1} \big | \Big |_{(\bxi',\tilde \bxi ')=(\x',\x)} 
\nonumber \\
  &\leq& C \|u\|_{L^\infty(C_{b,\d}^{2+\gamma})}(t_0-t) \d^{\gamma}(\m_{t_0,t}^{\x}(\x'),\btheta_{t_0,t}(\x'))
  \nonumber \\
&\leq& C c_0^{1+\frac \gamma{2n-1}}\|u\|_{L^\infty(C_{b,\d}^{2+\gamma})} \d^{2+\gamma}(\x,\x').
\end{eqnarray}
The last term of \eqref{control_hoped_mathbf_R_ineq_trigu1} is handled like in \eqref{tildeP_g2_Holder}.
Namely, by Lemma \ref{SENS_COV} (equation \eqref{CTR_SENSI_COV}), we obtain:
\begin{eqnarray}\label{tildeP_u2_Holder_Bis}
&&\big |\frac 12  D_{\x_1}^2u(t_0,\m_{t_0,t}^{\tilde \bxi'}(\x'))\big( [\tilde \K^{\bxi'}_{t_0,t}]_{1,1}-[\tilde \K^{\tilde \bxi'}_{t_0,t}]_{1,1}\big) \big| \Big|_{(\bxi',\tilde \bxi')=(\x',\x)}
\nonumber \\
&\leq&  C \|u(t_0,\cdot)\|_{C_{b,\d}^{2+\gamma}} \big( (t_0-t)^{\frac {2+\gamma} 2}+(t_0-t)\d^{\gamma}(\x,\x') \big )
\nonumber \\
&\leq & C c_0 \|u\|_{L^\infty(C_{b,\d}^{2+\gamma})} \d^{2+\gamma}(\x,\x') 
.
\end{eqnarray}
Plugging  \eqref{tildeP_u1_Holder}, \eqref{tildeP_u2_Holder2} and \eqref{tildeP_u2_Holder_Bis} into \eqref{control_hoped_mathbf_R_ineq_trigu1}, recalling as well that $c_0<1$, we derive the second statement  of Lemma \ref{CTR_TERME_DISC}.

\begin{REM}[Concluding remark concerning the \textit{a priori} estimates]\label{REM_PAS_SA_PLACE}From the results of Sections \ref{Sectio_Sup_D2} and \ref{HOLDER}, i.e. Proposition \ref{PROP_SUP_CTRL} for the supremum norms and Proposition \ref{PROP_HOLDER_CTRL} for the H\"older norms,
 we actually derive the following bound. There exist constants $C:=C(\A{A})$ and $\Lambda:=\Lambda(\A{A},T) $ as in Remark \ref{REM_LAMBDA} s.t. 
\begin{equation}\label{control_voulu}
\|u\|_{L^\infty(C^{2+\gamma}_{b,\d})} \leq C(\|g\|_{C^{2+\gamma}_{b,\d}}+ \|f\|_{L^\infty(C^\gamma_{b,\d})})+ \Big(\Lambda \big( (c_0^{-(n-\frac 12)+\frac \gamma 2}+c_0^{\frac \gamma2})+T^{\frac \gamma 2}\big) +Cc_0^{\frac{\gamma}{2n-1}}\Big) \|u\|_{L^\infty(C^{2+\gamma}_{b,\d})}.
\end{equation}
The above estimate then readily yields \eqref{SCHAU_UG} provided $c_0$ and $\Lambda $ are small enough (also with $\Lambda\ll c_0 $). 
Recalling from Remark \ref{REM_LAMBDA} that $\Lambda $ is small provided the coefficients do not vary much,
the remaining delicate part consists in getting rid of the small H\"older moduli constraint.  This can be done through suitable scaling arguments that are exposed in Section \ref{scaling}.


\end{REM}

\mysection{Scaling issues and final proof of Theorem \ref{THEO_SCHAU}}\label{scaling}

The purpose of this section is to first introduce a suitable scaling procedure for the system with mollified coefficients for which we will be able to show equation \eqref{control_voulu}. This intuitively means that the scaling has to make the H\"older moduli of the considered coefficients small. The expected control is then obtained going back to the initial variables through the \textit{inverse} scaling procedure. Also, once the estimate is established for \textit{small} final time horizon $T$, it can be deduced through iteration up to an arbitrary given time precisely because it provides a kind of stability for the solution in the space $L^\infty([0,T],C_{b,\d}^{2+\gamma} (\R^{nd},\R))$.  We then conclude the proof of our main result, Theorem \ref{THEO_SCHAU} through compactness arguments.

\subsection{Scaling settings and controls}
\label{section_scaling_properties}

We start here from the smooth solution $u$ to equation \eqref{KOLMO} with mollified coefficients (that we again denote by a slight abuse of notation without the mollifying parameter $m$). For an additional parameter $\lambda>0$ to be specified later (but meant to be \textit{small}), introducing the scaled function $u^\lambda(t,\x)=u(t,\lambda^{-1/2}\T_\lambda \x) $, it is then clear that this latter satisfies
\begin{eqnarray}
\label{KOLMO_LAMBDA}
\begin{cases}
\partial_t u^\lambda(t,\x)+ \langle \gF(t,\lambda^{-1/2}\T_\lambda \x),  \lambda^{1/2}\T_\lambda^{-1} \mathbf D u^\lambda(t,\x)\rangle\\
 +\frac{\lambda^{-1 }}2{\rm Tr}\big( D_{\x_1}^2u^\lambda (t,\x)a(t,\lambda^{-1/2}\T_\lambda \x)\big)=-f(t,\lambda^{-1/2}\T_\lambda \x),\ (t,\x)\in [0,T)\times \R^{nd},\\
 u^\lambda(T,\x)=g(\lambda^{-1/2}\T_\lambda \x),\ \x\in \R^{nd}.
 \end{cases}
\end{eqnarray}
This choice of rescaling is natural in view of the invariance by dilatation property \eqref{delta_lambda}, i.e. each single variable $\x_i$ is scaled by the parameter $\lambda$ according to its corresponding intrinsic scale.

We rewrite in short form 
 the above equation as for any $\x \in \R^{nd}$
\begin{eqnarray}\label{KOLMO_LAMBDA_SHORT}
\begin{cases}
\partial_tu^\lambda(t,\x)+ \langle \gF_\lambda(t,\x),  \mathbf D u^\lambda(t,\x)\rangle
 +\frac{1}2{\rm Tr}\big( D_{\x_1}^2 u^\lambda(t,\x)a_\lambda(t, \x)\big)=-f_\lambda(t, \x),\ t\in [0,T),\\ 
 u^\lambda(T,\x)=g_\lambda(\x),
 \end{cases}
\end{eqnarray}
where 
\begin{eqnarray}
 f_\lambda(t, \x)&:=&f(t,\lambda^{-1/2}\T_\lambda \x),\ g_\lambda(\x):=g(\lambda^{-1/2}\T_\lambda \x),\notag\\
a_\lambda(t,\x)&:=&\lambda^{-1} a(t,\lambda^{-1/2}\T_\lambda \x),\notag \\ 
\gF_\lambda(t, \x)&:=&\lambda^{1/2}\T_\lambda^{-1}\gF(t,\lambda^{-1/2}\T_\lambda \x).\label{CORRESP_COEFF_LAMBDA}
\end{eqnarray}
Accordingly, we introduce the spatial operator $(L_s^\lambda)_{s\ge 0} $ appearing in \eqref{KOLMO_LAMBDA_SHORT} which writes explicitly for any $\varphi \in C^2_0(\R^{nd},\R)$ as:
$$L_t^\lambda \varphi= \langle \gF_\lambda(t,\cdot) , \mathbf D\varphi \rangle  +\frac 12 {\rm Tr}\Big(a_\lambda(t,\cdot) D_{\x_1}^2 \varphi\Big),\ \lambda>0.$$

The dynamics of the SDE associated with the second order differential operator $(L_t^\lambda)_{t\in [0,T]} $ appearing in \eqref{KOLMO_LAMBDA}-\eqref{KOLMO_LAMBDA_SHORT} writes for a given starting point $(t,\x)\in [0,T]\times \R^{nd} $:
\begin{equation}
\label{RESCALED_SDE}
\X_{s}^{\lambda,t,\x}=\x+\int_t^s \gF_\lambda(u,\X_{u}^{\lambda,t,\x}) du+\int_t^s B\sigma_\lambda(u,\X_{u}^{\lambda,t,\x}) dW_u,\ s\ge t,
\end{equation}
where $(W_u)_{u\ge 0} $ is a $d$-dimensional Brownian motion on some filtered probability space $(\Omega,\F,(\F_t)_{t\ge 0},\P) $ and 
$\sigma_\lambda $ is a square root of the diffusion matrix $a_\lambda$ introduced in \eqref{CORRESP_COEFF_LAMBDA}.
\\

Equation \eqref{RESCALED_SDE}  then naturally leads to consider, for fixed $(s,\y)\in [t,T]\times \R^{nd} $, and with the notations of Section \ref{SEC2}, the corresponding linearized model
\begin{eqnarray}\label{RESCALED__LINEARIZEDSDE}
d\tilde \X_v^{ \bxi,\lambda}&=&[\gF_\lambda(v,\btheta_{v,t}^\lambda(\bxi^\lambda))+ D\gF_\lambda(v,\btheta^\lambda_{v,t}(\bxi^\lambda))(\tilde \X_v^{ \bxi,\lambda}-\btheta_{v,t}^\lambda(\bxi^\lambda))]dv +B\sigma_\lambda(v,\btheta^\lambda_{v,t}(\bxi^\lambda )) dW_v, \ \forall v\in  [t,s] \nonumber\\
 \tilde \X_t^{ \bxi,\lambda}&=&\x, 
 \end{eqnarray}  
 where $ \btheta_{v,t}^\lambda(\bxi^\lambda)=\lambda^{\frac 12}\T_{\lambda }^{-1}\btheta_{v,t}(\bxi^\lambda),\ \bxi^\lambda=\lambda^{-\frac 12}\T_\lambda \bxi$. The associated generator writes for $\varphi \in C_0^2(\R^{nd},\R) $ and $(s,\y)\in [t,T]\times \R^{nd} $ as:
\begin{equation*}
\tilde  L_s^{\lambda, \bxi}\varphi(\y)=\langle \gF_\lambda(s,\btheta_{s,t}^\lambda(\bxi^\lambda))+D\gF_\lambda(s,\btheta_{s,t}^\lambda(\bxi^\lambda))(\y-\btheta_{s,t}^\lambda(\bxi^\lambda)), \mathbf D \varphi(\y)\rangle +\frac 12{\rm Tr}\big(a_\lambda(t,\btheta_{s,t}^\lambda(\bxi^\lambda))D_{\y_1}^2 \varphi(\y) \big).
\end{equation*}
Observe from \eqref{FROZ} and \eqref{RESCALED__LINEARIZEDSDE} that the following very important correspondence holds: 
\begin{equation}
\label{CORRESP_SCALE}
\forall v\in [t,T],\ \tilde \X_{v}^{\bxi,\lambda}:=\lambda^{1/2}\T_\lambda^{-1} \tilde \X_v^{\bxi^\lambda}.
\end{equation}
We thus derive from Proposition \ref{THE_PROP} the following important correspondence for the densities. Denoting by $\tilde p_\lambda^\bxi(t,s,\x,\y) $ the density of $\tilde \X_{v}^{\bxi,\lambda}$ starting from $\x$ at time $t$ at point $\y$ in $s$, and $\x^\lambda =\lambda^{-\frac 12}\T_\lambda \x $, we have
\begin{eqnarray}
\label{CORRESP_DENS_SCALING_IMP}
\tilde p_\lambda^\bxi(t,s,\x,\y)&=&\lambda^{\frac{n^2d}2}\tilde p^{\bxi^\lambda}(t,s, \lambda^{-\frac 12}\T_\lambda \x,\lambda^{-\frac 12}\T_\lambda \y )\notag\\
&=&\frac{\lambda^{\frac{n^2d}2}}{(2\pi)^{\frac{nd}2}\det(\tilde \K_{s,t}^{\bxi^\lambda})^{\frac 12}}\exp\left( -\frac {\lambda^{-1}}2 \left\langle  \big[\T_\lambda ( \tilde \K_{s,t}^{\bxi^\lambda})^{-1} \T_\lambda\big](\lambda^{\frac 12}\T_{\lambda}^{-1}\m_{s,t}^{\bxi^\lambda}(\x^\lambda)-\y),\lambda^{\frac 12}\T_\lambda^{-1} \m_{s,t}^{\bxi^\lambda}(\x^\lambda)-\y\right\rangle\right).\notag\\
\end{eqnarray}
In particular, for $\bxi=\x $ one derives:
\begin{eqnarray}
\label{CORRESP_DENS_SCALING_IMP_BIS}
\tilde p_\lambda^\bxi(t,s,\x,\y)&=&\lambda^{\frac{n^2d}2}\tilde p^{\bxi^\lambda}(t,s, \lambda^{-\frac 12}\T_\lambda \x,\lambda^{-\frac 12}\T_\lambda \y )\notag\\
&=&\frac{\lambda^{\frac{n^2d}2}}{(2\pi)^{\frac{nd}2}\det(\tilde \K_{s,t}^{\bxi^\lambda})^{\frac 12}}\exp\left( -\frac {\lambda^{-1}}2 \left\langle  \big[\T_\lambda ( \tilde \K_{s,t}^{\bxi^\lambda})^{-1} \T_\lambda\big](\btheta_{s,t}^{\lambda}(\x^\lambda)-\y), \btheta_{s,t}^\lambda(\x^\lambda)-\y\right\rangle\right).\notag\\
\end{eqnarray}
Equation \eqref{CORRESP_DENS_SCALING_IMP} in turn yields the following important control:
\begin{eqnarray}
|D_{\x}^\vartheta\tilde p_\lambda^{\bxi}(t,s,\x,\y)|&\le& C \left(\frac{ \lambda}{(s-t)}\right)^{\sum_{i=1}^n \vartheta_i(i-\frac{1}{2}) +\frac{n^2d}2}\exp\left(-C^{-1}\frac{(s-t)}{\lambda} \big|\T_{\frac{s-t}{\lambda}}^{-1}\big(\lambda^{\frac 12}\T_\lambda^{-1}\m_{s,t}^{\bxi^\lambda}(\x^\lambda)-\y\big)\big|^2\right)\notag\\
&=:& C\left(\frac{ \lambda}{(s-t)}\right)^{\sum_{i=1}^n \vartheta_i(i-\frac{1}{2})}
\bar p_{C^{-1},\lambda}^{\bxi}(t,s,\x,\y).\label{THE_GOOD_SCALE_DER_CTR}
\end{eqnarray} 
In the following, we will also denote $\bar p_{C^{-1},\lambda}(t,s,\x,\y):=\bar p_{C^{-1},\lambda}^\bxi(t,s,\x,\y)\Big|_{\bxi=\x}$ in \eqref{THE_GOOD_SCALE_DER_CTR}.

\begin{REM}
We emphasize that \eqref{THE_GOOD_SCALE_DER_CTR} gives that, each derivation of the Gaussian kernel $\tilde p_\lambda^\bxi $ makes the small parameter $\lambda $ appear. Hence, up to the additional time singularities, the iterated derivatives become smaller and smaller. 
\end{REM}

\subsection{Scaling properties}
The point is now to reproduce the previous perturbative approach for the solution of \eqref{KOLMO_LAMBDA_SHORT} in order to obtain ad hoc versions of Propositions \ref{PROP_SUP_CTRL} (Section \ref{Sectio_Sup_D2}) and \ref{PROP_HOLDER_CTRL} (Section \ref{HOLDER}). Such versions then allow us to derive the analogous control of \eqref{control_voulu} in this rescaled setting involving precisely a positive exponent of $\lambda$ in front of the term $\Lambda c_0^{-(n-\frac 12) + \frac\gamma{2}}$. As underlined in Remark \ref{REM_PAS_SA_PLACE}, this type of control will give the expected final bound provided the scaling parameter $\lambda$ is small enough.\\

We will mainly focus on the H\"older norm of the remainder term (i.e. rescaled version of estimate \eqref{esti_holder_deriv_sec_perturpart} in Lemma \ref{HOLDER_CTR_PERT_PART}) associated with the second order derivatives w.r.t. the non-degenerate variables. It can indeed be seen from the previous computations that the other contributions can be dealt similarly.\\


\textbf{H\"older estimate on the rescaled remainder.}   According to the notations introduced in the beginning of Section \ref{HOLDER}, we denote for $\lambda>0 $, $(t,\x,\x',\x')\in [0,T]\times(\R^{nd})^2 $ the quantity:
\begin{eqnarray}
D^2_{\x_1}(\Delta^{\lambda})^{\bxi,\bxi'}(t,\x,\x')
&:=&
\Bigg( \int_t^{(t+\textcolor{black}{c_0}\lambda \d^2(\x,\x')) \wedge T} ds \int_{\R^{nd}}\Big(D^2_{\x_1} \tilde p_{\lambda}^{\bxi}(t,s,\x,\y) (L_s^\lambda-\tilde L_s^{\lambda,\bxi}) u^\lambda(s,\y)
\notag\\ &&
-D^2_{\x_1} \tilde p_{\lambda}^{\bxi'}(t,s,\x',\y) (L_s^\lambda-\tilde L_s^{\lambda,\bxi'}) u^\lambda(s,\y) \Big) d\y\Bigg)\Bigg|_{(\bxi,\bxi')=(\x,\x')}\notag\\
 &&+\Bigg( \int_{(t+\textcolor{black}{c_0} \lambda \d^2(\x,\x')) \wedge T}^T ds \int_{\R^{nd}}\Big(D^2_{\x_1} \tilde p_{\lambda}^{\bxi}(t,s,\x,\y) (L_s^\lambda-\tilde L_s^{\lambda,\bxi}) u^\lambda(s,\y)\notag\\
 &&-D^2_{\x_1} \tilde p_{\lambda}^{\bxi'}(t,s,\x',\y) (L_s^\lambda-\tilde L_s^{\lambda,\bxi'}) u^\lambda(s,\y) \Big)
 d\y\Bigg)\Bigg|_{(\bxi,\bxi')=(\x,\x)}\notag\\
 &=:& D^2_{\x_1}(\Delta^{\lambda})_{{\rm \textbf{off-diag}}}^{\bxi,\bxi'}(t,\x,\x')+D^2_{\x_1}(\Delta^{\lambda})_{{\rm \textbf{diag}}}^{\bxi,\bxi'}(t,\x,\x'),\label{RESCALED_QUANTITY_TO_CONTROL}
\end{eqnarray}
using again as in \eqref{EXP_MIXED_DER_NON_DEG} different freezing point\textcolor{black}{s} according to the spatial regime w.r.t. integration time $s$.
Note however carefully that the cutting threshold here depends on the scaling parameter $\lambda $. This is very important in order to balance the various scales that will appear. \textcolor{black}{Pay attention as well that the parameter $c_0$ also remains. Actually, a subtle balance between those two parameters will be needed to derive the expected control.
}  
Proceeding as in \textcolor{black}{Section \ref{SUB_HOLDER_CTR_PERT_PART}}, we aim at showing that there exist constants $C:=C(\A{A},T)$ and  $\Lambda:=\Lambda(\A{A},T)$ as in Remark \ref{REM_LAMBDA} s.t. for any $ (\x,\x')\in (\R^{nd})^2$:
\begin{equation}
\label{THE_CTR_RESCALED}
\sup_{t\in [0,T]}\frac{|D^2_{\x_1}(\Delta^{\lambda})^{\bxi,\bxi'}(t,\x,\x')|}{\d^\gamma(\x,\x')}\leq\Big( \Lambda \textcolor{black}{(} \textcolor{black}{c_0^{-(n-\frac 12)+ \frac \gamma 2}}+c_0^{\frac \gamma2} )\lambda^{\frac{\gamma}{2}}\Big) \|u^\lambda\|_{L^\infty(C^{2+\gamma}_{b,\d})}.
\end{equation}

\emph{Proof.} Let us first consider the diagonal term in \eqref{RESCALED_QUANTITY_TO_CONTROL} assuming w.l.o.g. that $t+\textcolor{black}{c_0}\lambda \d^2(\x,\x')\le T $ (otherwise we only have the off-diagonal contribution) \textcolor{black}{and recalling that in the \emph{diagonal regime} we chose $\bxi'=\bxi=\x$}. Write:
\begin{eqnarray}
\label{R_OFF_DIAG_LAMBDA_1}
&&|D^2_{\x_1}(\Delta^{\lambda})_{{\rm \textbf{diag}}}^{\bxi,\bxi'}(t,\x,\x')|\Big|_{\bxi'=\bxi=\x}
\notag \\
&\leq& \Big |\int_{t+\textcolor{black}{c_0}\lambda \d^2(\x,\x')}^T ds \int_{\R^{nd}} d\y \Big [ D^2_{\x_1}  \tilde p_{\lambda}^{\bxi} (t,s,\x,\y)- D^2_{\x_1}  \tilde p_{\lambda}^{\bxi} (t,s,\x',\y) \Big ] \Big [ \langle \gF_{\lambda,1}(t,\y)-\gF_{\lambda,1}(t,\btheta_{s,t}^\lambda(\bxi^\lambda)), D_{\y_1} \rangle\notag\\
 &&\quad+ \frac{1}2{\rm Tr} \Big( \big(a_\lambda(t,\y)-a_\lambda(t,\btheta_{s,t}^\lambda(\bxi^\lambda))\big) D_{\y_1}^2\Big) \Big ]u^\lambda(s,\y) \Big| \textcolor{black}{\bigg |_{\bxi=\x}}
\notag\\
&&+ \Big| \sum_{i=2}^n \int_{t+\textcolor{black}{c_0}\lambda \d^2(\x,\x')}^T ds\int_{\R^{nd}}d\y \Big[ D^2_{\x_1}  \tilde p_{\lambda}^{\bxi} (t,s,\x,\y)- D^2_{\x_1}  \tilde p_{\lambda}^{\bxi} (t,s,\x',\y) \Big]
\notag \\
&&\quad  \Big \langle \gF_{\lambda,i}(t,\y)-[\gF_{\lambda,i}(t,\btheta_{s,t}^\lambda(\bxi^\lambda))+D\gF_{\lambda,i}(t,\btheta_{s,t}^\lambda(\bxi^\lambda))(\y-\btheta_{s,t}^\lambda(\bxi^\lambda))_{i-1}], D_{\y_i} u^\lambda(s,\y) \Big \rangle \Big|\textcolor{black}{\bigg |_{\bxi=\x}}
\notag\\
&=& \Big|\int_{t+\textcolor{black}{c_0} \lambda \d^2(\x,\x')}^T ds \int_{\R^{nd}} d\y \int_0^1 d\mu  \mathbf D D^2_{\x_1}  \tilde p_{\lambda}^{\bxi} (t,s,\x+\mu(\x'-\x),\y) \cdot (\x-\x') \notag\\
&&\quad \bigg ( \Big \langle \gF_{\lambda,1}(t,\y)-\gF_{\lambda,1}(t,\btheta_{s,t}^\lambda(\bxi^\lambda)), D_{\y_1} \Big \rangle 
 + \frac{1}2{\rm Tr} \Big( \big(a_\lambda(t,\y)-a_\lambda(t,\btheta_{s,t}^\lambda(\bxi^\lambda))\big) D_{\y_1}^2\Big) \bigg) u^\lambda(s,\y) 
 \Big | \textcolor{black}{\bigg |_{\bxi=\x}}
\notag\\
&&+ \Big| \sum_{i=2}^n \int_{t+\textcolor{black}{c_0}\lambda \d^2(\x,\x')}^T ds\int_{\R^{nd}}d\y \int_0^1 d\mu  \mathbf D D^2_{\x_1}  \tilde p_{\lambda}^{\bxi} (t,s,\x+\mu(\x'-\x),\y) \cdot (\x-\x')
\notag \\
&&\quad \Big \langle \gF_{\lambda,i}(t,\y)-[\gF_{\lambda,i}(t,\btheta_{s,t}^\lambda(\bxi^\lambda))+D\gF_{\lambda,i}(t,\btheta_{s,t}^\lambda(\bxi^\lambda))(\y-\btheta_{s,t}^\lambda(\bxi^\lambda)_{\textcolor{black}{i-1}})], D_{\y_i} u^\lambda(s,\y)\Big \rangle \Big|\textcolor{black}{\bigg |_{\bxi=\x}}
\notag\\
&=:& 
\Big\{D^2_{\x_1}(\Delta^{\lambda})_{{\rm \textbf{diag}}}^{\bxi,1}(t,\x,\x')+D^2_{\x_1}(\Delta^{\lambda})_{\rm {\mathbf {diag}}}^{\bxi,2:n}(t,\x,\x')\Big\}\bigg |_{\bxi=\x}.
\nonumber  \\
\end{eqnarray}
We will now control the first term of the above right hand side. A key point for the analysis is to observe that, on the considered diagonal regime, we actually have from equations \eqref{CORRESP_DENS_SCALING_IMP_BIS}-\eqref{THE_GOOD_SCALE_DER_CTR}, recalling that $\z\mapsto \m_{s,t}^{\bxi^\lambda}(\z) $ is affine and using the good scaling property in \eqref{GSP} and \eqref{RESCALED_RES}, that:
\begin{eqnarray}
|D_{\x_k}D_{\x_1}^2 \tilde p_{\lambda}^{\bxi}(t,s,\x+\mu(\x'-\x),\y)|&\le& \frac{C \lambda^{1+(k-\frac 12)}}{(s-t)^{1+(k-\frac 12)}} \bar p_{{C}^{-1},\lambda}(t,s,\x,\y)\exp\Big(\sum_{j=1}^n\frac{|(\x-\x')_j|^2 \lambda^{2j-1}}{(s-t)^{2j-1}}\Big)\notag\\
&\le & \frac{C \lambda^{1+(k-\frac 12)}}{(s-t)^{1+(k-\frac 12)}} \bar p_{{C}^{-1},\lambda}(t,s,\x,\y).\label{CONVEX_INEQ_SCALED_DIAG}
\end{eqnarray}
To obtain the last inequality, we indeed observe from the homogeneity of the metric (see \eqref{DIST}) that $\lambda^{1/2}\d(\x,\x')=\d\big( \lambda^{-1/2}\T_\lambda \x,\lambda^{-1/2}\T_\lambda \x'\big)=\sum_{j=1}^n \big|(\x-\x')_j \lambda^{(2j-1)/2}\big|^{1/(2j-1)}$. For the diagonal regime $\lambda \d^2(\x,\x') \le (s-t)$ in turn implies that for each $j\in \leftB 1,n \rightB,\ (s-t)^{-2j+1} |(\x-\x')_j|^2 \lambda^{2j-1}\le 1  $.

Another key point is to observe that the contribution $(s-t)^{-1/2}\d\big( \lambda^{-1/2}\T_\lambda \y,\lambda^{-1/2}\T_\lambda \btheta_{s,t}^\lambda(\x^\lambda)\big) $ is homogeneous to the argument of the exponential term in $\bar p_{{C}^{-1},\lambda}(t,s,\x,\y) $. Namely, for any given $\beta_0>0 $ and $\beta\in (0,\beta_0] $, there exists $C_{\beta_0}$ s.t. 
\begin{equation}
\label{ABSORB_LAMBDA_BIS}
\left( \frac{\d\big( \lambda^{-\frac 12}\T_\lambda \y,\lambda^{-\frac 12}\T_\lambda \btheta_{s,t}^\lambda(\x^\lambda)\big)}{(s-t)^{\frac 12}} \right)^{\beta}\bar p_{{C}^{-1},\lambda} (t,s,\x,\y) \le C_{\beta_0} \bar p_{C_{\beta_0}^{-1},\lambda }(t,s,\x,\y),
\end{equation}
Equation \eqref{ABSORB_LAMBDA_BIS} is a direct consequence of the expression of $\bar p_{C^{-1},\lambda}^{t,\x} $ in Proposition \ref{THE_PROP} \textcolor{black}{(see also \eqref{THE_GOOD_SCALE_DER_CTR})} and the definition of $\d $ in \eqref{DIST}.

Hence, from the definition of $a_\lambda, \gF_\lambda $ in \eqref{CORRESP_COEFF_LAMBDA}\textcolor{black}{,} 
equations \eqref{CONVEX_INEQ_SCALED_DIAG} and \eqref{ABSORB_LAMBDA_BIS}, we derive that, under \A{A}\textcolor{black}{,} for each $k \in \leftB1,n \rightB$:
\begin{eqnarray}
&&\Big\{|\gF_{\lambda,1}(t,\y)-\gF_{\lambda,1}(t,\btheta_{s,t}^\lambda(\bxi^\lambda))| |D_{\x_k}D_{\x_1}^2 \tilde p_{\lambda}^{\bxi}(t,s,\x+\mu(\x'-\x),\y)|\Big\}
\Big |_{  \bxi=\x}
\notag \\
\!&\!\le\!&\! C\Big \{ \Big(\lambda^{-\frac 12}\|\gF_1(t,\cdot)\|_{C_\d^{\gamma}} \d^\gamma\big( \lambda^{-\frac 12}\T_\lambda \y,\lambda^{-\frac 12}\T_\lambda \btheta_{s,t}^\lambda(\bxi^\lambda)\big) \Big)
\lambda^{1+k-\frac 12}(s-t)^{-1-(k-\frac 12)}\bar p_{{C}^{-1},\lambda}(t,s,\x,\y)\Big\}
\Big |_{  \bxi=\x}
\notag\\
\!&\!\le\!&\! C 
\|\gF_1\|_{L^\infty(C^\gamma_\d)}\lambda^{k}
(s-t)^{-1-(k-\frac 12)+\frac \gamma 2}\bar p_{{C}^{-1},\lambda}(t,s,\x,\y).\label{CTR_DRIFT_1_lambda}
\end{eqnarray}
Similarly,
\begin{eqnarray}
 &&\Big\{|a_\lambda(t,\y)-a_\lambda (t,\btheta_{s,t}^\lambda(\bxi^\lambda)) |D_{\x_k}D_{\x_1}^2 \tilde p_{\lambda}^{\bxi}(t,s,\x+\mu(\x'-\x),\y)|
\Big\}\Big |_{  \bxi=\x}
 \notag \\
 \!&\!\le\!&\! C\Big\{\Big(\lambda^{-1}\|a(t,\cdot)\|_{C_\d^{\gamma}} \d^\gamma\big( \lambda^{-\frac 12}\T_\lambda \y,\lambda^{-\frac 12}\T_\lambda \btheta_{s,t}^\lambda(\bxi^\lambda)\big) \Big)
\lambda^{1+(k-\frac 12)}(s-t)^{-1-(k-\frac 12)}\bar p_{{C}^{-1},\lambda}(t,s,\x,\y) \Big\}\Big |_{  \bxi=\x}
\notag \\
\!&\!\le\!&\!C\|a\|_{L^\infty(C^\gamma_\d)}\lambda^{k-\frac 12} (s-t)^{-1-(k-\frac 12)+\frac \gamma 2}\bar p_{{C}^{-1},\lambda}(t,s,\x,\y). \label{CTR_DIFF_1_lambda}
\end{eqnarray}

Observe that both r.h.s. of \eqref{CTR_DRIFT_1_lambda} and \eqref{CTR_DIFF_1_lambda} exhibit a positive power of $ \lambda$. Hence, those quantities are \textit{small} provided $\lambda $ is.
The key point in the above computations is that the potentially explosive  H\"older norms of $\gF_{\lambda,1}, a_{\lambda}$(when $\lambda $ goes to zero) are compensated by the derivatives of  $\tilde p_\lambda^\bxi(s,t,\x,\y) $ (see equation \eqref{THE_GOOD_SCALE_DER_CTR}).
 We again carefully point-out that the previous bounds only depend on the supremum in time of the H\"older moduli of the coefficients (denoted $ \|\gF_1\|_{L^\infty(C^\gamma_\d)},\|a\|_{L^\infty(C^\gamma_\d)}$ respectively) and not on their supremum norm. In particular, we get from \eqref{CTR_DRIFT_1_lambda}, \eqref{CTR_DIFF_1_lambda} with the notation of \eqref{R_OFF_DIAG_LAMBDA_1}:
\begin{eqnarray}
\label{CTR_OFF_DIAG_NON_DEG}
&&|D^2_{\x_1}(\Delta^{\lambda})_{\rm {\mathbf {diag}}}^{\bxi,1}(t,\x,\x')|\Big|_{\bxi=\x}\nonumber \\
&\leq& \sum_{k=1}^n \int_{t+\textcolor{black}{c_0}\lambda\d^2(\x,\x)}^T ds \int_{\R^{nd}} d\y \lambda^{k-\frac 12}(s-t)^{-1-(k-\frac 12)+\gamma/2}  \nonumber \\
&&\times [\lambda^{\frac 12}\|D_{\x_1}u^\lambda\|_{L^\infty} \|\gF_1\|_{L^\infty(C^\gamma_\d)}+ \|D_{\x_1}^2 u^\lambda\|_{L^\infty}\|a\|_{L^\infty(C^\gamma_\d)}] 
\bar p_{{C}^{-1},\lambda}(t,s,\x,\y )|\x_k-\x_k'|
\nonumber \\
&\le&  \textcolor{black}{c_0^{-(n-\frac 12) +\frac \gamma 2}} \Lambda(\|D_{\x_1}u^\lambda\|_{L^\infty}+\|D_{\x_1}^2 u^\lambda\|_{L^\infty})\lambda^{\frac \gamma2}\d^\gamma(\x,\x').
\end{eqnarray}
Thanks to the inequality \eqref{ineq_Rm_2_n} and the previous analysis (to be performed according to the current scaling procedure replacing $(s-t) $ in Sections \ref{Sectio_Sup_D2} and \ref{HOLDER} by $(s-t)/\lambda $), we deduce:
\begin{eqnarray}
&&|D^2_{\x_1}(\Delta^{\lambda})_{\rm {\mathbf {diag}}}^{\bxi,2:n}(t,\x,\x')|\Big|_{\bxi=\x}\notag \\
\!\!&\!\! \leq\!\! &\!\!
 \bigg| \sum_{i=2}^n \int_{t+\textcolor{black}{c_0}\lambda \d^2(\x,\x')}^T \! \!\!\!\!ds\int_{\R^{nd}}\!\!\!d\y \int_0^1 \!\!\!d\mu  
 \Big  \langle D_{\y_i}\bigg \{ \Big (\gF_{\lambda,i}(t,\y)-[\gF_{\lambda,i}(t,\btheta_{s,t}^\lambda(\bxi^\lambda))+D\gF_{\lambda,i}(t,\btheta_{s,t}^\lambda(\bxi^\lambda))(\y-\btheta_{s,t}^\lambda(\bxi^\lambda))_{\textcolor{black}{i-1}}] \Big )
 \notag \\
 &&\quad \times \Big (\big\langle \x-\x', D_\x D^2_{\x_1}  \tilde p_{\lambda}^{\bxi} (t,s,\x+\mu(\x'-\x),\y)\big  \rangle\Big ) \bigg \}
,  u^\lambda(s,\y)\Big \rangle \bigg|\Bigg |_{\bxi=\x}
\notag \\
\!\!&\!\! \leq\!\! &\!\!\Lambda \|u^\lambda\|_{L^\infty (C_{b,\d}^{2+\gamma})} \sum_{k=1}^n \int_{t+\textcolor{black}{c_0} \lambda \d^2(\x,\x')}^T ds \frac{\lambda^{k-\frac 12}|\x_k-\x'_k|}{(s-t)^{1+(k-\frac 12)-\frac \gamma 2}}\notag\\
\!\!&\!\! \leq\!\! &\!\!\Lambda\|u^\lambda\|_{L^\infty (C_{b,\d}^{2+\gamma})}  \sum_{k=1}^n\textcolor{black}{c_0^{-(k-\frac 12)+ \frac \gamma 2}} \lambda^{(k-\frac 12)-(k-\frac 12)+\frac {\gamma} 2}  \d^{\gamma-(2k-1)}(\x,\x') |\x_k-\x'_k|
\nonumber \\
\!\!&\!\! \leq\!\! &\!\! \Lambda \|u^\lambda\|_{L^\infty (C_{b,\d}^{2+\gamma})} \sum_{k=1}^n\textcolor{black}{c_0^{-(k-\frac 12)+ \frac \gamma 2}}   \lambda^{\frac{ \gamma} 2} \d^\gamma(\x,\x')\le \Big(\Lambda  \textcolor{black}{c_0^{-(n-\frac 12)+ \frac \gamma 2}}  \lambda^{\frac \gamma 2} \Big)\|u^\lambda\|_{L^\infty(C_{b,\d}^{2+\gamma}) } \d^\gamma(\x,\x').\label{CTR_OFF_DIAG_DEG}
\nonumber \\
\end{eqnarray}
Plugging \eqref{CTR_OFF_DIAG_NON_DEG} and \eqref{CTR_OFF_DIAG_DEG} into \eqref{R_OFF_DIAG_LAMBDA_1} gives a diagonal control which precisely matches the r.h.s. of the expected final bound \eqref{THE_CTR_RESCALED}. 

It therefore remains to handle the off-diagonal contributions. With the notations of \eqref{RESCALED_QUANTITY_TO_CONTROL}, we deduce from the analysis of Section \ref{SUB_HOLDER_CTR_PERT_PART} (recall that in that case we chose $(\bxi,\bxi')=(\x,\x')$ as freezing points) and the previous arguments that:
\begin{eqnarray}\label{SCALE_OFF_DIAG}
&&| D^2_{\x_1}(\Delta^{\lambda})_{{\rm \textbf{off-diag}}}^{\bxi,\bxi'}(t,\x,\x')|\Big|_{(\bxi,\bxi')=(\x,\x')}\notag\\
&\le& \Lambda\| u^\lambda\|_{L^\infty(C_{b,\d}^{2+\gamma})}\int_{t}^{t+ \textcolor{black}{c_0}\lambda \d^2(\x,\x')} \frac{ds}{(s-t)^{1-\frac \gamma 2}}\le \Big(\Lambda \textcolor{black}{c_0^{ \frac \gamma 2}}  \lambda^{\frac \gamma 2}\Big) \| u^\lambda\|_{L^\infty(C_{b,\d}^{2+\gamma})} \d^\gamma(\x,\x').
\end{eqnarray}
In some sense the controls in \eqref{CTR_OFF_DIAG_DEG} and \eqref{SCALE_OFF_DIAG} can be seen as a mere consequence of the intrinsic scaling of the system. This is indeed the case, but, in order to avoid any ambiguity, we provide in Appendix \ref{SEC_BESOV_DUAL_FIRST_Scalling} a proof of the rescaled key Besov \textcolor{black}{control of} Lemma \ref{LEMME_BESOV_DEG}. The bound of equation \eqref{SCALE_OFF_DIAG} completes the proof of \eqref{THE_CTR_RESCALED}.\qed\\

Now, using the notations of \eqref{RESCALED_QUANTITY_TO_CONTROL}, it can be deduced from the same procedure (exploiting the control \eqref{SCALED_COV_AND_SENSI_CRITICAL_TIME} for the difference of the scaled covariances) that the following rescaled version of estimate \eqref{esti_holder_perturpart} in Lemma \ref{HOLDER} holds:
\begin{equation}\label{THE_CTR_RESCALED_2}
\sup_{t\in [0,T],(\x,\x')\in (\R^{nd})^2,\ \x_1=\x_1'}\frac{|(\Delta^{\lambda})^{\bxi,\bxi',\tilde \bxi'}(t,\x,\x')|\Big|_{(\bxi,\bxi',\tilde \bxi')=(\x,\x'\x)}}{\d^{2+\gamma}(\x,\x')}\leq  
 \Lambda \lambda^{\frac \gamma{2}}(c_0^{-(n-\frac 12) + \frac \gamma{2}}+ c_0
 ) \|u^\lambda\|_{L^\infty(C^{2+\gamma}_{b,\d})}.
\end{equation}

\textbf{H\"older estimate for the semi group at the discontinuity point.} 
We would also derive from \eqref{control_hoped_mathbf_R_ineq_trig1} and the same previous arguments  \textcolor{black}{by denoting $t_0=t+c_0 \lambda \d^2(\x,\x')$} 
and using the bound of \eqref{SCALED_DIFF_MEANS} for the difference of the scaled flows, that the analogous of estimates in Lemma \ref{lem_Holder_tildePg} : \begin{eqnarray}
\label{control_hoped_mathbf_R_ineq_trig1_lambda}
\big |D_{\x_1}^2 \tilde P_{t_0,t}^{\bxi'} u^\lambda(t_0, \x')- D_{\x_1}^2 \tilde P_{t_0,t}^{\tilde \bxi'} u^\lambda(t_0, \x')\big |  \Big |_{(\bxi',\tilde \bxi ')=(\x,\x')} 
\leq C \textcolor{black}{c_0^{\frac{\textcolor{black}{\gamma}}{2n-1}}}
\|u^\lambda\|_{C_{b,\d}^{2+\gamma}}  \d^\gamma(\x,\x'). 
\end{eqnarray}
and 
\begin{eqnarray*}
   \big |\tilde P_{t_0,t}^{\bxi'} u^{\lambda}(t_0, \x')- \tilde P_{t_0,t}^{\tilde \bxi'} u^{\lambda}(t_0, \x') \big | &\leq & C c_0  \|u^{\lambda}\|_{L^\infty(C^{2+\gamma}_{b,\d})}  \d^{2+\gamma}(\x,\x'), \text{ for } \x_1=\x_1',
\end{eqnarray*}
hold. Note that above we denoted with a slight abuse of notation 
$$\tilde P_{t_0,t}^{\bxi'} u^{\lambda}(t_0, \x') = \int_{\R^{nd}} d\y \tilde p_\lambda^{\bxi'}(t,t_0,\x',\y)u^{\lambda}(t_0, \y).$$

\textbf{H\"older estimates for the frozen semi-group and associated Green kernel.} It easily follows from \eqref{DIST_P} that, in this rescaled setting, the ad hoc estimates in Lemmas \ref{lem_Holder_tildePg} and \ref{lemme_Holder_Gg_VAR_DEG_ET_D2_NON_DEG} remain valid (i.e. with $f_\lambda,g_\lambda$ intead of $f,g$ and for the Green and frozen kernel associated with the rescaled system \eqref{KOLMO_LAMBDA} therein).\\

\textbf{Estimates of the supremum of the derivates w.r.t. the non degenerate variable.}

We importantly point out that for the H\"older moduli, we have benefitted from a regularizing effect in the scaling parameter $\lambda $. Note carefully that this is not the case as far as supremum derivatives are concerned. We indeed get, using the same previous arguments, that
\begin{eqnarray*}
|D_{\x_1}^2 u^\lambda(t,\x)|&=& |\int_t^{T} ds \int_{\R^{nd}}D^2_{\x_1} \tilde p_{\lambda}^{\bxi}(t,s,\x,\y) (L_s^\lambda-\tilde L_s^{\lambda,\bxi}) u^\lambda(s,\y)d\y|\\
&\le &C\big( \|g\|_{C_{b,\d}^{2+\gamma}}+(T-t)^{\frac \gamma 2}\|f\|_{L^\infty(C_{b,\d}^\gamma)}\big) +\Lambda (T-t)^{\frac{\gamma}{2}}\|u^\lambda\|_{L^\infty(C_{b,\d}^{2+\gamma})}.
\end{eqnarray*} 
The difference w.r.t. e.g. equations \eqref{CTR_OFF_DIAG_NON_DEG}, \eqref{CTR_OFF_DIAG_DEG} is that we have no additional spatial differentiations and no cutting threshold in time which precisely allowed to make the scaling parameter appear. Equivalently, we have to integrate in time the control \eqref{SCALE_OFF_DIAG} of the previous off-diagonal regime but on the whole time interval $[t,T] $. \\
 
\textbf{Conclusion: Schauder estimate for the rescaled system.} Gathering the above estimates, we eventually derive similarly to \eqref{control_voulu}, recalling that $c_0<1$, that:
\begin{equation}
 \label{control_voulu_2}
\|u^\lambda\|_{L^\infty(C^{2+\gamma}_{b,\d})} \leq C(\|g_\lambda\|_{C^{2+\gamma}_{b,\d}}+ \|f_\lambda\|_{L^\infty(C^\gamma_{b,\d})})+ \Big[\Lambda \big(\lambda^{\frac{\gamma}{2
}}(\textcolor{black}{c_0^{-(n-\frac 12)}}+c_0^{\frac \gamma 2})+T^{\frac \gamma 2}\big) +C \textcolor{black}{c_0^{\frac 1{2n-1}}} \Big]\|u^\lambda\|_{L^\infty(C^{2+\gamma}_{b,\d})}.
 \end{equation}
\textcolor{black}{Therefore}, for $T$,  $\textcolor{black}{c_0}:=c_0(\A{A},T)$ and $\lambda:=\lambda(\A{A},T) $ small enough, \textcolor{black}{ with $\lambda \ll c_0 $ }, i.e. \textcolor{black}{such that} $\bar c_0:=\Big[\Lambda \big(\lambda^{\frac \gamma2}(c_0^{-(n- \frac 12)} +c_0^{\frac{\gamma}{2
}})+T^{\frac \gamma 2}\big)+C \textcolor{black}{c_0^{\frac 1{2n-1}}}\Big]<1 $, the expect final control holds:
\begin{equation}
 \label{control_voulu_2_BIS}
\|u^\lambda\|_{L^\infty(C^{2+\gamma}_{b,\d})} \leq \frac{C}{1-\bar c_0}(\|g_\lambda\|_{C^{2+\gamma}_{b,\d}}+ \|f_\lambda\|_{L^\infty(C^\gamma_{b,\d})}). 
\end{equation}

\section{Conclusion: final proof of Theorem \ref{THEO_SCHAU}.}

Equation \eqref{control_voulu_2_BIS} provides the expected Schauder estimate for the rescaled system with mollified coefficients for a \textit{small} time horizon $T>0$. The scaling parameter $\lambda $ must precisely be tuned w.r.t. $\Lambda$ (associated with the H\"older moduli of the coefficients, see Remark \ref{REM_LAMBDA}) in order that the above constant $\bar c_0$ be strictly less than one. Recalling that $u^\lambda(t,\x)=u(t,\lambda^{-1/2}\T_\lambda \x) $, we then derive from \eqref{control_voulu_2_BIS} and  \eqref{CORRESP_COEFF_LAMBDA} that for $T>0 $ and $\lambda  $ small enough there exists $\bar C_0:=\bar C_0(\A{A},T,\lambda)>1 $ s.t.
\begin{equation}
 \label{control_voulu_3}
\|u\|_{L^\infty(C^{2+\gamma}_{b,\d})} \leq \bar C_0(\|g\|_{C^{2+\gamma}_{b,\d}}+ \|f\|_{L^\infty(C^\gamma_{b,\d})}),
\end{equation}
which precisely provides the required estimate for the initial system with mollified coefficients.

The point is now to extend the previous bound to an arbitrary fixed time $T>0$ not necessarily small. The stability resulting from estimate \eqref{control_voulu_3} allows to proceed by simple iterative application of the bound changing the final condition. 

 \subsection{Schauder estimates for the mollified system for a general time}

Equation \eqref{control_voulu_3} is valid for any $T<T_0$ with $T_0 \in (0,+\infty)$ sufficiently small.
Now, for a given $T>0$ (not necessary small), we can solve iteratively $ N =\lceil \frac{T}{T_0}\rceil$ (where $\lceil \cdot \rceil$ is the ceiling function) Cauchy problems in the following way.   Consider first:
\begin{equation*}
\label{KOLMO_1}
\begin{cases}
\partial_t u_{(1)}(t,\x)+ \langle \gF(t,\x), \mathbf  Du_{(1)}(t,\x)\rangle +\frac 12{\rm Tr}\big( D_{\x_1}^2u_{(1)}(t,\x)a(t,\x)\big)=-f(t,\x),\ (t,\x)\in [T(1-\frac 1N),T)\times \R^{nd},\\
 u_{(1)}(T,\x)=g(\x),\ \x\in \R^{nd}.
 \end{cases}
 \end{equation*}
In other words, our previous analysis, and the previous inequalities, are still available for $T-(1-\frac 1N)T= \frac 1N T\le T_0$ small enough. Precisely, from \eqref{control_voulu_3}:
 \begin{equation}
\label{CTR_U1}
 \|u_{(1)}\big(T(1-\frac 1N),\cdot\big)\|_{C_{b,\d}^{2+\gamma}}\le \bar C_0(\|g\|_{{C_{b,\d}^{2+\gamma}}}+\|f\|_{L^\infty([T(1-\frac 1N),T] ,C_{b,\d}^{\gamma})}).
 \end{equation}
 Also, for mollified coefficients, it is plain from the Feynman-Kac formula to identify $u_{(1)}$ and $u$ on $[T(1-\frac 1N),T] $ where $u$ solves \eqref{KOLMO} with mollified coefficients on $[0,T] $. Hence, \eqref{CTR_U1} gives in particular that 
 $u_{(1)}\big(T(1-\frac 1N),\cdot\big)=u\big(T(1-\frac 1N),\cdot\big)\in C_{b,\d}^{2+\gamma}(\R^{nd},\R) $ so that it is in particular natural to consider  the following second Cauchy problem \textcolor{black}{for any $\x \in \R^{nd}$}:
\begin{equation*}
\begin{cases}
\partial_t u_{(2)}(t,\x)+ \langle \gF(t,\x), \mathbf Du_{(2)}(t,\x)\rangle +\frac 12{\rm Tr}\big( D_{\x_1}^2u_{(2)}(t,\x)a(t,\x)\big)=-f(t,\x),\ t\in [(1-\frac 2N)T,(1-\frac 1N)T),\\ 
 u_{(2)}((1-\frac 1N)T,\x)= u((1-\frac 1N)T,\x).
 \end{cases}
 \end{equation*}
 Hence  $u_{(2)}$ satisfies identity \eqref{control_voulu_3} for the corresponding time interval and the associated source and final condition. It also coincides with $u$ on $[T(1-\frac 2N),T(1-\frac 1N)] $. We get:
 \begin{eqnarray*}
&& \|u_{(2)}\|_{L^\infty([T(1-\frac 2N),T(1-\frac 1N)],C_{b,\d}^{2+\gamma})}
\nonumber \\
&\le& \bar C_0(\|u(T(1-\frac 1N),\cdot)\|_{{C_{b,\d}^{2+\gamma}}}+\|f\|_{L^\infty([T(1-\frac 2N),T(1-\frac 1N)] ,C_{b,\d}^{\gamma})})\\
 &\le & \bar C_0 \Big( \bar C_0(\|g\|_{{C_{b,\d}^{2+\gamma}}}+\|f\|_{L^\infty([T(1-\frac 1N),T] ,C_{b,\d}^{\gamma})})
 +\|f\|_{L^\infty([T(1-\frac 2N),T(1-\frac 1N)] ,C_{b,\d}^{\gamma})}\Big) \\
 &\le & \bar C_0^2  (\|g\|_{{C_{b,\d}^{2+\gamma}}}+\|f\|_{L^\infty([T(1-\frac 2N),T] ,C_{b,\d}^{\gamma})}).
 \end{eqnarray*}
Repeating the analysis $N$-times, introducing for $k\in \leftB 3,n\rightB$ the auxiliary Cauchy problems \textcolor{black}{for any $ \x\in \R^{nd}$}:
\begin{equation*}
\begin{cases}
\partial_t u_{(k)}(t,\x)+ \langle \gF(t,\x), \mathbf Du_{(k)}(t,\x)\rangle +\frac 12{\rm Tr}\big( D_{\x_1}^2u_{(k)}(t,\x)a(t,\x)\big)=-f(t,\x),\ t,\in [T(1-\frac kN),T(1-\frac {k-1}N)),\\
 u_{(k)}((1-\frac {k-1}N)T,\x)= u((1-\frac {k-1}N)T,\x), 
 \end{cases}
 \end{equation*}
we derive that 
\begin{equation}
\label{controle_voulu_4}
 \|u\|_{L^\infty([0,T],C_{b,\d}^{2+\gamma})}
 \le  \bar C_0^N  (\|g\|_{{C_{b,\d}^{2+\gamma}}}+\|f\|_{L^\infty([0,T] ,C_{b,\d}^{\gamma})}).
\end{equation}
This precisely gives our main estimate for the system \eqref{KOLMO} with mollified coefficients. Again, even though the coefficients are smooth, all the constants appearing in \eqref{controle_voulu_4} only depend on the H\"older setting of assumption \A{A}.
\begin{REM}[About the constants in the final estimate] We could actually have slightly better bounds than those in  \eqref{controle_voulu_4}. Namely a direct induction shows that the following control holds.
\begin{equation}
\label{controle_voulu_4_bis}
 \|u\|_{L^\infty([0,T],C_{b,\d}^{2+\gamma})}
 \le  \bar C_0^N  \|g\|_{{C_{b,\d}^{2+\gamma}}}+\sum_{i=1}^N C_0^{i}\|f\|_{L^\infty([T(1-\frac{N-i+1}N),T(1-\frac{N-i}{N} )] ,C_{b,\d}^{\gamma})}).
\end{equation}
We chose to write the bound in the form of equation \eqref{controle_voulu_4} for simplicity. Note however that in any case, equation \eqref{controle_voulu_4} or \eqref{controle_voulu_4_bis}, we still have geometric constants coming from the iterative procedure. This is the main drawback of our approach, which anyhow seems, to the best of our knowledge, the only available one to consider degenerate Kolmogorov systems with non-linear drifts.
\end{REM}

 \subsection{Compactness arguments}
We now make the mollifying parameter appear again using the notations introduced in the detailed guide (see equation \eqref{KOLMO_m}). Equation \eqref{control_voulu_3} rewrites in the following way. There exists a constant $C_0$ s.t. for each $ m\in \N$:
\begin{equation}
\label{controle_voulu_4_MOL}
 \|u_m\|_{L^\infty([0,T],C_{b,\d}^{2+\gamma})}
 \le  \bar C_0  (\|g\|_{{C_{b,\d}^{2+\gamma}}}+\|f\|_{L^\infty([0,T] ,C_{b,\d}^{\gamma})}).
\end{equation}
\textcolor{black}{From Ascoli's theorem we deduce that there exists 
\textcolor{black}{a sequence $(u_{m_k})_{k \geq 1}$ of smooth solutions of \eqref{KOLMO_m}, with $m_k \rightarrow_k+\infty$, satisfying \eqref{controle_voulu_4_MOL}} and s.t. $u_{m_k} \rightarrow_k u $ in  $L^\infty([0,T],C_{b,\d}^{2+\gamma-\varepsilon})$, for all $0<\varepsilon< \gamma$. 
We then deduce from the previous analysis that such a $u$ \textcolor{black}{also satisfies \eqref{controle_voulu_4_MOL}}, and therefore lies in $L^\infty([0,T],C_{b,\d}^{2+\gamma})$.} Since we also have from \cite{chau:meno:17}
that $u_{m_k}(t,\x) \rightarrow_k \E[g(\X_T^{t,\x})]+\int_t^T \E[f(s,\X_{s}^{t,\x})]ds $ where $(\X_s^{t,\x}) $ denotes the unique in law solution to \eqref{SYST}, we deduce that $u(t,\x)=\E[g(\X_T^{t,\x})]+\int_t^T \E[f(s,\X_{s}^{t,\x})]ds $ corresponds to the \textit{unique} martingale, or \textit{mild},  solution of \eqref{KOLMO} which satisfies the stated Schauder estimate \eqref{control_voulu_3}.

\subsection{From mild to weak solutions}
\textcolor{black}{For this paragraph, we will consider the previous converging subsequence obtained from the compactness argument and will still denote it by $(u_m)_{m\in \N}$ for notational simplicity. Namely,  $\|u-u_m\|_{L^\infty(C_{b,\d}^{2+\gamma-\varepsilon})} \underset{m}{\longrightarrow} 0,\ \varepsilon \in (0,\gamma) $.}

Let $\varphi $ be a smooth given function with compact support, i.e. $\varphi \in C_0^\infty(\R^{nd},\R) $. It is clear that for the solution of \eqref{KOLMO} with mollified coefficients one indeed has:
\begin{eqnarray}
\label{VAR_M}
\int_{0}^T dt \int_{\R^{nd}} d\x f_m(t,\x) \varphi(t,\x)= \int_{0}^T dt \int_{\R^{nd}} d\x\varphi(t,\x)\Big(\partial_t +L_t^m \Big)u_m(t,\x).
\end{eqnarray}
Indeed, both the solution and the coefficients are smooth. Integrating by parts yields:
\begin{eqnarray*}
\int_{0}^T dt \int_{\R^{nd}} d\x f_m(t,\x) \varphi(t,\x)= \int_{0}^T dt \int_{\R^{nd}} d\x \Big(-\partial_t+(L_t^ {m})^*\Big)\varphi(t,\x)u_m(t,\x),
\end{eqnarray*} 
where $(L_t^ {m})^* $ denotes the adjoint of $ L_t^ {m}$.

Write now:
\begin{eqnarray}
\int_{0}^T dt \int_{\R^{nd}} d\x f_m(t,\x) \varphi(t,\x)&=&\int_{0}^T dt \int_{\R^{nd}} d\x f(t,\x) \varphi(t,\x)+\int_{0}^T dt \int_{\R^{nd}} d\x (f_m-f)(t,\x) \varphi(t,\x)\notag \\
&=:&\int_{0}^T dt \int_{\R^{nd}} d\x f(t,\x) \varphi(t,\x)+R_m(T,f). \label{DEF_RM}
\end{eqnarray}
It is clear that under \A{A},  recall that $f\in L^\infty([0,T],C_{b,\d}^{\gamma}(\R^{nd},\R)) $, $R_m(T,f) \underset{m}{\longrightarrow} 0$.
On the other hand, we now decompose:
\begin{equation}
\int_{0}^T dt \int_{\R^{nd}} d\x \Big(-\partial_t+(L_t^ {m})^*\Big)\varphi(t,\x)u_m(t,\x)
=\int_{0}^T dt \int_{\R^{nd}} d\x \Big(-\partial_t+(L_t)^*\Big)\varphi(t,\x)u(t,\x)+R_m(T,u),
\label{THE_RM}
\end{equation}
with 
\begin{eqnarray}\label{THE_OTHER_RM}
&&R_m(T,u)
\nonumber \\
&:=&\int_{0}^T dt \int_{\R^{nd}} d\x \Big(
(L_t^ {m})^*-L_t^*\Big)\varphi(t,\x)u_m(t,\x)
+ \int_{0}^T dt\int_{\R^{nd}} d\x \Big(-\partial_t+(L_t)^*\Big)\varphi(t,\x)(u_m(t,\x)-u(t,\x))
\nonumber \\
&=:&(R_m^1+R_m^2)(T,u), \notag\\
\end{eqnarray}
where $L_t^* $ is the formal adjoint of $L_t$. Observe first that:
\begin{eqnarray*}
R_{m,0}^2(T,u)&:=&\int_{0}^T dt\int_{\R^{nd}} d\x \partial_t \varphi(t,\x)\big( u_m(t,\x)-u(t,\x) \big) \underset{m}{\longrightarrow} 0,\\
\end{eqnarray*}
since $\|u-u_m\|_{L^\infty(C_{b,\d}^{2+\gamma-\varepsilon})} \underset{m}{\longrightarrow} 0$ for any $0<\varepsilon<\gamma$. For the terms of $R_m^2(T,u) $ which involve the adjoint,
the point is again to use the Besov duality to control the remainders. Let $0<\varepsilon<\gamma$, from the previous analysis we get that:
\begin{eqnarray*}
|R_{m,2:n}^2(T,u)|&:=&|\sum_{i=2}^n\int_{0}^T dt\int_{\R^{nd}} d\x D_{\x_i} \big(\varphi(t,\x)   \gF_i(t,\x)\big)\big( u_m(t,\x)-u(t,\x) \big)| \\
&\le& \sum_{i=2}^n \int_0^T dt \int_{\R^{(n-1)d}} \prod_{j\neq i} d\x_j \|D_{\x_i}\big( \varphi  \gF_i (t,\x_{\backslash i},\cdot) \big)\|_{B_{1,1}^{-\frac{ 2+\gamma-\varepsilon}{2i-1}}(\R^d,\R)} \|(u_m-u)(t,\x_{\backslash i},\cdot)\|_{C_{b}^{\frac{2+\gamma-\varepsilon}{2i-1}}(\R^d,\R)},
\end{eqnarray*}
denoting $\x_{\backslash i}=(\x_1, \cdots, \x_{i-1},\x_{i+1},\cdots, \x_n) $ and $ \varphi \gF_i(t,\x_{\backslash i},\cdot): \x_i\in \R^d \mapsto \varphi \gF_i(t,\x)$, $(u-u_m)(t,\x_{\backslash i},\cdot): \x_i\in \R^d \mapsto (u-u_m)(t,\x) $. On the one hand we know that  $\|u-u_m\|_{L^\infty(C_{b,\d}^{2+\gamma-\varepsilon})} \underset{m}{\longrightarrow} 0$. On the other hand, $\varphi$ is a smooth test function and from \A{S} we have that $D_{\x_i}\gF_i(t,\x_{\backslash i},\cdot)$ belongs to $B_{\infty,\infty}^{\frac{-2+\gamma}{2i-1}}(\R^d,\R)$. Since  $-\frac{2-\gamma}{2i-1} \ge -\frac{2+\gamma-\varepsilon}{2i-1}$ for any $0<\varepsilon <2\gamma$ we deduce from the arguments of the proof of Lemma \ref{LEMME_BESOV_DEG} that 
there exists $C$ s.t. for each $i\in \leftB 2,n \rightB$, $\|D_{\x_i}\big( \varphi  \gF_i (t,\x_{\backslash i},\cdot) \big)\|_{B_{1,1}^{-\frac{ 2+\gamma-\varepsilon}{2i-1}}(\R^d,\R)}\le C \psi_i(t,\x_{\backslash i}) $ where $\psi_i$ has compact support on $\R^{(n-1)d} $. 

We thus readily derive $|R_{m,2:n}^2(T,u)| \underset{m}{\longrightarrow} 0 $. 

Eventually,
\begin{eqnarray*}
|R_{m,1}^2(T,u)|&:=& \Big |\int_{0}^T dt\int_{\R^{nd}} d\x \Big[D_{\x_1} \big(\varphi(t,\x)   \gF_1(t,\x)\big)+D_{\x_1}^2 \big(\varphi(t,\x)   a(t,\x)\big)\Big]\big( u_m(t,\x)-u(t,\x) \big) \Big | \\
&=& \Big |\int_{0}^T dt\int_{\R^{nd}} d\x \Big[ \big(\varphi(t,\x)   \gF_1(t,\x)\Big)D_{\x_1}+ \big(\varphi(t,\x)   a(t,\x)\big)D_{\x_1}^2\Big]\big( u_m(t,\x)-u(t,\x) \big) \Big |,
\end{eqnarray*}
 which again tends to 0 with $m$ since $\|u-u_m\|_{L^\infty(C_{b,\d}^{2+\gamma-\varepsilon})} \underset{m}{\longrightarrow} 0$ for any $0<\varepsilon<\gamma$.
 
 The contributions involving $\big((L_t^m)^*-L_t^*\big) \varphi$ in $R_m^1(T,u) $ defined in \eqref{THE_OTHER_RM} can be handled as in the proof of Lemma \ref{LEMME_BESOV_DEG} exploiting that $\|a-a_m\|_{C_{b,\d}^\gamma}+\|\gF_1-\gF_{m,1}\|_{C_{b,\d}^\gamma}+\sum_{i=2}^n \|\gF_i-\gF_{m,i}\|_{C_{b,\d}^{
 2i-3+\gamma
 }}\rightarrow_m 0 $. 
 We now deduce from \eqref{THE_RM}, \eqref{THE_OTHER_RM} and the previous controls that $R_m(T,u) \rightarrow_m 0$. The same computations also give that the term $\int_{0}^T dt \int_{\R^{nd}} d\x \Big(-\partial_t+(L_t)^*\Big)\varphi(t,\x)u(t,\x)$ is well defined under \A{A}. From  \eqref{DEF_RM}, \eqref{THE_RM}, we thus finally derive:
 $$\int_0^T dt \int_{\R^{nd}} d\x \varphi(t,\x)f(t,\x)=\int_{0}^T dt \int_{\R^{nd}} d\x \big(-\partial_t+(L_t)^*\big)\varphi(t,\x)u(t,\x),$$
 which gives the statement.

\appendix 
\mysection{Proof of technical results}
\label{APP}

\subsection{Technical results associated with the flow}


We begin this paragraph stating and proving a key result for the sensitivity of H\"older flows, i.e. when the coefficients satisfy \A{A}. Those results are of course uniform w.r.t. a mollification procedure of the coefficients as the one previously considered from Section \ref{SEC2} to \ref{HOLDER}. Also, Lemma \ref{lem_theta_theta} is a direct consequence of Lemma \ref{lem_theta_theta_bis} below and Young/convexity inequalities. 
 
\subsubsection{A first sensitivity result for the flow}\label{sec:ProofLem_d_theta}
\begin{lem}\label{lem_theta_theta_bis} Under \A{A}, there exists $C:=C(\A{A},T)$ s.t. for all $(\x,\x')\in (\R^{nd})^{2}$, $\textcolor{black}{\d(\x,\x')\le 1}$, $0\le t<s\le T$ and $i \in \leftB 1,n\rightB$:
\begin{equation*}
|(\btheta_{s,t}(\x)-\btheta_{s,t}(\x'))_i| \leq 
 C\Big( (s-t)^{i-\frac 12}+\d^{2i-1}(\x-\x') \Big).
\end{equation*}
\end{lem}
The previous bound can be interpreted as follows. We somehow have the expected bound involving the spatial points, 
up to an additional contribution in time, which is precisely due to the quasi-distance $\d$. Indeed,
  this type of result already appeared (for Lipschitz drifts) in Proposition 4.1 of \cite{meno:17}. Through an appropriate mollifying procedure, this result remains unchanged.


\begin{proof}[Proof]
The main idea to prove this control relies on Gr\"onwall's Lemma. However, under \A{A}, the function $\gF $ is not Lipschitz (\textcolor{black}{only} H\"older continuous).
We have then to mollify suitably $\gF$.
\textcolor{black}{Let us denote by $\delta \in \R^n$, a vector with positive entries $\delta_{i} >0$ for $i\in \leftB 2,n\rightB $.
Define as well for all $\ v\in [0,T]$, $\z\in \R^{nd} $, $i\in \leftB 2,n\rightB $, 
\begin{equation}
\label{DEF_CONV_ADHOC}
\gF_i^\delta(v,\z^{i-1:n}):=\gF_i(v,\cdot)\star \rho_{\delta_{i}}(\z)=\int_{\R^{d}} \gF_i(v,\z_{i-1},\z_i-w,\z_{i+1}\cdots ,\z_n)\rho_{\delta_{i}}(w)dw,
\end{equation} 
with  $\rho_{\delta_{i}}(w):=\frac{1}{\delta_{i}^{d}}\rho\left(\frac{w}{\delta_{i}} \right) $ where $ \rho:\R^{d}\rightarrow \R_{+}$ is a usual mollifier, namely $\rho$ has compact support and $\int_{\R^{d}} \rho(\z)d\z=1 $. }
Finally, we define $\gF^\delta(v,\z):=(\gF_1 (v,\z),\gF_2^\delta(v,\z), \cdots,\gF_n^\delta(v,\z)) $. In this definition, we make a slight abuse of notation since the first component $\gF_1$ is not mollified. Due to the final control we want to prove and the intrinsic scale of the first component, the sublinearity of $\gF_1$ (implied by its H\"older property) is enough and  it is not needed to mollify this component. 
\\

To be at the \textit{good} current time scale for the contributions associated with the mollification, we pick $\delta_{i}$ in order to have $C:=C(\A{A},T)>0$ s.t. for all $\z \in \R^{nd}$, $\textcolor{black}{v} \in [t,s]$:
\begin{eqnarray}\label{approxrescaleddrfit}
\Big|(s-t)^{\frac 12}\T_{s-t}^{-1}\Big(\gF(\textcolor{black}{v},\z)-\gF^\delta(\textcolor{black}{v},\z)\Big)\Big|\leq C(s-t)^{-1}.
\end{eqnarray}
By the previous definition of $\gF^\delta$ \textcolor{black}{and assumptions \A{A}}, 
 identity \eqref{approxrescaleddrfit} is equivalent to:
\begin{eqnarray}\label{eq:controleT1}
\textcolor{black}{\sum_{i=2}^n (s-t)^{\frac 12-i}  \delta_{i}^{\frac{2i-3+\gamma}{2i-1}} 
\leq C(s-t)^{-1}.}
\end{eqnarray}
Hence, we choose from now on, for all $i\in \leftB 2,n \rightB $:
\begin{equation}\label{choice_deltaij}
\delta_{i}= (s-t)^{(i-\frac 32)\frac{2i-1}{2i-3+\gamma}}.
\end{equation}

Next, let us control the last components of the flow.
By the definition of $\btheta_{s,t}$ in \eqref{DYN_DET_SMOOTH}, we get:
\begin{eqnarray*}
&&|(\btheta_{s,t}(\x)-\btheta_{s,t}(\x'))_n|
\nonumber \\
\!\!&\!\!\le\!\!& |(\x-\x')_n|
+
 \int_{t}^s \Big ( |\gF_n^\delta (v,\btheta_{v,t}(\x))- \gF_n^\delta (v,\btheta_{v,t}(\x'))|
\nonumber \\
&&+ |\gF_n^\delta (v,\btheta_{v,t}(\x))- \gF_n (v,\btheta_{v,t}(\x))|
+ |\gF_n^\delta (v,\btheta_{v,t}(\x'))- \gF_n (v,\btheta_{v,t}(\x'))|
 \Big)dv 
 \\
\!\!&\!\!
\leq\!\!&
|(\x-\x')_n|
+ C
 \int_{t}^s \Big ( 
\big  |\big (\btheta_{v,t}(\x)-\btheta_{v,t}(\x') \big )_{n-1} \big|
+   \delta_{n}^{-1+\frac{2n-3+\gamma}{2n-1}} \big  |\big (\btheta_{v,t}(\x)-\btheta_{v,t}(\x')\big )_n \big|
 \Big)
 dv +(s-t)  \delta_{n}^{\frac{2n-3+\gamma}{2n-1}} \!.
\end{eqnarray*}
Hence by Gr\"onwall' s Lemma, we get:
\begin{eqnarray}\label{ineq_theta_x_x_n}
&&|(\btheta_{s,t}(\x)-\btheta_{s,t}(\x'))_n|
\nonumber \\
&\le& C \exp \Big(C (s-t)  \delta_{n}^{-1+\frac{2n-3+\gamma}{2n-1}} \Big)
\Big(|(\x-\x')_n| + (s-t)  \delta_{n}^{\frac{2n-3+\gamma}{2n-1}}  
+\int_{t}^s 
\big  |\big (\btheta_{v,t}(\x)-\btheta_{v,t}(\x') \big )_{n-1} \big| dv \Big)
\nonumber \\
&\le& C \exp \Big(C (s-t)^{\frac \gamma 2} \Big)
\Big(|(\x-\x')_n| + (s-t)^{n-\frac 12} 
+\int_t^s\big  |\big (\btheta_{v,t}(\x)-\btheta_{v,t}(\x') \big )_{n-1} \big| dv \Big),
\end{eqnarray}
using \eqref{eq:controleT1} for the last inequality.
We proceed similarly for the $(n-1)^{\rm {th}}$ component, but in this case we have to handle the non-Lipschitz continuity of $\gF^\delta_{n-1}$ in its $n^{\rm {th}}$ variable. 

For the rescaled flows see e.g. Lemma 2 in \cite{chau:meno:17} this difficulty could also be circumvented through mollification, the situation is here slightly different and it seems that Young type controls are more appropriate. Write:
\begin{eqnarray}\label{ineq_theta_x_x_n1}
&&|(\btheta_{s,t}(\x)-\btheta_{s,t}(\x'))_{n-1}|
\nonumber \\
&\le& C \exp \Big (C (s-t)  \delta_{n-1}^{-1+\frac{2(n-1)-3+\gamma}{2(n-1)-1}} \Big)
\Big(|(\x-\x')_{n-1}| + (s-t)  \delta_{n-1}^{\frac{2(n-1)-3+\gamma}{2(n-1)-1}}  
\nonumber \\
&&
+\int_{t}^s \Big \{
\big  |\big (\btheta_{v,t}(\x)-\btheta_{v,t}(\x') \big )_{n-2} \big| +  |\big (\btheta_{v,t}(\x)-\btheta_{v,t}(\x') \big )_{n} \big|^{\frac{2(n-1)-3+\gamma}{2n-1}}  \Big \}dv \Big)
\nonumber \\
&\le& C \exp(C (s-t)^{\frac \gamma 2} )
\Big(|(\x-\x')_{n-1}| + (s-t)^{n-\frac 32} 
+\int_{t}^s \bigg \{
\big  |\big (\btheta_{v,t}(\x)-\btheta_{v,t}(\x') \big )_{n-2} \big| 
\nonumber \\
&&
+ |(\x-\x')_n|^{\frac{2(n-1)-3+\gamma}{2n-1}} + (v-t)^{\frac{2(n-1)-3+\gamma}{2}}
+
\Big (\int_t^v\big  |\big (\btheta_{w,t}(\x)-\btheta_{w,t}(\x') \big )_{n-1} \big| dw \Big )^{\frac{2(n-1)-3+\gamma}{2n-1}}\bigg \}dv \bigg).
\nonumber \\
\end{eqnarray}
The last inequality is a consequence of our choice of $\delta_{n-1}$ in \eqref{choice_deltaij}, identity \eqref{ineq_theta_x_x_n} and convexity inequality.
\\

We aim, now, to proceed with  Gr\"onwall's Lemma. To do so, we \textcolor{black}{first  of all } need to use a Young inequality. Namely, we write for any $\tilde \delta_{n-1,n} > 0$ (where the two indexes in the subscript respectively denote the \textcolor{black}{level} of the chain, i.e. $n-1$, and the considered variable, i.e. $n$):
\begin{eqnarray*}
&&\Big (\int_t^v\big  |\big (\btheta_{w,t}(\x)-\btheta_{w,t}(\x')) \big )_{n-1} \big| dw \Big )^{\frac{2(n-1)-3+\gamma}{2n-1}}
\nonumber \\
&\leq& C \bigg  (
\Big (\int_t^v\big  |\big (\btheta_{w,t}(\x)-\btheta_{w,t}(\x')) \big )_{n-1} \big| dw \Big ) \tilde \delta_{n-1,n}^{-\frac{2n-1}{2(n-1)-3+\gamma}} + \tilde \delta_{n-1,n}^{\frac{2n-1}{4-\gamma}} \bigg) .
\end{eqnarray*}
In order to obtain the suitable time scale, we choose $\tilde \delta_n$ s.t.
\begin{equation*}
 \tilde \delta_{n-1,n}^{\frac{2n-1}{4-\gamma}} = (v-t)^{\frac{2(n-1)-3+\gamma}{2}} \Longleftrightarrow   \tilde \delta_{n-1,n} = (v-t)^{\frac{(2(n-1)-3+\gamma)(4-\gamma)}{2(2n-1)}} ,
\end{equation*}
which also yields that
\begin{equation*}
\Big (\int_t^v\big  |\big (\btheta_{w,t}(\x)-\btheta_{w,t}(\x')) \big )_{n-1} \big| dw \Big ) \tilde \delta_{n-1,n}^{-\frac{2n-1}{2(n-1)-3+\gamma}} 
\leq \Big (\int_t^v\big  |\big (\btheta_{w,t}(\x)-\btheta_{w,t}(\x')) \big )_{n-1} \big| dw \Big ) (v-t)^{-\frac{4-\gamma}{2}}.
\end{equation*}
Hence we get   from  \eqref{ineq_theta_x_x_n1} and  the previous controls that for any $\tilde s\in [t,s] $:
\begin{eqnarray*}
&&|(\btheta_{\tilde s,t}(\x)-\btheta_{\tilde s,t}(\x'))_{n-1}|
\nonumber \\
&\le& C \exp(C (\tilde s-t)^{\frac \gamma 2} )
\Big(|(\x-\x')_{n-1}| + (\tilde s-t)^{n-\frac 32} 
+\int_{t}^{\tilde s} \bigg \{
\big  |\big (\btheta_{v,t}(\x)-\btheta_{v,t}(\x')) \big )_{n-2} \big| 
\nonumber \\
&&
+ |(\x-\x')_n|^{\frac{2(n-1)-3+\gamma}{2n-1}} + (v-t)^{\frac{2(n-1)-3+\gamma}{2}}
+
 \Big (\int_t^v\big  |\big (\btheta_{w,t}(\x)-\btheta_{w,t}(\x')) \big )_{n-1} \big| dw \Big ) (v-t)^{-\frac{4-\gamma}{2}}\bigg \}dv \bigg)
\end{eqnarray*} 
 The point is now to take the supremum in $\tilde s\in [s,t]$ in the above equation. This yields: 
\begin{eqnarray*} 
&&\sup_{\tilde s\in [t,s]}|(\btheta_{\tilde s,t}(\x)-\btheta_{\tilde s,t}(\x'))_{n-1}| 
\nonumber \\
&\le& C \exp(C (s-t)^{\frac \gamma 2} )
\Big(|(\x-\x')_{n-1}| + (s-t)^{n-\frac 32} 
+\int_{t}^s 
\big  |\big (\btheta_{v,t}(\x)-\btheta_{v,t}(\x')) \big )_{n-2} \big| dv
\nonumber \\
&&
+ |(\x-\x')_n|^{\frac{2(n-1)-3+\gamma}{2n-1}} (s-t) + (s-t)^{1+\frac{2(n-1)-3+\gamma}{2}}
\nonumber \\
&&
+
 \Big (\int_t^s  
\sup_{w\in [t,v]} |\big (\btheta_{w,t}(\x)-\btheta_{w,t}(\x')) \big )_{n-1}\big| (v-t)^{1-\frac{4-\gamma}{2}} dv \Big )  \bigg)
 .
\end{eqnarray*}
We get then by  Gr\"onwall's Lemma:
\begin{eqnarray}
&&|(\btheta_{s,t}(\x)-\btheta_{s,t}(\x'))_{n-1}|
\nonumber \\
&\le& C \exp(C (s-t)^{\frac \gamma 2} )
\nonumber \\
&&\times \Big(|(\x-\x')_{n-1}| + (s-t)^{n-\frac 32} 
+\int_{t}^s 
\big  |\big (\btheta_{v,t}(\x)-\btheta_{v,t}(\x')) \big )_{n-2} \big| dv
+ |(\x-\x')_n|^{\frac{2(n-1)-3+\gamma}{2n-1}} (s-t) 
 \Big)\nonumber \\ 
&\le& C \exp(C (s-t)^{\frac \gamma 2} )
\Big(|(\x-\x')_{n-1}| + (s-t)^{n-\frac 32} 
+\int_{t}^s 
\big  |\big (\btheta_{v,t}(\x)-\btheta_{v,t}(\x')) \big )_{n-2} \big| dv
+ |(\x-\x')_n|^{\frac{2n-3}{2n-1}}  
 \Big)\label{I_N_1_2},
 \nonumber \\
\end{eqnarray}
using again the Young inequality $ |(\x-\x')_n|^{\frac{2(n-1)-3+\gamma}{2n-1}} (s-t) \le C((s-t)^{n-\frac 32}+|(\x-\x'	)_n|^{\frac{2(n-1)-3+\gamma}{2n-1}\frac{2n-3}{2n-5}})$ for the last identity, recalling as well that $\d(\x,\x')\le 1$, and therefore $\textcolor{black}{|(\x-\x')_n|^{\frac{2(n-1)-3+\gamma}{2n-5}}}
\le |(\x-\x')_n| $, for the last identity. 
The purpose of \eqref{I_N_1_2} is that each entry of the difference of the starting points appears at its intrinsic scale for the homogeneous distance $\d$.

Plugging the above inequality into \eqref{ineq_theta_x_x_n} we derive:
\begin{eqnarray*}
&&|(\btheta_{s,t}(\x)-\btheta_{s,t}(\x'))_n|
\nonumber \\
\!\!&\!\!\leq\!\!&\!\! C \exp \Big(C (s-t)^{\frac \gamma 2} \Big)
\bigg(|(\x-\x')_n| + (s-t)^{n-\frac 12} + |(\x-\x')_{n-1}| (s-t) 
\nonumber \\
&&+ 
|(\x-\x')_n|^{\frac{2n-3}{2n-1}} (s-t) + \int_{t}^s \! \int_{t}^v 
\big  |\big (\btheta_{w,t}(\x)-\btheta_{w,t}(\x')) \big )_{n-2} \big| dw dv \bigg)\nonumber\\
\!\!&\!\!\leq\!\!&\!\! C \exp \Big(C (s-t)^{\frac \gamma 2} \Big)
\Big(|(\x-\x')_n| + (s-t)^{n-\frac 12} + |(\x-\x')_{n-1}|^{\frac{2n-1}{2n-3}}  
 + \int_{t}^s \! \int_{t}^v 
\big  |\big (\btheta_{w,t}(\x)-\btheta_{w,t}(\x')) \big )_{n-2} \big| dw dv \Big),
\end{eqnarray*}
using again the Young inequalities $|(\x-\x')_n|^{\frac{2n-3}{2n-1}}(s-t)\le C(|(\x-\x')_n|+(s-t)^{n-\frac 12}) $ and $|(\x-\x')_{n-1}|(s-t)\le C\big(|(\x-\x')_{n-1}|^{\frac{2n-1}{2n-3}}+(s-t)^{n-\frac 12}\big) $ for the last bound.
Iterating these computations, we obtain:
\begin{equation}\label{CTR_N}
|(\btheta_{s,t}(\x)-\btheta_{s,t}(\x'))_n|
\le C
\Big ( (s-t)^{n-\frac 12}+\sum_{j=2}^n |(\x-\x')_j|^{\frac{2n-1}{2j-1}}
+ \int_ t^{v_n=s}\!\!\! dv_{n-1} \hdots \int_t^{v_{2}} \!\!\! dv_{1} \big| \big(\btheta_{v_1,t}(\x')-\btheta_{v_1,t}(\x)\big)_{1}\big|\Big).
\end{equation}
Similarly, for $i\in \leftB 2,n\rightB  $ we derive:
\begin{equation}\label{CTR_I}
|(\btheta_{s,t}(\x)-\btheta_{s,t}(\x'))_i|
\le C
\Big ( (s-t)^{i-\frac 12}+\sum_{j=2}^n |(\x-\x')_j|^{\frac{2i-1}{2j-1}}
+ \int_ t^{v_i=s}\!\!\! dv_{i-1} \hdots \int_t^{v_{2}} \!\!\! dv_{1} \big| \big(\btheta_{v_1,t}(\x')-\btheta_{v_1,t}(\x)\big)_{1}\big|\Big).
\end{equation}
\begin{REM}\label{EXT_TEMP_GR}
We importantly point out that equations \eqref{CTR_N} and \eqref{CTR_I} actually hold not only for the fixed time $s$ but also for any $v\in [t,T] $. 
\end{REM}

The term for $i=1 $ is treated slightly differently. Namely, for any $\tilde s\in [t,s] $, write:
\begin{equation*}
\label{CTR_1}
|(\btheta_{\tilde s,t}(\x)-\btheta_{\tilde s,t}(\x'))_1|
\le |(\x-\x')_1|+C \sum_{j=1}^n\int_t^{\tilde s} |(\btheta_{v,t}(\x)-\btheta_{v,t}(\x'))_j|^{\frac{\gamma}{2j-1}}dv,
\end{equation*}
which in turn implies\textcolor{black}{, using \eqref{CTR_I} and Remark \ref{EXT_TEMP_GR}}, 
\begin{eqnarray*}
&&\sup_{\tilde s\in [t,s]}|(\btheta_{\tilde s,t}(\x)-\btheta_{\tilde s,t}(\x'))_1|\nonumber\\
&\le& |(\x-\x')_1|+C\Big( (s-t)\big(\sup_{v\in [t,s]}|(\btheta_{\textcolor{black}{v},t}(\x)-\btheta_{\textcolor{black}{v},t}(\x'))_1|\big)^\gamma 
+\sum_{j=2}^n\int_t^s |(\btheta_{v,t}(\x)-\btheta_{v,t}(\x'))_j|^{\frac{\gamma}{2j-1}}dv\Big)\nonumber\\
&\le &|(\x-\x')_1|+C\bigg( (s-t)\big(\sup_{v\in [t,s]}|(\btheta_{\textcolor{black}{v},t}(\x)-\btheta_{\textcolor{black}{v},t}(\x'))_1|\big)^\gamma \nonumber\\
&&+\sum_{j=2}^n C_j (s-t) \Big( (s-t)^{j-\frac 12}+\sum_{k=2}^n |(\x-\x')_k|^{\frac{2j-1}{2k-1}}+(s-t)^{j-1} \sup_{v\in [t,s]}|(\btheta_{v,t}(\x)-\btheta_{v,t}(\x'))_1|\Big)^{\frac{\gamma}{2j-1}}\bigg)\nonumber\\
&\le &|(\x-\x')_1|+C\bigg( (s-t)\big(\sup_{v\in [t,s]}|(\btheta_{\textcolor{black}{v},t}(\x)-\btheta_{\textcolor{black}{v},t}(\x'))_1|\big)^\gamma \nonumber\\
&&+\sum_{j=2}^n C_j (s-t) \Big( (s-t)^{\frac \gamma 2}+\sum_{k=2}^n |(\x-\x')_k|^{\frac{\gamma}{2k-1}}+(s-t)^{(j-1)\frac{\gamma}{2j-1}} \sup_{v\in [t,s]}|(\btheta_{v,t}(\x)-\btheta_{v,t}(\x'))_1|^{\frac{\gamma}{2j-1}}\Big)\bigg),
\end{eqnarray*}
using as well convexity inequalities for the last bound. We now write,
\begin{eqnarray}
&&\sup_{\tilde s\in [t,s]}|(\btheta_{\tilde s,t}(\x)-\btheta_{\tilde s,t}(\x'))_1|\nonumber\\
&\le &C\bigg(|(\x-\x')_1|+(s-t)^{1+\frac \gamma 2}+(s-t)\sum_{k=2}^n |(\x-\x')_k|^{\frac{\gamma}{2k-1}} \nonumber\\
&&+\sum_{j=1}^n (s-t)^{1+(j-1)\frac{\gamma}{2j-1}} \sup_{v\in [t,s]}|(\btheta_{v,t}(\x)-\btheta_{v,t}(\x'))_1|^{\frac{\gamma}{2j-1}}\bigg)\nonumber\\
&\le & C(|(\x-\x')_1|+(s-t)+\sum_{k=2}^n |(\x-\x')_k|^{\frac{1}{2k-1}})\label{THE_CTR_1_AFTER_YOUNG}
\end{eqnarray}
recalling that $(s-t)\le T$ is small, and using again Young inequalities for the last bound. Namely,
 \begin{equation*}
 (s-t)^{1+(j-1)\frac{\gamma}{2j-1}} \sup_{v\in [t,s]}|(\btheta_{v,t}(\x)-\btheta_{v,t}(\x'))_1|^{\frac{\gamma}{2j-1}}\le C(s-t)\Big(1+\sup_{v\in [t,s]}|(\btheta_{v,t}(\x)-\btheta_{v,t}(\x'))_1|\Big) ,
 \end{equation*}
 \textcolor{black}{and}
 \begin{equation*}
 (s-t)|(\x-\x')_k|^{\frac{\gamma}{2k-1}}\le C((s-t)^{\frac{1}{1-\gamma}}+|(\x-\x')_k|^{\frac{1}{2k-1}}).
  \end{equation*}
We eventually derive from \eqref{THE_CTR_1_AFTER_YOUNG} that:
\begin{equation*}
\sup_{v\in [t,s]}|(\btheta_{v,t}(\x)-\btheta_{v,t}(\x'))_1|\le C\big((s-t)^{\frac 12}+\d(\x,\x')\big),
\end{equation*}
which gives the stated bound for $i=1$. It now remains to plug this control into \eqref{CTR_I}. We obtain for each $i\in \leftB 2,n \rightB  $:
\begin{eqnarray*}
|(\btheta_{s,t}(\x)-\btheta_{s,t}(\x'))_i|
&\le& C
\big( (s-t)^{i-\frac 12}+\d^{2i-1}(\x,\x')+(s-t)^{i-1}\sup_{v\in [t,s]}|(\btheta_{v,t}(\x)-\btheta_{v,t}(\x'))_1|\big)
\nonumber \\
&\le& C\Big( (s-t)^{i-\frac 12}+\d^{2i-1}(\x,\x')+(s-t)^{i-1}\big( (s-t)^{\frac 12}+\d(\x,\x')\big)\Big)\\
&\le& C\big( (s-t)^{i-\frac 12}+\d^{2i-1}(\x,\x')\big),
\end{eqnarray*}
using again the Young inequality to derive that $(s-t)^{i-1}\d(\x,\x') \le C\big((s-t)^{i-\frac 12}+\d^{2i-1}(\x,\x') \big) $. The proof is complete.
\end{proof}

Again, Lemma \ref{lem_theta_theta} is  a direct consequence of the previous Lemma \ref{lem_theta_theta_bis} and Young/convexity inequalities. 
\\

We are now in position to prove the sensitivity results for the linearized system w.r.t. the freezing parameter.
\subsubsection{Sensitivity  results for  the mean}
\label{sec:ProofLemma_flow}
\begin{proof}[Proof of the Technical Lemma \ref{Lemme_d_theta_theta_x_x}]
We assume w.l.o.g. that $\d(\x,\x')\le 1$. The idea of the proof is to separate the term to control into two contributions.
Namely, we write:
\begin{equation}\label{decomp_tilde_m_theta}
\m_{s,t}^{\x}(\x')-  \btheta_{s,t}(\x')=[ \m_{s,t}^{\x}(\x')-  \btheta_{s,t}(\x)]+   [\btheta_{s,t}(\x)-  \btheta_{s,t}(\x')].
\end{equation}
The definition of the proxy \eqref{FROZ_MOL_FOR_NO_SUPER_SCRIPTS} yields that the mean value of $\tilde  \X^{m, \bxi}_v$, $ \m_{v,t}^{\bxi}$ is s.t.
\begin{equation}
\m_{s,t}^{\x}(\x')-  \btheta_{s,t}(\x)= \x'- \x+ \int_ t^s dv
 D\gF(v,\btheta_{v,t}(\x))[\m_{v,t}^\x(\x')-\btheta_ {v,t}(\x)].
\end{equation}
The sub-triangular structure of $D\gF$ yields that  for each $i \in \leftB2,n\rightB$:
\begin{equation*}
\big (\m_{s,t}^{\x}(\x')-  \btheta_{s,t}(\x) \big )_i= \x'_i- \x_i +\int_ t^s dv
 D_{i-1}\gF_i(v,\btheta_{v,t}(\x))[\m_{v,t}^\x(\x')_{i-1}-\btheta_ {v,t}(\x)_{i-1}]
.
\end{equation*}
Also, since  $\m_{v,t}^{\x}(\x')_1= \x_1'+\int_t^s \gF_1(v,\btheta_{v,t}(\x)) dv$, so that $[\m_{v,t}^\x(\x')_{1}-\btheta_ {v,t}(\x)_{1}]=\x'_1-\x_1 $,  we  then obtain by iteration that:
\begin{equation*}
\big (\m_{s,t}^{\x}(\x')-  \btheta_{s,t}(\x) \big )_i = \x'_{i}-\x_{i}
 + \sum_{k=2}^i \Big [ \int_ t^{v_i=s}\!\!\! dv_{i-1} \hdots \int_t^{v_{k}} \!\!\! dv_{k-1} \prod_{j=k}^{i}   D_{j-1}\gF_{j}(v_j,\btheta_{v_j,t}(\x)) \Big ] [\x'_{k-1}-\x_{k-1}],
\end{equation*}
with the convention that for $i=1$, $ \sum_{k=2}^i=0$. From the above control, equation \eqref{decomp_tilde_m_theta} and the dynamics of the flow, recalling that the starting points are the same, so that the contributions involving differences of the spatial points or flows only appear in iterated time integrals, we derive: 
\begin{eqnarray*}
&&|\big (\m_{s,t}^{\x}(\x')-  \btheta_{s,t}(\x') \big )_i|
\nonumber \\
&\le &
  \bigg |\sum_{k=2}^i \Big [ \int_ t^{v_i=s}\!\!\! dv_{i-1} \hdots \int_t^{v_{k}} \!\!\! dv_{k-1} \prod_{j=k}^{i}   D_{j-1}\gF_{j}(v_j,\btheta_{v_j,t}(\x)) \Big ] [\x'_{k-1}-\x_{k-1}] \bigg |
\nonumber \\
&&+
 \int_{t}^s  |\gF_i(v,\btheta_{v,t}(\x))- \gF_i (v,\btheta_{v,t}(\x'))| dv
 \\
&\leq& C \Big(\sum_{k=2}^{i-1} (s-t)^{i-k} |\x_k-\x_k'|+ \int_{t}^s   \Big(\sum_{j=i}^n 
\big  |\big (\btheta_{v,t}(\x)-\btheta_{v,t}(\x') \big )_{j} \big|^{\frac{2i-3+\gamma}{2j-1}}
+\big  |\big (\btheta_{v,t}(\x)-\btheta_{v,t}(\x') \big )_{i-1} \big| \Big)
dv\Big). \nonumber 
\end{eqnarray*}
From the previous Lemma \ref{lem_theta_theta_bis}, we thus obtain:
\begin{eqnarray}
|\big (\m_{s,t}^{\x}(\x')-  \btheta_{s,t}(\x') \big )_i|
&\le& 
  C \bigg(\sum_{k=2}^{i-1} (s-t)^{i-k} |\x_k-\x_k'|
+ 
(s-t)^{\frac{2i-3+\gamma}{2}+1}
\nonumber \\
&&+ \d^{2i-3+\gamma}(\x,\x')(s-t)+ \big((s-t)^{(i-1)-\frac 12} +\d^{2(i-1)-1}(\x,\x')\big)(s-t) 
\bigg).\label{PREAL_SPLENDID_HOMOGENEOUS_SMOOTHING_SCALED_EFFECT_ou_pas}
\end{eqnarray}
In particular, for $s=t_0=t+c_0\d^2(\x,\x')$ with $c_0<1$, the previous equation yields:
\begin{equation*}
|\big (\m_{t_0,t}^{\x}(\x')-  \btheta_{t_0,t}(\x') \big )_i|
\le 
  C   \Big( c_0  \d^{2i-1}(\x,\x')  
  + (c_0^{i-\frac 12+\frac \gamma 2}+c_0)\d^{2i-1+\gamma}(\x,\x')+(c_0^{i- \frac 12}+c_0)\d^{2i-1}(\x,\x') \Big ).
\end{equation*}
So, after summing and by convexity inequalities, for $\d(\x,\x') \leq 1$:
\begin{equation*}
\d\big (\m_{t_0,t}^{\x}(\x'),  \btheta_{t_0,t}(\x') \big ) \leq  Cc_0^{\frac 1{2n-1}} \d(\x,\x').
\end{equation*}
\end{proof}

\subsection{Sensitivities for the scaled flows}\label{SCALED_FLOWS_SEC}
For the scaling analysis of Section \ref{scaling} we also need the scaled versions of the previous \textcolor{black}{l}emmas. Recalling the notations introduced therein, i.e. for $\lambda>0 $, $0\le t\le v\le T $, $\btheta_{v,t}^\lambda(\x^\lambda)=\lambda^{\frac 12}\T_{\lambda}^{-1} \btheta_{v,t}(\x^\lambda),\ \x^\lambda:=\lambda^{-\frac 12}\T_\lambda \x $,  
we readily get:
\begin{equation}\label{ECHELLE_MEAN}
\textcolor{black}{\d\big (\lambda^{\frac 12}\T_\lambda^{-1}\m_{v,t}^{\x^\lambda}({\x'}^\lambda), \btheta_{v,t}^\lambda({\x'}^\lambda)\big)=\lambda^{-\frac 12}\d\big(\m_{v,t}^{\x^\lambda}({\x'}^\lambda),\btheta_{v,t}({\x'}^\lambda)\big)}.
\end{equation}
Now, the discontinuity term leads to consider $v=t+\textcolor{black}{c_0}\lambda \d^2(\x,\x') =t+\textcolor{black}{c_0} \d^2(\x^\lambda,{\x'}^\lambda)$. So, from \eqref{ECHELLE_MEAN}, we can readily apply
Lemma \ref{Lemme_d_theta_theta_x_x} to the quantity $\d\big(\m_{v,t}^{\x^\lambda}({\x'}^\lambda),\btheta_{v,t}({\x'}^\lambda)\big) $ for the spatial points $\x^\lambda,{\x'}^\lambda $ and the corresponding critical time. This precisely yields $\d\big(\m_{v,t}^{\x^\lambda}({\x'}^\lambda),\btheta_{v,t}({\x'}^\lambda)\big)\le Cc_0^{\frac{1}{2n-1}}\d(\x^\lambda,{\x'}^\lambda)=Cc_0^{\frac{1}{2n-1}}\lambda^{\frac 12}\d(\x,{\x'}) $ which plugged into \eqref{ECHELLE_MEAN} finally leads to:
\begin{equation}\label{SCALED_DIFF_MEANS}
\d\big (\lambda^{\frac 12}\T_\lambda^{-1}\m_{v,t}^{\x^\lambda}({\x'}^\lambda),  \btheta_{v,t}^\lambda({\x'}^\lambda) \big ) \leq  Cc_0^{\frac 1{2n-1}} \d(\x,\x').
\end{equation} 
In other words, Lemma \ref{Lemme_d_theta_theta_x_x} is invariant for the scaled flows.
\mysection{Sensitivity results for the resolvent and covariance}
\label{sec_resolv_cov}
\subsection{Sensitivity Lemma for the \textcolor{black}{R}esolvent}
\label{proof_diff_R_K}
\begin{lem}[Controls of the Sensitivities for the Resolven\textcolor{black}{t}]
\label{LEM_SENSI_TO_CONCLUDE}
There exists $\tilde C$
s.t. for all $0 \leq t \leq s \leq T$, $(\x,\x') \in (\R^{nd})^2$, 
the following control holds. For all $1\le j<i\le n $, with the notation of \eqref{DEF_NORME_HOLDER_HOMO}:
\begin{eqnarray*}
\big|\big[\tilde \gR^{(\tau,\x)}(t,s)-\tilde \gR^{(\tau,\x')}(t,s)\big]_{i,j} \big| &\le&  \tilde C (s-t)^{i-j}  
\Big(\sum_{k=2}^n  
\|\gF_k\|_{L^\infty( C_{\d,\gH}^{2k-3+\gamma})}\Big)
\big( (s-t)^{\frac \gamma 2}+ \d^{\gamma} (\x,\x') \big)\\
& \le& \Lambda (s-t)^{i-j}\big( (s-t)^{\frac \gamma 2}+ \d^{\gamma} (\x,\x') \big).
\end{eqnarray*}
\end{lem}
\begin{proof}[Proof of Lemma \ref{LEM_SENSI_TO_CONCLUDE}]
\textcolor{black}{From the scaling properties of the resolvent, see e.g. the proof of Proposition \ref{THE_PROP} or Lemma 6.2 in \cite{meno:17}, we have that:
\begin{equation}
\label{scaled_resolvent}
 \tilde \gR^{\x}(s,t)=\T_{s-t} \widehat{\tilde  \gR}_1^{s,t,\x} \T_{s-t}^{-1} ,\ \tilde \gR^{\x'}(s,t)=\T_{s-t} \widehat{\tilde  \gR}_1^{s,t,\x'} \T_{s-t}^{-1} ,
 \end{equation}
where $\widehat{\tilde  \gR}_1^{s,t,\x},\ \widehat{\tilde  \gR}_1^{s,t,\x'}$ are non-degenerate bounded matrices. 
We define then:
\begin{equation}
\Delta {\widehat {\tilde \gR}}_1^{s,t,\x,\x'}:= \widehat{\tilde  \gR}_1^{s,t,\x} -\widehat{\tilde  \gR}_1^{s,t,\x'} .
\label{BIG_DECOUPAGE_R}
\end{equation}
Hence, from \eqref{scaled_resolvent} and the definitions in \eqref{BIG_DECOUPAGE_R}:}
\begin{eqnarray*}
|\Delta {\widehat {\tilde \gR}}_1^{s,t,\x,\x'}|&=&|\widehat{\tilde \gR}_1^{s,t,\x}-\widehat{\tilde \gR}_1^{s,t,\x'}| 
=|\T_{s-t}^{-1}  (\tilde \gR^{\x}-\tilde \gR^{\x'})(s,t) \T_{s-t}|
\\
&=&\Big|\T_{s-t}^{-1}\int_t^{s} \Big(D\gF(v,\btheta_{v,t}(\x))\tilde \gR^{\x}(v,t)-D\gF(v,\btheta_{v,t}(\x'))\tilde  \gR^{\x'}(v,t)\Big) dv\T_{s-t}\Big|\\
&\le& \int_{t}^{s} |\T_{s-t}^{-1}D\gF(v,\btheta_{v,t}(\x))\T_{s-t} | |\T_{s-t}^{-1} ({\tilde \gR}^{\x}-\tilde \gR^{\x'})(v,t)\T_{s-t}| dv\\
&&  +\int_{t}^{s} \Big| \T_{s-t}^{-1}\Big(D\gF(v,\btheta_{v,t}(\x))-D\gF(v,\btheta_{v,t}(\x'))\Big)\T_{s-t} \Big| |\T_{s-t}^{-1}\tilde \gR^{\x'}(v,t)\T_{s-t}|dv\\
\nonumber \\
&=&\textcolor{black}{ 
\int_{t}^{s} |\T_{s-t}^{-1}D\gF(v,\btheta_{v,t}(\x))\T_{s-t} | |\T_{s-t}^{-1} \T_{v-t} \Delta {\widehat {\tilde \gR}}_1^{v,t,\x,\x'}\T_{v-t}^{-1}\T_{s-t}| dv}\\
&&  \textcolor{black}{+\int_{t}^{s} \Big| \T_{s-t}^{-1}\Big(D\gF(v,\btheta_{v,t}(\x))-D\gF(v,\btheta_{v,t}(\x'))\Big)\T_{s-t} \Big| |\T_{s-t}^{-1}\tilde \gR^{\x'}(v,t)\T_{s-t}|dv.}
\end{eqnarray*}
\textcolor{black}{Using the Gr\"onwall's Lemma and the structure of the resolvent, we get:
\begin{equation*}
|\Delta {\widehat {\tilde \gR}}_1^{s,t,\x,\x'}|
\le C\int_{t}^{s}  {(s-t)}^{-1}|D\gF(v,\btheta_{v,t}(\x))-D\gF(v,\btheta_{v,t}(\x')| dv.
\end{equation*}
}
Pay attention that we only know from our smoothness assumption \A{S} that for all $i\in \leftB 2,n\rightB$, 
$\z^{i:n}=(\z_i,\cdots,\z_n)\in \R^{(n-i+1)d} $, $\z_{i-1}\mapsto D_{\x_{i-1}}\gF_i(\z_{i-1},\z^{i:n}) $ is $C^{\frac{\gamma}{2(i-1)-1}}(\R^d,\R^d\otimes \R^d) $-H\"older continuous for $\eta>0$. Hence, we proceed carefully like in \cite{chau:meno:17} and we obtain, from the above bound, that
\begin{eqnarray}\label{Diff_K_R}
|\Delta {\widehat {\tilde \gR}}_1^{s,t,\x,\x'}|
&\le & C\int_{t}^{s}  {(s-t)}^{-1}\sum_{i=2}^n\Big(
|D_{i-1}\gF_i(v,\btheta_{v,t}(\x))-D_{i-1}\gF_i(v,\btheta_{v,t}(\x)_{i-1},(\btheta_{v,t}(\x'))^{i:n})| \nonumber \\
&&+\|(D_{\x_{i-1}} \gF_i)_{i-1}\|_{L^\infty(C_\d^{\gamma})}|(\btheta_{v,t}(\x)-\btheta_{v,t}(\x'))_{i-1}|^{\eta_i}\Big) dv
\nonumber \\
&=:&(R_1+R_2)(s,t,\x,\x'),
\end{eqnarray} 
where $\eta_i:= \frac{\gamma}{2(i-1)-1}$ and the notation $(D_{\x_{i-1}} \gF_i)_{i-1}$ indicates that $D_{\x_{i-1}} \gF_i$ is viewed as a function of its variable $(i-1) $ and the supremum is taken over the other ones.
From Lemma \ref{lem_d_theta} and the definition of $\d$ (see also Lemma \ref{lem_theta_theta_bis}), we readily get
\begin{eqnarray}\label{Diff_K_R2}
 |R_2(s,t,\x,\y)|
&\le&   C \|D \gF\|_{L^\infty(C_\d^{\gamma})} \sum_{i=2}^n \int_{t}^{s}  {(s-t)}^{-1} \Big( (v-t)^{(i-1-\frac 12)}+ \d^{2(i-1)-1}(\x,\x')\Big)^{\eta_i}
 dv
\nonumber \\
&
\le  & C \|D \gF\|_{L^\infty(C_\d^{\gamma})} \big( (s-t)^{\frac \gamma 2}+ \d^{\gamma }(\x,\x')  \big )
,
\end{eqnarray}
denoting with a slight abuse of notation $\|D \gF\|_{L^\infty(C_\d^{\gamma})}:=\sum_{i=2}^n \|(D_{\x_{i-1}} \gF_i)_{i-1}\|_{L^\infty(C_\d^{\gamma})} $.
To control the difference of the gradients terms in $R_1(s,t,\x,\x')$ in \eqref{Diff_K_R}, we need the following result whose proof is postponed to Appendix \ref{sec:ProofLemma_Taylor_reverse}.
\begin{lem}[Reverse Taylor expansion]\label{Lemme_Taylor_reverse}
There is a constant $C>0$ s.t, for all $(\z,\z') \in \R^{nd}\times \R^{nd}$, 
$v\in [0,T]$ :
\begin{equation*}
  \big | D_{\x_{i-1}} \gF_{i}(v, \z)  -D_{\x_{i-1}} \gF_{i}(v, \z' ) \big |
\leq
C 
\|\gF_i\|_{L^\infty( C_{\d,\gH}^{
2i-3+\gamma})} \d(\z,\z')^\gamma,
\end{equation*}
with the notations of equation \eqref{DEF_NORME_HOLDER_HOMO}.
\end{lem}
From the reverse Taylor expansion of Lemma \ref{Lemme_Taylor_reverse} and the definition in  \eqref{Diff_K_R}, we obtain:
\begin{eqnarray}\label{Diff_K_R1}
|R_1(s,t,\x,\y)|
&\le&  C (s-t)^{-1} \int_t^s dv \sum_{i=2}^n \|\gF_i\|_{L^\infty( C_{\d,\gH}^{2i-3+\gamma})} \d^\gamma( \btheta_{v,t}(\x),\btheta_{v,t}(\x'))\nonumber\\
&\le & C  \sum_{i=2}^n \|\gF_i\|_{L^\infty( C_{\d,\gH}^{2i-3+\gamma})} \big((s-t)^{\frac \gamma 2}+\d^\gamma( \x,\x')\big),
\end{eqnarray} 
using again  Lemma \ref{lem_d_theta} for the last inequality.
Gathering \eqref{Diff_K_R2}, \eqref{Diff_K_R1} into \eqref{Diff_K_R} and  recalling the definition in \eqref{DEF_NORME_HOLDER_HOMO},  we obtain:
\begin{equation}\label{BORNE_RES_SCALEE}
|\Delta {\widehat {\tilde \gR}}_1^{s,t,\x,\x'}| \le C\big(\sum_{i=2}^n \|\gF_i\|_{L^\infty( C_{\d,\gH}^{2i-3+\gamma})}\big)\big((s-t)^{\frac \gamma 2}+\d^\gamma( \x,\x')\big).
\end{equation}
The result follows from the previous bound, the definition in \eqref{BIG_DECOUPAGE_R} and the scalings of equation  \eqref{scaled_resolvent}.
\end{proof}


\subsection{ Sensitivity Lemma for the covariances}
\subsubsection{Proof of Lemma \ref{SENS_COV}}\label{proof_lemma_SENCOV}
Let us first explicitly write the covariance matrices 
\begin{equation}\label{equ_esplicite_K}
\tilde \K ^{\bxi}_{s,t}:=\int_t^s {\tilde  \gR}^{\bxi}(s,\textcolor{black}{v}) Ba(\textcolor{black}{v},\btheta_{\textcolor{black}{v},t}(\bxi))B^*{\tilde \gR}^{\bxi}(s,\textcolor{black}{v})^* d\textcolor{black}{v}.
 \end{equation}

So\textcolor{black}{,} we have to control the term
\begin{eqnarray}
\tilde \K^\bxi_{s,t}-\tilde \K^{\bxi'}_{s,t} &= :&\Delta_1^{\bxi,\bxi'}(s,t)+\Delta_2^{\bxi,\bxi'}(s,t),\notag\\
\Delta_1^{\bxi,\bxi'}(s,t)&:=& \int_{t}^{s} d\textcolor{black}{v} {\tilde  \gR}^{\bxi}(s,\textcolor{black}{v}) B \Delta_{\textcolor{black}{v}} a(\btheta_{\textcolor{black}{v},t}(\bxi), \btheta_{\textcolor{black}{v},t}(\bxi'))
B^* {\tilde  \gR}^{\bxi}(s,\textcolor{black}{v}) ^*,\notag\\
 \Delta_{\textcolor{black}{v}} a(\btheta_{\textcolor{black}{v},t}(\bxi), \btheta_{\textcolor{black}{v},t}(\bxi'))&:=&a(\textcolor{black}{v},\btheta_{\textcolor{black}{v},t}(\bxi))-a(\textcolor{black}{v},\btheta_{\textcolor{black}{v},t}(\bxi')),\notag\\
\Delta_2^{\bxi,\bxi'}(s,t)&:=& \int_{t}^{s} d\textcolor{black}{v} \Delta {\tilde { \gR}}^{\bxi,\bxi'}(s,\textcolor{black}{v}) B a(\textcolor{black}{v},\btheta_{\textcolor{black}{v},t}(\bxi'))B^*   {\tilde  \gR}^{\bxi}(s,\textcolor{black}{v})^*
\notag\\
&& 
+ \int_{t}^{s} d\textcolor{black}{v}  {\tilde  \gR}^{\bxi'}(s,\textcolor{black}{v})B a(\textcolor{black}{v},\btheta_{\textcolor{black}{v},t}(\bxi'))B^*
\Delta {\tilde { \gR}}^{\bxi,\bxi'}(s,\textcolor{black}{v})^*,\notag\\
\Delta {\tilde { \gR}}^{\bxi,\bxi'}(s,\textcolor{black}{v})&=&{\tilde  \gR}^{\bxi}(s,\textcolor{black}{v}) -{\tilde  \gR}^{\bxi'}(s,\textcolor{black}{v}).
\label{BIG_DECOUPAGE}
\end{eqnarray}
Hence, from the scalings of \eqref{scaled_resolvent} and the definitions in \eqref{BIG_DECOUPAGE}, for all $1\le j\le i\le n $:
\begin{eqnarray*}
|[\Delta_1^{\bxi,\bxi'}(s,t)]_{i,j}|
&\le& C(s-t)^{-2}\int_{t}^{s} |[\T_{s-t }\textcolor{black}{B}\Delta_{\textcolor{black}{v}} a(\btheta_{\textcolor{black}{v},t}(\bxi), \btheta_{\textcolor{black}{v},t}(\bxi'))\textcolor{black}{B^*} \T_{s-t}]_{i,j}| d\textcolor{black}{v}\\
&\le& \Lambda(s-t)^{i+j-2} \int_{t}^{s} \d(\btheta_{\textcolor{black}{v},t}(\bxi),\btheta_{\textcolor{black}{v},t}(\bxi'))^\gamma d\textcolor{black}{v} .
\end{eqnarray*}
We deduce by Lemma \ref{lem_theta_theta}, that
\begin{eqnarray}\label{CTR_DELTA_1}
|[\Delta_1^{\bxi,\bxi'}(s,t)]_{i,j}|\le \Lambda (s-t)^{i+j-1}\big( (s-t)^{\frac {\gamma }2}+ \d^{\gamma}(\bxi,\bxi') \big ).
\end{eqnarray}
\textcolor{black}{Still from \eqref{scaled_resolvent} and the definitions in \eqref{BIG_DECOUPAGE}}, write now that:
\begin{eqnarray}
\label{CTR_DELTA_2}
|[\Delta_2^{\bxi,\bxi'}(s,t)]_{i,j}|&\le& C(s-t)^{-2}\Bigg(\int_t^{s}\Big( |[\T_{s-t} (\widehat{\tilde \gR}_1^{s,\textcolor{black}{v},\bxi}-\widehat{\tilde \gR}_1^{s,\textcolor{black}{v},\bxi'}) \textcolor{black}{B} a(\textcolor{black}{v},\btheta_{\textcolor{black}{v},t}(\bxi')  )   \textcolor{black}{B^*}
(\widehat {\tilde { \gR}}_1^{s,\textcolor{black}{v},\bxi'})^*\T_{s-t}]_{i,j}|\nonumber\\
&&+  |[\T_{s-t} \widehat{\tilde \gR}_1^{s,\textcolor{black}{v},\bxi'} \textcolor{black}{B} a(\textcolor{black}{v},\btheta_{\textcolor{black}{v},t}(\bxi')  )   \textcolor{black}{B^*}
(\widehat {\tilde { \gR}}_1^{s,\textcolor{black}{v},\bxi'}-\widehat{\tilde \gR}_1^{s,\textcolor{black}{v},\bxi})^*\T_{s-t}]_{i,j}|\Big) d\textcolor{black}{v}\Bigg).
\end{eqnarray}
Thanks to equation \eqref{BORNE_RES_SCALEE} in the proof of Lemma \ref{LEM_SENSI_TO_CONCLUDE}, we thus obtain:
\begin{eqnarray}
\label{CTR_DELTA_22}
|[\Delta_2^{\bxi,\bxi'}(s,t)]_{i,j}| \leq \Lambda (s-t)^{i+j-1}\big( (s-t)^{\frac {\gamma }2}+ \d^{\gamma}(\bxi,\bxi') \big ).
\end{eqnarray}
Gathering \eqref{CTR_DELTA_1} and \eqref{CTR_DELTA_22} in \eqref{BIG_DECOUPAGE} yields:
\begin{equation}
\label{FINAL_CTR_COV}
|[\tilde \K^\bxi_{s,t}]_{i,j}-[\tilde \K^{\bxi'}_{s,t}]_{i,j}|\le \Lambda(s-t)^{i+j-1}\big( (s-t)^{\frac {\gamma }2}+ \d^{\gamma}(\bxi,\bxi') \big ),
\end{equation}
which precisely gives  \eqref{CTR_SENSI_COV} for $i=j=1$ \textcolor{black}{and then concludes the proof of Lemma \ref{SENS_COV}}. \hfill $\square $
\subsubsection{Sensitivities for the scaled covariance.}
In connection with Section \ref{scaling}, we recall the identity in law \eqref{CORRESP_SCALE}, i.e. $\tilde \X_{v}^{\bxi,\lambda}:=\lambda^{1/2}\T_\lambda^{-1} \tilde \X_v^{\bxi^\lambda},\ v\in [t,T] $, which readily gives:
$$\tilde \K_{v,t}^{\bxi,\lambda}:={\rm {Cov}}(\tilde \X_{v}^{\bxi,\lambda})=\lambda \T_\lambda^{-1}\tilde \K_{v,t}^{\bxi^\lambda }\T_{\lambda}^{-1}.$$
In particular, we thus derive from the analysis of the previous paragraph:
\begin{eqnarray*}
\label{SCALED_COV_AND_SENSI}
[\tilde \K_{v,t}^{\bxi,\lambda}]_{1,1}&\le& C\lambda^{-1}(v-t),\notag\\
  \big|[\tilde \K_{v,t}^{\bxi,\lambda}]_{1,1}-[\tilde \K_{v,t}^{{\bxi'},\lambda}]_{1,1}\big|&\le& C\lambda^{-1}(v-t)\Big( \d^\gamma(\bxi^\lambda, {\bxi'}^\lambda)+(v-t)^{\frac \gamma 2}\Big)\le C\lambda^{-1}(v-t)\Big( \lambda^{\frac\gamma 2}\d^\gamma(\bxi, {\bxi'})+(v-t)^{\frac \gamma 2}\Big),\notag
\end{eqnarray*}
recalling that $\bxi^\lambda=\lambda^{-\frac 12}\T_\lambda \bxi $ and using the homogeneity properties of $\d$ for the last inequality. Also, for $v=t+\textcolor{black}{c_0}\lambda \d^2(\x,\x')  $ and taking $\bxi=\x$, $\bxi'=\x' $, the above controls rewrite:
\begin{equation}
\label{SCALED_COV_AND_SENSI_CRITICAL_TIME}
[\tilde \K_{v,t}^{\bxi,\lambda}]_{1,1}\le C \textcolor{black}{c_0} \d^2(\x,\x'),\  \big|[\tilde \K_{v,t}^{\bxi,\lambda}]_{1,1}-[\tilde \K_{v,t}^{{\bxi'},\lambda}]_{1,1}\big|\le C 
\textcolor{black}{c_0}
\lambda^{\frac \gamma 2}\d^{2+\gamma}(\x,\x') .
\end{equation}
%

\textcolor{black}{Note that the \textcolor{black}{sensitivity} of the scaled covariance  yields a contribution of the scaling coefficient in $\textcolor{black}{\lambda^{\frac \gamma 2}}$. Unlike for the control of the scaled mean in \eqref{SCALED_DIFF_MEANS}, where, as previously noticed, we could not exploit the full regularity of $\gF_i$ w.r.t. the $(i -1)^{\rm th}$ variable, \textcolor{black}{we can here precisely take advantage of such a regularity.}
Indeed, this follows from the expression of the covariance \eqref{equ_esplicite_K} which only involves $D_{i-1}{\gF_i}$ so that one can exploit the associated $\frac{\gamma}{2i-3}$-H\"older regulariry  w.r.t. $\x_{i-1}$.}

\subsection{Reverse Taylor formula}
\label{sec:ProofLemma_Taylor_reverse}

\begin{proof}[Proof of Lemma \ref{Lemme_Taylor_reverse}]
We assume here, for the sake of simplicity and without loss of generality, that $d=1$ (scalar case). 
When $d>1$, the proof below can be reproduced componentwise.
Let us decompose the expression around the variables which do/do not transmit the noise.
Namely, we write for any $\delta_i>0$:
\begin{eqnarray} \label{def_mathcal_F}
&& D_{\x_{i-1}} \gF_{i}(v, \z) -D_{\x_{i-1}} \gF_{i}(v, \z' )
\nonumber \\
 &=&
\int_0^1 d \mu 
\{D_{\x_{i-1}} \gF_{i}(v, \z ) -D_{\x_{i-1}} \gF_{i}(v, \z _{i-1}+\mu\d(\z,\z')^{\delta_i},\z^{i:n})  \}
\nonumber \\
&&+\{D_{\x_{i-1}} \gF_{i}(v, \z _{i-1}+\mu\d(\z,\z')^{\delta_i },(\z')^{i:n}) -D_{\x_{i-1}} \gF_{i}(v, \z' )  \}
\nonumber \\
&&+\{D_{\x_{i-1}} \gF_{i}(v, \z _{i-1}+\mu\d(\z,\z')^{\delta_i},\z^{i:n}) -D_{\x_{i-1}} \gF_{i}(v, \z _{i-1}+\mu\d(\z,\z')^{\delta_i},(\z')^{i:n}) \}
\nonumber \\ 
&=:& \sum_{\ell=1}^3 \Delta \gF_{i}^\ell(v,\z,\z') .
\end{eqnarray}
The  first two terms can be dealt directly. From \A{A} we get:
\begin{equation} \label{mathcal_F_1}
|  \Delta \gF_{i}^1(v,\z,\z')  |
\leq 
\|(D_{\x_{i-1}}\gF_i)_{i-1}\|_{L^\infty(C_{\d}^{\gamma})}
\d(\z,\z')^{\delta_i \frac \gamma{2(i-1)-1}}. 
\end{equation}
Similarly,
\begin{eqnarray} \label{mathcal_F_2}
 | \Delta \gF_{i}^2(v,\z,\z') | 
&\leq &
\|(D_{\x_{i-1}}\gF_i)_{i-1}\|_{L^\infty(C_{\d}^{\gamma})}
\big ( \d(\z,\z')^{\delta_i \frac \gamma{2(i-1)-1}} + |(\z-\z')_{i-1}|^{\textcolor{black}{\frac \gamma {2(i-1)-1}} }\big )
\nonumber \\
&\leq & C
\|(D_{\x_{i-1}}\gF_i)_{i-1}\|_{L^\infty(C_{\d}^{\gamma})}
\d(\z,\z')^{\delta_i \frac \gamma{2(i-1)-1}}
.
\end{eqnarray}
For $\Delta \gF_{i}^3(t,\z,\z')$, we use an explicit reverse Taylor expansion which yields together with the smoothness assumption of $\gF_i $ in \A{A}:
\begin{eqnarray} \label{mathcal_F_3}
 | \Delta \gF_{i}^3(t,\z,\z') |
&=&\d(\z,\z')^{-\delta_i} 
\Big |\Big [ \gF_i(t, \z _{i-1}+\d(\z,\z')^{\delta_i},\z^{i:n}) - \gF_i(t, \z _{i-1}+\d(\z,\z')^{\delta_i},(\z')^{i:n}) 
\nonumber \\
&&+
\gF_i(t, \textcolor{black}{\z_{i-1}},\z^{i:n}) - \gF_i(t, \z _{i-1},(\z')^{i:n}) 
\Big ] \Big |
\nonumber \\ 
&\leq& 2
 \|\gF_i\|_{L^\infty(C_{\d,\gH}^{2i-3+\gamma
 }
 )}
\d(\z,\z')^{2i-3+\gamma-\delta_i}.
\end{eqnarray}
Taking $\delta_i$ s.t. $\delta_i \frac{\gamma}{2(i-1)-1}=2i-3+\gamma-\delta_i$, which implies that $\delta_i= 2i-3$, gives
in \eqref{mathcal_F_1}, \eqref{mathcal_F_2} and \eqref{mathcal_F_3} a global bound of order $C \|\gF_i\|_{L^\infty(C_{\d,\gH}^{2i-3+\gamma
})}
\d^\gamma(\z,\z') $.
 The result then follows from \eqref{def_mathcal_F}.

\end{proof}

\mysection{Scaling Control of the degenerate part of the perturbative term}
\label{SEC_BESOV_DUAL_FIRST_Scalling}
This section is dedicated to the proof of the scaled version of the key Besov \textcolor{black}{control of} Lemma \ref{LEMME_BESOV_DEG}. We recall that, with the definitions of Section \ref{scaling}, for all multi-index $\vartheta=(\vartheta_1, \hdots,\vartheta_n) \in \R^ {nd}$, $ i\in \leftB 2, n \rightB$, we aim to control the terms 
\begin{equation*} \int_t^T ds \int_{\R^{nd}} 
D^ \vartheta D_{\x_1}^2\tilde p_\lambda ^{\bxi}(t,s,\x,\y) \big \langle \Delta_{i,\gF}^ \lambda (t,s,\btheta_{s,t}(\bxi),\y), D_{\y_i} u^ \lambda(s,\y)  \big \rangle d\y ,
\end{equation*}
 with
\begin{eqnarray*}
&&\sum_{i=2}^n\big \langle \gF_{\lambda,i}(s,\y)-\gF_{\lambda,i}(s,\btheta_{s,t}^\lambda(\bxi))-D_{\x_{i-1}} \gF_{\lambda,i}(s,\btheta^ \lambda_{s,t}(\bxi)) (\y-\btheta_{s,t}^\lambda(\bxi))_{i-1} , D_{\y_i} u^ \lambda(s,\y) \big \rangle\notag\\
&=:&\sum_{i=2}^n\big \langle \Delta_{i,\gF}^ \lambda(t,s,\btheta_{s,t}^\lambda(\bxi),\y),D_{\y_i} u^ \lambda(s,\y) \big \rangle, \label{DEF_TERMES_DEG_SCALING}
\end{eqnarray*}
which appear in equation \eqref{DUALITE_PREAL_BESOV} of the detailed guide for the scaled system. We precisely want to specify how the scaling procedure impacts the constants in equation \eqref{CTR_DIAG_PREAL}. 

This is exactly what equations \eqref{CTR_OFF_DIAG_DEG} and \eqref{SCALE_OFF_DIAG} reflect. Those controls actually follow from the more general following result, which will again be useful for the H\"older norm in Section \ref{HOLDER}.
\label{CTR_DER_SUP_scalling}
\begin{lem}[Scaled Besov Control Lemma] \label{LEMME_BESOV_DEG_scaling}
There exists $\Lambda:=\Lambda(\A{A},T) $ as in Remark \ref{REM_LAMBDA} s.t. for each multi-index $\vartheta=(\vartheta_1, \hdots, \vartheta_n) \in \N^{\textcolor{black}{n}}$:
\begin{eqnarray}\label{ineq_lemme_Besov_scalling}
&&\sum_{i=2}^n\Big|\int_{\R^{nd}} D^\vartheta \tilde p_\lambda^{\bxi}(t,s,\x,\y) \big \langle \Delta_{i,\gF}^ \lambda(t,s,\btheta_{s,t}(\bxi),\y)  ,D_{\y_i}  u(s,\y) \big \rangle
   d\y\Big| \bigg|_{\bxi=\x}\nonumber \\
   & \leq& \Lambda \lambda^ {-1+\sum_{j=1}^n \vartheta_j (j-\frac 12)}\|u^ \lambda\|_{L^\infty(C_{b,\d}^{2+\gamma})}(s-t)^{-\sum_{j=1}^n \vartheta_j (j-\frac 12)+\frac{\gamma}2} \!.
\end{eqnarray}
\end{lem}
With Lemma \ref{LEMME_BESOV_DEG_scaling} at hand, we indeed readily derive \eqref{CTR_OFF_DIAG_DEG} and \eqref{SCALE_OFF_DIAG} taking $\vartheta=(2,0,\cdots,0)+e_k $ (where $e_k$ stands for the $k^{\rm th} $ vector of the orthonormal basis) for each $k\in \leftB 2,n \rightB $ and $\vartheta=(2,0,\cdots,0) $ respectively. Let us now turn to the proof of the Lemma \ref{LEMME_BESOV_DEG_scaling}.

\begin{proof}[Proof of Lemma \ref{LEMME_BESOV_DEG_scaling}]
The analysis of singularities is \textcolor{black}{identical} to the ones in the proof of Lemma \ref{LEMME_BESOV_DEG}.
However, here, we have to track the scalling coefficient $\lambda$ through the identites.
Note carefully, we write the upper-script/sub-script $\lambda$ to mean that we manage the scaled variables.
In particular, we write:
\begin{equation}
\label{DEF_GI_TO_BELONG_TO_BESOV_SPACE_scalling}
\vartheta_{i,(t,\x)}^{\vartheta,\lambda} (s,\y)
:=D_{\y_i}\cdot \Big(D^\vartheta\tilde p_\lambda^{\bxi}(t,s,\x,\y) \otimes \Delta_{i,\gF}^ \lambda(t,s,\btheta_{s,t}(\bxi),\y)
\Big)
=:D_{\y_i}\cdot \Theta_{i,(t,\x)}^{\vartheta,\lambda}(s,\y),
\end{equation}
and
\begin{equation}
\label{DEF_PSI_M_SCALING}
\Psi_{i,(t,\x),(s,\y_{1:i-1},\y_{i+1:n})}^{\vartheta,\lambda}: \y_i\mapsto  D_{\y_i}\cdot \big(\Theta_{i,(t,\x)}^{\vartheta,\lambda}(s,\y)\big).
\end{equation}
With these notations, we have:
\begin{equation}
\sum_{i=2}^n\Big|\int_{\R^{nd}} \!D^\vartheta \tilde p_\lambda^{(t,\bxi)}(t,s,\x,\y) \big \langle \Delta_{i,\gF}^ \lambda(t,s,\btheta_{s,t}(\bxi),\y),  D_{\y_i}  u(s,\y) \big \rangle
   d\y\Big| \bigg|_{\bxi=\x} 
\!\!=\sum_{i=2}^n\Big| \int_{\R^{nd}} \! D_{\y_i}\cdot \big( \Theta_{i,(t,\x)}^{\vartheta,\lambda} (s,\y)\big) u^ \lambda(s,\y) d\y\Big|. \label{FROM_WHERE_TO_USE_BESOV_DUALITY_scaling}
\end{equation}
The point is here again to control, for each $i\in \leftB 2,n \rightB$, the quantity $\|D_{\y_i}\cdot \big(\Theta_{i,(t,\x)}^{\vartheta,\lambda} (s,\y_{1:i-1},\cdot,\y_{i+1:n})\big)\|_{B_{1,1}^{\tilde \alpha_i}},\ \tilde \alpha_i=\frac{2+\gamma}{2i-1} $ with the indicated bounds in the scaling parameter $\lambda$. Accordingly with what can be seen e.g. in \eqref{CONVEX_INEQ_SCALED_DIAG}, the previous analysis of the proof of the (non-scaled) Lemma \ref{LEMME_BESOV_DEG} can be adapted replacing $(s-t) $ therein by $(s-t)/\lambda $ in the computations involving the thermic characterization of Besov spaces. 

We also point out that, w.l.o.g., we assume that $T/\lambda\le 1 $ so that in particular for $0\le t<s\le T $,  $\lambda^{-1}(s-t)\le 1 $. Indeed, the parameter $\lambda $ is meant to be \textit{small} (at least $\lambda\le 1 $) but \textit{macro} as well. From the previous analysis and the statement of Lemma \ref{LEMME_BESOV_DEG_scaling} it can be seen that the \textit{optimal} $\lambda $, i.e. the largest one, actually depends on the H\"older moduli of the coefficients. Hence, the condition $T/\lambda\le 1 $ is, up to a possible modification of $T$, not restrictive.
\\

Let us first introduce some notation:
\begin{equation*}
\hat q_{c,\lambda} (t,s,\x,\y):=\prod_{j=1}^n {\mathcal N}_{c\lambda^ {\frac{2j-1}2}(s-t)^{2j-1}}\big( (\btheta_{s,t}(\x)-\y)_j\big)=\bar p_{c^{-1}}(t,s,\x,\y),
\end{equation*}
where for $\varsigma>0,\ z\in \R^d$, like before ${\mathcal N}_{\varsigma}(z)=\frac{1}{(2\pi \varsigma)^{\frac d2}}\exp\big(-\frac{|z|^2}{2\varsigma}\big) $ is the standard Gaussian density of $\R^d $ with covariance matrix $\varsigma I_d$, and:
\begin{equation}
\label{DEF_HAT_SETMINUS_scalling}
\hat q_{c\setminus i,\lambda}(t,s,\x,(\y_{1:i-1},\y_{i+1:n}))=\prod_{j\in \leftB 1, n\rightB, j\neq i} {\mathcal N}_{c\lambda^ {\frac{2j-1}2}(s-t)^{2j-1}}\big( (\btheta_{s,t}(\x)-\y)_j\big).
\end{equation}
We recall from \eqref{DEF_BETA_I}, that the parameter $\beta_i = \frac{(2i-3)(2i-1)}{2i-3-\gamma}$.
The first contribution of the scaled Besov control is:
\begin{eqnarray*}
&&\int_{[\lambda^{-1}(s-t)]^{\beta_i}}^1 \frac{dv}{v}v^{\frac{\tilde \alpha_i}2 }\|h_v\star \Psi_{i,(t,\x),(s,\y_{1:i-1},\y_{i+1:n})}^{\vartheta, \lambda} \|_{L^1(\R^d,\R)}\\
&\le& \int_{[\lambda^{-1}(s-t)]^{\beta_i}}^1 \frac{dv}{v}v^{\frac{\tilde \alpha_i}2 } \int_{\R^d} dz 
 \Big|\int_{\R^d} D^\vartheta \tilde p_\lambda^{\bxi}(t,s,\x,\y)
\big \langle \Delta_{i,\gF}^ \lambda(t,s,\btheta_{s,t}^ \lambda(\bxi),\y), D_z h_v(z-\y_i) \big \rangle
d\y_i\Big| \bigg |_{\bxi=\x}\\
&\le& \Lambda\int_{[\lambda^{-1}(s-t)]^{\beta_i}}^1 \frac{dv}{v}v^{\frac{\tilde \alpha_i}2 } \int_{\R^d} dz 
 \int_{\R^d} d\y_i \frac{h_{cv}(z-\y_i)}{v^{\frac 12}} \frac{\lambda ^{\sum_{j=1}^n \vartheta_j(j-\frac 12)} \hat q_{c,\lambda}(t,s,\x,\y)}{(s-t)^{\sum_{j=1}^n \vartheta_j(j-\frac 12)}}
 \nonumber \\
 && \times \lambda^{-i+\frac 12} \d^{2i-3+\gamma}(\lambda^ {-1/2} \T_{\lambda}\btheta_{s,t}^ \lambda (\x),\lambda^ {-1/2} \T_{\lambda}\y)\\
&\le& \Lambda\lambda ^{\sum_{j=1}^n \vartheta_j(j-\frac 12)-(i-\frac 12)} \int_{[\lambda^{-1}(s-t)]^{\beta_i}}^1 \frac{dv}{v}v^{\frac{\tilde \alpha_i}2 } \int_{\R^d} dz 
 \int_{\R^d} d\y_i \frac{h_{cv}(z-\y_i)}{v^{\frac 12}}\frac{\hat q_{c, \lambda} (t,s,\x,\y)}{(s-t)^{\sum_{j=1}^n \vartheta_j(j-\frac 12)}} (s-t)^{\frac{2i-3+\gamma}2},
 \end{eqnarray*}
 exploiting \eqref{CONVEX_INEQ_SCALED_DIAG} for the last inequality. Then
  \begin{eqnarray}
&& \int_{[\lambda^{-1}(s-t)]^{\beta_i}}^1 \frac{dv}{v}v^{\frac{\tilde \alpha_i}2 }\|h_v\star \Psi_{i,(t,\x),(s,\y_{1:i-1},\y_{i+1:n})}^{\vartheta, \lambda} \|_{L^1(\R^d,\R)}\notag\\
&\le& \Lambda \lambda ^{\sum_{j=1}^n \vartheta_j(j-\frac 12)-(i-\frac12)} \hat q_{c\setminus i, \lambda}(t,s,\x,(\y_{1:i-1},\y_{i+1:n}))\int_{[\lambda^{-1}(s-t)]^{\beta_i}}^1 dvv^{-\frac 32+\frac{\tilde \alpha_i} 2 }(s-t)^{-\sum_{j=1}^n \vartheta_j(j-\frac 12)+\frac{2i-3+\gamma}2}\notag\\
&\le& \Lambda\lambda ^{\sum_{j=1}^n \vartheta_j(j-\frac 12)-(i-\frac 12)} \hat q_{c\setminus i, \lambda}(t,s,\x,(\y_{1:i-1},\y_{i+1:n})) \lambda^{[\frac 12-\frac{\tilde \alpha_i} 2]\beta_i}(s-t)^{[-\frac 12+\frac{\tilde \alpha_i} 2]\beta_i -\sum_{j=1}^n \vartheta_j(j-\frac 12)+\frac{2i-3+\gamma}2}
\notag\\
&\le& \Lambda\lambda ^{\sum_{j=1}^n \vartheta_j(j-\frac 12)-1} \hat q_{c\setminus i, \lambda}(t,s,\x,(\y_{1:i-1},\y_{i+1:n})) (s-t)^{-\sum_{j=1}^n \vartheta_j(j-\frac 12)+\frac \gamma 2},\label{CTR_BESOV_LAMBDA_HD}
\nonumber \\
\end{eqnarray}
the third inequality is a consequence of Proposition \ref{THE_PROP}, and the last identity comes from the pick of $\beta_i$ which in particular gives $ -(i-\frac 12)+[1-\tilde \alpha_i]\frac{\beta_i}{2}=-1$.

Let us now consider the second contribution of the scaled Besov control, i.e. we take $v\in \big[0,[\lambda^{-1}(s-t)]^{\beta_i}\big] $. Write:
\begin{eqnarray}
&&\int_{\R^d}  h_v(z-\y_i) D_{\y_i} \cdot  \big ( \Theta_{i,(t,\x)}^{\vartheta, \lambda} (s,\y) \big )d\y_i
\nonumber \\
&=&\int_{\R^d} h_v(z-\y_i) D_{\y_i} \cdot  \Big(\Theta_{i,(t,\x)}^{\vartheta, \lambda } (s,\y)- \Theta_{i,(t,\x)}^{\vartheta, \lambda} (s,\y_{1:i-1},z,\y_{i+1:n}) \Big) d\y_i\notag \\
&=& \int_{\R^d}  D^\vartheta \tilde p^{\bxi}_ \lambda (t,s,\x,\y) \big \langle \gF_{\lambda, i}(s,\y)-\gF_{\lambda, i}(s,\y_{1:i-1},z,\y_{i+1:n}) , D_z h_v(z-\y_i) \big \rangle  d\y_i \notag \\
&&+\int_{\R^d} \Big(D^\vartheta \tilde p_\lambda^{\bxi}(t,s,\x,\y)-D^\vartheta\tilde p_\lambda^{\bxi}(t,s,\x,\y_{1:i-1},z,\y_{i+1:n})\Big)\notag \\
&&\times \Big \langle \big (\gF_{\lambda, i}(s,\y_{1:i-1},z,\y_{i+1:n})-\gF_{\lambda, i}(s,\btheta^ \lambda_{s,t}(\bxi)) -D_{\x_{i-1}}\gF_{\lambda,i}(s,\btheta_{s,t}^ \lambda(\bxi))(\y-\btheta_{s,t}^ \lambda(\bxi))_{i-1}\big), D_z h_v(z-\y_i) \Big \rangle d\y_i\notag\\
&=:& \Big({\mathscr T}_{\lambda,1}+{\mathscr T}_{\lambda,2}\Big) \big(v,t,s,\x,(\y_{1:i-1},z,\y_{i+1:n})\big), \label{DECOUP_CAR_THERMIC_scalling}
\nonumber \\
\end{eqnarray}
thanks to the definition in \eqref{DEF_GI_TO_BELONG_TO_BESOV_SPACE_scalling} for the last identity.
\begin{eqnarray}
&&|{\mathscr T}_{\lambda, 1} \big(v,t,s,\x,(\y_{1:i-1},z,\y_{i+1:n})\big)|
\nonumber \\
&\le& \Lambda \int_{\R^d} \frac{h_{cv}(z-\y_i)}{v^{\frac12}} \frac{\lambda^{\sum_{j=1}^n\vartheta_j(j-\frac 12)}}{(s-t)^{\sum_{j=1}^n \vartheta_j(j-\frac 12)}}\hat q_{c,\lambda}(t,s,\x,\y)\lambda^{-i+\frac 12}(\lambda^ {\frac {2i-1}2}|z-\y_i|)^{\frac{2i-3+\gamma}{2i-1}} d\y_i\notag \\
&\le & \Lambda \lambda ^{\sum_{j=1}^n \vartheta_j(j-\frac 12)-(i-\frac 12)+\frac{2i-3+\gamma}{2}} \int_{\R^d} \frac{h_{cv}(z-\y_i)}{v^{\frac{2-\gamma}{4i-2}}} \frac{\hat q_{c,\lambda}(t,s,\x,\y)}{(s-t)^{\sum_{j=1}^n \vartheta_j(j-\frac 12)}} d\y_{i}.
\nonumber \\
\label{CTR_T1_scalling}
&\le& \frac{\Lambda\lambda ^{\sum_{j=1}^n \vartheta_j(j-\frac 12)- 1+\frac \gamma 2
} 
}{v^{\frac{2-\gamma}{4i-2}} (s-t)^{\sum_{j=1}^n \vartheta_j(j-\frac 12)}} \hat q_{c\setminus i, \lambda}(t,s,\x,(\y_{1:i-1},\y_{i+1:n})) {\mathcal N}_{cv+\lambda^ {\frac{2i-1}2}(s-t)^{2i-1}}\big ( z-\btheta_{s,t}(\x)_i\big).
\end{eqnarray}
Write now from \eqref{DECOUP_CAR_THERMIC_scalling}:
\begin{eqnarray*}
&&|{\mathscr T}_{\lambda,2} \big(v,t,s,\x,(\y_{1:i-1},z,\y_{i+1:n})\big)|
\nonumber \\
&\le& \Lambda\int_{\R^d} d\y_i\frac{h_{cv}(z-\y_i)}{v^{\frac 12}}\int_0^1 d\mu  \frac{ \lambda^{\sum_{j=1}^n \vartheta_j(j-\frac 12)+\frac{2i-1}{2}} \hat q_{c,\lambda} (t,s,\x,\y_{1:i-1},z+\mu (\y_i-z), \y_{i+1:n})}{(s-t)^{\sum_{j=1}^n \vartheta_j(j-\frac 12)+\frac{2i-1}{2}}}\\
&&\times |\y_i-z| \lambda^{-i+\frac 12}\Big (\Big|\gF_{i}\big(s,\lambda^{-\frac 12}\T_\lambda(\y_{1:i-1},z,\y_{i+1:n})\big)-\gF_{i}\big(s,\lambda^{-\frac 12}\T_\lambda(\y_{1:i-1},\btheta_{s,t}(\x)_{i:n})\big)\Big|\\
&&+\Big|\gF_{i}\big(s,\lambda^{-\frac 12}\T_\lambda(\y_{1:i-1},\btheta_{s,t}(\x)_{i:n})\big)-\gF_{i}\big(s,\lambda^{-\frac 12}\T_\lambda\btheta_{s,t}(\bxi)\big) \\
&&-D_{\x_{i-1}}\gF_{i}\big(s,\lambda^{-\frac 12}\T_\lambda \btheta_{s,t}(\x)\big)\big( \lambda^{-\frac 12}\T_\lambda(\y-\btheta_{s,t}(\x))\big)_{i-1} \Big|\Big)\\
&\le& \Lambda  \lambda^{\sum_{j=1}^n \vartheta_j(j-\frac 12)}  \int_{\R^d} d\y_ih_{cv}(z-\y_i)\int_0^1 d\mu  \frac{\hat q_{c,\lambda}(t,s,\x,\y_{1:i-1},z+\mu (\y_i-z), \y_{i+1:n})}{(s-t)^{\sum_{j=1}^n \vartheta_j(j-\frac 12)+\frac{2i-1}{2}}}\\
&&\times \Big(
(\lambda^{\frac {2i-1}2}|z-\btheta_{s,t}(\x)_i|)^{\frac{2i-3+\gamma}{2i-1}} +(\lambda^{\frac{2(i-1)-1}2}|(\btheta_{s,t}(\x)-\y)_{i-1}|)^{1+\frac{\gamma}{2(i-1)-1}}
\nonumber \\
&&+\sum_{k=i+1}^n (\lambda^ {\frac{2j-1}2} |(\btheta_{s,t}(\x)-\y)_k|)^{\frac{2i-3+\gamma}{2k-1}}\Big).
\end{eqnarray*}
We have for any  $\mu \in [0,1]  $,
$$|z-\btheta_{s,t}(\x)_i|\le \mu  |z-\y_i|+|z +\mu (\y_i-z)-(\btheta_{s,t}(\x))_i|\textcolor{black}{,} $$
we thus derive
\begin{eqnarray}
\!\!\! && \!\!\! |{\mathscr T}_{\lambda,2} \big(v,t,s,\x,(\y_{1:i-1},z,\y_{i+1:n})\big)|
\nonumber \\
&\le& \Lambda \lambda^{\sum_{j=1}^n \vartheta_j(j-\frac 12)
}  \int_{\R^d} d\y_ih_{cv}(z-\y_i)\int_0^1 d\mu  \frac{\hat q_{c,\lambda}(t,s,\x,\y_{1:i-1},z+\mu (\y_i-z), \y_{i+1:n})}{(s-t)^{\sum_{j=1}^n \vartheta_j(j-\frac 12)+\frac{2i-1}{2}}}\notag \\
&&\times \bigg(\lambda^ {\frac{2i-3+\gamma}2}|\y_i-z|^{\frac{2i-3+\gamma}{2i-1}}
+ \d^{2i-3+\gamma}\bigg(\lambda^ {-1/2} \T_ {\lambda} \btheta_{s,t}(\x), \lambda^ {-1/2} \T_ {\lambda} \Big (\y_{1:i-1},z +\mu (\y_i-z)
, \y_{i+1:n})\Big) \bigg) \notag \\
&\le&  \Lambda \lambda^{\sum_{j=1}^n \vartheta_j(j-\frac 12)
} \int_{\R^d} d\y_ih_{cv}(z-\y_i)\int_0^1 d\mu  \hat q_{c,\lambda}(t,s,\x,\y_{1:i-1},z+\mu (\y_i-z), \y_{i+1:n})
\nonumber \\
&&\times \Big(\frac{\lambda^ {\frac{2i-3+\gamma}2} v^{\frac{2i-3+\gamma}{2(2i-1)}}}{(s-t)^{\sum_{j=1}^n \vartheta_j(j-\frac 12)+\frac{2i-1}{2}}}    +\frac{1}{(s-t)^{\sum_{j=1}^n \vartheta_j(j-\frac 12)+1-\frac \gamma 2}}\Big)\notag\\
&\le&  \Lambda  \lambda^{\sum_{j=1}^n \vartheta_j(j-\frac 12)
}\hat q_{c\setminus i,\lambda}(t,s,\x,\y_{1:i-1}, \y_{i+1:n})
\nonumber \\
&&\times \int_0^1 d\mu \int_{\R^d} h_{cv}(z-\y_i) {\mathcal N}_{c\lambda^ {\frac{2i-1}2}(s-t)^{2i-1}}(z+\mu (\y_i-z)-(\btheta_{s,t}(\x))_i) d\y_i\notag\\
&&\times \Big(\frac{\lambda^ {\frac{2i-3+\gamma}2}v^{\frac{2i-3+\gamma}{2(2i-1)}}}{(s-t)^{\sum_{j=1}^n \vartheta_j(j-\frac 12)+\frac{2i-1}{2}}}    +\frac{1}{(s-t)^{\sum_{j=1}^n \vartheta_j(j-\frac 12)+1-\frac \gamma 2}}\Big).
\label{CTR_T2_scalling}
\end{eqnarray}
From \eqref{DECOUP_CAR_THERMIC_scalling}, \eqref{CTR_T1_scalling} and \eqref{CTR_T2_scalling} we deduce, with the notation of \eqref{DEF_HAT_SETMINUS_scalling}:
\begin{eqnarray*}
&&\|h_v\star \Psi_{i,(t,\x),(s,\y_{1:i-1},\y_{i+1})}^{\vartheta,\lambda}\|_{L^1(\R^d,\R)} 
\\
&\!\!\le\!\!&  \Lambda\lambda^{\sum_{j=1}^n \vartheta_j(j-\frac 12)} \hat q_{c\setminus i, \lambda}(t,s,\x,(\y_{1:i-1},\y_{i+1:n}))
\nonumber \\
&& \times \bigg(\frac{\lambda^ {-1+\frac{\gamma}2}}{v^{\frac{2-\gamma}{4i-2}} (s-t)^{\sum_{j=1}^n \vartheta_j(j-\frac 12)}}
+\frac{\lambda^ {\frac{2i-3+\gamma}2
}v^{\frac{2i-3+\gamma}{2(2i-1)}}}{(s-t)^{\sum_{j=1}^n \vartheta_j(j-\frac 12)+\frac{2i-1}{2}}}    +\frac{1
}{(s-t)^{\sum_{j=1}^n \vartheta_j(j-\frac 12)+1-\frac \gamma 2}} \bigg)
\\
&&\times  \int_0^1 d\mu \int_{\R^d} dz \int_{\R^d} d\y_i h_{cv}(z-\y_i) {\mathcal N}_{c\lambda^ {\frac{2i-1}2}(s-t)^{2i-1}}(z+\mu (\y_i-z)-(\btheta_{s,t}(\x))_i)\\
&\!\!\le\!\!& \Lambda   \lambda^{\sum_{j=1}^n \vartheta_j(j-\frac 12)}\hat q_{c\setminus i,\lambda}(t,s,\x,(\y_{1:i-1},\y_{i+1:n}))\\
&&\times \bigg(\frac{\lambda^{-1+\frac \gamma2}}{v^{\frac{2-\gamma}{4i-2}} (s-t)^{\sum_{j=1}^n \vartheta_j(j-\frac 12)}}
+\frac{\lambda^{\frac{2i-3+\gamma}{2}}v^{\frac{2i-3+\gamma}{2(2i-1)}}}{(s-t)^{\sum_{j=1}^n \vartheta_j(j-\frac 12)+\frac{2i-1}{2}}}    +\frac{1}{(s-t)^{\sum_{j=1}^n \vartheta_j(j-\frac 12)+1-\frac \gamma 2}} \bigg).
\end{eqnarray*}
The last identity is a again consequence of the change of variable $(w_1, w_2)=(z-\y_i,z+\mu (\y_i-z)-(\btheta_{s,t}(\x))_i) $.

We now write:
\begin{eqnarray}
&&\int_0^{[\lambda^{-1}(s-t)]^{\beta_i}} dv v^{\frac{\tilde \alpha_i}2-1}\|h_v\star \Psi_{i,(t,\x),(s,\y_{1:i-1},\y_{i+1})}^{\vartheta,\lambda}\|_{L^1(\R^d,\R)}\notag\\
&\le& \Lambda \lambda^{\sum_{j=1}^n \vartheta_j(j-\frac 12)} \hat q_{c\setminus i}(t,s,\x,(\y_{1:i-1},\y_{i+1:n}))\int_0^{[\lambda^{-1}(s-t)]^{\beta_i}} \frac{dv}{v}v^{\frac{\tilde \alpha_i}{2}}
\notag\\
&&\!\!\!\times \bigg(\frac{\lambda^{-1+\frac \gamma 2} }{v^{\frac{2-\gamma}{4i-2}} (s-t)^{\sum_{j=1}^n \vartheta_j(j-\frac 12)}}
+\frac{\lambda^{\frac{2i-3+\gamma}2}v^{\frac{2i-3+\gamma}{2(2i-1)}}}{(s-t)^{\sum_{j=1}^n \vartheta_j(j-\frac 12)+\frac{2i-1}{2}}}    +\frac{1}{(s-t)^{\sum_{j=1}^n \vartheta_j(j-\frac 12)+1-\frac \gamma 2}} \bigg) \label{THE_BESOV_LAMBDA}\nonumber \\
&=:&\Lambda \lambda^{\sum_{j=1}^n \vartheta_j(j-\frac 12)} \hat q_{c\setminus i,\lambda}(t,s,\x,(\y_{1:i-1},\y_{i+1:n}))  {\mathcal B}_{\vartheta,\beta_i}^\lambda(t,s).
\end{eqnarray}
Let us now prove that
\begin{equation}
\label{CTR_EXP_TEMPS_PETIT_LAMBDA}
{\mathcal B}_{\vartheta,\beta_i}^\lambda(t,s) \le \frac{C\lambda^{-1}}{(s-t)^{\sum_{j=1}^n \vartheta_j(j-\frac 12)-\frac \gamma 2}}.
\end{equation}
Integrating in $v $ in \eqref{THE_BESOV_LAMBDA} we derive:
\begin{eqnarray*}
{\mathcal B}_{\vartheta,\beta_i}^\lambda(t,s) &\le&  C(s-t)^{-\sum_{j=1}^n \vartheta_j(j-\frac 12)} \Bigg[\lambda^{-1+\frac \gamma 2}[\lambda^{-1}(s-t)]^{\beta_i(\frac{\tilde \alpha_i}{2}-\frac{2-\gamma}{4i-2})}\\
&& +\lambda^{\frac{2i-3+\gamma}2}[\lambda^{-1}(s-t)]^{\beta_i(\frac{\tilde \alpha_i}2+\frac{2i-3+\gamma}{2(2i-1)}) )}(s-t)^{i-\frac 12}  
   +[\lambda^{-1}(s-t)]^{ \beta_i\frac{ \tilde \alpha_i} 2 }(s-t)^{-1+\frac \gamma 2}\Bigg].
\end{eqnarray*}
Recall now from the proof of Lemma \ref{LEMME_BESOV_DEG} that:
\begin{eqnarray*}
\beta_i\bigg(\frac{\tilde \alpha_i}{2}-\frac{2-\gamma}{4i-2}\bigg)-\frac{\gamma}2\ge 0,\ \beta_i\bigg(\frac{\tilde \alpha_i}{2}+\frac{2i-3+\gamma}{2(2i-1)}\bigg) -\frac{2i-1}{2}-\frac \gamma 2\ge 0,\ \beta_i \frac{\tilde \alpha_i}2-1\ge 0,
\end{eqnarray*}
with $\beta_i=\frac{(2i-3)(2i-1)}{2i-3-\gamma}, \tilde \alpha_i=\frac{2+\gamma}{2i-1}$. Therefore, since $(s-t)/\lambda\le 1$:
\begin{equation*}
{\mathcal B}_{\vartheta,\beta_i}^\lambda(t,s) \le  C(s-t)^{-\sum_{j=1}^n \vartheta_j(j-\frac 12)} \Bigg[\lambda^{-1+\frac \gamma 2}[\lambda^{-1}(s-t)]^{\frac \gamma 2}
 +\lambda^{-1+\frac{\gamma}2}[\lambda^{-1}(s-t)]^{\frac \gamma 2}  
   +[\lambda^{-1}(s-t)](s-t)^{-1+\frac \gamma 2}\Bigg],
\end{equation*}
which precisely gives \eqref{CTR_EXP_TEMPS_PETIT_LAMBDA}.

Plugging \eqref{CTR_EXP_TEMPS_PETIT_LAMBDA} into \eqref{THE_BESOV_LAMBDA} and from \eqref{CTR_BESOV_LAMBDA_HD} we eventually get:
\begin{equation*}
\label{CTR_IN_TIME_BESOV_NORM_scaling}
\int_0^1 \frac{dv}{v}v^{\frac{\tilde \alpha_i}2 }\|h_v\star \Psi_{i,(t,\x),(s,\y_{1:i-1},\y_{i+1:n})}^{\vartheta,\lambda}\|_{L^1(\R^d,\R)}
\le \frac{\Lambda \lambda^{\sum_{j=1}^n \vartheta_j(j-\frac 12)-1}}{(s-t)^{\sum_{j=1}^n \vartheta_j(j-\frac 12)-\frac \gamma 2}}\hat q_{c\setminus i,\lambda}(t,s,\x,(\y_{1:i-1},\y_{i+1:n})),\notag\\
\end{equation*}
which is precisely the stated control.
The term $\|\varphi (D) \Psi_{i,(t,\x),(s,\y_{1:i-1},\y_{i+1:n})}^{\vartheta, \lambda}\|_{L^1(\R^d,\R)}$ appearing in the Besov norm could be handled similarly. The result is complete.
\end{proof}

\vspace*{1cm}

 \textbf{Paul-\'Eric Chaudru de Raynal}. \emph{Universit\'e de Savoie Mont Blanc, CNRS, LAMA, F-73000 Chamb\'ery, France. \texttt{E-mail}: pe.deraynal@univ-smb.fr}.\\
 
\textbf{Igor Honor\'e}. \emph{Laboratoire de Mod\'elisation Math\'ematique d'Evry (LaMME), Universit\'e d'Evry Val d'Essonne, 23 Boulevard de France 91037 Evry. \texttt{E-mail}: igor.honore@univ-evry.fr}.\\

\textbf{St\'ephane Menozzi}. \emph{Laboratoire de Mod\'elisation Math\'ematique d'Evry (LaMME, UMR CNRS 8071), Universit\'e d'Evry Val d'Essonne, Paris Saclay, 23 Boulevard de France 91037 Evry, France and Laboratory of Stochastic Analysis, National Research University Higher School of Economics, Pokrovsky Boulevard, 11, Moscow, Russian Federation. \texttt{E-mail}: stephane.menozzi@univ-evry.fr}

\bibliographystyle{plain}
\bibliography{bibli}

\end{document}